\documentclass[11pt]{article}

\usepackage[T1]{fontenc}
\usepackage[utf8]{inputenc}
\usepackage[normalem]{ulem}

\usepackage{mathtools}%
\usepackage{amsthm}
\usepackage{amsxtra}%
\usepackage{amsfonts}%
\usepackage{amssymb}%
\usepackage[margin=1.25in]{geometry}

\usepackage{cite}

\usepackage{color}
\usepackage{graphicx}
\usepackage{subfig}
\usepackage{array}
\usepackage{booktabs}
\usepackage{enumitem}
\usepackage{dsfont}
\usepackage{authblk}
\usepackage{todonotes}

\usepackage{hyperref}
\hypersetup{colorlinks,breaklinks,
             linkcolor=blue,urlcolor=blue,
             anchorcolor=blue,citecolor=blue}
\usepackage[nameinlink, capitalise, noabbrev]{cleveref}
\crefname{assumption}{Assumption}{Assumptions}

\newcommand{\defeq}{=}

\DeclareMathOperator*{\argmin}{argmin}

\newcommand{\one}{\mathds{1}}

\DeclareMathOperator{\Lip}{Lip}
\renewcommand{\O}{{\mathcal O}}
\renewcommand{\bar}[1]{{\overline{#1}}}

\newcommand{\E}{\mathcal{E}}
\newcommand{\bE}{\mathbb{E}}
\renewcommand{\P}{{\mathbb P}}

\newcommand{\X}{{\mathcal X}}
\renewcommand{\H}{{\mathcal H}}

\newcommand{\R}{\mathbb{R}}
\newcommand{\N}{\mathbb{N}}

\DeclareMathOperator{\supp}{supp}
\newcommand{\eps}{\varepsilon}
\renewcommand{\tilde}[1]{\widetilde{#1}}
\renewcommand{\phi}{\varphi}

\renewcommand{\epsilon}{\varepsilon}
\newcommand{\interior}[1]{%
  {\kern0pt#1}^{\mathrm{o}}%
}

\renewcommand{\l}[1]{{\ell^{#1}(\mathcal{X}_{n})}}
\newcommand{\lo}[1]{{\ell^{#1}_0(\mathcal{X}_{n})}}
\newcommand{\lx}[1]{{\ell^{#1}(\mathcal{X}^2_{n})}}
\newcommand{\sprod}[2]{\langle #1\rangle_{#2}}
\newcommand{\ipg}[1]{\langle #1\rangle_{\l2}}
\newcommand{\ips}[1]{\langle #1\rangle_{H^1(\X_n)}}
\newcommand{\ipv}[1]{\langle #1\rangle_{\lx2}}
\newcommand{\ip}[1]{\langle #1\rangle_{\ell^2(\X_n)}}

\renewcommand{\ng}[2]{\left\| #2\right\|_{\l{#1}}}
\newcommand{\nv}[1]{\left\| #1\right\|_{\lx2}}

\newcommand{\ene}{\E_{n,\varepsilon}}
\newcommand{\Ene}[1]{\E_{n,\varepsilon}^{(#1)}}
\newcommand{\ted}{\theta_{\varepsilon,\delta}}
\newcommand{\ied}{I_{\varepsilon,\delta}}

\newcommand{\ie}{I_{\varepsilon}}
\newcommand{\Ied}[1]{I_{\varepsilon,\delta}^{(#1)}}

\newcommand{\Ie}[1]{I_{\varepsilon}^{(#1)}}
\newcommand{\I}[1]{I^{(#1)}}
\let\deg\relax
\DeclareMathOperator{\deg}{deg}
\DeclareMathOperator{\sign}{sign}
\DeclareMathOperator{\diam}{diam}

\renewcommand{\hat}[1]{\widehat{#1}}
\renewcommand{\L}{\mathcal{L}}
\newcommand{\Lr}{\mathcal{L}_{rw}}
\renewcommand{\O}{{\mathcal O}}

\newcommand{\set}[1]{\left\{#1\right\}}
\DeclareMathOperator{\dist}{dist}
\DeclareMathOperator{\osc}{osc}
\DeclareMathOperator{\loc}{loc}
\newcommand{\cp}[1]{\tau(#1)}
\newcommand{\sa}{A}
\newcommand{\cdk}{\Theta_{d,k}}
\newcommand{\cdj}{\Theta_{d,j}}

\def\XXint#1#2#3{{\setbox0=\hbox{$#1{#2#3}{\int}$ }
\vcenter{\hbox{$#2#3$ }}\kern-.6\wd0}}

\newtheorem{theorem}{Theorem}
\newtheorem{lemma}[theorem]{Lemma}
\newtheorem{corollary}[theorem]{Corollary}

\newtheorem{proposition}[theorem]{Proposition}
\theoremstyle{definition}
\newtheorem{remark}[theorem]{Remark}
\newtheorem{definition}[theorem]{Definition}
\newtheorem{assumption}[theorem]{Assumption}

\numberwithin{equation}{section}
\numberwithin{theorem}{section}

\newcommand{\norm}[1]{\left\|#1\right\|}

\newcommand{\abs}[1]{\left\lvert #1 \right\rvert}
\DeclareMathOperator{\divtemp}{div}
\renewcommand{\div}{\divtemp}
\renewcommand{\d}{\,\mathrm{d}}
\newcommand{\dx}{\d x}
\newcommand{\dy}{\d y}

\newcommand{\M}{\mathcal{M}}
\newcommand{\st}{\,:\,}

\newcommand{\Ri}{r}

\newlist{assenum}{enumerate}{1} % should only occur inside assumption env.
\setlist[assenum]{label=(\alph*),ref=\theassumption\,(\alph*)}
\crefname{assenumi}{Assumption}{Assumptions}

\newlist{thmenum}{enumerate}{1} % should only occur inside theorem env.
\setlist[thmenum]{label=(\roman*),ref=\thetheorem\,(\roman*)}
\crefname{thmenumi}{Theorem}{Theorem}

\renewcommand{\eqref}{\labelcref}

\graphicspath{{./figs/}}

\title{Convergence rates for Poisson learning to a Poisson equation with measure data\thanks{{\bf Funding:} Calder was supported by NSF grants DMS:1944925 and MoDL+ CCF:2212318, the Alfred P. Sloan foundation, the McKnight foundation, and an Albert and Dorothy Marden Professorship. Mihailescu was supported by DFG SFB 1060. Houssou was supported by an internal University of Minnesota CSE InterS\&Ections Seed Grant. }}

\newcommand*\samethanks[1][\value{footnote}]{\footnotemark[#1]}
\author{Leon Bungert\thanks{Institute of Mathematics, Center for Artificial Intelligence and Data Science (CAIDAS), University of Würzburg. \texttt{leon.bungert@uni-wuerzburg.de}},
Jeff Calder\thanks{School of Mathematics, University of Minnesota. \{\tt jwcalder, houss001, yuanx290\}@umn.edu}, 
Max Mihailescu\thanks{Institute for Applied Mathematics \& Hausdorff Center for Mathematics,  University of Bonn. \texttt{mihailescu@iam.uni-bonn.de}}, 
Kodjo Houssou\samethanks[2], 
Amber Yuan\samethanks[2]}

\begin{document}

\maketitle

\begin{abstract}
In this paper we prove discrete to continuum convergence rates for Poisson Learning, a graph-based semi-supervised learning algorithm that is based on solving the graph Poisson equation with a source term consisting of a linear combination of Dirac deltas located at labeled points and carrying label information. The corresponding continuum equation is a Poisson equation with measure data in a Euclidean domain $\Omega \subset \R^d$. The singular nature of these equations is challenging and requires an approach with several distinct parts: (1) We prove quantitative error estimates when convolving the measure data of a Poisson equation with (approximately) radial function supported on balls. (2) We use quantitative variational techniques to prove discrete to continuum convergence rates on random geometric graphs with bandwidth $\eps>0$ for bounded source terms. (3) We show how to regularize the graph Poisson equation via mollification with the graph heat kernel, and we study fine asymptotics of the heat kernel on random geometric graphs.  
Combining these three pillars we obtain $L^1$ convergence rates that scale, up to logarithmic factors, like $\O(\eps^{\frac{1}{d+2}})$ for general data distributions, and $\O(\eps^{\frac{2-\sigma}{d+4}})$ for uniformly distributed data, for all $\sigma>0$.
These rates are valid with high probability if $\eps\gg\left({\log n}/{n}\right)^q$ where $n$ denotes the number of vertices of the graph and $q \approx \frac{1}{3d}$.
\end{abstract}

\tableofcontents

%\listoftodos
%\clearpage

\section{Introduction}
\subsection{Motivation}

Machine learning methods for fully supervised learning (e.g., image classification via convolutional neural networks) and generative tasks (e.g., large language models powered by transformers) have experienced tremendous success in recent years, due the availability of massive data sets and computational resources \cite{goodfellow2016deep}. However, for many real world problems (e.g., medical image classification), large training sets are not available or would be costly to create. Thus, there has been significant interest recently in machine learning methods that can learn from limited amounts of labeled training data, such as transfer learning \cite{zhuang2020comprehensive}, few-shot learning \cite{wang2020generalizing}, and semi-supervised learning \cite{van2020survey}.

Semi-supervised learning algorithms learn from both labeled and \emph{unlabeled} data, the latter typically being widely available for many tasks (e.g., large databases of natural images). In order to utilize unlabeled data, it is common to construct a \emph{similarity graph} over a data set, which gives a convenient representation for high dimensional data, while in other problems, such as in network science, the data has an intrinsic graph structure (e.g., links between papers in a citation data set). This leads to a field called \emph{graph-based semi-supervised learning}, which utilizes graph structures to train classifiers with fewer labeled examples than are required with fully supervised learning. See \cite{calder2024consistency} for a survey of graph Laplacian based learning algorithms, and \cite{song2022graph} for graph-neural network approaches.  

The field of graph-based learning has recently seen an infusion of theoretically well-founded machine learning problems by identifying graph-based learning algorithms with partial differential equation (PDE) or variational problems in the continuum limit. Here, we specifically discuss semi-supervised learning problems, where one is given data points $\X_n=\{x_1,\dots,x_n\}$ in $\R^d$ and a subset of labeled data points $\Gamma_n \subset \X_n$ together with labels $g : \Gamma_n \to\R$. The task is to extend these labels to the whole data set $\X_n$ in a reasonable way.  A simplistic approach for this problem would be nearest neighbor classification, i.e., a point in $\X_n$ gets the label that the closest point in $\Gamma_n$ is carrying.  Notably, this method completely neglects the presence of the remaining (unlabeled) data points and typically leads to inferior results.  In contrast, graph-based semi-supervised learning builds on the manifold hypothesis which assumes that the data points in $\X_n$ are samples from a probability distribution supported on a manifold or domain in $\R^d$.  In order to extract this inherent geometry from the data, geometric graphs have proven to be a useful tool. For this, the data is converted into a weighted graph $G_n=(\X_n,w)$, where $w:\X_n\times\X_n\to[0,\infty)$ is a (symmetric) weight function that assigns high weight to similar data points and low weight to dissimilar ones. 

\begin{figure}[!t]
\centering
\subfloat[Laplace Learning]{\includegraphics[width=0.32\textwidth]{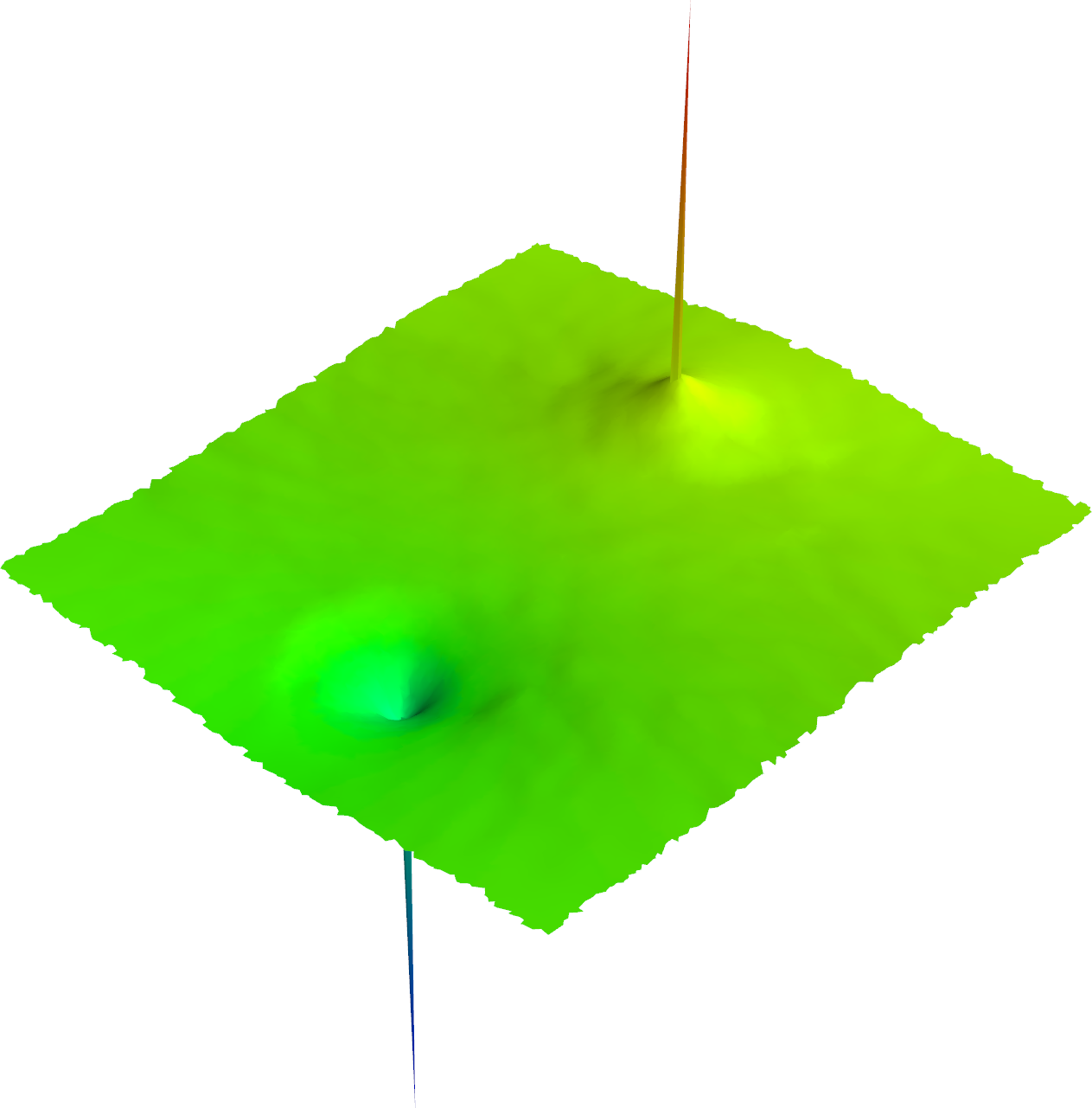}\label{fig:laplace}}
\subfloat[Poisson Learning]{\includegraphics[width=0.32\textwidth]{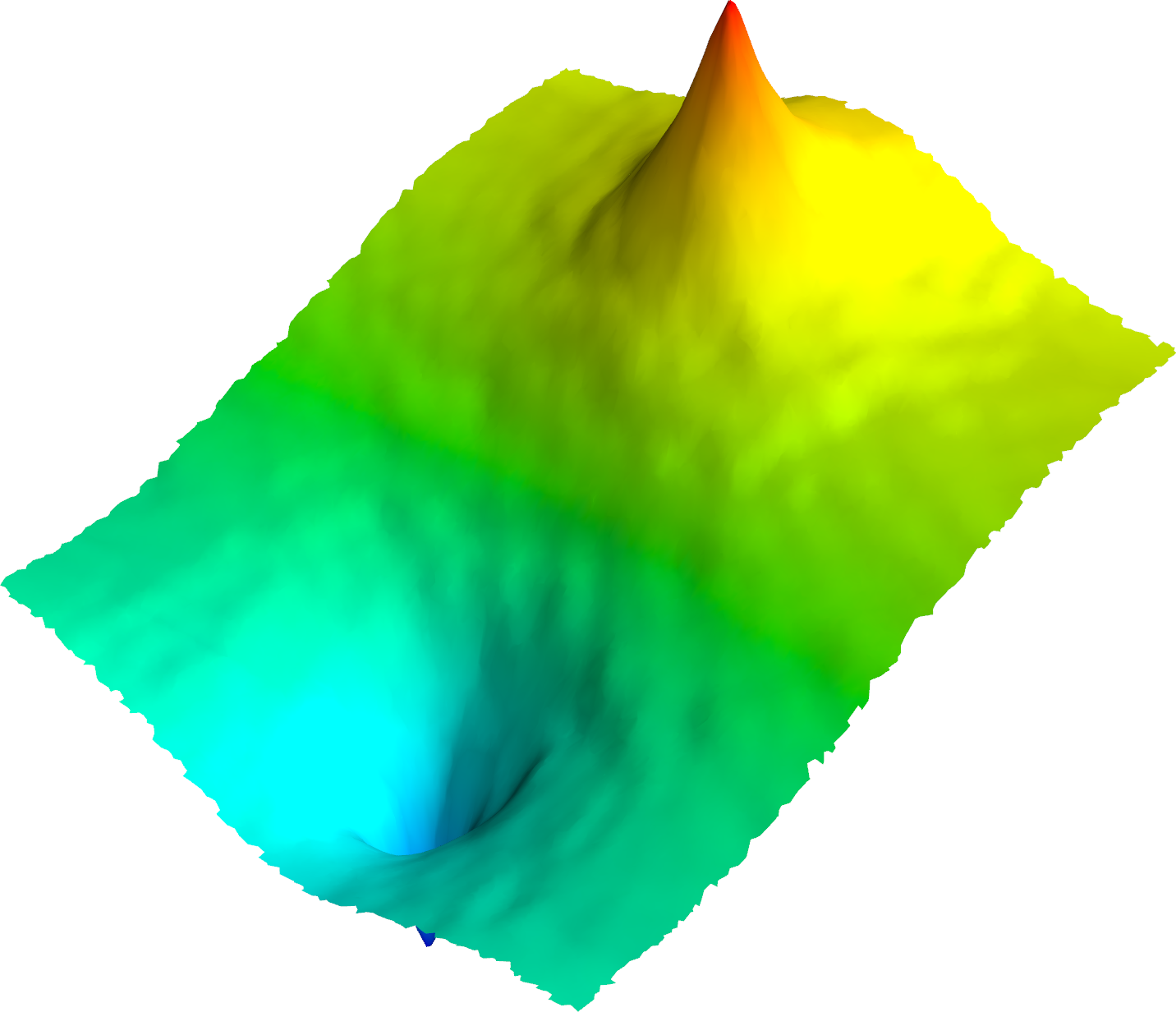}\label{fig:poisson}}
\subfloat[PWLL]{\includegraphics[width=0.32\textwidth]{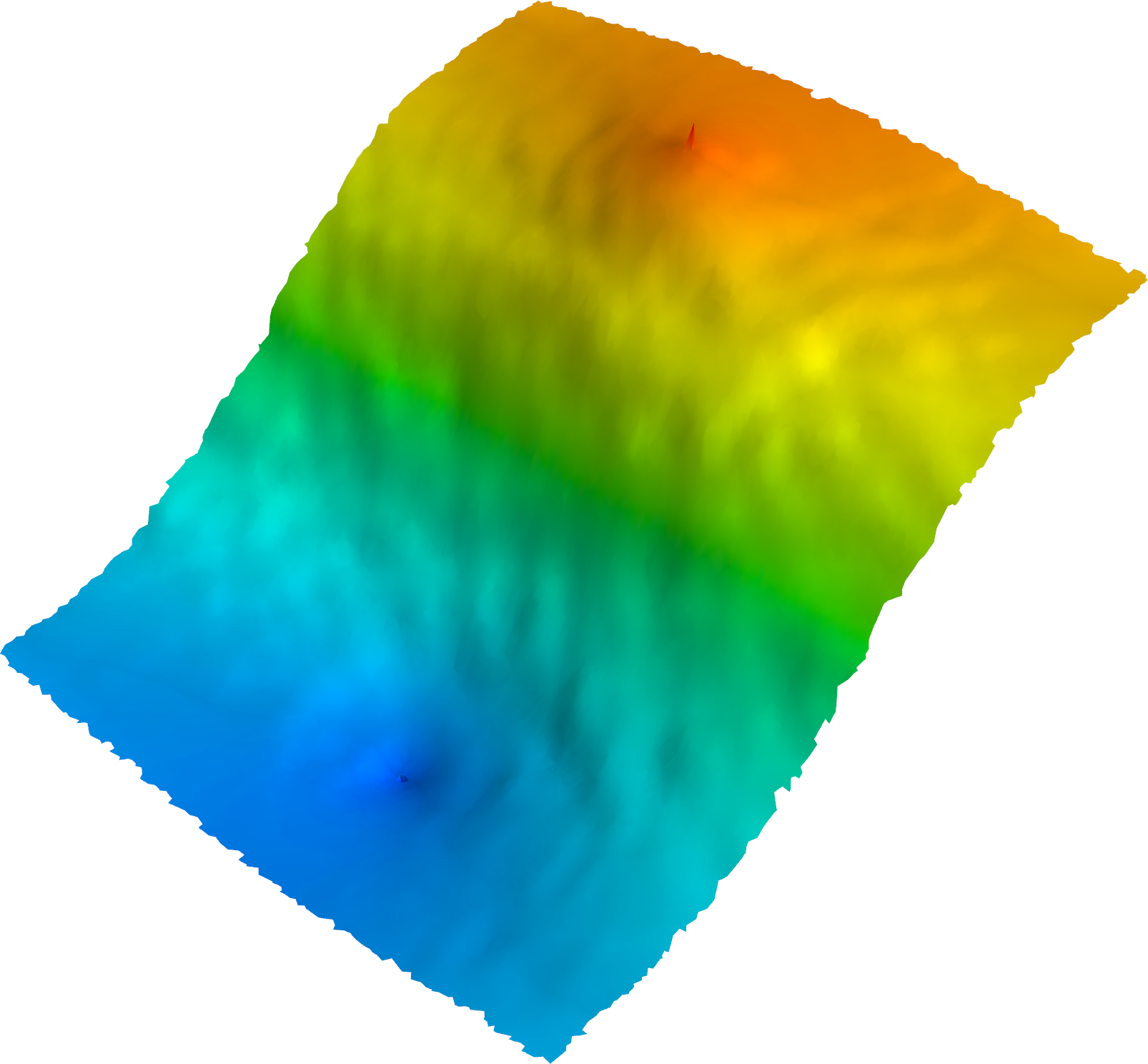}\label{fig:pwll}}
\caption{Comparisons of (a) Laplace Learning (b) Poisson Learning and (c) Poisson Weighted Laplace Learning (PWLL) for a problem with two points in opposite classes, with labels of $+1$ and $-1$. Laplace learning develops spikes, while Poisson Learning approximates the fundamental solution of Laplace's equation, and PWLL smoothly interpolates the labels. }
\label{fig:laplace_demo}
\end{figure}

One of the earliest and most popular graph-based methods for semi-supervised learning is based on solving the graph Laplace equation with ``boundary conditions'' on the labeled data points  \cite{zhu2003semi}.
This amounts to solving the following linear system of equations
\begin{equation}\label{eq:Laplace_Learning}
\left\{
\begin{aligned}
\sum_{y\in\X_n} w_{xy}(u(x)-u(y)) &= 0,&& \text{if } x\in \X_n\setminus \Gamma_n \\
u(x) &= g(x),&& \text{if } x\in \Gamma_n
\end{aligned}
\right.
\end{equation}
for the vector $(u(x_1),u(x_2),\dots,u(x_n))\in\R^n$ which describes the final labeling of the whole data set, and we interpret as a function $u:\X_n\to \R$.\footnote{This is the setting of binary classification, but the method can be easily extended to the multi-class setting by taking $u(x)\in \R^k$ if we have $k$ classes.}    Note that the resulting labeling can be equivalently characterized through having the mean value property 
\begin{align}\label{eq:mean_value_property}
    u(x) = \frac{\sum_{y\in\X_n}w_{xy}u(y)}{\sum_{y\in\X_n}w_{xy}}    
\end{align}
for all $x\in\X_n\setminus\Gamma_n$. While this approach, later termed \emph{Laplace learning}, gives satisfactory results if the set of labeled points $\Gamma_n$ is sufficiently large \cite{calder2023rates}, Laplace learning dramatically fails for only few labels, which was first pointed out in \cite{nadler2009semi}. In the latter case, as we can see in \cref{fig:laplace} the solution forms spikes at the labeled data points and is close to being constant otherwise. 
An intuitive explanation for this is offered by the mean value property \labelcref{eq:mean_value_property}: If the number of labeled points in the sum is very small compared to the total number of summands, a function which is constant everywhere outside the labeled points will approximately satisfy \labelcref{eq:mean_value_property}.

For a rigorous analysis one can study the \emph{continuum limit} of \labelcref{eq:Laplace_Learning} as the number of data points in $\X_n$ grows to infinity; an analysis of this type was carried out in \cite{calder2023rates}.
If $\X_n$ is an independent and identically distributed (\textit{i.i.d.}) sample from some probability distribution $\rho$ on a Euclidean domain $\Omega \subset\R^d$, if the weights $w_{xy}$ are of the form $w_{xy}=\eta(\abs{x-y}/\eps)$ for some 
non-increasing and non-negative function $\eta$ and some scaling parameter $\eps>0$, and if $\eps$ goes to zero sufficiently slow depending on $n$, then solutions of \labelcref{eq:Laplace_Learning} converge to solutions of the weighted Laplace equation
\begin{equation}\label{eq:weigted_elliptic}
\div(\rho^2\nabla u) = 0.
\end{equation}
However, the constraint that $u=g$ on the set of labeled points $\Gamma_n$ only carries through to the limiting partial differential equation if the sets $\Gamma_n$ approximate a set $\Gamma\subset\overline\Omega$ as $n\to\infty$, where $\Gamma$ has positive capacity, see \cite{calder2023rates} for a few cases. Otherwise the constraints are ignored and the solutions of \labelcref{eq:Laplace_Learning} converge to the trivial constant solution of \labelcref{eq:weigted_elliptic} \cite{calder2023rates}.

As a consequence, many different streams of work suggested alternatives for \labelcref{eq:Laplace_Learning}, most of which are based on the idea of enforcing higher regularity of solutions in the continuum limit. 
For instance, replacing the Laplacian by the \emph{variational} $p$-Laplacian \cite{el2016asymptotic}, it was proved in \cite{slepcev2019analysis} that the label constraints are preserved if $p>d$ (essentially because $W^{1,p}$-functions are H\"older continuous in this regime). Similar results were obtained for the game-theoretic $p$-Laplacian in \cite{calder2018game}, where it was also shown that graph $p$-harmonic functions are approximately H\"older continuous when $p>d$.  In the limit case as $p\to\infty$, one obtains the infinity Laplace operator and the corresponding problem is called \emph{Lipschitz Learning} since solutions are globally Lipschitz continuous. The method was introduced in \cite{kyng2015algorithms}, qualitative discrete to continuum limits were proved in \cite{calder2019consistency,roith2023continuum} and convergence rates were recently established in \cite{bungert2023uniform,bungert2024ratio}. Despite strong theoretical results and the fact that Lipschitz Learning has a well-posed continuum limit in any dimension $d$, it suffers from the drawback that---while it captures the \emph{geometry} of the underlying space very well---it does not at all capture the distribution of the data points (though see \cite{calder2019consistency} for reweighting techniques that can partially address this).  Other approaches enforce sufficient regularity by using higher-order differential operators like powers of graph Laplacian \cite{zhou2011semi}, the poly-Laplacian \cite{GTMT2022ratespoly}, or eikonal-type equations \cite{dunbar2023models,calder2022hamilton} but have similar drawbacks.

In contrast, the \emph{Poisson learning} algorithm, proposed in \cite{calder2020poisson}, builds on the simple but powerful idea of replacing ``boundary value problems'' for certain differential operators on graphs with  Poisson equations.  To achieve this, the information about the labels is transferred from a pointwise constraint of the form $u(x)=g(x)$ for $x\in\Gamma_n$ to the source term of a graph Poisson equation of the form
\begin{align}\label{eq:Poisson_Learning}
\sum_{y\in\X_n} w_{xy}(u(x)-u(y)) = \sum_{z\in\Gamma_n} (g(z) - \overline g)\delta_{z}(x),\qquad& \text{for all }  x\in\X_n,
\end{align}
subject to a constraint on the mean value of $u$ to ensure uniqueness.
Here we let $\delta_z:\X_n\to\R$ be defined as $\delta_z(z)=n$ and $\delta_z(x)=0$ for $x\in\X_n\setminus\{z\}$.
Centering by the constant $\overline g \defeq   \frac{1}{\#\Gamma_n}\sum_{z\in\Gamma_n}g(z)$ ensures that the source term sums to zero, which is the necessary compatibility condition for the Poisson equation.
Poisson learning was shown in \cite{calder2020poisson} to significantly outperform other semi-supervised learning methods, in particular, at low labeling rates.

The authors of \cite{calder2020poisson} partially attribute the success of Poisson learning to the fact that it possesses a well-posed continuum limit without any assumptions on the labeled set $\Gamma_n$.
The continuum limit was conjectured to be the Poisson equation
\begin{align}\label{eq:Poisson_Learning_CL}
-\div\left(\rho(x)^2\nabla u\right) = \sum_{z\in\Gamma} (g(z) - \overline g)\delta_{z}, \ \ \text{in} \ \  \Omega,
\end{align}
%where $\rho$ is the probability distribution on a bounded domain $\Omega\subset\R^d$ from which the points in $\X_n$ are sampled, 
where $\Gamma$ is a set of continuum labels, which could even coincide with $\Gamma_n$ for every $n\in\N$, and $\delta_z$ is the Dirac measure concentrated at $z\in\Gamma$. Note we commit a slight abuse of notation by using the same symbol $\delta_z$ for both the Dirac measure and its graph approximation; the intended choice will be clear by context.
%, and as before $\overline g \defeq   \frac{1}{\#\Gamma}\sum_{z\in\Gamma}g(z)$. 
\cref{eq:Poisson_Learning_CL} has measure-valued data and its solutions are therefore to be understood in the distributional sense.
Furthermore, the equation is complemented with homogeneous Neumann boundary conditions on $\partial\Omega$ and a constraint on the mean value of $u$ to ensure uniqueness.  
\cref{fig:poisson} shows how Poisson learning resolves the spike problem in a simple toy example. 

Another related approach to the low-label rate problem is to reweight the graph more heavily near labeled data points \cite{shi2017weighted,calder2020properly}. That is, we replace the graph weights $w_{xy}$ in \labelcref{eq:Laplace_Learning} with $\tilde w_{xy} = \gamma(x)\gamma(y) w_{xy}$, where $\gamma(x)>0$ and increases rapidly in a neighborhood of the labeled set $\Gamma_n$,  so as to penalize large gradients (i.e., spikes). A localized reweighting idea that considered only edges adjacent to labeled nodes was originally proposed in \cite{shi2017weighted}, while in \cite{calder2020properly} the authors identified that a singular non-local weighting with $\gamma(x)^2 \sim \dist(x,\Gamma)^{-\alpha}$, where $\alpha>d-2$, was required to ensure well-posedness in the continuum limit, and proposed the \emph{properly weighted} graph Laplacian based on this scaling. In a forthcoming paper \cite{calder2024poisson}, as well as in \cite{miller2023active}, a method called \emph{Poisson Reweighted Laplace Learning} (PWLL) was proposed that selects the reweighting function $\gamma$ in the properly weighted graph Laplacian by solving a graph Poisson equation of the form
\[\sum_{y\in\X_n} w_{xy}(\gamma(x)-\gamma(y)) = \sum_{z\in\Gamma_n} \left(\delta_{z}(x) - \tfrac{1}{n}\right) \ \  \text{for all }  x\in\X_n.\]
As before, the idea is that the reweighting function $\gamma$ should converge to the solution of a continuum Poisson equation like \labelcref{eq:Poisson_Learning_CL} with measure-valued data, and will hence have the correct rate of blow-up at the labeled set to be utilized in the properly weighted Laplacian.  \cref{fig:pwll} shows how the PWLL method interpolates between two labeled data points. 

The goal of this paper is to prove that Poisson learning is a well-posed and stable method for propagating labels on graphs at arbitrarily low label rates.  We do this by establishing that Poisson learning has a well-posed continuum limit, given by a Poisson equation with measure-valued data, in the setting where the number of unlabeled data points tends to infinity while the number of labeled data points is fixed and finite.  To the best of our knowledge, there are no results in the literature that rigorously prove convergence of solutions to graph Poisson equations to solutions of the respective continuum Poisson equation for measure data. 
The main difficulties with tackling this problem are twofold:
First, the limit equation \labelcref{eq:Poisson_Learning_CL} does not admit a variational interpretation in the sense that its solutions are not characterized as minimizers to a variational problem. 
This is in stark contrast to the case of the Poisson equation $-\Delta u = f$ with more regular data, e.g., $f\in L^2(\Omega)$.
Here, solutions are minimizers of the convex energy $u\mapsto\frac12\int_\Omega\abs{\nabla u}^2\dx - \int_\Omega f u \d x$ over $H^1(\Omega)$.
Second, solutions of \labelcref{eq:Poisson_Learning_CL} are not regular.
This can be seen from the case $\rho\equiv const$ where solutions to \labelcref{eq:Poisson_Learning_CL} are linear combinations of fundamental solutions of the Laplace equation of the form $u_z(x)=\abs{x-z}^{2-d}$ for $d\geq 3$ and smooth correctors.
In particular, the maximal regularity of solutions to \labelcref{eq:Poisson_Learning_CL} is $W^{1,p}(\Omega)$ for $1\leq p < \frac{d}{d-1}$ and solutions do not have a continuous representative.
These difficulties render standard approaches for proving discrete to continuum convergence of graph PDEs inapplicable, e.g., those based on qualitative variational tool like Gamma-convergence \cite{garcia2016continuum,slepcev2019analysis,roith2023continuum}, or on quantitative consistency of the graph operators with the limiting differential operator for sufficiently regular functions \cite{calder2018game,calder2019consistency,calder2023rates,yuan2022continuum}.

In the present work we leverage a combination of variational and PDE tools to prove convergence with quantitative high probability rates of solutions to Poisson learning \labelcref{eq:Poisson_Learning} to its continuum limit \labelcref{eq:Poisson_Learning_CL}.
We adopt the following \textbf{proof strategy}: 
\begin{enumerate}
    \item We consider continuum Poisson equations with measure data of the form \labelcref{eq:Poisson_Learning_CL} and prove error estimates for replacing the measure data by a convolution with functions that are compactly supported on a ball of radius $\Ri>0$. 
    Using Green's functions representations, we show that weak solutions of the corresponding Poisson equations converge to the distributional solution of \labelcref{eq:Poisson_Learning_CL} at a rate of approximately $O(\Ri^2)$ in the $L^1$-norm and, as a consequence, also in the graph $\ell^1$-norm with high probability. 
    This is the content of \cref{sec:smoothed_poisson}.
    \item We prove discrete to continuum rates of convergence for Poisson equations with bounded source terms, which have a variational interpretation as minimizers of an energy functional, as explained above. 
    For this we build on quantitative variational techniques that go back to \cite{burago2015graph} and were further developed in \cite{garcia2020error,trillos2018variational,calder2022improved,CalculusofVariationsLN,garcia2022graph}.
    See also \cite{cueto2023variational} for a similar approach for fractional problems.
    The main idea is use strong convexity of the energy functional or some other sort of quantitative stability around minimizers, to prove rates of convergence based on consistency of the energy functionals instead of the associated differential operators. 
    This requires modifying the solution of the discrete problem to be feasible for the continuum one, and vice versa.
    While the latter is typically easy to achieve, the former requires the usage of tailored mollification procedures which grant a precise control of the continuum energy.
    This is the content of \cref{sec:continuum_limits}.
    \item We perform an analysis, similar to \cref{sec:smoothed_poisson}, albeit for Poisson equations on graphs, where we replace the right hand side in \labelcref{eq:Poisson_Learning} by its mollification through $k\in\N$ steps of the heat equation on a random geometric graph with bandwidth $\eps>0$. 
    Analogously, we obtain a convergence rate of approximately $O(\eps_k^2)$ in the $\ell^1$-norm on the graph, where $\eps_k \defeq   \eps\sqrt{k}$ is the effective support radius of the graph heat kernel.
    For this we derive asymptotics of the graph heat kernel in terms of a nonlocal averaging operator, and an easier-to-deal-with $k$-fold convolution operator.
    This is the content of \cref{sec:heat_kernel}
    \item  Finally, in \cref{sec:combination} we combine all results to derive discrete to continuum convergence rates of the Poisson learning problem \labelcref{eq:Poisson_Learning} to \labelcref{eq:Poisson_Learning_CL}, by passing through regularized equations on the graph and the continuum and using the results of \cref{sec:smoothed_poisson,sec:continuum_limits,sec:heat_kernel}. 
    We also go on to prove discrete to continuum results for Poisson equations with source terms that are signed Radon measures, utilizing our main results for atomic measures and $L^1$ stability results for Poisson equations. 
\end{enumerate}

\subsection{Setting and main results}\label{sec:main_results}

Our results hold in the setting of a random geometric graph. Let $x_1,x_2,\dots,x_n$ be an \emph{i.i.d.} sequence of random variables on a bounded domain $\Omega \subset \R^{d}$, distributed according to a probability density $\rho$.  The set of points $\mathcal{X}_{n}\defeq   \left\{ x_1,x_2,\dots,x_n\right\}$ form the vertices of the graph. To endow the points with a graph structure, we let $\eta:[0,\infty)\to[0,\infty)$,  $\eta_\eps(t)\defeq  \eps^{-d}\eta(t/\eps)$ and define edge weights of the form 
\[w_{xy}\defeq  \eta_\eps(\abs{x-y}).\]
The vertices $\X_n$ equipped with edge weights $w_{xy}$ between all pairs of vertices $x,y\in \X_n$ form a \emph{random geometric graph} with bandwidth $\eps$, which controls the distance at which we connect points in the graph. 

We place the following assumptions on $\eta$. 
\begin{assumption}\label{ass:eta}
    The function $\eta:[0,\infty)\to[0,\infty)$ satisfies the following:
    \begin{enumerate}
    \item $\eta$ is continuous at $0$ and $\eta(0)>0$.
    \item $\eta$ is non-increasing and $\supp \eta \subset [0,1]$.
    \item $\eta$ has unit mass $\int_{B(0,1)}\eta(\abs{z})\d z=1$.
    \end{enumerate}
\end{assumption}

Associated with $\eta$ we define the constant
\[\sigma_{\eta}\defeq   \int_{\R^{d}}|z_{1}|^{2}\eta \left(|z|\right)\d z.\]
By 3.~in \cref{ass:eta} we have $\sigma_\eta<\infty$. Since $\eta(|x|)$ is rotationally invariant, the constant $\sigma_\eta$ is also given by
\[\sigma_{\eta} = \int_{\R^{d}}|z\cdot v|^{2}\eta \left(|z|\right)\d z\]
for any unit vector $v$. In addition, using the assumptions on $\eta$ we can also write
\begin{equation}\label{eq:sigma_eta_identity}
\int_{B(0,\epsilon)}|z\cdot w|^2\eta_\eps(|z|)  \d z = \epsilon^2|w|^2\int_{\R^d}\left|y\cdot \frac{w}{|w|}\right|^2\eta(|y|) \d y = \epsilon^2|w|^2\sigma_\eta
\end{equation}
for any $w\in \R^d$. At various points in the paper we will identify $\eta$ with the function $z\mapsto \eta(|z|)$ by writing $\eta(z)$ in place of $\eta(|z|)$, for notational simplicity. 

We introduce the following assumptions on the domain $\Omega$ and density $\rho$. To ensure our intermediate results are as general as possible, we specify various levels of regularity. 
\begin{assumption}\label{ass:omega}
    Let $\Omega \subset \R^d$ be open and bounded
    \begin{assenum}
        \item with Lipschitz boundary.\label{ass:omega_lipschitz}
        \item with $C^{1,\alpha}$ boundary for some $\alpha \in (0,1)$.\label{ass:omega_hoelder}
        \item with $C^{1,1}$ boundary.\label{ass:omega_c11} 
    \end{assenum}
\end{assumption}

\begin{assumption}\label{ass:rho}
    Let $\rho \in L^\infty(\Omega)$ such that $0 < \rho_{\min} \leq \rho \leq \rho_{\max} < \infty$,
    \begin{assenum}
        \item without further restriction.\label{ass:rho_bounded}
        \item such that $\rho \in C^{0,\alpha}(\Omega)$ for some $\alpha \in (0,1)$.\label{ass:rho_hoelder}
        \item such that $\rho \in \Lip(\Omega)$.\label{ass:rho_lipschitz}
        \item such that $\rho \in C^{1,\alpha}(\Omega)$ for some $\alpha \in [0,1]$, where we identify $C^{1,0}(\Omega)$ with $C^{0,1}(\Omega)$.\label{ass:rho_c1a}
    \end{assenum}
\end{assumption}

For the entire paper, we also make the following standing assumption on $n$ and $\eps$.
\begin{assumption}\label{ass:neps}
We assume $n\geq 2$ and $0 < \eps \leq 1$ such that $n\eps^d \geq 1$.
\end{assumption}
\cref{ass:neps} stipulates that the average number of neighbors of each node, which scales with $n\eps^d$, is at least a constant. In all of our results in this paper, we will usually require far stricter conditions, such as $n\eps^d \geq C \log(n)$ or $n\eps^d \sim \eps^{-q}$ for some $q>0$, so this standing assumption is not restrictive in any way, and it allows us to make simplifications to some error terms, such as the estimates $\frac{1}{n}\leq \eps^d$ and $n-1\geq \frac{1}{2}n$. 

Our main results require the strongest assumptions, \cref{ass:eta,ass:omega_c11,ass:rho_c1a} with $\alpha=1$. In this case, we have two main results, which we state informally here; the reader may skip to  \cref{sec:combination} to see the rigorous versions of each result. We first show in \cref{cor:main_nonconstant} that there exists a constant $C>0$ such that the following graph $\ell^1$ convergence rate holds with high probability:
\begin{equation}\label{eq:informal_nonconstant}
\frac{1}{n}\sum_{i=1}^n |u_{n,\epsilon}(x_i) - u(x_i)| \leq C \log(\eps^{-1})^{\frac{d}{2}}\eps^{\frac{1}{d+2}}.
\end{equation}
In \labelcref{eq:informal_nonconstant}, $u_{n,\epsilon}:\X_n\to\R$ is the solution to the graph Poisson equation \labelcref{eq:Poisson_Learning} with $n$ vertices and graph bandwidth $\eps$, properly normalized, and $u\in W^{1,p}(\Omega)$ for $1 \leq p < \frac{d}{d-1}$ is the solution of the continuum Poisson equation \labelcref{eq:Poisson_Learning_CL} with measure data (we refer the reader to  \cref{cor:main_nonconstant} for precise details).  Furthermore, we show in  \cref{cor:main_constant} that in the special case that $\rho\equiv |\Omega|^{-1}$ is constant, we can improve the rate to read
\begin{equation}\label{eq:informal_constant}
\frac{1}{n}\sum_{i=1}^n |u_{n,\epsilon}(x_i) - u(x_i)| \leq C \log(\eps^{-1})^{\frac{d}{2}+1}\eps^{\frac{2-\sigma}{d+4}},
\end{equation}
for any $\sigma>0$. Our theory is also able to establish graph $\ell^p$ convergence rates for $p>1$ very close to one; we leave the discussion of this to   \cref{rem:lprates} in \cref{sec:combination}.

We mention that the ``high probability'' condition in both results \labelcref{eq:informal_nonconstant} and \labelcref{eq:informal_constant} requires that $\eps>0$ is not too small, compared to $n$. In particular, for \labelcref{eq:informal_nonconstant} to hold with high probability we require that $\eps$ and $n$ satisfy
\[\eps \geq C\left( \frac{\log n}{n}\right)^{q},\]
where $q > 0$ and $C>0$ is a constant. The value of $q$ varies with dimension $d$ and depends on whether $\rho$ is constant or not, but in all cases is close to $q = \frac{1}{3d}$; see  \cref{rem:eps_constant,rem:eps_nonconstant} for precise details. These length scale restrictions on $\eps$ are much larger than the graph connectivity length scale which corresponds to $q=\frac{1}{d}$, and arise from our treatment of the heat kernel asymptotics in  \cref{sec:heat_kernel}. An interesting and challenging problem for future work is to establish convergence rates under less restrictive assumptions on the graph bandwidth~$\eps$. 

We can also prove results for general measures, using a stability estimate for distributional solutions of the Poisson equation with measure data. 
In this case, the rates \labelcref{eq:informal_nonconstant,eq:informal_constant} have an additive error term which measures the Wasserstein-1 distance of the measure data and the empirical measure used for the graph problem, as in \labelcref{eq:Poisson_Learning}.
For precise statements we refer to \cref{thm:rate_general_measure,cor:general_measure,rem:general_measure}.

\subsection{Calculus on general graphs}
\label{sec:graph_calculus}

Let $\X_n$ be a graph with $n \in \N$ vertices, together with symmetric edge weights $w_{xy} \geq 0$, for $x,y \in \X_n$. This section introduces graph norms, inner products, and calculus on general abstract graphs. In  \cref{sec:graph_calculus_rg} we specialize some of these notions for random geometric graphs. 

We let $\l2$ denote the Hilbert space of functions $u:\X_n\to \R$, equipped with the inner product
\begin{equation}\label{eq:graph_inner}
\ipg{u,v} = \frac{1}{n}\sum_{x\in \X_n}^n u(x)v(x),
\end{equation}
and norm  $\ng2{u}^2 = \ipg{u,u}$. 
For $p\geq 1$ we also define $p$-norms
\begin{equation}\label{eq:graph_pnorm}
\ng{p}{u}^p = \frac{1}{n}\sum_{i=1}^n |u(x_i)|^p.
\end{equation}
The \emph{degree} is a function $\deg \in \l2$ defined by
\[\deg(x) = \sum_{y \in \X_n} w_{xy}.\]
For a function $u\in \l2$ we also define the weighted mean value
\begin{equation}\label{eq:weighted_graph_mean}
(u)_{\deg} = \frac{\sum_{x \in \X_n} \deg(x) u(x)}{\sum_{x \in \X_n} \deg(x)},
\end{equation}
and the space of weighted mean zero graph functions
\begin{equation}
\label{eq:def_ell20}
    \lo2 = \left\{ u\in \l2 \st (u)_{\deg}=0\right\}.
\end{equation}
We let $\lx2$ denote the space of functions $V:\X_n^2\to \R$, which we view as \emph{vector fields} over the graph. The gradient $\nabla_{n} u\in \lx2$ of $u\in \l2$ is defined by
\begin{equation}\label{eq:graph_gradient}
\nabla_{n} u(x,y) = \sqrt{w_{xy}} \left(u(x) - u(y)\right).
\end{equation}
For two vector fields $U,V\in \lx2$ we define an inner product 
%\todo{Do we need $x \neq y$ in the sum? L: Doesn't matter since all vectorfields we use are zero for $x=y$.}
\begin{equation}\label{eq:vector_inner}
	%\ipv{U,V} =\frac{1}{\sigma_{\eta}n(n-1)\epsilon^{2}} \sum_{i,j=1}^n\eta_{\varepsilon}\left(|x_i-x_j|\right) U(x_i,x_j)V(x_i,x),
	%\ipv{U,V} =\frac{1}{n(n-1)} \sum_{x,y \in \X_n} w_{xy} U(x,y)V(x,y),
    \ipv{U,V} =\frac{1}{n(n-1)} \sum_{x,y \in \X_n} U(x,y)V(x,y),
\end{equation}
together with a norm $\nv{U}^2 = \ipv{U,U}$. 
Moreover, for a function $u \in \ell^2(\mathcal X_n)$, we define two graph Laplacians. The unnormalized graph Laplacian is given by
\begin{equation}\label{eq:def_unnormalized_graph_laplacian}
\L u(x) = \sum_{y \in \X_n} w_{xy} \left(u(x)-u(y)\right),
\end{equation}
while the random walk Laplacian is defined as
\begin{equation}\label{eq:def_rw_graph_laplacian}
\Lr u(x) = \frac{1}{\deg(x)}\sum_{y \in \X_n} w_{xy}\left(u(x)-u(y)\right) = u(x) - \sum_{y \in \X_n} \frac{w_{xy}}{\deg(x)} u(y).
\end{equation}
Both are connected via the identity $\Lr u = \deg^{-1}\L u$.  The adjoint $\Lr^T$ of $\Lr$ is defined as
\begin{equation}
\label{eq:def_rw_graph_laplacian_adjoint}
\Lr^T u(x) \defeq   u(x) - \sum_{y \in \X_n} \frac{w_{xy}}{\deg(y)} u(y).
\end{equation}
The random walk Laplacian and its adjoint are operators $\Lr,\Lr^T:\ell^2(\X_n) \to \ell^2(\X_n)$ and satisfy 
\[\ipg{\Lr u,v} = \ipg{u,\Lr^T v}.\]
Another important identity relating $\Lr$ and $\Lr^T$ that is readily verified is
\begin{equation}\label{eq:Lr}
\Lr^T u = \deg \Lr (\deg^{-1}u).
\end{equation}

\subsection{Calculus on random geometric graphs}
\label{sec:graph_calculus_rg}

In the setting of a random geometric graphs, we make slightly different definitions of gradients and Laplacians, so that all objects are consistent with the analogous objects in the continuum limit. Recall from the start of \cref{sec:main_results} that a random geometric graph has node set $\X_n=\{x_1,\dots,x_n\}$, where $x_1,\dots,x_n$ are $(i.i.d.)$ random variables on a domain $\Omega \subset \R^d$ with probability density function $\rho$ and edge weights $w_{xy} =  \eta_\eps\left(\abs{x-y}\right)$ for all $x,y \in \X_{n}$ where $\eta_\eps(t) =\frac{1}{\eps^d}\eta(t/\eps)$.  
As in the case of an abstract graph, the node degree $\deg_{n,\eps}$ at $x \in \X_{n}$ is given by 
\begin{equation}\label{eq:def_deg_ne}
\deg_{n,\eps}(x) = \sum_{y \in \X_{n}} \eta_\eps\left(\abs{x-y}\right),
\end{equation}
but when $n$ and $\eps$ are fixed, we sometimes drop the subscript and write $\deg \equiv \deg_{n,\eps}$.  It is important to note that the definition of degree \labelcref{eq:def_deg_ne} also makes sense for general $x\in\R^d$. Notice that this is the same definition as in  \cref{sec:graph_calculus} with $w_{xy} = \eta_\eps(|x-y|)$. The definition of $\l2$ inner product \labelcref{eq:graph_inner}, $p$-norms \labelcref{eq:graph_pnorm}, weighted mean value \labelcref{eq:weighted_graph_mean}, space of mean zero graph functions \labelcref{eq:def_ell20}, and inner product between vector fields \labelcref{eq:vector_inner}, as well as the induced norm, are all the same as in  \cref{sec:graph_calculus}. 

However, to obtain the correct continuum limits, it is necessary to consider a different scaling for the gradient and Laplacians. In particular, we will scale the discrete gradient in the following way. For $u \in \ell^2(\X_{n})$ and $x,y \in \X_{n}$ we set
\begin{equation*}
%\nabla_{n,\eps} u(x,y) \defeq   \frac{u(x) - u(y)}{\sigma_\eta \eps^2},
\nabla_{n,\eps} u(x,y) 
\defeq   \sqrt{\frac{w_{xy}}{\sigma_\eta \eps^2}} \left(u(x) - u(y)\right)
= \sqrt{\frac{\eta_\eps\left(\abs{x-y}\right)}{\sigma_\eta \eps^2}} \left(u(x) - u(y)\right),
\end{equation*}
where $\sigma_\eta$ is the constant from \cref{ass:eta}. With this definition of the gradient, its squared norm is given by 
\[\|\nabla_{n,\eps}u\|_{\lx2}^2 = \ipv{\nabla u_{n,\epsilon},\nabla u_{n,\epsilon}} =  \frac{1}{\sigma_\eta \eps^2n(n-1)}\sum_{x,y\in \X_n}\eta_\eps(|x-y|)(u(x) - u(y))^2, \]
which is the \emph{graph Dirichlet energy}, scaled in a way so that it is consistent with the continuum Dirichlet energy $\int_\Omega \rho^2 |\nabla u|^2 \d x$ as $n\to \infty$ and $\eps\to 0$.   We define the graph Laplacian by
\begin{equation}\label{eq:graph_laplacian_rg}
\L_{n,\eps} u(x) = \frac{1}{\sigma_\eta \varepsilon^2 (n-1)} \L u(x) = \frac{1}{\sigma_\eta \varepsilon^2 (n-1)} \sum_{y \in \X_n} \eta_\eps(|x-y|) \left(u(x)-u(y)\right),  
\end{equation}
which satisfies
\[\ip{u,\L_{n,\epsilon} v} = \ipv{\nabla_{n,\epsilon} u, \nabla_{n,\epsilon}v}\]
for all $u,v\in \l2$. 

In the random geometric setting, we also introduce the $H^1(\X_{n})$ graph inner product for $u,v \in \ell^2(\X_{n})$, by setting
\begin{equation*}
\ips{u,v} = \ip{u,v} + \ipv{\nabla_{n,\epsilon}u,\nabla_{n,\eps}}
\end{equation*}
and the $H^1(\X_{n})$ norm $\norm{u}_{H^1(\X_{n})}^2 \defeq   \ips{u,u}$. This is, again, consistent in the continuum limit with an $H^1(\Omega)$ inner product weighted by the density $\rho$.

\subsection{Notation}

We denote the Lipschitz constant of a function $f:\Omega\to \R$ by 
\[\Lip(f;\Omega) \defeq   \sup_{\substack{x,y\in\Omega\\x\neq y}}\frac{\abs{f(x)-f(y)}}{\abs{x-y}},\]
and we say a function is Lipschitz continuous on $\Omega$ if $\Lip(f;\Omega)<\infty$.  We define the H\"older semi-norm of a function $f:\Omega\to \R^m$ by 
\[ [f]_\alpha = \sup_{\substack{x,y\in \Omega \\ x\neq y}} \frac{|f(x) -f(y)|}{|x-y|}, \]
where $0 < \alpha \leq 1$. Note that $[f]_1 = \Lip(f;\Omega)$. 

We let $\omega_d$ denote the volume of the unit ball in $\R^{d}$. $\mathcal B(E)$ denotes the Borel $\sigma$-algebra of a set $E \subset \R^d$. For $\tau\geq 0$ we define the inner parallel set and the strip of width $\tau$ around the boundary as
\[\Omega_\tau \defeq   \left\lbrace x\in\Omega\st\dist(x,\partial\Omega)>\tau\right\rbrace,
\qquad
\partial_\tau \Omega = \Omega \setminus \Omega_\tau.
\]
for $f \in L^\infty(\Omega)$ and $D \subset \Omega$ we define the \emph{oscillation} of $f$ on $D$ by
\begin{equation*}
\osc_D f \defeq   \sup_{x,y \in D} \abs{f(x)-f(y)}.
\end{equation*}
We also define the weighted mean zero Sobolev space by 
\begin{equation}\label{eq:H1rho}
H^1_\rho(\Omega) \defeq   \set{v \in H^1_\rho(\Omega) \colon (v)_\rho \defeq   \frac{\int_\Omega v \rho^2 \dx}{\int_\Omega \rho^2 \dx} = 0}.
\end{equation}

\begin{remark}[The symbol $\lesssim$]
Throughout this paper, we make heavy use of the notation $f \lesssim g$, which is standard in the PDE literature and, in our case, means that $f \leq C g$ for a constant $C>0$ that just depends on various quantities from \cref{ass:eta,ass:omega,ass:rho}, i.e., the domain $\Omega$, the probability density $\rho$ with its Lipschitz constant and bounds, as well as the kernel function $\eta$ with the associated quantities $\eta(0)$, $\sigma_\eta$, etc. We will specify in each section which quantities the constants depend upon.
\end{remark}

\begin{remark}[The big-$\O$ notation]
We mention that our use of big-$\O$ notation is non-asymptotic so that $f(x) = \O(g(x))$ means there exists $C>0$ such that $|f(x)|\leq Cg(x)$ for all $x$.  Also, the notation $a \ll 1$ means there exists a constant $0 < c <1$ such that $a\leq c$, and $A \gtrsim B$ means there exists $C>1$ such that $A \geq C B$. The constants in the big-$\O$ notation depend on the same quantities as the constants in the $\lesssim$ symbol. 
\end{remark}

%! TeX root: ../main.tex
\section{Poisson equations with measure data and their approximation}
\label{sec:smoothed_poisson}

In this section we shall introduce the limiting equation of Poisson learning rigorously. 
For this we will first study the well-posedness of weighted Poisson equations with measure data and Neumann boundary conditions, then turn to Green's functions, investigate refined regularity of solutions with regular data, and finally prove stability both for measure and for regular data. 
While many of these results are widely known in the PDE community, we have to reprove most of them. 
This is necessary because of the lack of references for Neumann boundary conditions, and because we shall require explicit constants in what follows later.

The ultimate goal of this section is to prove $L^p$-rates of convergence of solutions with mollified right hand side to the distributional solution with measure data, where typically $1\leq p < \frac{d}{d-1}$. 
For this we will require the regularity statements for the Green's functions mentioned above.
It turns out that these convergence rates can be significantly improved for unweighted Poisson equations, i.e., where the differential operator is the Laplacian.

The Poisson equation which we study in this section is defined in the following.
\begin{definition}\label{def:distr_sol_Poisson_eq}
Let $\Omega$ satisfy \cref{ass:omega_lipschitz} and $\varrho$ satisfy \cref{ass:rho_bounded}, and let $f\in\M\left(\overline\Omega\right)$ be a  finite, real-valued Radon measure which satisfies the compatibility condition $f(\overline{\Omega}) = 0$.
We say that $u\in W^{1,p}(\Omega)$, $p>1$ is a \emph{distributional solution} to
\begin{equation} \label{eq:continuum_pde}
    \begin{dcases}
        -\div \varrho \nabla u = f & \text{in } \Omega, \\
        \frac{\partial u}{\partial \nu} = 0 & \text{on } \partial \Omega, \\
        \int_\Omega u \varrho \d x = 0,
    \end{dcases}
\end{equation}
if for all $\psi \in C^\infty(\overline\Omega)$ it holds
\begin{equation*}
    \int_\Omega \varrho \, \nabla u \cdot \nabla \psi \d x 
    = 
    \int_{\overline\Omega} \psi \d f,
    %$+\int_{\partial\Omega} \psi \d g
    \qquad 
    \int_\Omega u \varrho \d x = 0\,.
\end{equation*}
\end{definition}

% \begin{remark}
% Note that by setting $\varrho = \rho^2 / \int_\Omega \rho^2 \d x$, $f = \rho \tilde f$ in \labelcref{eq:continuum_pde}, and $g=0$, we recover the definition of the continuum Poisson equation given in \labelcref{eq:PDE} with right hand side $\tilde f$. In this section we will use the convention of \labelcref{eq:continuum_pde} for ease of notation, and explain the necessary adjustments in \cref{sec:combination} when we combine all results.
% \end{remark}

% \begin{remark}
%     In most of this section we will actually be concerned with the case of homogeneous Neumann data $g=0$, which corresponds to the continuum limit of the Poisson learning problem. 
%     However, we shall also need some properties of the inhomogeneous equation which is why \labelcref{eq:continuum_pde} has its general form.
% \end{remark}

Elliptic equations with measure data like \labelcref{eq:continuum_pde} have been a very active field of research in the last decades.
In particular, there are a couple of different solutions concepts which account for the fact that gradients of solutions are not square integrable, in general. 
One of the early approaches is due to Stampacchia \cite{stampacchia1965probleme} who proved existence and uniqueness of so-called duality solutions, which works best for linear problems.
The drawback of this approach is that it just asserts the existence of a solution $u \in L^p(\Omega)$ for $p<\frac{d}{d-2}$ but does not allow any statements regarding the regularity of its gradient. 
Taking into account that the right hand side of the weak formulation in \cref{def:distr_sol_Poisson_eq} is meaningful for functions $\psi \in L^q(\Omega)$ for $q>d$ (which consequently possess a Hölder-continuous representative) indicates that one should expect $\nabla u \in L^p(\Omega)$ for $1\leq p < \frac{d}{d-1}$.
Indeed, existence of distributional solutions $u\in W^{1,p}(\Omega)$ for $1\leq p<\frac{d}{d-1}$ can be proved by mollification and goes back to \cite{boccardo1989non}, where the authors used this approach for nonlinear equations.
However, uniqueness cannot be proved with this approach but was established for linear equations by showing that such distributional solutions are also duality solutions \cite{dal1999renormalized}. 
For further reading on elliptic equations with measure data and a discussion of the literature we also refer to \cite{mingione2007calderon}, where the author extends Calder\'on--Zygmund techniques to equations with measure data and proves that the gradient of distributional solutions possesses a fractional derivative.
Since most references, including the ones mentioned above, deal with the case of Dirichlet boundary data, we keep this section self-contained and provide all proofs.

In the context of Poisson learning, we are  interested in measure data of the form
\begin{align}\label{eq:Poisson_data}
    f \defeq   \sum_{i=1}^m a_i \delta_{x_i}
    \quad 
    \text{with the compatibility condition}
   \quad
    \sum_{i=1}^m a_i = 0.
\end{align}
% and homogeneous Neumann data $g=0$.
We will show that we can express a solution $u$ of the Poisson learning problem in terms of Green's functions $G^x \in W^{1,p}(\Omega)$, which we introduce in \cref{sec:Greensfunctions}, meaning that $u = \sum_{i=1}^m a_i G^{x_i} \in W^{1,p}(\Omega)$.
Furthermore, we shall prove convergence rates for solutions of the Poisson equation with mollified data to this special solution with data of the form \labelcref{eq:Poisson_data} in terms of the mollification parameter.

\subsection{Existence of solutions}

We begin with a general existence result for solutions $u\in W^{1,p}(\Omega)$ with $p \in \left[1, d/(d-1)\right)$ of \labelcref{eq:continuum_pde} for general measure data.%, which, in particular, implies the existence of Green's functions.

\begin{proposition}\label{prop:existence_solution_continuum_pde}
Let $\Omega$ satisfy \cref{ass:omega_lipschitz} and $\varrho$ satisfy \cref{ass:rho_bounded}. Moreover, let $f\in\M\left(\overline\Omega\right)$ be a finite, real-valued Radon measure with $f({\overline\Omega})=0$. 

Assume there exists a family $\set{f_n} \subset W^{-1,2}(\Omega)\cap L^1(\Omega)$ 
%and $\set{g_n} \subset H^{-1/2}(\partial\Omega)\cap L^1(\partial\Omega)$, 
such that% 
% \todo[inline]{Suggestion: Let's assume that $g \in L^1(\partial \Omega)$ is a fixed function. This way our statement has something new for the Neumann problem, but we do not have to worry about approximating a measure on the boundary.
% \red But for $g\in L^1$ the variational problem doesn't make sense
% \blue You are right. I suppose it should be $\set{g} \subset H^{-1/2}(\partial\Omega)\cap L^1(\partial\Omega)$}
\begin{enumerate}[label=(\roman*)]
    \item $\sup_{n \in \N} \norm{f_n}_{L^1(\Omega)} < \infty$,
    % \item $\sup_{n \in \N} \norm{g_n}_{L^1(\partial\Omega)} < \infty$,
    \item $\int_\Omega f_n \d x=0$,%$\int_\Omega f_n \d x + \int_{\partial\Omega} g_n \d\H^{d-1} = 0$,
    \item for all $\psi \in C^\infty(\overline \Omega)$ it holds $\int_\Omega \psi \, f_n \d x \to \int_{\overline\Omega} \psi \d f$ as $n \to \infty$.
    % \item for all $\psi \in C^\infty(\overline \Omega)$ it holds $\int_{\partial\Omega} \psi g_n \d\H^{d-1} \to \int_{\partial\Omega} \psi  \d g$ as $n \to \infty$.
\end{enumerate}
Then, there exist unique functions $u_n \in W^{1,2}(\Omega)$ which are weak solutions to \labelcref{eq:continuum_pde} with data $f_n$ such that for all $p \in \left[1, d/(d-1)\right)$,
\begin{equation}\label{eq:bound:W1p_norm_un}
    \sup_{n \in \N} \norm{u_n}_{W^{1,p}(\Omega)} \leq C,
\end{equation}
where $C=C(\varrho, p, d, \Omega, \sup_{n \in \N} \norm{f_n}_{L^1(\Omega)})$ is a constant.

Moreover, there exists a function $u \in W^{1,p}\left(\Omega\right)$, which is a distributional solution to \labelcref{eq:continuum_pde}, such that $u_{n_k} \rightharpoonup u$ in $W^{1,p}\left(\Omega\right)$ as $k \to \infty$ for a subsequence $\set{u_{n_k}} \subset \set{u_n}$.
\end{proposition}
\begin{remark}[Related results]
    The arguments presented in this proof follow very closely the paper \cite{boccardo1989non}, where this result was shown for a large class of (possibly non-linear) operators in divergence form with homogeneous Dirichlet boundary conditions. 
    We also refer to \cite{gruter1982green}, where a Green's function for uniformly elliptic coefficients and Dirichlet boundary conditions is constructed, corresponding to the case when $f$ is a Dirac measure. 
    See also \cite{hofmann2007green} for a Green's matrix for systems.
\end{remark}
\begin{remark}[Uniqueness]
    At this point we refrain from proving uniqueness as in \cite{dal1999renormalized} since this would require introducing the concept of Stampacchia's duality solutions for \labelcref{eq:continuum_pde} and some regularity properties.
    Later in this section we shall prove uniqueness for Green's functions and then later for general measure data, under some regularity assumptions on $\partial\Omega$ and $\varrho$, cf. \cref{rem:uniqueness_distributional}.
\end{remark}
\begin{remark}[Larger class of test functions]\label{rem:more_test_functions}
    Since \cref{prop:existence_solution_continuum_pde} asserts the existence of a distributional solution $u \in W^{1,p}(\Omega)$ for $1\leq p < \frac{d}{d-1}$, arising as weak limits of variational solutions, it is obvious to see that in fact one can enlarge the class of test functions in \cref{def:distr_sol_Poisson_eq} and obtain that
    \begin{align*}
        \int_\Omega \varrho \, \nabla u \cdot \nabla \psi \d x 
        = 
        \int_{\overline\Omega} \psi \d f
    \end{align*}
    holds even for all test functions $\psi \in W^{1,q}(\Omega)$ with $q>d$. 
\end{remark}
\begin{remark}[Approximating sequences]
    One my ask under which conditions on the measure $f$ a suitable approximating sequence $f_n$ exists.
    To construct such a sequence, one can convolve $f$ with a mollifier and subtract the mass, i.e., $f_n \defeq   f\star\phi_n - (f\star\phi_n)(\Omega)$. 
    Under the condition that $\abs{f}(\{x\in\Omega\st\dist(x,\partial\Omega)<\frac{1}{n}\}) \to 0$ as $n\to \infty$ (i.e., not too much mass concentrates near or on the boundary) one can prove that $f_n$ satisfies the properties above.
    In particular, the condition is satisfied if the support of $f$ is compactly contained in $\Omega$ which is an assumption that we will have to make for many other statements as well.
\end{remark}
\begin{proof}[Proof of \cref{prop:existence_solution_continuum_pde}]
    Since the proof from the original reference \cite{boccardo1989non} can be adapted easily, we postpone the proof of this proposition to \cref{app:smoothed_poisson}.
\end{proof}

\subsection{Green's functions}\label{sec:Greensfunctions}

For the study of the Poisson learning problem in the continuum, we will make frequent use of the Green's function $G^y$ with pole $y \in \Omega$, which is defined as a solution to \labelcref{eq:continuum_pde} with right hand side given by $\delta_y - \varrho / \int_\Omega \varrho \dx$.

\begin{definition}\label{def:green_function}
    Let $\Omega$ satisfy \cref{ass:omega_lipschitz} and $\varrho$ satisfy \cref{ass:rho_bounded}, and let $y \in \Omega$.
    Define $G^y \in W^{1,p}(\Omega)$, $p \in [1, d/(d-1))$, to be a distributional solution to \labelcref{eq:continuum_pde} with right hand side $f\defeq  \delta_y - \varrho / \int_\Omega \varrho \dx$. Then, we say that $G^y$ is a \emph{Green's function with pole in $y$.}
\end{definition}
By \cref{prop:existence_solution_continuum_pde} these function to indeed exist and can be constructed by the following approximation scheme.
For a fixed $y\in\Omega$, consider 
\begin{equation*}
        f_n \defeq   \phi_n^y - \frac{\varrho}{\int_\Omega \varrho \dx},
\end{equation*}
where $\supp(\phi)=B(0,1)$, $\phi \in L^\infty(B(0,1))$ with $\phi \geq 0$ such that $\norm \phi _{L^1(B(0,1))}=1$ and $\phi_n^y(x) \defeq   n^{d}\phi\left(n\left(x-y\right)\right)$, defined for $n>\frac{1}{\dist(y,\partial\Omega)}$. This choice $\{f_n\}$ as right hand side for \labelcref{eq:continuum_pde} satisfies the necessary conditions, and so we obtain weak solutions  $G_n^y \in W^{1,2}(\Omega)$ that converge (up to a subsequence) to the desired function $G^y \in W^{1,p}\left(\Omega\right)$. 

In this section we will show that these function $G^y$ are indeed Green's functions as well as collect several useful regularity results. 
We begin with the Green's function property for bounded right hand side.
\begin{lemma}\label{lem:v_as_Greens_convolution}
    Let $\Omega$ satisfy \cref{ass:omega_lipschitz} and $\varrho$ satisfy \cref{ass:rho_bounded}. 
    For $y \in \Omega$ let $G^y \in W^{1,p}$, $p \in [1, d/(d-1))$ be as in \cref{def:green_function}. Moreover, set $p^* = dp/(d-p)$ and $q^* = p^*/(p^*-1)$.
    
    Let $f \in L^{q^*}(\Omega)$
    %, $g \in W^{1,{q^*}}(\Omega)$ 
    with $\int_\Omega f \d x = 0$ 
    %$\int_{\partial \Omega} g \d \mathcal H^{d-1}$.
    Then $v \in W^{1,2}(\Omega)$ is a weak solution to \labelcref{eq:continuum_pde} with right hand side $f$ if and only if    \begin{equation}\label{eq:v_as_Greens_convolution}
        v(x) = \int_\Omega G^x(y) f(y) \d y. % + \int_{\partial \Omega} G^x(y) g(y) \d \mathcal H^{d-1}(y).
    \end{equation}
\end{lemma}
\begin{proof}
    We follow the proof given in \cite[Theorem 3.1]{hofmann2007green}.
    Let $v \in W^{1,2}(\Omega)$ be a weak solution to \labelcref{eq:continuum_pde}. Hence it holds for all $\psi \in W^{1,2}(\Omega)$
    \begin{equation*}
        \int_\Omega \varrho \nabla v \cdot \nabla \psi \d x = \int_\Omega f \psi \d x. %+ \int_{\partial \Omega} g \psi \d \mathcal H^{d-1}.
    \end{equation*}
    Choosing $\psi = G_n^x$ thus yields
    \begin{equation*}
        \int_\Omega \varrho \nabla v \cdot \nabla G_n^x \d y 
        = \int_\Omega f\, G_n^x \d y.% + \int_{\partial \Omega} g \,G_n^x \d \mathcal H^{d-1}.
    \end{equation*}
    Letting $n \to \infty$, the right hand side converges to 
    \begin{equation*}
        \int_\Omega f(y) G^x(y) \d y.% + \int_{\partial \Omega} g(y) G^x(y) \d \mathcal H^{d-1}(y).
    \end{equation*}
    By elliptic regularity, it follows that $v$ is continuous in the interior of the domain, in particular at $x \in \Omega$. Hence,
    \begin{equation*}
        \lim_{n \to \infty} \int_\Omega \varrho \nabla v \cdot \nabla G_n^x \d y 
        = \lim_{n \to \infty} \int_\Omega v \left(\phi_n^x - \varrho\right) \d y = v(x),
    \end{equation*}
    due to the continuity and zero mean condition of $v$.
    % For the converse statement we first note that thanks to \cref{thm:Gy_symmetric}, $v$ satisfies the mean zero condition:
    % \begin{align*}
    %     \int_\Omega v\varrho\d x
    %     &=
    %     \int_\Omega \int_\Omega G^x(y) g(y) \d y\varrho(x)\d x
    %     \\
    %     &=
    %     \int_\Omega g(y)\int_\Omega G^x(y)\varrho(x)\d x\d y
    %     \\
    %     &=
    %     \int_\Omega g(y)\int_\Omega G^y(x)\varrho(x)\d x\d y
    %     =0.
    % \end{align*}
    % Let us show that $v$ is a weak solution of the PDE \labelcref{eq:continuum_pde} with right hand side $g$.
    The reverse implication follows trivially from the uniqueness of solutions to \labelcref{eq:continuum_pde}:
    Letting $v$ and $w$ be weak solutions, one obtains by linearity for all $\psi \in W^{1,2}(\Omega)$:
    \begin{align*}
        \int_\Omega\varrho\nabla (v-w)\cdot\nabla\psi\d x = 0.
    \end{align*}
    Choosing $\psi\defeq  v-w \in W^{1,2}(\Omega)$ we obtain $\nabla (v-w) = 0$ almost everywhere on $\Omega$ and hence $v-w=c$ for some constant $c\in\R$. 
    The zero mean condition then implies
    \begin{align*}
        c = \int_\Omega \varrho (v-w)\d x = \int_\Omega \varrho v\d x- \int_\Omega \varrho w\d x = 0
    \end{align*}
    and hence $v=w$ almost everywhere in $\Omega$.
\end{proof}
Next, we show a bound in $W^{1,p}(\Omega)$, uniform in the pole $y \in \Omega$.
\begin{lemma}\label{lemma:Greens_functions_W1p_bound}
Let $\Omega$ satisfy \cref{ass:omega_lipschitz} and $\varrho$ satisfy \cref{ass:rho_bounded}. For $y \in \Omega$ let $G^y \in W^{1,p}$, $p \in [1, d/(d-1))$, be as in \cref{def:green_function}. 
Then there exists $C=C(\varrho, p, d, \Omega) > 0$ such that
    \begin{equation*}
        \sup_{y \in \Omega} \norm{G^y}_{W^{1,p}(\Omega)} \leq C.
    \end{equation*}
\end{lemma}
\begin{proof}
    Note that we have the uniform bound
    \begin{equation*}
        \norm{\phi_n^y - \rho}_{L^1(\Omega)} \leq 1+ \abs{\Omega} \rho_{\max}.
    \end{equation*}
    Then the proof follows immediately from the construction and \cref{prop:existence_solution_continuum_pde}.
\end{proof}

% By adapting classical elliptic regularity arguments to the case of Neumann boundary conditions, we can show that this limiting function $G^y$ is indeed the unique Green's function in the sense of \cref{def:green_function}.

Finally, we have a regularity statement of Green's functions away from their poles.

%\todo[inline]{The following proposition can actually be summarized in a single statement, namely the second one, if we are willing to assume a $C^{1,\alpha}$ boundary throughout.}
\begin{proposition}[Regularity of Green's functions]\label{prop:Greens-regularity}
    Let $\Omega$ satisfy \cref{ass:omega_lipschitz} and $\varrho$ satisfy \cref{ass:rho_hoelder}. For $y \in \Omega$ let $G^y \in W^{1,p}$, $p \in [1, d/(d-1))$ be as in \cref{def:green_function}. Moreover, let $\beta \in (0, \alpha)$.  
    \begin{enumerate}[label=(\roman*)]
        \item Let $x_0 \in \Omega$ and $R > 0$ such that $\dist \left(x_0, \partial \Omega\right) > 5R$ and $\abs{x_0 - y} > 5R$. Then it holds that
        \begin{equation*}   
            \norm{G^y}_{C^{1,\beta}\left(B(x_0,R)\right)} \leq C,
        \end{equation*}
        where $C = C(d, \varrho, R, \Omega, \alpha, \beta)>0$ is a constant and does not depend on $x_0$ and $y$. In particular, $G^y$ is continuous for any $x \in \Omega$, $x\neq y$.
        \item Let further $\Omega$ satisfy \cref{ass:omega_hoelder}, and $0 < R < \dist(y, \partial\Omega)$. Then,
        \begin{equation*}
            \norm{G^y}_{C^{1,\beta}\left(\Omega \setminus B(y,R)\right)} \leq C,
        \end{equation*}
        where $C = C(d, \varrho, R, \Omega, \alpha, \beta)>0$ is a constant which does not depend on $y$.
    \end{enumerate}
\end{proposition}

\begin{proof}
    This result follows immediately by combining \cref{thm:C1alpha_approximation_Gy,thm:Gy_boundary_regularity}.
\end{proof}

\subsection{Regularity for regular data}
\label{sec:regularity_estimates}

It will be necessary to estimate the Lipschitz constant of weak solution to \labelcref{eq:continuum_pde} on the whole domain and close to the boundary in terms of the data. We do so in the present section, and start with a Lipschitz estimate for the whole domain.

It will be convenient later to drop the mean zero condition on $f$. We will consider the weak solution $u\in H^1(\Omega)$ of the variational problem
\begin{equation}\label{eq:variational_problem}
\min\left\{\frac12\int_\Omega |\nabla u|^2\varrho\d x - \int_\Omega fu \d x \st u\in H^1(\Omega) \text{ and } \int_\Omega u\varrho \d x = 0  \right\},
\end{equation}
where $f\in L^\infty(\Omega)$. The minimizer $u\in H^1(\Omega)$ is a weak solution of the PDE
\[ -\div(\varrho \nabla u) = f - c_f \varrho,\]
where $c_f = \frac{1}{|\Omega|}\int_\Omega f\d x$. By \cref{lem:v_as_Greens_convolution} we can write the solution $u$ as
\begin{equation}\label{eq:green_rep_formula}
u(x) = \int_\Omega G^x(y) (f(y) - c_f\varrho) \d y =  \int_\Omega G^x(y) f(y)\d y,
\end{equation}
due to the fact that $\int_\Omega G^y(x)\varrho \d x = 0$. 

% \red I'm recording below the regularity proof that uses $C^{1,\alpha}$ boundary, since it's currently the only one I am sure is correct\nc
\begin{proposition}[Global regularity]\label{prop:global_regularity}
Let $\Omega$ satisfy \cref{ass:omega_lipschitz} and $\varrho$ satisfy \cref{ass:rho_bounded}. Let $p \in [1,d/(d-1))$ and set $p^* = dp/(d-p)$ and $q^* = p^*/(p^*-1)$.

Let $f\in L^{q^*}(\Omega)$ and let $u\in H^1(\Omega)$ be the minimizer of \labelcref{eq:variational_problem}. Then the following hold:
\begin{enumerate}[label=(\roman*)]
\item Then $u\in L^\infty(\Omega)$ and we have
\begin{equation}\label{eq:Linfty_bound}
\|u\|_{L^\infty(\Omega)} \leq C\|f\|_{L^{q^*}(\Omega)},
\end{equation}
where $C = C(\varrho, p, d, \Omega) > 0$. 
\item If we have moreover that $\Omega$ satisfies \cref{ass:omega_hoelder} and $\varrho$ satisfies \cref{ass:rho_hoelder} and $f \in L^\infty(\Omega)$, then $u\in C^{1,\alpha}(\Omega)$ and 
\begin{equation}\label{eq:C1a_bound}
\|u\|_{C^{1,\alpha}(\Omega)} \leq C\|f\|_{L^{\infty}(\Omega)},
\end{equation}
where $C = C(\varrho, d, \Omega,\alpha) > 0$.
\end{enumerate}
\end{proposition}
\begin{remark}
In fact, under \cref{ass:omega_lipschitz,ass:rho_bounded} the function $u$ is not just in $L^\infty(\Omega)$ but even Hölder continuous and one could replace (i) by asserting that $u\in C^{0,\gamma}(\Omega)$ for some $\gamma>0$ and it holds
\begin{align*}
    \|u\|_{C^{0,\gamma}(\Omega)} \leq C\|f\|_{L^{q^*}(\Omega)},
\end{align*}
where $C = C(\varrho, p, \Omega, \gamma) > 0$. 
This was proved in \cite[Theorem 3.14]{nittka2011regularity}, where we note that $p<\frac{d}{d-1}$ implies that $q^*>\frac{d}{2}$ which makes this result applicable.
\end{remark}
\begin{proof}
Using \labelcref{eq:green_rep_formula} and the Hölder and Sobolev inequalities we have
\begin{equation*}
|u(x)| \leq \|G^x\|_{L^{p^*}(\Omega)}\|f\|_{L^{q^*}(\Omega)} \leq C\|G^x\|_{W^{1,p}(\Omega)}\|f\|_{L^{q^*}(\Omega)} \leq C\|f\|_{L^{q^*}(\Omega)},
\end{equation*}
where we use that $\|G^x\|_{W^{1,p}(\Omega)}$ is bounded, thanks to \cref{lemma:Greens_functions_W1p_bound}. This proves (i).

To prove (ii), we use the $C^{1,\alpha}$ estimate \cite[Theorem 5.54]{lieberman2003oblique} to obtain
\[\|u\|_{C^{1,\alpha}(\Omega)} \leq C \left(\|u\|_{L^\infty(\Omega)} + \|f\|_{L^\infty(\Omega)}\right),\]
where $C=C(\varrho,d,\Omega,\alpha)$. Combining this with (i) completes the proof of (ii).
\end{proof}
\begin{remark}\label{rem:specific_bounds}
We remark that in \cref{prop:global_regularity}, we can make a specific choice of $p \in (1,d/(d-1))$ such that $q^* = \frac{d+1}{2}$ (by noting that $p=1$ yields $q^*=d$ and the invalid choice $p=d/(d-1)$ yields $q^*=d/2$). Therefore it follows from \labelcref{eq:Linfty_bound} that
\begin{equation}\label{eq:specific_Linfty_bound}
\|u\|_{L^\infty(\Omega)} \leq C\|f\|_{L^{\frac{d+1}{2}}(\Omega)}.
\end{equation}
\end{remark}

We have improved boundary regularity results. Recall the notation
\begin{equation*}
    \Omega_\tau \defeq   \set{x \in \Omega \colon \dist (x, \partial \Omega) \geq \tau}
\end{equation*}
and
\begin{equation*}
    \partial_\tau \Omega \defeq   \Omega \setminus \Omega_\tau.
\end{equation*}
introduced above. 
We can now give improved Lipschitz regularity statements close to the boundary which just depend on the $L^1$-norm of the right hand side.
\begin{proposition}[Boundary regularity]\label{prop:boundary_c01}
Let $\Omega$ satisfy \cref{ass:omega_hoelder} and $\varrho$ satisfy \cref{ass:rho_hoelder}.
Let $f\in L^\infty(\Omega)$ and let $u\in H^1(\Omega)$ be the minimizer of \labelcref{eq:variational_problem}. Then for any $\tau > 0$ it holds that
\begin{equation}\label{eq:boundary_c01}
\|u\|_{W^{1,\infty}(\partial_\tau \Omega)}\leq C\left(\|f\|_{L^\infty(\partial_{2\tau}\Omega)}+ \|f\|_{L^1(\Omega)}\right),
\end{equation}
where $C = C(\varrho, d, \Omega,\alpha,\tau) > 0$. 
\end{proposition}
\begin{proof}
We use the Green's function representation formula \labelcref{eq:v_as_Greens_convolution} and its symmetry from \cref{thm:Gy_symmetric} to get $u = w +v$, where
\begin{equation*}
w(x) = \int_{\partial_{2\tau}\Omega} G^x(y) f(y) \d y \quad\text{and}\quad v(x) = \int_{\Omega_{2\tau}} G^y(x) f(y) \d y.
\end{equation*}
By \cref{lem:v_as_Greens_convolution} we see that $w$ is a weak solution of $-\div(\varrho w)=f\chi_{\partial_{2\tau}\Omega}$ and hence by \cref{prop:global_regularity} (ii) we have
\[\|w\|_{C^{1,\alpha}(\Omega)} \leq C\|f\|_{L^\infty(\partial_{2\tau}\Omega)}.\]
To bound $v$, for $x \in \partial_\tau\Omega$ we have
\begin{align*}
\abs{v(x)} &\leq \norm{f}_{L^1(\Omega)} \sup_{y \in \Omega_{2\tau}} \abs{G^y(x)} \leq  \norm{f}_{L^1(\Omega)} \sup_{y \in \Omega_{2\tau}} \norm{G^y}_{L^\infty(\partial_\tau \Omega)}
\end{align*}
% and similarly 
% \begin{align*}
% \abs{ \nabla v(x)} &\leq  \norm{f}_{L^1(\Omega)} \sup_{y \in \Omega_{2\tau}} \norm{\nabla G^y}_{L^\infty(\partial_\tau \Omega)}
% \end{align*}
and similarly that
\begin{equation*}
\abs{\nabla v(x)} 
\leq \norm{f}_{L^1(\Omega)} \sup_{y \in \Omega_{2\tau}} \abs{\nabla_x G^y(x)}
\leq \norm{f}_{L^1(\Omega)} \sup_{y \in \Omega_{2\tau}} \norm{\nabla G^y}_{L^\infty(\partial_\tau \Omega)}.
\end{equation*}
Combining both estimates we thus see that
\begin{equation*}
\norm{v}_{W^{1,\infty}(\partial_\tau \Omega)} 
\leq \norm{f}_{L^1(\Omega)} \sup_{y \in \Omega_{2\tau}} \norm{G^y}_{W^{1,\infty}(\partial_\tau \Omega)} 
\leq C \norm{f}_{L^1(\Omega)},
\end{equation*}
where we used \cref{prop:Greens-regularity} (ii) to bound $\norm{G^y}_{W^{1,\infty}(\partial_\tau \Omega)} \leq C(d, \varrho, \tau, \Omega, \alpha)$ uniformly in~$y$, since $\dist(y, \partial\Omega) > 2\tau$.
% \todo[inline]{Plz check, I replaced $C^{0,1}$ norms of $v$ by $W^{1,\infty}$ norms to avoid $\tau$-dependent constants. 
% Furthermore, I used that $\norm{G^y}_{W^{1,\infty}(\partial_\tau\Omega)}\leq \norm{G^y}_{C^{0,1}(\partial_\tau\Omega)}$ with constant $1$.}
% By \cref{lem:green_c01} there exsits $C=C(d,\Omega,\varrho,\tau)$ such that
% \[|\nabla G^y(x)| \leq C \ \ \text{for all}  \ \ x\in \partial_\tau \Omega,\, y\in \Omega_{2\tau},\]
% since in this case we always have $|x-y|\geq \tau$. Therefore, for any $x\in \partial_\tau\Omega$ we have
% \[|u(x)| \leq \int_{\Omega_{2\tau}}|G^y(x)||f(y)|\d y \leq C\int_{\Omega_{2\tau}}|f(y)|\d y \leq C \|f\|_{L^1(\Omega)},\]
% and similarly
% \[|\nabla u(x)| \leq \int_{\Omega_{2\tau}}| \nabla G^y(x)||f(y)|\d y\leq C \|f\|_{L^1(\Omega)},\]
% which completes the proof.
\end{proof}

\subsection{Stability}

The global regularity results in \cref{prop:global_regularity} allow us to prove $L^p$ stability of distributional solutions of Poisson equations. 
The key result for deriving convergence rates for solutions of the Poisson equation with mollified data to the one with measure data is the following stability result.
It requires some regularity of the boundary $\partial\Omega$ and the density $\varrho$, however, it has very strong implications, in particular, uniqueness of the distributional solutions and Green's functions, cf. \cref{rem:uniqueness_distributional} further down.
\begin{theorem}[$L^p$-stability for measure data]\label{thm:Lp_stability}
Let $\Omega$ satisfy \cref{ass:omega_hoelder} and $\varrho$ satisfy \cref{ass:rho_hoelder}.
Let $f\in\M\left(\overline\Omega\right)$ be a finite real-valued Radon measure that satisfies the compatibility condition $f(\overline{\Omega})=0$, and let $u\in W^{1,p}(\Omega)$, $p \in [1,d/(d-1))$, be a \emph{distributional solution} to \labelcref{eq:continuum_pde}. Moreover, let $\beta \defeq   \min\set{\alpha, 1-d(p-1)/p}$. 
Then there exists $C=C(\varrho,d,p,\Omega,\alpha)$ such that
\begin{equation}\label{eq:Lp_stability}
\|u\|_{L^p(\Omega)} \leq C\sup\left\{ \int_{\overline \Omega} \psi \d f \, : \, \psi\in C^{1,\beta}(\Omega), \|\psi\|_{C^{1,\beta}(\Omega)} \leq 1\right\}.
\end{equation}
% \todo[inline]{L: For $p=\frac{d}{d-1}$ this lemma should still be true, replacing $C^{1,\beta}$ by $C^{0,1}$}
\end{theorem}
\begin{proof}
Let $\psi\in W^{1,2}(\Omega)$ be the weak solution of 
%\[-\div(\varrho \nabla \psi) = \sign(u) - \varrho\int_\Omega \sign(u)\d x,\]
%\red
\[-\div(\varrho \nabla \psi) = \sign(u)\abs{u}^{p-1} - \varrho\int_\Omega \sign(u)\abs{u}^{p-1}\d x,\]
%\nc
with homogeneous Neumann boundary condition $\frac{\partial \psi}{\partial \nu} = 0$ on $\partial \Omega$ and zero mean $\int_\Omega \varrho \psi \d x = 0$. 
%By \cref{prop:global_regularity}  (ii) we have $\psi\in C^{1,\alpha}(\Omega)$, and $\|\psi\|_{C^{1,\alpha}(\Omega)}\leq C=C(\varrho,d,\Omega,\alpha)$---in particular, $C$ is independent of $u$. 
%\red
By \cite[Theorem 5.54]{lieberman2003oblique}, we have that for $\beta \in (0, \alpha]$ and $C=C(\rho, d, \Omega) > 0$\
\begin{equation*}
    \norm{\psi}_{C^{1, \beta}(\Omega)} 
    \leq C \left(\norm{\psi}_{L^\infty(\Omega)} + \norm{\sign(u)\abs{u}^{p-1} - \varrho\int_\Omega \sign(u)\abs{u}^{p-1}\d x}_{M^{2,d-2+2\beta}(\Omega)}\right).
\end{equation*}
Here, $M^{2,d-2+2\beta}(\Omega)$ is the Morrey space, defined by all functions $w \in L^2(\Omega)$ such that the norm
\begin{equation*}
    \norm{w}_{M^{2,d-2+2\beta}(\Omega)}^2 \defeq   \sup_{x \in \Omega} \sup_{0 < r < \diam(\Omega)} \frac1{r^{d-2+2\beta}}\int_{B(x,r) \cap \Omega} \abs{w}^2 \d x
\end{equation*}
is finite. An application of Hölder's inequality shows that $L^r(\Omega) \hookrightarrow M^{2,d-2+2\beta}(\Omega)$ for all $r \geq d/(1-\beta)$.
We wish to choose $r = p/(p-1)$, which requires $\beta \leq 1-d(p-1)/p$. Then,
\begin{equation*}
\begin{aligned}
    \norm{\psi}_{C^{1, \beta}(\Omega)} 
    &\leq C \left(\norm{\psi}_{L^\infty(\Omega)} + C(d, p, \beta, \Omega)\norm{\sign(u)\abs{u}^{p-1} - \varrho\int_\Omega \sign(u)\abs{u}^{p-1}\d x}_{L^{p/(p-1)}(\Omega)}\right) \\
    &\leq C \left(\norm{\psi}_{L^\infty(\Omega)} + C(d, p, \beta, \varrho, \Omega)\norm{u}_{L^{p}(\Omega)}^{p-1}\right).
\end{aligned}
\end{equation*}
Using \cref{prop:global_regularity} (i), we see that
\begin{equation*}
    \norm{\psi}_{L^\infty(\Omega)} 
    \leq C\norm{\sign(u)\abs{u}^{p-1} - \varrho\int_\Omega \sign(u)\abs{u}^{p-1}\d x}_{L^{p/(p-1)}(\Omega)} 
    \leq C\norm{u}_{L^{p}(\Omega)}^{p-1},
\end{equation*}
where $C>0$ depends only on $\Omega$, $d$, $p$, $\varrho$. Therefore, for $0 < \beta \leq \min\set{\alpha, 1-\frac{d(p-1)}{p}}$ we have
\begin{equation*}
    \norm{\psi}_{C^{1, \beta}(\Omega)} \leq C \norm{u}_{L^{p}(\Omega)}^{p-1}.
\end{equation*}
%\nc
By the definition of weak solution we have 
%\begin{equation}\label{eq:weak_psi}
%\int_\Omega \varrho \nabla \psi \cdot \nabla w \d x = \int_\Omega \sign(u) w \d x - \int_\Omega \varrho  w\d x \int_\Omega \sign(u)\d x,
%\end{equation}
%\red
\begin{equation}\label{eq:weak_psi}
\int_\Omega \varrho \nabla \psi \cdot \nabla w \d x = \int_\Omega \sign(u)\abs{u}^{p-1} w \d x - \int_\Omega \varrho  w\d x \int_\Omega \sign(u)\abs{u}^{p-1}\d x,
\end{equation}
%\nc
for all $w \in W^{1,2}(\Omega)$. Since $\psi \in C^{1,\beta}(\Omega) \subset W^{1,\infty}(\Omega)$, we may, by approximation, take $w\in W^{1,1}(\Omega)$ in \labelcref{eq:weak_psi}. In particular, we can set $w=u$ to obtain
%\begin{equation}\label{eq:psi_weak}
%\int_\Omega \varrho \nabla \psi \cdot \nabla u \d x = \int_\Omega |u| \d x  = \|u\|_{L^1(\Omega)},
%\end{equation}
%\red
\begin{equation}\label{eq:psi_weak}
\int_\Omega \varrho \nabla \psi \cdot \nabla u \d x = \int_\Omega \abs{u}^p \d x  = \|u\|_{L^p(\Omega)}^p,
\end{equation}
%\nc
due to the fact that $\int_\Omega \varrho u \d x= 0$. Since $u\in W^{1,p}(\Omega)$, we can use $\psi\in C^{1,\beta}(\Omega)$ as a test function in the definition of distributional solution of \labelcref{eq:continuum_pde} given in \cref{def:distr_sol_Poisson_eq}, yielding
%Combining this with \labelcref{eq:psi_weak} we have
%\begin{equation*}
%\|u\|_{L^1(\Omega)}=\int_\Omega \varrho \, \nabla u \cdot \nabla \psi \d x = \int_{\overline\Omega} \psi \d f = \|\psi\|_{C^{1,\alpha}(\Omega)} %\int_{\overline\Omega }\phi \d f,
%\end{equation*}
%where $\phi = \frac{\psi}{\|\psi\|_{C^{1,\alpha}(\Omega)}}$. Since $\|\phi\|_{C^{1,\alpha}(\Omega)}\leq 1$ and $\|\psi\|_{C^{1,\alpha}(\Omega)}\leq %C$, we obtain \labelcref{eq:Lp_stability}, which completes the proof.
%\red
\begin{equation*}
\|u\|_{L^p(\Omega)}^p=\int_\Omega \varrho \, \nabla u \cdot \nabla \psi \d x = \int_{\overline\Omega} \psi \d f = \|\psi\|_{C^{1,\beta}(\Omega)} \int_{\overline\Omega }\phi \d f,
\end{equation*}
where $\phi = \frac{\psi}{\|\psi\|_{C^{1,\beta}(\Omega)}}$. Since $\|\phi\|_{C^{1,\beta}(\Omega)}\leq 1$ and $\|\psi\|_{C^{1,\beta}(\Omega)}\leq C \norm{u}_{L^p}^{p-1}$, we obtain \labelcref{eq:Lp_stability}, which completes the proof.
%\nc
\end{proof}
\begin{remark}[Uniqueness]\label{rem:uniqueness_distributional}
An immediate application of \cref{thm:Lp_stability} is the uniqueness of distributional solutions of \labelcref{eq:continuum_pde} in the space $W^{1,p}(\Omega)$ for $p \in [1, d/(d-1))$. In particular, this shows that the Green's functions constructed in \cref{sec:Greensfunctions} are unique.
\end{remark}
\begin{remark}[Wasserstein stability]
    Another immediate consequence is the stability of solutions with respect to the Wasserstein-1 distance of the positive and negativ parts of $f$, i.e., $\norm{u}_{L^p(\Omega)}\leq C W_1(f^+,f^-)$.
\end{remark}

We also have the following stability results for regular data.
\begin{proposition}[Stability for $L^1$ data]\label{prop:stability_L1}
Let $\Omega$ satisfy \cref{ass:omega_hoelder} and $\varrho$ satisfy \cref{ass:rho_hoelder}. Let  $f_1,f_2\in L^1(\Omega)$ satisfy the compatibility condition $\int_\Omega f_i\d x = 0$. Let $p \in [1,d/(d-1))$ and let $u_1,u_2\in W^{1,p}(\Omega)$ be the corresponding \emph{distributional solutions} to \labelcref{eq:continuum_pde}. There exists exists $C=C(\varrho,d,p,\Omega,\alpha)$ such that
\begin{equation}\label{eq:L1_stability}
\|u_1-u_2\|_{L^p(\Omega)} \leq C\|f_1 - f_2\|_{L^1(\Omega)}.
\end{equation}
\end{proposition}
\begin{proof}
The difference $u=u_1 - u_2$ is the distributional solution of \labelcref{eq:continuum_pde} with $f=f_1-f_2$. We now apply  \cref{thm:Lp_stability} to the Radon measure $f\d x$, noting that
\[\int_{\overline\Omega} \psi f \d x \leq \|\psi\|_{L^\infty(\Omega)}\|f\|_{L^1(\Omega)} \leq \|f\|_{L^1(\Omega)},\]
since $\|\psi\|_{L^\infty(\Omega)}\leq \|\psi\|_{C^{1,\beta}(\Omega)}\leq 1$. This yields
\[\|u_1-u_1\|_{L^p(\Omega)} = \|u\|_{L^p(\Omega)} \leq C\|f\|_{L^1(\Omega)} = \|f_1-f_2\|_{L^1(\Omega)}.\qedhere\]
\end{proof}

\begin{proposition}[Uniform stability]\label{prop:stability_Linfty}
Let $\Omega$ satisfy \cref{ass:omega_lipschitz} and $\varrho$ satisfy \cref{ass:rho_bounded}. Let $p \in [1,d/(d-1))$ and set $p^* = dp/(d-p)$ and $q^* = p^*/(p^*-1)$.

Let $f_1,f_2\in L^{q^*}(\Omega)$ and let $u_1,u_2\in H^1(\Omega)$ be the corresponding minimizers of \labelcref{eq:variational_problem}.  Then it holds that
\begin{equation}\label{eq:stability}
\|u_1-u_2\|_{L^\infty(\Omega)} \leq C\|f_1-f_2\|_{L^{q^*}(\Omega)},
\end{equation}
\end{proposition}
\begin{proof}
We note that each $u_i$ is a weak solution of 
\[ -\div(\varrho \nabla u_i) = f_i - c_{f_i} \varrho,\]
and so $u\defeq  u_1-u_2$ is a weak solution of 
\[ -\div(\varrho \nabla u) = f - c_f \varrho, \  \ \text{where } f\defeq  f_1-f_2. \]
Thus, $u$ is the minimizer of \labelcref{eq:variational_problem} with data $f$. Applying \cref{prop:global_regularity} (i) completes the proof. 
\end{proof}

\subsection{Convergence rates for mollified data}

Now we are ready to prove the main convergence statement of this section, proving global convergence rates for solutions of the PDE \labelcref{eq:continuum_pde} with mollified right hand side to the solution of the continuum Poisson learning problem with data \labelcref{eq:Poisson_data}.

We first prove a result about approximation of Green's functions which will then translate to an approximation result for general solutions.
\begin{lemma}\label{lem:approx_greens}
Let $\Omega$ satisfy \cref{ass:omega_hoelder} and $\varrho$ satisfy \cref{ass:rho_hoelder}.
Let $\phi\in L^{\infty}(\Omega)$ be a nonnegative function such that $b = \int_\Omega \phi \d x > 0$, and let $v \in W^{1,2}(\Omega)$ be the minimizer of \labelcref{eq:variational_problem} with source term  $\phi$. Let $\beta \defeq   \min\set{\alpha, 1- d(p-1)/p}$. For $p \in [1, d/(d-1))$, there exists $C=C(\varrho,d,p,\Omega,\alpha)$ such that for any $x\in \Omega$ 
\[\|G^{x} - v\|_{L^p(\Omega)} \leq C\left( |b-1| + \sup\left\{ b\psi(x) - \int_\Omega \psi \phi \d y \, : \, \psi\in C^{1,\beta}(\Omega), \|\psi\|_{C^{1,\beta}(\Omega)} \leq 1\right\}\right).\]
\end{lemma}
\begin{proof}
We first note that $v$ is the weak solution of 
\[-\div(\varrho \nabla v) = \phi - b\varrho \ \ \text{in } \Omega\]
with homogenous Neumann boundary condition, and mean-zero condition $\int_\Omega \varrho v \d x =0$. Let us set $w = G^{x}-b^{-1}v\in W^{1,p}(\Omega)$. Then $w$ is the distributional solution of
\[-\div(\varrho \nabla w) = \delta_{x}-b^{-1}\phi\ \ \text{in } \Omega,\]
and hence by \cref{thm:Lp_stability} we have
\[\|G^{x}-b^{-1}v\|_{L^p(\Omega)} \leq C\sup\left\{ \psi(x) - b^{-1}\int_\Omega \psi \phi \d y \, : \, \psi\in C^{1,\beta}(\Omega), \|\psi\|_{C^{1,\beta}(\Omega)} \leq 1\right\}.\]
Multiplying by $b$ on both sides and using that
\[\|bG^{x}-v\|_{L^p(\Omega)} \geq \|G^{x} - v\|_{L^p(\Omega)} - |b-1|\|G^{x}\|_{L^1(\Omega)} \geq \|G^{x} - v\|_{L^p(\Omega)} - C|b-1|\]
completes the proof.
\end{proof}

We can now show $L^p$ convergence with an explicit rate.
For convenience we will prove a slightly stronger result than the one outlined above. 
We shall let $u$ and $v$ be the solutions  
\begin{align*}
    -\div(\varrho\nabla v) = f - c_f\rho
    \qquad
    \text{and}
    \qquad
    -\div(\varrho\nabla u) = \sum_{i=1}^m a_i\delta_{x_i}    
\end{align*}
and will bound the $L^p$-difference of $u$ and $v$ by the difference of $f$ and the mollified data $g\defeq  \sum_{i=1}^m a_i\varphi_i$, as well as the parameters of the mollifiers $\varphi_i$. The reason we do not take $f=g$ is that the source terms $f$ that we encounter will not be compactly supported in $\Omega$, so we need to allow for an approximation error between $f$ and $g$ so that we may truncate their support. We also note that the functions $\varphi_i$ that we shall use later on will not be classical mollifiers, but merely compactly supported bounded functions whose integral is close to one, in some cases satisfying approximate symmetry conditions. Thus, we state our theorems, the first of which is below, in as general a setting as possible.    

\begin{theorem}[$L^p$-convergence rates]\label{thm:convergence_rate_varrho}
Let $\Omega$ satisfy \cref{ass:omega_hoelder} and $\varrho$ satisfy \cref{ass:rho_hoelder}. Let $\{x_1,\dots,x_m\}\subset\Omega$ and $\set{a_1,\dots,a_m}\subset\R$ satisfy $\sum_{i=1}^m a_i = 0$.  Let $\Ri_1,\dots,\Ri_m>0$ such that $B(x_i,\Ri_i)\subset \Omega$, and let $\phi_i\in L^\infty(\R^d)$ be nonzero and nonnegative functions such that $\phi_i(x)=0$ for $|x|\geq \Ri_i$. Define $u \in W^{1,p}(\Omega)$, $g\in L^\infty(\Omega)$ and $b_i>0$ by 
\begin{equation}\label{eq:b_g_u}
u = \sum_{i=1}^m a_i G^{x_i}, \ \   g = \sum_{i=1}^m a_i \phi_i(\cdot - x_i),\ \ \text{and} \ \ b_i = \int_{\R^d}\phi_i \d x>0.
\end{equation}
Let $f\in L^\infty(\Omega)$ and let $v \in W^{1,2}(\Omega)$ be the minimizer of \labelcref{eq:variational_problem} with source term $f$.  Then for all $p \in \left[1,\frac{d}{d-1}\right)$ it holds that
\begin{equation}\label{eq:convergence_rate_smoothing_L1}
\|u-v\|_{L^p(\Omega)} \leq C\left(\sum_{i=1}^m |a_i| \left(b_i\Ri_i^{1+\beta} + |\xi_i| + |b_i-1|\right)  + \|f-g\|_{L^1(\Omega)}\right).
\end{equation}
where $\xi_i=\int_{B(0,\Ri_i)}z\phi_i(z)\d z$ and $\beta = \min\set{\alpha, 1-d(p-1)/p)}$.
\end{theorem}
\begin{proof}
Let $w$ be the minimizer of \labelcref{eq:variational_problem} with source term $g$. Then $v$ and $w$ are weak solutions of 
\[ -\div(\varrho \nabla v) = f - c_{f} \varrho \ \ \text{and} \ \ -\div(\varrho \nabla w) = g - c_{g}\varrho,\]
and hence also distributional solutions of the same equations. By \cref{prop:stability_L1} we have
\[\|v - w\|_{L^p(\Omega)}\leq C\|f-g + (c_g - c_f)\varrho \|_{L^1(\Omega)} \leq C\|f-g\|_{L^1(\Omega)}.\]
The remainder of the proof will focus on bounding $u-w$. Let $w_i\in W^{1,2}(\Omega)$ be the minimizer of \labelcref{eq:variational_problem} with source term  $\phi_i(\cdot-x_i)$. Then we can write
\begin{equation}\label{eq:lp_estimate_wu}
w = \sum_{i=1}^m a_i w_i, \ \ \text{and so} \ \ \|u - w\|_{L^p(\Omega)} \leq  \sum_{i=1}^m\abs{a_i}\|G^{x_i}- w_i\|_{L^p(\Omega)}.
\end{equation}
We now use \cref{lem:approx_greens} to bound $\|G^{x_i} - w_i\|_{L^p(\Omega)}$. Let $\psi\in C^{1,\beta}(\overline\Omega)$ with $\|\psi\|_{C^{1,\beta}(\Omega)}\leq 1$. Then we have the Taylor expansion 
\[\left| \psi(x_i+z) - \psi(x_i) - \nabla \psi(x_i)\cdot z \right| \leq \frac{1}{2}|z|^{1+\beta},\]
which holds for all $z\in B(0,\Ri_i)$ since $\|\psi\|_{C^{1,\beta}(\Omega)}\leq 1$. Using this and the definition of $\xi_i$ we compute
\begin{align*}
b_i&\psi(x_i) - \int_{B(x_i,\Ri_i)} \psi(y) \phi_i(y-x_i) \d y \\
&= \int_{B(0,\Ri_i)}\phi_i(z)(\psi(x_i) - \psi(x_i+z))\d y\\
&=\int_{B(0,\Ri_i)}\phi_i(z)(\psi(x_i) - \psi(x_i+z) + \nabla \psi(x_i)\cdot z )\d z -\nabla \psi(x_i)\cdot\int_{B(0,\Ri_i)}z\phi_i(z)\d z  \\
&\leq\int_{B(0,\Ri_i)}\phi_i(z)|\psi(x_i) - \psi(x_i+z) + \nabla \psi(x_i)\cdot z |\d y + |\nabla \psi(x_i)| \left|\int_{B(0,\Ri_i)}z\phi_i(z)\d z\right|  \\
&\leq\frac{1}{2} \int_{B(0,\Ri_i)}\phi_i(z)|z|^{1+\beta}\, dy + |\nabla \psi(x_i)| |\xi_i| \leq \frac{b_i}{2}\Ri_i^{1+\beta} + |\xi_i|.
\end{align*}
Therefore
\[\|G^{x_i} - w_i\|_{L^p(\Omega)} \leq C(b_i\Ri_i^{1+\beta} + |\xi_i| +  |b_i-1|),\]
which, upon inserting into \labelcref{eq:lp_estimate_wu}, completes the proof. 
\end{proof}
\begin{remark}\label{rem:opt_rate}
In  \cref{thm:convergence_rate_varrho}, we have $\xi_i=\int_{\R^d}z\phi_i(z)\d z=0$ for radial kernels $\phi_i(x)=\phi_i(|x|e_1)$, as well as kernels with other symmetries, such as $\phi(x)=\phi(-x)$. Since $|\xi_i|$ appears as an error term, the result can also handle kernels with approximate symmetries. 

Also note that if we assume $p>1$ and $\alpha \geq 1-d(p-1)/p$, then we obtain an $O( \Ri_i^{2 - \frac{d}{p}(p-1)})$ convergence rate, neglecting the error terms $|b_i-1|$, $|\xi_i|$, and $\|f-g\|_{L^1(\Omega)}$, which in applications will be far smaller. Since we cannot take $\alpha=1$, we obtain an $O(\Ri_i^{1 + \alpha})$ rate for any $\alpha < 1$ when $p=1$. In  \cref{thm:convergence_rate_smoothing} below, we prove an $O(\Ri_i^2)$ rate when $p=1$, $\varrho$ is constant, and the $\phi_i$ are radial.  
\end{remark}

\subsection{Improved rates for constant density}

It turns out that we can improve the convergence rates from \cref{thm:convergence_rate_varrho} if we assume that the density $\varrho$ is constant and that the kernels $\phi_i$ are radial functions.
Assuming that they are supported on a ball of radius $R$ we can improve the $L^1$-convergence rate to quadratic.
Furthermore, the following theorem requires much less regularity on the domain boundary than \cref{thm:convergence_rate_varrho} did and extends the validity of the $L^p$-convergence rate from $p<\frac{d}{d-1}$ to the maximal exponent $p<\frac{d}{d-2}$ which can be expected from the Sobolev embedding.

\begin{theorem}[Improved $L^p$-convergence rate]\label{thm:convergence_rate_smoothing}
Let $\Omega$ satisfy \cref{ass:omega_lipschitz} and assume that $\varrho\equiv \beta \defeq   \abs{\Omega}^{-1}>0$ is a constant. Let $\{x_1,\dots,x_m\}\subset\Omega$ and $\set{a_1,\dots,a_m}\subset\R$ satisfy $\sum_{i=1}^m a_i = 0$.  Let $\Ri_1,\dots,\Ri_m>0$ such that $B(x_i,\Ri_i)\subset \Omega$, and let $\phi_i\in L^\infty(\R^d)$ be nonzero and nonnegative radial functions such that $\phi_i(x)=0$ for $|x|\geq \Ri_i$. Define $u \in W^{1,p}(\Omega)$, $g\in L^\infty(\Omega)$ and $b_i>0$ by \labelcref{eq:b_g_u}. Let $f\in L^\infty(\Omega)$ and let $v \in W^{1,2}(\Omega)$ be the minimizer of \labelcref{eq:variational_problem} with source term $f$. Then there exists $C=C(p, d, \varrho, \Omega)$ such that for all $p \in \left[1,\frac{d}{d-2}\right)$ for $d\geq3$ and for all $p\in[1,\infty)$ if $d\in\{1,2\}$ it holds
\begin{align*}
\norm{u-v}_{L^p(\Omega)} \leq C\sum_{i=1}^m |a_i| \left(\|\phi_i\|_{L^p(\R^d)}\Ri_i^2 + b_i \Ri_i^2 + |b_i-1|\right) + C\|f-g\|_{L^\infty(\Omega)}.
\end{align*}
\end{theorem}
\begin{proof}
Let $w$ be the minimizer of \labelcref{eq:variational_problem} with source term $g$. Then by \cref{prop:stability_Linfty} we have
\begin{equation}\label{eq:vw_bound}
\|v - w\|_{L^\infty(\Omega)}\leq C\|f-g\|_{L^\infty(\Omega)}.
\end{equation}
The remainder of the proof will focus on bounding $u-w$.

Since $\phi_i \in L^\infty(\R^d)$, by \cref{lem:v_as_Greens_convolution} we know that $w(x) = \sum_{i=1}^m a_i \int_\Omega G^y(x)\phi_i(y-x_i)\d y$ and we define the functions $w_i \defeq   \int_\Omega G^y(x)\phi_i(y-x_i)\d y$. Subtracting $u$ and $w$ and using the symmetry of the Greens' function from \cref{thm:Gy_symmetric} we get that for almost every $x\in\Omega$ it holds
\begin{align}\label{eq:estimate_1}
G^{x_i}(x) - b_i^{-1}&w_i(x) \notag \\
&=  \int_\Omega \left(G^x(x_i)-G^x(y)\right)b_i^{-1}\phi_i(y-x_i)\d y\notag\\
&=-b_i^{-1}\int_{B(x_i,\Ri_i)} \int_0^1\frac{\d}{\d t}G^x(x_i+t(y-x_i))\d t\phi_i(y-x_i)\d y \notag\\
&=-b_i^{-1}\int_0^1 \int_{B(x_i,\Ri_i)} \nabla G^x(x_i+t(y-x_i))\cdot (y-x_i)\phi_i(y-x_i)\d y\d t.
\end{align}
We now construct $\Psi_i\in W^{1,\infty}(\R^d)$ so that $-\nabla \Psi_i (z) = z\phi_i(z)$ for all $z\in \R^d$ and $\Psi_i(z)=0$ for $|z|\geq \Ri_i$. Since $\phi_i$ is radial, the construction is simply
\[\Psi_i(z) = \int_{|z|}^\infty s\,\phi_i(se_1) \d s.\]
We also define the rescaled kernels $\Psi^t_i(z) = \frac{1}{t^d}\Psi_i(\frac{z}{t})$.  
A simple computation shows that 
\begin{equation}\label{eq:poincare_psi}
% \|\Psi_i\|_{L^p(\R^d)} = \|\Psi_i\|_{L^p(B(0,\Ri_i))} \leq C_1\Ri_i \|\nabla \Psi_i\|_{L^p(B(0,\Ri_i))} \leq C\Ri_i^2\|\phi_i\|_{L^p(\R^d)}
\|\Psi_i\|_{L^p(\R^d)} = \|\Psi_i\|_{L^p(B(0,\Ri_i))} \leq C \Ri_i^2\|\phi_i\|_{L^p(\R^d)}
\end{equation}
for any $p\geq 1$, where $C\defeq  d^{-\frac{1}{p}}$. 
Inserting the definition of $\Psi_i$ in \labelcref{eq:estimate_1} we have
\begin{align}\label{eq:estimate_3}
G^{x_i}(x) - b_i^{-1}w_i(x) &=  b_i^{-1}\int_0^1\int_{B(x_i, \Ri_i)}\nabla_z G^x(x_i+t(y-x_i)) \cdot \nabla \Psi_i\left(y-x_i\right) \d y \d t \notag \\
&=  b_i^{-1}\int_0^1t^{-d}\int_{B(x_i, t \Ri_i)}\nabla_z G^x(z) \cdot \nabla \Psi_i\left(\frac{z - x_i}{t}\right) \d z \d t \notag \\
&=b_i^{-1} \int_0^1 t\int_{\Omega} \nabla_z G^x(z) \cdot \nabla_z \Psi_i^t(z-x_i) \d z \d t.
\end{align}
Using that $G^x$ is a distributional solution of $-\Delta G^x = \delta_x - \beta$ with $\beta>0$ on $\Omega$ with homogeneous Neumann boundary conditions, and taking into account that according to \cref{rem:more_test_functions} the function $\Psi^t_i\in W^{1,\infty}(\Omega)$ is a valid test function we obtain
\begin{align}\label{eq:estimate_4}
\int_{\Omega}\nabla_z G^x(z) & \cdot \nabla_z \Psi_i^t(z-x_i) \d z=\Psi^t_i(x-x_i)-\beta\int_{\Omega}\Psi^t_i(z-x_i)\d z.
\end{align}
Combining \labelcref{eq:estimate_1,eq:estimate_3,eq:estimate_4} and using the triangle inequality yields
\begin{align*}
\|G^{x_i}-b_i^{-1}w_i\|_{L^p(\Omega)}
%&= \norm{\int_0^{1}t^{-d+1}\left[\Psi_i\left(\frac{\cdot-x_i}{t}\right)-\beta \int_{\Omega}\Psi_i\left(\frac{z-x_i}{t}\right)\d z\right]\d t}_{L^p(\Omega)}\\
%&= \int_0^{1}t^{-d+1}\left[\norm{\Psi_i\left(\frac{\cdot-x_i}{t}\right)}_{L^p(\Omega)}+\beta \int_{\Omega}\Psi_i\left(\frac{z-x_i}{t}\right)\d z\right]\d t\\
&\leq b_i^{-1}\int_0^{1}t^{1}\left[\norm{\Psi^t_i(\cdot-x_i)}_{L^p(\R^d)}+\beta|\Omega|^{\frac{1}{p}} \norm{\Psi_i^t(\cdot-x_i)}_{L^1(\R^d)}\d z\right]\d t.
\end{align*}
We now use that $\|\Psi_i^t\|_{L^p(R^d)} \leq t^{-d + \frac{d}{p}}\|\Psi_i\|_{L^p(\R^d)}$ 
%\[\norm{\Psi_i\left(\frac{\cdot-x_i}{t}\right)}_{L^p(\R^d)}^p = \int_{\R^d}\Psi_i\left(\frac{x-x_i}{t}\right)^p\d x = t^d\int_{\R^d}\Psi_i(z)^p\d z = t^d \|\Psi_i\|^p_{L^p(\R^d)},\]
and \labelcref{eq:poincare_psi} to obtain
\begin{align*}
\|G^{x_i}-b_i^{-1}w_i\|_{L^p(\Omega)} &\leq b_i^{-1}\int_0^1 t\left(t^{-d+\frac{d}{p}}\|\Psi_i\|_{L^p(\Omega)} + \beta|\Omega|^{\frac{1}{p}}\|\Psi_i\|_{L^1(\Omega)} \right) \d t\\
&\leq Cb_i^{-1}\Ri_i^2 \int_0^1 t^{1-d+\frac{d}{p}}\|\phi_i\|_{L^p(\Omega)} + t\beta|\Omega|^{\frac{1}{p}}\|\phi_i\|_{L^1(\Omega)} \d t\\
&= Cb_i^{-1}\Ri_i^2\left(\frac{p}{d(1-p) + 2p}\|\phi_i\|_{L^p(\R^d)}  + \frac{1}{2}\beta|\Omega|^{\frac{1}{p}} \|\phi_i\|_{L^1(\R^d)}\right)\\
&\leq Cb_i^{-1}\Ri_i^2\left(\|\phi_i\|_{L^p(\R^d)} + b_i\right),
\end{align*}
where, if $d>2$, we used that $p < d/(d-2)$ to integrate the first term, and $C$ changes from line to line. The proof is completed by the computation
\begin{align*}
\norm{u - w}_{L^p(\Omega)} &\leq \sum_{i=1}^m |a_i| \norm{G^{x_i} - w_i}_{L^p(\Omega)}\\
&\leq\sum_{i=1}^m |a_i| \left(b_i\norm{G^{x_i} - b_i^{-1}w_i}_{L^p(\Omega)} + |b_i-1|\norm{G^{x_i}}_{L^p(\Omega)}\right)\\
%&\leq\sum_{i=1}^m |a_i| \left(C\Ri_i^2\left(\|\phi_i\|_{L^p(\R^d)} + b_i\right) + |b_i-1|\norm{G^{x_i}}_{L^p(\Omega)}\right)\\
&\leq C\sum_{i=1}^m |a_i| \left(\Ri_i^2\left(\|\phi_i\|_{L^p(\R^d)} + b_i\right) + |b_i-1|\right),
\end{align*}
and combining this with \labelcref{eq:vw_bound}.
\end{proof}
\begin{remark}\label{rem:actual_rate}
When $p=1$,  \cref{thm:convergence_rate_smoothing} gives an $O(\Ri_i^2)$ convergence rate, since we expect $\|\phi_i\|_{L^1(\Omega)} = b_i$ to be uniformly bounded; indeed, the functions $\phi_i$ will be chosen as approximations to the measures $\delta_{x_i}$. 

When $1 < p < \frac{d}{d-2}$, then the terms $\|\phi_i\|_{L^p(\Omega)}$ will be larger and produce a worse rate. For example, if $\phi_i = \Ri_i^{-d}\phi\left( \frac{z}{\Ri_i}\right)$ is a rescaled unit kernel  $\phi$, then 
\[\|\phi_i\|_{L^p(\R^d)} = \Ri_i^{-d + \frac{d}{p}} \|\phi\|_{L^p(\R^d)}.\]
Thus, the rate becomes $O(\Ri_i^{2 - d + \frac{d}{p}})$, which matches the rate in  \cref{thm:convergence_rate_varrho}. 
\end{remark}

\section{Convergence rates of Poisson learning for smooth data}
\label{sec:continuum_limits}

This section aims to establish discrete to continuum convergence rates for solutions of graph Poisson equations with regular data as the sample size tends to infinity. In this setting, both the continuum PDE and discrete graph problem have variational interpretations and we use the quantitative stability of the variational problems to establish convergence rates. The proofs are similar to recent work on spectral convergence rates \cite{calder2022improved,garcia2020error,burago2015graph}, with some modifications to handle the boundary of the domain and simplify parts of the proof. Some of the main ideas in this section appeared previously in lecture notes by the second author \cite{CalculusofVariationsLN}.

In this section, we take the setting of a random geometric graph, introduced in  \cref{sec:graph_calculus_rg}. If not stated differently, we assume throughout this section that $\eta$ satisfies \cref{ass:eta}, $\Omega$ satisfies \cref{ass:omega_lipschitz} and $\rho$ satisfies \cref{ass:rho_lipschitz}. In this section, the constants in the $\lesssim$ symbol depend on all quantities from the assumptions. In some cases we will denote the dependence of constants more explicitly, e.g., $C=C(\rho_{min})$. For simplicity we will write $C(\rho) = C(\rho_{min},\rho_{max},\Lip(\rho))$.

We now introduce the graph and continuum energies used in this section. For $f_n\in \ell^2(\X_{n})$, we define the discrete graph energy $\ene(\cdot; f_n) \colon \ell^2(\X_{n}) \to \R$ via
\begin{equation}\label{eq:def_ene}
    \ene(u_n; f_n) \defeq   \ene^{(1)}(u_n) + \ene^{(2)}(u_n; f_n),
\end{equation}
where
\begin{gather*}
    \ene^{(1)}(u_n) \defeq   \frac12 \norm{\nabla_{n, \eps} u_n}_{\ell^2(\X_{n}^2)}^2 = \frac{1}{2 n(n-1)\sigma_\eta \eps^2} \sum_{x,y \in \X_{n}} \eta_\eps\left(\abs{x-y}\right) \abs{u_n(x) - u_n(y)}^2; \\
    \qquad \ene^{(2)}(u_n; f_n) \defeq   -\sprod{u_n, f_n}{\ell^2(\X_{n})} = -\frac1n \sum_{x \in \X_{n}} u(x)f(x).
\end{gather*}
For $f, \rho \in L^\infty(\Omega)$, we define the continuum counterpart $I(\cdot; f, \rho) \colon H^1(\Omega) \to \R$ by
\begin{equation}\label{eq:def_I}
    I(u; f, \rho) \defeq   I^{(1)}(u; \rho) + I^{(2)}(u; f, \rho),
\end{equation}
where
\begin{equation*}
    I^{(1)}(u; \rho) \defeq   \frac12 \int_\Omega \abs{\nabla u}^2 \rho^2 \dx; \qquad I^{(2)} \defeq   -\int_\Omega f u \rho \dx.
\end{equation*}
Frequently, we will consider minimizers of \labelcref{eq:def_I} in the space $H^1_\rho(\Omega)$ defined in \labelcref{eq:H1rho}.  The associated Euler--Lagrange equation for a minimizer $u \defeq   \argmin_{v \in H^1_\rho(\Omega)} I(v;f, \rho)$ reads
\begin{equation}\label{eq:continuum_pde_EL}
    \forall v \in H^1_\rho(\Omega) \colon \int_\Omega \rho^2 ~\nabla u \cdot \nabla v \dx = \int_\Omega \rho f v \dx.
\end{equation}
Moreover, when $f \in L^\infty$ satisfies the compatibility condition $\int_\Omega f\rho \dx = 0$, one can extend the space of test functions to all of $H^1(\Omega)$, which implies that $u$ is a weak solution to
\begin{equation}\label{eq:def_continuum_pde_limit_formulation}
\left\{
\begin{aligned}
-\frac1\rho (\div \rho^2 \nabla u) &=f && \text{in $\Omega$},\\
\frac{\partial u}{\partial \nu} &= 0 && \text{on $\partial \Omega$}.
\end{aligned}
\right.
\end{equation}
In particular, since $\rho \geq \rho_{\min} > 0$, we have an elliptic problem and can apply the tools of \cref{sec:smoothed_poisson}, with $\varrho \equiv \rho^2$ and right hand side $f\rho$, to study $u$. We consider this slightly different continuum equation compared to the one in the previous section, because it is the natural limit arising from \labelcref{eq:def_ene}, but it is not difficult to transform estimates from one to the other.

Similarly, we will frequently minimize $\ene$ in the space $\lo2$ (cf. \cref{eq:def_ell20}). Let $u_{n, \eps} \defeq   \argmin_{v \in \lo2} \ene(v, f_n)$ be the unique minimizer, then it solves an Euler--Lagrange equation (see e.g. \cite[Theorem 2.3]{calder2020poisson}) given by
\begin{equation}\label{eq:discrete_pde_EL}
    \forall v\in \ell^2_0(\X_{n}) \colon \sprod{\nabla_{n,\eps}u_{n,\eps},\nabla_{n,\eps}v}{\ell^2(\X_{n}^2)} = \sprod{f_n, v}{\ell^2(\X_{n})}.
\end{equation}
If $f_n$ satisfies the compatibility condition $\sprod{f_n, \one}{\ell^2(\X_{n})}=0$, then one can extend the space of test functions to the whole $\ell^2(\X_{n})$, which implies that $u_{n,\eps}$ solves the graph PDE 
\begin{equation}\label{eq:graph-laplacian-geometric}
\L_{n,\eps} u_{n,\eps} =  f_n,
\end{equation}
where $\L_{n,\epsilon}$ is the random geometric graph Laplacian defined in \labelcref{eq:graph_laplacian_rg}.

Our main result of this section will be the following estimate regarding convergence to the continuum.

\begin{theorem}[Continuum limit for bounded data]\label{thm:main_smooth}
Let $f \colon \Omega \to \R$ be Borel-measurable and bounded, let $u \in H^1_\rho(\Omega)$ be the unique minimizer of $I(\cdot; f, \rho)$ over $H^1_\rho(\Omega)$, and let $q > \frac d2$.
    There exist positive constants $C_1(\Omega, \eta, \rho_{\min}, \rho_{\max})$, $C_2(\Omega)$, $C_3(\Omega, \eta, \rho_{\min}, \rho_{\max})$, $C_4(\Omega, \eta, \rho)$, $R(\Omega)$, $K(q)$, $\eps_1(\Omega, \rho, \eta)$, $\hat\lambda_1(\Omega, \rho, \eta)$ and $\hat\lambda_2(\Omega, \rho)$, such that for any $n \in \N$, $0 < \eps \leq \eps_1$, $n^{-\frac1d} < \delta \leq \frac{\rho_{\min}}{8 \Lip(\rho)}$, $\delta/\eps \leq C_4$, $0 < \lambda_1 \leq \hat\lambda_1$, and $0 < \lambda_2 \leq \hat\lambda_2$ the event that
    \begin{align*}
	\|u - u_{n,\epsilon}\|_{H^1(\X_{n})}^2 
        &\lesssim \left(\norm{u}_{L^\infty(\Omega)}  + K(q) \norm{f_n}_{\ell^q(\X_{n})}\right) \norm{f - f_n}_{\ell^1(\X_{n})} \\
        & + \left(\Lip(u)^2 + \varepsilon^{\frac d2} \norm{fu}_{L^\infty(\Omega)}  + \eps^d \lambda_1 \norm{u}_{L^\infty(\Omega)}^2\right) \lambda_1 +             
            %\norm{\nabla u}_{L^\infty(\partial_{2\varepsilon}\Omega)}^2 
            \Lip(u;\partial_{2\epsilon}\Omega)^2\eps \\
        & + \left(\frac{\delta}{\varepsilon} + \varepsilon + \lambda_1^2 + \lambda_2\right) \norm{f}_{L^2(\Omega)} ^2+ \norm{f_n}_{\ell^2(\X_{n})} \norm{f_n}_{\ell^2(\X_{n} \cap \partial_{2R} \Omega)} \\
		& + (\eps + \lambda_2) \norm{f_n}_{\ell^2(\X_{n})}^2 + K(q) \norm{\osc_{B(\delta;\cdot)}f}_{L^1(\Omega)} \norm{f_n}_{\ell^q(\X_{n})}
    \end{align*}
   holds for all $f_n \in \ell^2(\X_{n})$ satisfying $\sprod{f_n,\one}{\ell^2(\X_{n})}=0$ has probability at least $1-4n \exp(-C_1 n \eps^d \lambda_1^2) - C_2n\exp(-C_3 n\delta^d \lambda_2^2)$, where $u_{n,\eps} \in \ell^2_0(\X_{n})$ is the unique minimizer of $\ene(\cdot;f_n)$.
    %
	% Let $f \in L^\infty(\Omega)$ and let $u \in H^{1}_\rho(\Omega)$ be the unique minimizer of \labelcref{eq:def_I} in $H_\rho^1(\Omega)$. 
 % %
 %    Let $n\in\N$, $\eps>0$, and $\delta>0$ be such that $n^{-1/d} < \delta \leq \frac{\rho_{\min}}{8\Lip(\rho)}$ and $\delta/\epsilon \leq C(\eta)$. For $R\equiv R(\Omega)>0$ sufficiently small, there exist constants $C_1,C_2,C_3,C_4>0$, depending only on $\Omega$, $\eta$, $\rho$, and $R$ such that for all $q>d/2$, $0<\lambda_1\leq C_1$ and $0\leq \lambda_2 \leq \tfrac{\rho_{\min}}{8\rho_{\max}}$ such that $\frac{\delta}{\eps}+\epsilon + \lambda_2  \leq C_2$  the event that
	% \begin{align*}
	% 		\|u - u_{n,\epsilon}\|_{H^1(\X_{n})}^2 \lesssim& \,\,C(q)\left(\ng1{f_n-f} +\|\osc_{\Omega\cap B(\cdot,\delta)}f\|_{L^1(\Omega)} \right)\left(\|f\|_{L^q(\Omega)} + \ng{q}{f_n}\right) \\
	% 		&+\|f\|_{L^\infty(\Omega)}^2\lambda_1+\|f\|_{L^\infty(\partial_{4\epsilon}\Omega)}^2\eps+ \ng2{f_n} \|f_n\|_{\ell^2(\X_{n}\cap \partial_{2R}\Omega)}  \\
	% 		&+\left(\ng2{f_n}^2 + \|f\|^2_{L^2(\Omega)}\right)\left(\frac{\delta}{\eps}+\epsilon +\lambda_1^2 + \lambda_2\right).
	% 	\end{align*}
	% holds for all $f_n\in \l2$ satisfying the compatibility condition $\sprod{f_n,\one}{\ell^2(\X_{n})} = 0$ has probability at least $1-C_3\left(\exp(-C_4n\epsilon^d\lambda_1^2)+\exp(-C_4n\delta^d\lambda_2^2)\right)$. Here, $u_{n,\eps} \in \ell_0^2(\X_{n})$ is the unique minimizer of \labelcref{eq:def_ene}.
\end{theorem}
We remark that the Borel measurability assumption on $f$ ensures that $f(x_i)$ is a well-defined random variable, where $x_i$ is a node in the random geometric graph, which allow us to \emph{restrict} $f$ to the graph. 

To show \cref{thm:main_smooth}, we will need the following two results, which will be proven below.

\begin{proposition}
	\label{prop:discrete_continuum_energy_continuum_function}
    %\todo{Need more boundary regularity here, to obtain that $u$ is indeed Lipschitz}
	Let $f \colon \Omega \to \R$ be Borel-measurable and bounded. Let $u \in H_\rho^1(\Omega)$ be the unique minimizer of $I(\cdot; f, \rho)$ in $H_\rho^1(\Omega)$. There exists a positive constant $C_1(\eta(0), \rho_{\max}, \sigma_\eta)$, such that for any $0 < \lambda \leq 1$ the event
    \begin{align*}
	{\ene(u;f_n) - I(u;f,\rho)} 
	&\lesssim 
        \norm{u}_{L^\infty(\Omega)} \norm{f - f_n}_{\ell^1(\X_{n})} + \left(\Lip(u)^2 + \norm{fu}_{L^\infty(\Omega)} \varepsilon^{\frac d2}\right) \lambda 
        \\
        &\qquad+ \left(\norm{f}_{L^2(\Omega)}^2 + 
            %\norm{\nabla u}_{L^\infty(\partial_{2\varepsilon}\Omega)}^2 
            \Lip(u;\partial_{2\epsilon}\Omega)^2
        \right)\varepsilon
    \end{align*}
    for all $f_n \in \ell^2(\X_{n})$ holds with probability at least $1-4\exp\left(-C_1n\varepsilon^d \lambda^2\right)$.
\end{proposition}

\begin{proposition} 
	\label{prop:continuum_energy_continuum_solution_discrete_energy_discrete_solution}
     Let $f \in L^\infty(\Omega)$ and $u \in H^1_\rho(\Omega)$ the unique minimizer of $I(\cdot; f, \rho)$ in $H^1_\rho(\Omega)$, $q > \frac d2$.
	There exist positive constants $C_1(\Omega)$, $C_2(\Omega, \rho_{\min})$, $C_3(\Omega, \eta, \rho)$, $R(\Omega)$, $K(q)$, $\eps_1(\Omega, \rho)$ and $\lambda_1(\Omega, \rho)$, such that for any $n \in \N$, $0 < \eps \leq \eps_1$, $n^{-\frac1d} < \delta \leq \frac{\rho_{\min}}{8 \Lip(\rho)}$, $\delta/\epsilon \leq C_3$, and $0 < \lambda \leq \lambda_1$ the event that
	\begin{align*}
		&\phantom{{}={}}I(u; f, \rho) - \ene(u_{n,\varepsilon}; f_n) \\
		&\lesssim \norm{u_{n,\eps}}_{\ell^2(\X_{n} \cap \partial_{4R} \Omega)} \norm{f_n}_{\ell^2(\X_{n} \cap \partial_{2R} \Omega)} + (\lambda + \eps)\norm{f_n}_{\ell^2(\X_{n})}^2  + \left(\frac{\delta}{\varepsilon} + \varepsilon + \lambda\right) \norm{f}_{L^2(\Omega)}^2\\
		&\qquad\qquad + K(q) \left(\norm{f_n -f}_{\ell^1(\X_{n})} + \norm{\osc_{B(\delta, \cdot)} f}_{L^1(\Omega)}\right) \norm{f_n}_{\ell^q(\X_{n})} 
	\end{align*}
    for all $f_n \in \ell^2(\X_{n})$ satisfying $\sprod{f_n,\one}{\ell^2(\X_{n})}=0$, has probability at least $1-C_1 n\exp(-C_2 n\delta^d \lambda^2)$. Here, $u_{n,\eps} \in \ell^2_0(\X_{n})$ is the unique minimizer of $\ene(\cdot;f_n)$.
\end{proposition}

The general idea in the proof of \cref{thm:main_smooth} is that the $H^1(\X_{n})$ norm of $u-u_{n,\varepsilon}$ can be controlled by the difference of the energies $\ene(u;f_n) - \ene(u_{n,\eps};f_n)$, by using the quadratic nature of this energy, a discrete Poincaré inequality, and some estimates on the discrete mean value $(u)_{\deg}$ of $u$. 
However, as the proof requires several prerequisites, we postpone it to a later section and begin with the proofs of the two propositions above.

In \cref{sec:continuum_limits} we derived regularity estimates for minimizers of the energy $I(\cdot; f, \rho)$ in $H^1_\rho(\Omega)$ (setting $\varrho \equiv \rho^2$). Thus we obtain the following
%
% \todo[inline]{We don't want to assume $\int_\Omega f \rho \dx = 0$, it make the combination results harder in the final section. This is why I restated all the regularity results without this assumption (as minimizers of the variational problem).\red M:Done }
% \todo[inline]{Need to replace $\|u\|_{L^\infty(\Omega)}$ with the appropriate bound below. \red M:Done}
% \todo[inline]{There used to be two parametes $\lambda_1$ and $\lambda_2$, which were set to different values in the combination section. \red M:Done}
% \todo[inline]{The way it was before, with the oscillation term, the transportation maps did not need to explicitly show up in the theorem statment. I did this to make the statement easier to understand and apply in the combination section, without needing to reference the transportation maps in that section. \red M:Done}
\begin{corollary}[Simplified continuum limit for bounded data]\label{thm:main_smooth_simplified}
    Under the assumptions of \cref{thm:main_smooth}, we have
	\begin{align*}
		\|u - u_{n,\epsilon}\|_{H^1(\X_{n})}^2 
        &\lesssim K(q) \left(\norm{f}_{L^q(\Omega)}  +  \norm{f_n}_{\ell^q(\X_{n})}\right) \norm{f - f_n}_{\ell^1(\X_{n})} + \norm{f}_{L^\infty(\Omega)}^2 \lambda_1\\
        &  + \norm{f}_{L^\infty(\partial_{4\eps} \Omega)}^2\eps + \left(\frac{\delta}{\varepsilon} + \varepsilon + \lambda_1^2 + \lambda_2\right) \norm{f}_{L^2(\Omega)} ^2 + (\eps + \lambda_2) \norm{f_n}_{\ell^2(\X_{n})}^2\\
        & + \norm{f_n}_{\ell^2(\X_{n})} \norm{f_n}_{\ell^2(\X_{n} \cap \partial_{2R} \Omega)} + K(q) \norm{\osc_{B(\delta;\cdot)}f}_{L^1(\Omega)} \norm{f_n}_{\ell^q(\X_{n})}
	\end{align*}
    for all $f_n \in \ell^2(\X_{n})$ satisfying $\sprod{f_n,\one}{\ell^2(\X_{n})}=0$ has probability at least $1-4n \exp(-C_1 n \eps^d \lambda_1^2) - C_2n\exp(-C_3 n\delta^d \lambda_2^2)$. Here, $u_{n,\eps} \in \ell^2_0(\X_{n})$ is the unique minimizer of $\ene(\cdot;f_n)$.
\end{corollary}
\begin{proof}
    \cref{sec:smoothed_poisson} allows us to bound the $u$ dependent quantities from \cref{thm:main_smooth}. By \cref{prop:global_regularity} we have $\norm{u}_{L^\infty(\Omega)} \lesssim K(q) \norm{\rho f}_{L^q(\Omega)}$, while $\Lip(u) \lesssim \norm{\rho f}_{L^\infty(\Omega)}$. Moreover, by \cref{prop:boundary_c01} we have that 
    \begin{equation*}
        \Lip(u;\partial_{2\epsilon}\Omega) 
        \lesssim \norm{\rho f}_{L^\infty(\partial_{4\eps} \Omega)} + \norm{\rho f}_{L^1(\Omega)}.
    \end{equation*}
    Pulling out $\rho_{\max}$ from all these bounds and plugging it into the constant hidden in $\lesssim$ concludes the proof.
\end{proof}
	
\subsection{Discrete to local convergence rate\texorpdfstring{, proof of \cref{prop:discrete_continuum_energy_continuum_function}}{}}

A central object, which is related to the expectation of the discrete energy $\ene(u;f)$ and depends on two parameters $\eps, \delta>0$, is the following non-local energy functional $\ied \colon L^2(\Omega) \to \R$, given by
\begin{equation}\label{eq:non-local_functional_eps}
	\ied(u;f,\rho)= \Ied1(u;\rho) + \I2(u;f,\rho),
\end{equation}
where $f \in L^2(\Omega)$ and $\rho \in L^\infty(\Omega)$ are fixed. The \emph{non-local Dirichlet energy} $\Ied1$ is given by 
\begin{equation}\label{eq:non-local_dirichlet_eps}
	\Ied1(u;\rho) = \frac{1}{2\sigma_{\eta} \varepsilon^{2}} \int_{\Omega} \int_{\Omega} \eta_{\varepsilon} (|x -y| + 2\delta)\abs{u(x)-u(y)}^{2} \rho(x)\rho(y)\d x \d y.
\end{equation}
For $\delta=0$ we write $\ie=I_{\eps,0}$ as well as $\Ie1=I^{(1)}_{\eps,0}$. Since we will have occasion to use different choices for $f$ and $\rho$ we make the notational dependence explicit.

First, we estimate the difference of the non-local and local energies. 

\begin{lemma}[Local to non-local]\label{lem:Local-to-non-local-Convergence}
	Let $u \in H^1(\Omega)$ be Lipschitz, $f \in L^\infty(\Omega)$ and $\epsilon > 0$. Then,
	\begin{equation}
			\ie(u;f,\rho) - I(u;f,\rho) \lesssim \left( \I1(u;\rho)  + 
			\Lip(u;\partial_{2\epsilon})^2
			%\norm{\nabla u}_{L^\infty(\partial_{2\epsilon}\Omega)}^2
            \right) \eps.
		\end{equation} 
\end{lemma}

\begin{proof}
	We first write the non-local Dirichlet energy as
	\begin{equation}
			\Ie1(u;\rho) =  A +B,
		\end{equation}
	where
	\[A = \frac{1}{2\sigma_{\eta} \varepsilon^{2}} \int_{\Omega_\eps} \int_{B(x,\epsilon)} \eta_{\varepsilon} (|x -y|)\abs{u(x)-u(y)}^{2} \rho(x)\rho(y)\d y \d x,\]
	and
	\[B = \frac{1}{2\sigma_{\eta} \varepsilon^{2}} \int_{\partial_\eps\Omega} \int_{B(x,\epsilon)\cap \Omega} \eta_{\varepsilon} (|x -y|)\abs{u(x)-u(y)}^{2} \rho(x)\rho(y)\d y \d x,\]
	and we bound $A$ and $B$ separately. We first focus on estimating $A$. Let $x\in \Omega_\eps$ and $y\in B(x,\epsilon)$. Since $B(x,\epsilon)\subset \Omega$, the line segment between $x$ and $y$ belongs to $\Omega$. Therefore we can use Jensen's inequality to obtain
	\begin{align*}
			\abs{u(x)-u(y)}^{2} &= \left(\int^{1}_{0} \frac{\d}{\d t}u(x+t(y-x))\d t \right)^{2} \\
			&= \left( \int^{1}_{0} \nabla u(x+t(y-x))\cdot (y-x)\d t \right)^{2} 
			\\
			&\leq 
			\int^{1}_{0} \abs{\nabla u(x+t(y-x))\cdot (y-x)}^{2}\d t.
		\end{align*}
	Since $|x-y|\leq \epsilon$ we also have
	\begin{equation}\label{eq:rhoxy}
			\rho(y) \leq \rho(x) + \Lip(\rho)\eps \leq \rho(x)\left( 1 + \frac{\Lip(\rho)}{\rho_{\min}}\eps\right) = \rho(x)(1+C\epsilon).
		\end{equation}
	Therefore we can bound $A$ as 
	\begin{align*}
			A &\leq \frac{1+C\epsilon}{2\sigma_{\eta} \varepsilon^{2}} \int_{\Omega_\eps} \int_{B(x,\epsilon)} \eta_{\varepsilon} (|x -y|)\int^{1}_{0} \abs{\nabla u(x+t(y-x))\cdot (y-x)}^{2} \d t\, \rho(x)^2\d y \d x\\
			&=\frac{1+C\epsilon}{2\sigma_{\eta} \varepsilon^{2}}\int_0^1 \int_{\Omega_\eps} \int_{B(x,\epsilon)} \eta_{\varepsilon} (|x -y|) \abs{\nabla u(x+t(y-x))\cdot (y-x)}^{2} \d y\, \rho(x)^2\d x\d t\\
			&=\frac{1+C\epsilon}{2\sigma_{\eta} \varepsilon^{2}}\int_0^1 \int_{\Omega_\eps} \int_{B(0,\epsilon)} \eta_{\varepsilon} (|z|) \abs{\nabla u(x+tz)\cdot z}^{2} \d z\, \rho(x)^2\d x\d t\\
			&=\frac{1+C\epsilon}{2\sigma_{\eta} \varepsilon^{2}}\int_0^1 \int_{B(0,\epsilon)} \eta_{\varepsilon} (|z|)\int_{\Omega_\eps}  \abs{\nabla u(x+tz)\cdot z}^{2}\rho(x)^2 \d x \d z\d t\\
			&=\frac{1+C\epsilon}{2\sigma_{\eta} \varepsilon^{2}}\int_0^1 \int_{B(0,\epsilon)} \eta_{\varepsilon} (|z|)\int_{\Omega_\eps+tz}  \abs{\nabla u(y)\cdot z}^{2}\rho(y-tz)^2 \d y \d z\d t\\
			&\leq\frac{(1+C\epsilon)^2}{2\sigma_{\eta} \varepsilon^{2}}\int_{B(0,\epsilon)} \eta_{\varepsilon} (|z|)\int_{\Omega}  \abs{\nabla u(y)\cdot z}^{2}\rho(y)^2 \d y \d z\\
			&\leq\frac{1+C\epsilon}{2\sigma_{\eta} \varepsilon^{2}}\int_{\Omega}\int_{B(0,\epsilon)} \eta_{\varepsilon} (|z|)  \abs{\nabla u(y)\cdot z}^{2} \d z \, \rho(y)^2 \d y\\
			&=\frac{1+C\epsilon}{2}\int_{\Omega}\abs{\nabla u(y)}^2 \rho(y)^2 \d y = (1+C\epsilon)\I1(u;\rho),
		\end{align*}
	where we used \labelcref{eq:sigma_eta_identity} in the last line, and the constant $C$ was increased between lines.
	Furthermore, we have 
	\begin{align*}
			B &=
			\frac{1}{2\sigma_{\eta} \varepsilon^{2}} \int_{\partial_\eps\Omega} \int_{B(x,\epsilon)\cap \Omega} \eta_{\varepsilon} (|x -y|)\abs{u(x)-u(y)}^{2} \rho(x)\rho(y)\d y \d x
			\\
			&\leq \frac{\rho_{\max}^2}{2\sigma_{\eta}}  \int_{\partial_\eps\Omega} \int_{B(x,\epsilon)\cap \Omega} \eta_{\varepsilon} (|x -y|) \d y \Lip(u;B(x,\eps)\cap\Omega)^2 \d x\\
			%&\leq \frac{C_\Omega\rho_{\max}^2}{2\sigma_{\eta}}  \int_{\partial_\eps\Omega} \int_{B(x,\epsilon)\cap \Omega} \eta_{\varepsilon} (|x -y|) \d y \norm{\nabla u}_{L^\infty(B(x,\eps)\cap\Omega)}^2 \d x\\
			&\lesssim 
			%\norm{\nabla u}_{L^\infty(\partial_{2\epsilon}\Omega)}^2
			\Lip(u;\partial_{2\epsilon}\Omega)^2
			\int_{\partial_\eps\Omega} \int_{\R^d} \eta(|z|) \d z \d x\\
			&\lesssim 
			%\norm{\nabla u}_{L^\infty(\partial_{2\epsilon}\Omega)}^2
			\Lip(u;\partial_{2\epsilon}\Omega)^2
			|\partial_{2\epsilon}\Omega|
			\\
			%&\lesssim \norm{\nabla u}_{L^\infty(\partial_{2\epsilon}\Omega)}^2\epsilon,
            &\lesssim \Lip(u;\partial_{2\epsilon})^2\epsilon,
		\end{align*}
	where we utilized that 
    %$\Lip(u;B(x,\eps)\cap\Omega)\leq C_\Omega \norm{\nabla u}_{L^\infty(B(x,\eps)\cap\Omega)}$ for $x\in\partial_\eps\Omega$ and 
    $\abs{\partial_\eps\Omega}\lesssim \epsilon$.%, since $\Omega$ (and $B(x,\eps)\cap\Omega$) are Lipschitz domains. 
    Combining the inequalities for $A$ and $B$ yields
	\[\ie(u;f,\rho) - I(u;f,\rho) = \Ie1(u;\rho) - \I1(u;\rho) \lesssim \left(\I1(u;\rho) + 
	\Lip(u;\partial_{2\epsilon}\Omega)^2
	%\norm{\nabla u}_{L^\infty(\partial_{2\epsilon}\Omega)}^2
	\right)\epsilon,\]
	which completes the proof.
\end{proof}

We now estimate the difference between the discrete and non-local energies
%\todo[inline]{In this lemma the event where the estimate is true depends on the choice of $f$ and $u$. Is this a problem later on? It might require some union bound. In contrast, all other probabilistic estimates only need that the transport map estimates hold. J: This fine since it applies to the continuum $f$ and $u$, which do not involve the graph probability space in any way. }
\begin{lemma} [non-local to discrete]\label{lem:non-local-to-discrete-energy}
    Let $u \in H^1(\Omega)$ be Lipschitz and $f:\Omega \to \R$ be Borel-measurable and bounded. There exists a positive constant $C_1(\eta(0), \rho_{\max}, \sigma_\eta)$, such that for any $0 < \lambda \leq 1$
	\begin{equation}
			\left|\ene(u;f) - \ie(u;f,\rho)\right| \leq \left(\Lip(u)^{2} + \|fu\|_{L^\infty(\Omega)}\eps^{\frac{d}{2}}\right)\lambda
		\end{equation}
	hold with probability at least $1-4\exp\left(-C_1n\epsilon^{d}\lambda^{2}\right)$.
\end{lemma}
\begin{proof}
	We bound $|\Ie1(u;\rho)-\Ene1(u)|$ by using Bernstein's inequality for $U$-statistics and the term $|\I2(u;f,\rho)-\Ene2(u;f)|$ by applying Hoeffding's inequality. 
	The final bound for $\abs{\ene(u;f) - \ie(u;f,\rho)}$ then follows from the triangle inequality.
	
	\textbf{Step 1:} 
	We begin by defining the $U$-statistic
	\begin{equation}
			U_{n} = \frac{1}{n(n-1)}\sum_{i\ne j} g(x_{i}, x_{j})
		\end{equation}
	where
	\begin{equation}
			g(x,y) = \frac{\eta_{\epsilon}\left(|x-y|\right)}{2\sigma_\eta \epsilon^2}(u(x)-u(y))^2,
		\end{equation}
	and we observe that $\Ene1(u) = U_{n}$. We also note that the expectation of $g(x,y)$ is $\Ie1(u;\rho)$; indeed
	\[\mathbb{E}[g(x,y)] = \frac{1}{2\sigma_\eta}\int_{\Omega}\int_{\Omega}\eta_{\epsilon}\left(|x-y|\right)\abs{u(x)-u(y)}^{2}\rho(x)\rho(y)\d x\d y =\Ie1(u;\rho).\]
	Since $u$ is  Lipschitz and $\eta$ is nonincreasing, we have $|\eta_{\epsilon}| \leq \eta(0)\epsilon^{-d}$, and $\eta_{\epsilon}(|x-y|) = 0$ for $|x - y| > \epsilon$. Therefore we obtain
	\[b \defeq   \sup_{x,y \in \Omega}|g(x,y)|=\sup_{x,y\in \Omega}\frac{\eta_{\epsilon}\left(|x-y|\right)}{2\sigma_\eta \epsilon^2}(u(x)-u(y))^2\leq \frac{\eta(0)}{2\sigma_\eta}\Lip(u)^2\eps^{-d}.\]
	Furthermore, the variance is bounded by
	\begin{flalign*}
			\sigma^{2} \defeq   \mathbb V\left(g(x,y)\right)&\leq  \mathbb E[g(x,y)^2]\\
			&\leq \frac{1}{4\sigma_\eta^2\epsilon^4}\int_{\Omega}\int_{\Omega}\eta_{\epsilon}\left(|x-y|\right)^{2} (u(x)-u(y))^4\rho(x)\rho(y)\d x\d y \\
			&\leq \frac{\eta(0)\rho_{\max}}{4\sigma_\eta^2}\Lip(u)^{4}\epsilon^{-d}\int_{\Omega}\int_{B(y,\epsilon)}\eta_\eps(|x-y|)\d x \rho(y)\d y \\
			&= \frac{\eta(0)\rho_{\max}}{4\sigma_\eta^2}\Lip(u)^{4}\epsilon^{-d},
		\end{flalign*}
	since $\rho$ is a probability density and $\eta_\eps$ has unit mass. Bernstein's inequality for $U$-statistics reads \cite[Theorem 5.15]{CalculusofVariationsLN}:
	\begin{align*}
			\mathbb{P}\left(\abs{U_{n}  - \Ie1(u;\rho)} \geq t\right) \leq 2\exp\left(-\frac{nt^2}{6\left(\sigma^2+\tfrac{1}{3}bt\right)}\right)
			\quad 
			\text{for all } t>0.
		\end{align*}
	Choosing $t\defeq  \Lip(u)^2\lambda_1$ with $0<\lambda_1\leq 1$ and using the bounds for $b$ and $\sigma$ we get
	\begin{equation}
			\mathbb{P}\left(\left|  U_{n}  - \Ie1(u;\rho) \right|\geq \Lip(u)^{2}\lambda_1\right) \leq 2\exp\left(-Cn\epsilon^{d}\lambda_1^{2}\right),
		\end{equation}
	where $C>0$ depends on $\eta(0)$, $\rho_{\max}$ and $\sigma_\eta$.  
	
	\textbf{Step 2:} 
	Now, we define the random variable $Y_{i} = -f(X_{i})u(X_{i})$---which is well-defined since $u$ is continuous and $f$ is Borel measurable---and observe that
	\[\Ene2(u;f) = -\frac{1}{n}\sum_{i=1}^{n}f(X_{i})u(X_i) = \frac{1}{n}\sum_{i=1}^{n}Y_{i}.\]
	The mean of $Y_i$ is given by
	\[\mu \defeq   \mathbb E(Y_i) = -\int_\Omega f u \rho \d x = \I2(u;f,\rho),\]
	and we have
	\[|Y_i - \mu| = \left|f(X_i)u(X_i) - \int_\Omega f u \rho \d x \right| \leq 2\|fu\|_{L^\infty(\Omega)} =:b. \]
	Therefore, the Hoeffding inequality reads \cite[Theorem 5.9]{CalculusofVariationsLN}:
	\begin{equation}
			\mathbb{P}\left(\abs{\Ene2(u;f) - \I2(u;f,\rho)} \geq t \right) \leq 2\exp\left(-\frac{nt^{2}}{2b^2} \right)\qquad\forall t>0.
		\end{equation}
	We set $t=\frac{1}{2}\lambda_2b = \|fu\|_{L^\infty(\Omega)}\lambda_2$, where $\lambda_2>0$, to obtain that
	\[\abs{\Ene2(u;f) - \I2(u;f,\rho)} \leq \|fu\|_{L^\infty(\Omega)}\lambda_2\]
	with probability at least $1-2\exp\left( -\frac{1}{8}n\lambda_2^2\right)$. 
	
	\textbf{Step 3:} 
	Using the results of Steps 1 and 2, and a union bound, we obtain
	\begin{align*}
			\left|\ene(u;f) - \ie(u;f,\rho)\right| &\leq \left|\Ene1(u)-\Ie1(u;\rho)\right| + \left|\Ene2(u;f) -\I1(u;f,\rho)\right| \\
			&\leq \Lip(u)^{2}\lambda_1 + \|fu\|_{L^\infty(\Omega)}\lambda_2
		\end{align*}
	with probability at least $1-2\exp\left( -\frac{1}{8}n\lambda_2^2\right)-2\exp\left(-Cn\epsilon^{d}\lambda_1^{2}\right)$. We now set $\lambda_1=\lambda$ and $\lambda_2 = \eps^{\frac{d}{2}}\lambda$, to match the probabilities and complete the proof.
\end{proof}

\begin{proof}[Proof of \cref{prop:discrete_continuum_energy_continuum_function}]
    %Since $(f)_\rho = 0$, we have that $u \in H^1_\rho(\Omega)$ is a weak solution of \labelcref{eq:def_continuum_pde_limit_formulation}. We can thus apply the regularity result \cref{prop:global_regularity} and obtain that $u \in C^{1,\alpha}(\Omega)$ for some $\alpha \in (0,1)$. In particular $u \in \Lip(\Omega)$.
    If $u$ is not Lipschitz, then the right hand side of \cref{prop:discrete_continuum_energy_continuum_function} is defined to be infinite, so the result trivially holds. Assume thus that $u$ is Lipschitz (as follows from some boundary regularity assumptions, cf. \cref{prop:global_regularity}).
    
    We consider a realization of the graph $\X_{n}$ such that the results of \cref{lem:non-local-to-discrete-energy} hold. This has probability at least $1-4\exp(-Cn\eps^d \lambda^2)$.

    Note that for any $f_n \in \ell^2(\Omega)$ we have that
    \begin{equation*}
        \abs{\ene(u;f_n) - \ene(u;f)} 
        \leq \norm{u}_{\ell^\infty(\X_{n})} \norm{f - f_n}_{\ell^1(\X_{n})} 
        \leq \norm{u}_{L^\infty(\Omega)} \norm{f - f_n}_{\ell^1(\X_{n})}. 
    \end{equation*}
    We can combine \cref{lem:Local-to-non-local-Convergence,lem:non-local-to-discrete-energy} to obtain
	\begin{equation*}
		{\ene(u;f) - I(u;f,\rho)} 
		\lesssim \left(\Lip(u)^2 + \norm{fu}_{L^\infty(\Omega)} \varepsilon^{\frac d2}\right) \lambda 
        + \left(I^{(1)}(u;\rho) + 
            %\norm{\nabla u}_{L^\infty(\partial_{2\varepsilon}\Omega)}^2 
            \Lip(u;\partial_{2\epsilon}\Omega)^2
        \right)\varepsilon
	\end{equation*}
	Moreover, the Euler--Lagrange equation \labelcref{eq:continuum_pde_EL} gives
	\begin{equation*}
		\forall v \in H^1_\rho(\Omega) \colon \quad \int_\Omega \rho^2 ~\nabla u \cdot \nabla v \dx = \int_\Omega f v \rho \dx.
	\end{equation*}
	Choosing $v \equiv u$ and applying the Poincaré inequality gives the estimate
     \begin{equation*}
         I^{(1)}(u; \rho) 
         = \frac12 \int_\Omega \abs{\nabla u}^2 \rho^2 \dx
         \lesssim \norm{f}_{L^2(\Omega} \norm{\nabla u}_{L^2(\Omega)} 
         \lesssim \norm{f}_{L^2(\Omega)} \sqrt{I^{(1)}(u; \rho)}
     \end{equation*}
     concluding the proof.
\end{proof}

\subsection{Local to discrete convergence rate\texorpdfstring{, proof of \cref{prop:continuum_energy_continuum_solution_discrete_energy_discrete_solution}}{}}

\subsubsection{Transportation maps}

In order to extend discrete functions on the graph to continuum functions while controlling the graph Dirichlet energies, we use an approach similar to the transportation map approach originally developed in \cite{garcia2016continuum}. The original idea in \cite{garcia2016continuum} is to use a \emph{transportation map} $T:\Omega \to \X_{n}$ that pushes forward the data distribution measure $\mu\defeq  \rho \d x$ onto the empirical data measure $\mu_n\defeq   \frac{1}{n}\sum_{i=1}^n \delta_{x_i}$. That is, $T_\# \mu = \mu_n$, which simply means that $\mu(T^{-1}(x_i)) = \frac{1}{n}$ for all $i$, or rather
\[\int_{T^{-1}(x_i)}\rho \d x = \frac{1}{n}  \ \ \text{for all } i=1,\dots,n.\] 
Given such a transportation map $T$, we can easily convert discrete summations into continuous integrals, since the definition of the push forward implies that
\begin{equation}\label{eq:transportation_ident}
	\frac{1}{n}\sum_{i=1}^n u(x_i) = \int_{\Omega} u(T(x_i)) \rho(x) \d x \ \ \text{for all } u\in \l2.
\end{equation}
This appears similar to the kinds of estimates one obtains from concentration of measure. However, there are important differences: (1) there are no error terms in \labelcref{eq:transportation_ident}, (2) the identity holds uniformly over all $u\in \l2$, once one constructs the transportation map $T$, and (3) the right hand side of \labelcref{eq:transportation_ident} is \emph{not} the expectation of the left hand side (since it involves $u\circ T$), as it would be in an application of the Bernstein or Hoeffding bounds.

In order to make sure that $u\approx u\circ T$, so the right hand side of \labelcref{eq:transportation_ident} is close to the expectation of the left hand side, we require that the transportation map $T$ does not move points too far (and that $f$ has some type of monotonicity or continuity). This naturally leads to the optimal transportation problem 
\begin{equation}\label{eq:optimal_transportation}
	\inf_{T_\# \mu = \mu_n} \sup_{x\in \Omega}|T(x)-x|.
\end{equation}
Thus, we seek the transportation map $T$ that moves points in the worst case by the smallest possible distance. This is called an $L^\infty$-optimal transportation problem. Let $\delta_n>0$ denote the infimal distance in \labelcref{eq:optimal_transportation}. It was shown in \cite{trillos2015rate} that optimal transportation maps exist in dimension $d\geq 3$ with 
\begin{equation}\label{eq:optimal_scaling}
	\delta_n \sim \left( \log n/n\right)^{\frac{1}{d}}.
\end{equation}
Up to constants this is optimal, since this agrees with the worst case distance from a point to its nearest neighbor in an \emph{i.i.d.}~point cloud (which is a natural lower bound for $\delta_n$). However, in dimension $d=2$, there are some topological obstructions to the arguments used in \cite{trillos2015rate} and it is only possible to show the existence of transportation maps with $\delta_n \sim (\log n)^{\frac{1}{4}}\left( \log n/n\right)^{\frac{1}{2}}$, which is suboptimal by a logarithmic factor. 

In \cite{calder2022improved} a simpler transportation map approach was developed, which yields the optimal scaling \labelcref{eq:optimal_scaling} in all dimensions $d\geq 1$, and does not require solving an $L^\infty$-optimal transportation problem, which greatly simplifies the proofs. The key idea is to relax the condition that $T_\#\mu = \mu_n$ slightly, and instead ask that $T_\#\tilde\mu=\mu_n$, where $\tilde\mu$ is some measure that is ``close'' to $\mu$ in a sense that will be made clear below. This approach was developed in \cite{calder2022improved} for data sampled from a closed manifold without boundary. Here, we are working on a Euclidean domain with boundary, and there are some additional details to verify. We state the main results in this section and include the proofs in the appendix for completeness. We recall that $x_1,\dots,x_n$ is an \emph{i.i.d.} sequence with density $\rho$.
\begin{theorem}\label{thm:transportation} 
    There exists constants $C=C(\Omega)>0$ and $c=c(\Omega,\rho_{\min})>0$ such that for any $n^{-1/d} < \delta \leq \frac{\rho_{\min}}{8\Lip(\rho)}$ there exists a probability density $\rho_\delta\in L^\infty(\Omega)$ and a measurable map $T_\delta:\Omega\to \X_{n}$ such that for any $0 \leq \lambda \leq\frac{\rho_{\min}}{8\rho_{\max}}$ the following hold with probability at least  $1-Cn\exp(-cn\delta^{d}\lambda^{2})$: %\todo{Are the first two statements actually probabilistic or are they true for any realization?}
	\begin{thmenum}
		\item ${T_\delta}_\#(\rho_\delta \d x) = \mu_n$, \label{thm:transportation_is_pushforward}
		\item $T_\delta(x_i)=x_i$ for all $i=1,\dots,n$, \label{thm:transportation_point_to_point}
		\item $|T_\delta(x) - x| \leq \delta$ for all $x\in \Omega$, and \label{thm:transportation_distance_T}
		\item $|\rho(x) - \rho_{\delta}(x)| \leq \Lip(\rho)\delta + \rho_{\max}\lambda$ for all $x\in \Omega$. \label{thm:transportation_distance_rho}
	\end{thmenum}
\end{theorem}
\begin{remark}\label{rem:Lipschitz}
	We remark that for the probability $1-Cn\exp(-cn\delta^{d}\lambda^{2})$ to be close to $1$, we need to choose $\delta$ so that $n\delta^d \gtrsim \log(n)$, which is equivalent to $\delta \sim \left( \log n/n\right)^{\frac{1}{d}}$. This is the same as the optimal scaling obtained by the transportation map approach used in \cite{garcia2016continuum} for $d\geq 3$, but sharper when $d=2$.  %Inspecting the proof in \cref{app:transportation}, we also notice that \cref{thm:transportation} requires only that the boundary of $\Omega$ is Lipschitz, and the $C^{1,1}$ assumption is not used. 
\end{remark}
To simplify notation, we introduce the extension operator $E_{\delta} : \ell^{2}(\mathcal{X}_{n}) \longrightarrow L^{2} (\Omega)$ by
%begin{definition}
%	We define the extension operator $E_{\delta} : \ell^{2}(\mathcal{X}_{n}) \longrightarrow L^{2} (\Omega)$ by
	\begin{equation}\label{eq:def_extension_operator}
		E_{\delta}u(x) = (u\circ T_{\delta})(x).
	\end{equation}
%\end{definition}
The extended function $E_\delta u$ is piecewise constant taking the value $u(x_i)$ on the set $T_\delta^{-1}(\{x_i\})$. \cref{thm:transportation_is_pushforward} and the definition of the extension operator allow us to write
\begin{equation}\label{eq:Td_ext}
	\frac{1}{n}\sum_{i=1}^n u(x_i) = \int_{\Omega} E_\delta u \,\rho_\delta \d x \ \ \text{for all } u\in \l2.
\end{equation}

\begin{lemma}[Discrete to non-local]\label{lem:Discrete-to-non-local-1}
    Fix $n^{-1/d} < \delta \leq \frac{\rho_{\min}}{8\Lip(\rho)}$. Let $\rho_\delta \in L^\infty(\Omega)$ be the probability density from \cref{thm:transportation}.
    Consider a realization of the random graph $\X_{n}$, such that for the transport map $T_\delta \colon \Omega \to \X_{n}$ \cref{thm:transportation_is_pushforward,thm:transportation_point_to_point,thm:transportation_distance_T} hold. 
	%Then, there exist constants $C=C(\Omega)>0$ and $c=c(\Omega,\rho_{\min})>0$, such that the event that
 
    Then, for all $u,f\in \ell^{2}(\mathcal{X}_{n})$, we have
	\begin{equation}\label{eq:discrete-to-non-local}
		\Ied1(E_\delta u;\rho_\delta) \leq \Ene1(u)  \ \ \text{and} \ \ \I2(E_\delta u;E_\delta f,\rho_\delta) = \Ene2(u;f).
	\end{equation} 
	%holds for all $u,f\in \ell^{2}(\mathcal{X}_{n})$ has probability at least $1-Cn\exp(-cn\delta^{d})$.
\end{lemma}
\begin{proof}
    %Consider a realization of the graph $\X_{n}$ such that $\abs{T_\delta(x) - x} < \delta$ for all $x \in \Omega$. This has probability at least $1-Cn\exp(-cn\delta^{d})$ (Take $\lambda = \frac{\rho_{\min}}{8\rho_{\max}}$ in \cref{thm:transportation}).

    To simplify the notation, we will write $u_\delta \defeq   E_\delta u$ and $f_\delta \defeq   E_\delta f$.
    %
	%For such a realization we have
    Under the assumptions of $T_\delta$, we have
	\[|T_\delta(x) - T_\delta(y)| \leq 2\delta \ \ \text{for all } x,y\in \Omega.\]
	Since $\eta$ is non-increasing, this implies
	\begin{align*}
		\Ene1(u) &= \frac{1}{2\sigma_{\eta}n(n-1)\epsilon^{2}}\sum_{i,j=1}^n\eta_{\varepsilon}\left(|x_i-x_j|\right)\abs{u(x_i)-u(x_j)}^{2}\\
		&=\left( \frac{n}{n-1}\right)\frac{1}{2\sigma_{\eta} \varepsilon^{2}} \int_{\Omega} \int_{\Omega} \eta_{\varepsilon} (|T_{\delta}(x)-T_{\delta}(y)|)\abs{u_{\delta}(x)-u_{\delta}(y)}^{2} \rho_{\delta}(x)\rho_{\delta}(y)\d x \d y \\
		&\geq \frac{1}{2\sigma_{\eta} \varepsilon^{2}} \int_{\Omega} \int_{\Omega} \eta_{\varepsilon} (|x -y| + 2\delta)\abs{u_{\delta}(x)-u_{\delta}(y)}^{2} \rho_{\delta}(x)\rho_{\delta}(y)\d x \d y
        = \Ied1(u_\delta;\rho_\delta).
	\end{align*}
	We also observe that 
	\[\Ene2(u;f) =  -\frac{1}{n} \sum_{i=1}^n f(x_{i})u(x_{i}) =  -\int_{\Omega} u_{\delta} f_{\delta} \rho_\delta\d x = \I2(u_\delta;f_\delta,\rho_\delta), \]
	which completes the proof.
\end{proof}

\subsubsection{Smoothing and Stretching}

We now need to estimate the local energy $I$ in terms of the non-local energy $\ied$. The main difficulty here is that $u_\delta$ is a piecewise continuous function, and hence it is discontinuous and does not belong to $H^1(\Omega)$ (hence $I(u_\delta) = \infty$). To handle this, we mollify $u_\delta$ slightly so that the mollified version belongs to $W^{1,\infty}(\Omega)$ while at the same time preserving control between the non-local and local Dirichlet energies. This requires a very careful choice of the mollification kernel, given below. The techniques in this section are similar to those used in \cite{garcia2020error,calder2022improved}, however, there are some modifications made to simplify the proofs, and to handle the boundary of the domain.  

For $\epsilon,\delta>0$ we define the mollification kernel
\begin{equation*}
	\psi_{\epsilon, \delta}(x) = \frac{1}{\sigma_{\eta, \delta/\epsilon}\varepsilon^{2}} \int^{\infty}_{|x|}\eta_{\epsilon}(s+2\delta)s\d s,
\end{equation*}
where for $t>0$ the constant $\sigma_{\eta,t}$ is given by
\begin{equation}\label{eq:sigma-definition}
	\sigma_{\eta, t}= \int_{\R^{d}}|z_{1}|^{2}\eta \left(|z|+2t\right)\d z.
\end{equation}
The kernel $\psi_{\epsilon,\delta}$ is clearly nonnegative, radially symmetric, and is supported in $B(0,\eps-2\delta)$. It also has a scaling identity $\psi_{\epsilon,\delta}(x) = \frac{1}{\epsilon^d}\psi_{1,\delta/\epsilon}\left(x/\epsilon\right)$, which is readily verified. Regarding the constant $\sigma_{\eta,t}$, we first note that $\sigma_{\eta,0}=\sigma_\eta$. It is also important to note that we can replace $z_1$ by $z\cdot \xi$ for any unit vector $\xi$, as well as scale by $\eps$ to obtain 
\begin{equation}\label{eq:scaling_sigma_eta}
	\sigma_{\eta, t}=\int_{\R^{d}}|z\cdot \xi|^{2}\eta \left(|z|+2t\right)\d z = \frac{1}{\epsilon^2}\int_{\R^{d}}|x\cdot \xi|^{2}\eta_\eps \left(|x|+2t\eps\right)\d x.
\end{equation}
In particular, by averaging both sides over $\xi=e_i$, $i=1,\dots,d$, so that $\xi\cdot z=z_i$, we also obtain
\begin{equation}\label{eq:sigma_alt}
	\sigma_{\eta, t}= \frac{1}{d}\int_{\R^{d}}|z|^{2}\eta \left(|z|+2t\right)\d z.
\end{equation}
We first estimate how how the constants $\sigma_{\eta,t}$ depend on $t$. 
\begin{lemma}\label{lem:sigma_eta_bound}
	For any $t\geq 0$ we have
	\begin{equation}\label{eq:lip}
		\sigma_{\eta} - 4Ct\leq \sigma_{\eta,t} \leq \sigma_{\eta}, \ \ \text{where} \ C=\int_{\R^d}\eta(|z|)|z|\d z.
	\end{equation}
\end{lemma}
The proof of \cref{lem:sigma_eta_bound} is given in \cref{app:mollification}. We also verify that $\psi_{\eps,\delta}$ is indeed a mollification kernel on the interior of the domain, in that it integrates to unity. 
Define
\begin{equation}\label{eq:ted}
	\ted(x) = \int_{\Omega} \psi_{\epsilon, \delta}(x-y) \d y.
\end{equation}
\begin{lemma}\label{lem:psi_moll}
	Let $\epsilon,\delta>0$ with $\delta/\epsilon \leq C(\eta)$. Then the following hold:
	\begin{enumerate}[label=(\roman*)]
		\item There exists $0 < C \leq 1$, depending only on $\Omega$ and $\eta$, such that  $C \leq \ted(x)\leq 1$ for all $x\in \Omega$.
		\item $\ted(x)=1$ for all $x\in \Omega_{\epsilon-2\delta}$.
	\end{enumerate}
\end{lemma}
The proof of \cref{lem:psi_moll} is a tedious but straightforward computation, and is also given in \cref{app:mollification}. 

We now define the smoothing operator $\Lambda_{\epsilon, \delta}$, which amounts to convolution with $\psi_{\epsilon, \delta}$ with normalization by $\ted$.
\begin{definition}
	The operator $\Lambda_{\epsilon, \delta}: L^{2}(\Omega) \longrightarrow L^2(\Omega)$ is defined by
	\begin{equation}\label{def:smoothing-operator}
		\Lambda_{\epsilon, \delta}u(x) 
		=\frac{1}{\ted(x)}\int_{\Omega}\psi_{\eps,\delta}(x-y)u(y)\d y.
	\end{equation}
\end{definition}
%\todo[inline]{
	%Since we use the transversal vector fields anyway, I think we can drastically simplify by defining $\Lambda_{\eps,\delta}: L^2(\Omega)\to L^2(\Omega_\eps)$.}

We now show that the smoothing operator indeed increases the regularity of $u$. 
\begin{proposition}\label{prop:grad_psi}
	Let $u\in L^2(\Omega)$ and $w = \Lambda_{\eps,\delta}u$. Then $w\in  W^{1,\infty}(\Omega)$  and
	\begin{equation}\label{eq:grad_psi}
		\nabla w(x) = \frac{1}{\sigma_{\eta, \delta/\epsilon}\ted(x)\epsilon^{2}} \int_{\Omega} \eta_{\epsilon} (|x-y|+2\delta)(y-x)(u(y)-w(x))\d y.
	\end{equation}
\end{proposition}
\begin{proof}
	Thanks to the standing \cref{ass:eta,ass:rho} the kernel $\psi_{\eps,\delta}$ is bounded, and by assumption $u\in L^2(\Omega)\subset L^1(\Omega)$. Hence $w\defeq  \Lambda_{\epsilon,\delta}u\in L^\infty(\Omega_\eps)$. We now write
	\[v(x) = \int_{\Omega}\psi_{\eps,\delta}(x-y)u(y)\d y\]
	so that $w(x) = v(x)/\ted(x)$. For any $x\in \R^d$ we have
	\begin{equation}\label{eq:key_ident}
		\nabla \psi_{\eps,\delta}(x) = -\frac{x}{\sigma_{\eta, \delta/\eps}\eps^2} \eta_\eps(|x|+2\delta).
	\end{equation}
	Therefore we have
	\begin{equation}\label{eq:key_grad}
		\nabla v(x) = \int_{\Omega}\nabla \psi_{\eps, \delta}(x-y)u(y)\d y=\frac{1}{\sigma_{\eta,\delta/\eps}\epsilon^2}\int_{\Omega}\eta_\eps(|x-y| + 2\delta)(y-x)u(y)\d y.
	\end{equation}
	and
	\[\nabla \ted(x) = \frac{1}{\sigma_{\eta,\delta/\eps}\epsilon^2}\int_{\Omega}\eta_\eps(|x-y| + 2\delta)(y-x)\d y.\]
	Therefore we have
	\begin{align*}
		\nabla w(x) &= \frac{1}{\ted(x)}\left(\nabla v(x) - w(x)\nabla \ted(x)\right)\\
		&= \frac{1}{\sigma_{\eta,\delta/\eps}\ted(x)\epsilon^2}\int_{\Omega}\eta_\eps(|x-y| + 2\delta)(y-x)(u(y) - w(x))\d y.
	\end{align*}
	Using again that $\psi_{\eps,\delta}$ is bounded and $u\in L^1(\Omega)$ shows that $\nabla v \in L^\infty(\Omega_\eps)$ and so $v\in W^{1,\infty}(\Omega_\eps)$, which completes the proof.
\end{proof}
\begin{remark}\label{rem:key_psi}
	\cref{prop:grad_psi} shows that $\Lambda_{\eps,\delta}:L^2(\Omega) \to W^{1,\infty}(\Omega)$. It furthermore sheds light on the definition of $\psi_{\eps,\delta}$, which was made precisely so that \labelcref{eq:key_ident,eq:key_grad} hold. We will see below that this will ensure that convolution with $\psi_{\eps,\delta}$ will allow us to pass from the non-local to local Dirichlet energies. 
\end{remark}

It will be important later to control the $L^p$ norm of the smoothing operator over subsets of the domain.
\begin{lemma}\label{prop:smoothing_l1}
	Let $1 \leq p \leq \infty$. For any $u\in L^p(\Omega)$ and $\Omega'\subset \Omega$ we have
	\[\|\Lambda_{\epsilon, \delta}u\|_{L^p(\Omega')} \leq C\|u\|_{L^p((\Omega'+B_\eps)\cap \Omega)},\]
	where $C$ depends only on $p,\Omega$ and $\eta$. 
\end{lemma}
\begin{proof}By computation, given in \cref{app:mollification}.
\end{proof}
We also need $L^2$ control the difference $\Lambda_{\eps,\delta}u - u$. 
\begin{lemma}\label{lem:conv_diff}
	Let $\epsilon,\delta>0$ with $\delta/\epsilon \leq C(\eta)$. There exists constants $C=C(\Omega,\eta)>0$ such that for all $u\in L^2(\Omega)$ we have
	\begin{equation}
		\|\Lambda_{\epsilon, \delta}u - u\|^{2}_{L^2(\Omega)} \leq C \Ied1(u;\one)\epsilon^{2}.
	\end{equation}
\end{lemma}
\begin{proof}
	We assume that
	\begin{equation}\label{eq:eps_delta}
		\eps \geq \frac{8 \delta}{\sigma_\eta}\int_{\R^d}\eta(|z|)|z|\d z.
	\end{equation}
	By \cref{lem:sigma_eta_bound,eq:eps_delta} we have $\sigma_{\eta, \delta/\epsilon} \geq \frac{\sigma_\eta}{2}$. By the monotonicity of $\eta$, for $|x|\leq \epsilon-2\delta$ we have
	\[\psi_{\eps,\delta}(x) = \frac{1}{\sigma_{\eta,\delta/\eps}\epsilon^2}\int_{|x|}^{\epsilon-2\delta} \eta_{\eps}(s + 2\delta)s\d s \leq \frac{\eta_{\eps}(|x| + 2\delta)}{\sigma_{\eta,\delta/\eps}\epsilon}\int_{|x|}^{\epsilon-2\delta} \d s \leq \frac{2}{\sigma_{\eta}}\eta_\eps(|x| + 2\delta),\]
	while for $|x|\geq \epsilon-2\delta$ we have $\psi_{\eps,\delta}(x)=0$. Hence, by Jensen's inequality  and we have
	\begin{align*}
		\abs{\Lambda_{\eps,\delta}u(x)  - u(x)}^{2}  &=\left(\frac{1}{\ted(x)}\int_\Omega \psi_{\eps,\delta}(x-y)(u(y)-u(x)) \right)^2\\
		&\leq\frac{1}{\ted(x)} \int_\Omega \psi_{\eps,\delta}(x-y)\abs{u(x)-u(y)}^2 \d y\\
		&\leq \frac{2}{\sigma_\eta\ted(x)}\int_\Omega \eta_\eps(|x-y|+2\delta)\abs{u(x)-u(y)}^2 \d y.
	\end{align*}
	Therefore we have
	\begin{align}\label{eq:psibound}
		\int_\Omega \ted(x) \abs{\Lambda_{\epsilon,\delta}u(x) - u(x)}^2 \d x &\leq \frac{2}{\sigma_\eta}\int_\Omega\int_\Omega \eta_\eps(|x-y|+2\delta)\abs{u(x)-u(y)}^2 \d x\d y\notag\\
		&= 4 \Ied1(u;\one)\eps^2.
	\end{align}
	The proof is completed by applying \cref{lem:psi_moll} (i) to the left hand side above.
\end{proof}

We now turn to our main result in this section, which compares the local and non-local Dirichlet energies, where one function is smoothed by $\Lambda_{\epsilon,\delta}$.  
\begin{theorem}[non-local to local]\label{thm:non-local_to_Local}
	Let $\epsilon,\delta>0$ with $\delta/\epsilon \leq C(\eta)$. There exists $C=C(\eta)>0$ such that for all $u\in L^2(\Omega)$ and $\Omega'\subset \Omega$ we have
	\begin{align*}
		\int_{\Omega'}|\nabla \Lambda_{\epsilon,\delta} u|^2\rho^2& \ted^2\d x \\
		&\leq \frac{1}{\sigma_{\eta, \delta/\epsilon}\epsilon^{2}}\int_{\Omega'}\int_{\Omega}(1 + \one_{\partial_\epsilon\Omega}(x)) \eta_{\epsilon} (|x-y|+2\delta) \abs{u(y) -u(x)}^{2}\rho(x)^2\d y \d x\notag \\
		&\qquad + \frac{C}{\epsilon^2}\|\rho^2(\Lambda_{\epsilon,\delta}u - u)\|_{L^2(\Omega'\cap\partial_\epsilon\Omega)}^2.
	\end{align*}
\end{theorem}
\begin{proof}
	Let $w = \Lambda_{\epsilon,\delta}u\in W^{1,\infty}(\Omega_\eps)$. 
	Fix $x\in \Omega$ and let $\xi \in \R^{d}$ with $|\xi| = 1$ so that $\nabla w(x)\cdot \xi = |\nabla w(x)|$. Then by \cref{prop:grad_psi} we have
	\begin{align*}
		|\nabla w(x)|\ted(x) &= \frac{1}{\sigma_{\eta, \delta/\epsilon}\epsilon^{2}} \Bigg[\int_{\Omega} \eta_{\epsilon} (|x-y|+2\delta)\left((y-x) \cdot \xi\right)\left(u(y) -u(x)\right)\d y \\
		&\hspace{1.5in} + (u(x) - w(x))\int_{\Omega} \eta_{\epsilon} (|x-y|+2\delta)(y-x) \cdot \xi\d y\Bigg]\\
		&=:A+(u(x)-w(x))B.
	\end{align*}
	By the Cauchy--Schwarz inequality we have
	\begin{align*}
		A^2&\leq \frac{1}{\sigma_{\eta, \delta/\epsilon}^2\epsilon^{4}}\int_{\R^d} \eta_{\epsilon} (|x-y|+2\delta)|(y-x) \cdot \xi|^{2}\d y\int_{\Omega} \eta_{\epsilon} (|x-y|+2\delta) \abs{u(y) -u(x)}^{2}\d y\\
		&= \frac{1}{\sigma_{\eta, \delta/\epsilon}\epsilon^{2}}\int_{\Omega} \eta_{\epsilon} (|x-y|+2\delta) \abs{u(y) -u(x)}^{2}\d y,
	\end{align*}
	where we used \labelcref{eq:scaling_sigma_eta} in the second line above. Since
	\[\int_{\Omega} \eta_{\epsilon} (|x-y|+2\delta)(y-x)\d y = 0\]
	for any $x\in \Omega_\eps$, we have
	\[|B| \leq \frac{\one_{\partial_\epsilon\Omega}(x)}{\sigma_{\eta, \delta/\epsilon}\epsilon^{2}}\int_{\Omega} \eta_{\epsilon} (|x-y|+2\delta)|y-x|\d y \lesssim \frac{1}{\epsilon}\one_{\partial_\epsilon\Omega}(x),\]
	where we used that $\sigma_{\eta,\delta/\eps}\geq \frac{\sigma_\eta}{2}$ by invoking \cref{lem:sigma_eta_bound}. It follows that
	\[|\nabla w(x)|^2\ted(x)^2 \leq A^2 + 2(u(x)-w(x))AB + \abs{u(x)-w(x)}^2B^2.\]
	Since $AB = AB\one_{\partial_\epsilon\Omega}(x)$, we can apply Cauchy's inequality $2ab\leq a^2+b^2$ to the middle term to obtain
	\begin{align*}
		|\nabla w(x)|^2\ted(x)^2 &\leq (1 + \one_{\partial_\epsilon\Omega}(x))A^2 + 2\abs{u(x)-w(x)}^2B^2\\
		&\leq \frac{1 + \one_{\partial_\epsilon\Omega}(x)}{\sigma_{\eta, \delta/\epsilon}\epsilon^{2}}\int_{\Omega} \eta_{\epsilon} (|x-y|+2\delta) \abs{u(y) -u(x)}^{2}\d y \\
		&\qquad+ \frac{C}{\epsilon^2}\abs{u(x)-w(x)}^2\one_{\partial_\epsilon\Omega}(x).
	\end{align*}
	Multiplying by $\rho(x)^2$ and integrating over $\Omega'\subset \Omega$ completes the proof.
\end{proof}

We have two important corollaries of \cref{thm:non-local_to_Local}.
\begin{corollary}\label{cor:nonlocal_to_local}
	Let $\epsilon,\delta>0$ with $\delta/\epsilon \leq C(\eta)$. Then for all $u\in L^2(\Omega)$ we have
	\[\frac{1}{2}\int_{\Omega_\eps}|\nabla \Lambda_{\epsilon,\delta} u|^2\rho^2\d x -\Ied1(u;\rho)  \lesssim \left(\tfrac{\delta}{\epsilon} + \epsilon \right)\Ied1(u;\rho).\]
\end{corollary}
\begin{proof}
	We use \cref{thm:non-local_to_Local} with $\Omega'=\Omega_\eps$. By \cref{lem:psi_moll} (ii) we have $\ted(x)=1$ for $x\in \Omega_\epsilon$, and therefore
	\begin{align*}
		\frac{1}{2}\int_{\Omega_\eps}|\nabla \Lambda_{\epsilon,\delta} u|^2\rho^2\d x 
		&\leq \frac{1}{2\sigma_{\eta, \delta/\epsilon}\epsilon^{2}}\int_{\Omega_\eps}\int_{\Omega}\eta_{\epsilon} (|x-y|+2\delta) \abs{u(y) -u(x)}^{2}\rho(x)^2\d y \d x\\
		&\leq \frac{1+C\eps}{2\sigma_{\eta, \delta/\epsilon}\epsilon^{2}}\int_{\Omega_\eps}\int_{\Omega}\eta_{\epsilon} (|x-y|+2\delta) \abs{u(y) -u(x)}^{2}\rho(x)\rho(y)\d y \d x\\
		&=\frac{\sigma_\eta}{\sigma_{\eta,\delta/\epsilon}}(1+C\eps)\Ied1(u;\rho),
	\end{align*}
	where we used \labelcref{eq:rhoxy} in the penultimate line.  By \cref{lem:sigma_eta_bound} we have
	\[\sigma_\eta \leq \sigma_{\eta,\delta/\epsilon} + \frac{4\delta}{\epsilon}\int_{\R^d}\eta(|z|)|z| \d z.\]
	Employing this and using the restriction \labelcref{eq:eps_delta} to ensure that $\sigma_{\eta,\delta/\epsilon}\geq \frac{\sigma_\eta}{2}>0$ completes the proof.
\end{proof}

As second corollary, we bound the Dirichlet energy on the whole domain, but without sharp constants. This cannot be used to prove convergence of the Poisson equation, but is used to prove a discrete Poincaré inequality later (see \cref{prop:discrete_poincare}).
\begin{corollary}\label{cor:poincare_bound}
	Let $\epsilon,\delta>0$ with $\delta/\epsilon \leq C(\eta)$. For all $u\in L^2(\Omega)$ 
	\[\I1(\Lambda_{\epsilon,\delta} u;\rho)\lesssim \Ied1(u;\rho).\]
\end{corollary}
\begin{proof}
	We use \cref{thm:non-local_to_Local} with $\Omega'=\Omega$. Using \cref{lem:psi_moll} (i) and the restriction \labelcref{eq:eps_delta} to ensure that $\sigma_{\eta,\delta/\epsilon}\geq \frac{\sigma_\eta}{2}>0$ we have
	\begin{align*}
		\int_{\Omega}|\nabla \Lambda_{\epsilon,\delta} u|^2\rho^2\d x 
		&\lesssim \frac{1}{2\sigma_{\eta}\epsilon^{2}}\int_{\Omega}\int_{\Omega} \eta_{\epsilon} (|x-y|+2\delta) \abs{u(y) -u(x)}^{2}\rho(x)^2\d y \d x \\
		&\qquad + \frac{1}{\epsilon^2}\|\rho^2(\Lambda_{\epsilon,\delta}u - u)\|_{L^2(\Omega\cap\partial_\epsilon\Omega)}^2\\
		&\lesssim \Ied1(u;\rho) +\frac{1}{\epsilon^2}\|\Lambda_{\epsilon,\delta}u - u\|_{L^2(\Omega)}^2,
	\end{align*}
	where we used \labelcref{eq:rhoxy} in the final line.  The proof is completed by invoking \cref{lem:conv_diff}, and using that $\Ied1(u;\one) \lesssim \Ied1(u;\rho)$. 
\end{proof}

To handle the boundary of the domain, where \cref{cor:nonlocal_to_local} does not give us any information, we introduce a stretching operator to be applied \emph{after} the smoothing operator $\Lambda_{\epsilon,\delta}$.
This will ensure that the convolution takes place \emph{interior} to the domain. %\todo{see prev. comment. Do we even need $\theta(x)$ then?}
For this we utilize the concept of transversal vector fields \cite{hofmann2007geometric}, as done, e.g., in \cite{ern2016mollification,katzourakis2022generalised}.
For strongly Lipschitz domains, i.e., domains with a uniform cone property, there exists a smooth vector field $j\in C^\infty(\R^d;\R^d)$ and $\kappa=\kappa(\Omega)>0$ such that for $\mathcal{H}^{d-1}$-almost $x\in\partial\Omega$ it holds
\begin{align*}
	\langle j(x),\nu_\Omega(x)\rangle \geq \kappa,
	\qquad
	\abs{j(x)} = 1,
\end{align*}
where $\nu_\Omega$ denotes the unit normal vector to $\partial\Omega$.
Note that for smooth domains one can choose $j$ as a smooth extension of the outer unit normal field $\nu_{\Omega}$ in which case $\kappa=1$.
It was shown in \cite[Lemma 2.1]{ern2016mollification} (see also \cite[page 18]{katzourakis2022generalised}) that there exists $\eps_0=\eps_0(\Omega,j)>0$ and $\ell=\ell(\Omega,j)>0$ such that the mapping
\begin{align}\label{eq:Seps}
	S_\eps(x) \defeq   x - \eps \ell j(x)
\end{align}
satisfies
\begin{align}\label{eq:inclusion_stretching}
	S_\eps(\Omega) \subset \Omega_\eps 
	\qquad
	\forall 0 < \eps < \eps_0.
\end{align}
Note that, as observed in \cite{katzourakis2022generalised}, by multiplying with a smooth cut-off function one can restrict $j$ to be supported in a small collar neighborhood $\{x\in\R^d\st\dist(x,\partial\Omega)\leq R\}$ where $R=R(\Omega)>0$ is a domain-dependent constant.
Consequently, it holds 
\begin{align}\label{eq:support_j}
	S_\eps(x)=x,\qquad\forall x\in\Omega_R.  
\end{align}
Furthermore, since $j$ is smooth and just depends on $\Omega$, it holds
\begin{align*}
	D S_\eps(x) = \operatorname{id} - \eps \ell Dj(x)
\end{align*}
and, using this and a Taylor expansion of the determinant, we obtain
\begin{align}\label{eq:bounds_jacobian}
	\|D S_\eps(x)\| = 1 + \O(\epsilon) \ \ \text{and} \ \ \det D S_\eps(x) = 1 + \O(\eps)
\end{align}
for $\eps \ll 1$.
In particular, for such $\eps\ll 1$ the Jacobian $DS_\eps(x)$ is invertible and $S_\eps$ is a diffeomorphism.
We will apply the stretching map $S_\eps$ \emph{after} the smoothing operator, so we will consider the smoothing operation $(\Lambda_{\epsilon,\delta}u)\circ S_\eps$ which follows the approach of \cite{ern2016mollification,katzourakis2022generalised}.

We first prove some preliminary results about stretching $L^p$ norms of functions and their gradients.
\begin{proposition}\label{prop:lp_bounds}
	Let $1 \leq p < \infty$. There exist positive constants $C(\Omega), \eps_0(\Omega)$, such that for $0 < \epsilon \leq \eps_0$ and $u\in L^p(\Omega)$ we have
	\[\int_\Omega |u\circ S_\eps |^p\d x \leq (1+C_2\epsilon)\int_{\Omega_\eps}|u|^p \d x.\]
\end{proposition}
\begin{proof}
	Performing a change of variables, and using \labelcref{eq:inclusion_stretching} we have
	\[\int_\Omega |u\circ S_\eps|^p \d x =\int_{\Omega_\eps}|u(y)|^p |\det(D S_\eps)(S_\eps^{-1}(y))|^{-1}\d y\leq (1+C\epsilon)\int_{\Omega_\eps}|u|^p \d y,\]
	where we used \labelcref{eq:bounds_jacobian} in the last line.
\end{proof}
\begin{proposition}\label{prop:Se_bounds}
	There exist positive constants $C(\Omega)$, $\eps_0(\Omega)$ such that for $0 < \epsilon \leq \eps_0$ and $u\in H^1(\Omega_\eps)$ we have
	\[\int_\Omega |\nabla (u\circ S_\eps)|^2\rho^2 \d x \leq (1+C\epsilon)\int_{\Omega_\eps}|\nabla u|^2\rho^2 \d x.\]
\end{proposition}
\begin{proof}
	Let $w = u\circ S_\epsilon \in H^1(\Omega)$. Then we have $\nabla w(x) = DS_\epsilon(x)^T\nabla u(S_\epsilon(x))$ and so by \labelcref{eq:bounds_jacobian} we have
	\[|\nabla w(x)| \leq (1+C\epsilon)\nabla u(S_\epsilon(x)).\]
	Squaring, integrating over $x\in \Omega$, performing a change of variables, and using \labelcref{eq:inclusion_stretching} we have
	\begin{align*}
		\int_\Omega |\nabla w|^2\rho^2 \d x 
		&\leq (1+C\epsilon)\int_{\Omega}|\nabla u(S_\epsilon(x))|^2\rho(x)^2\d x\\
		&=(1+C\epsilon)\int_{\Omega_\eps}|\nabla u|^2(\rho \circ S_\eps^{-1})^2 |DS_\eps(S_\eps^{-1}(y))|^{-1}\d y\\
		&\leq (1+C\epsilon)\int_{\Omega_\eps}|\nabla u|^2\rho^2 \d y,
	\end{align*}
	where we used that $\rho$ is positive and Lipschitz in the last line, along with \labelcref{eq:bounds_jacobian}.
\end{proof}

The stretching operator $S_\eps$ allows us to improve our non-local to local convergence result.
\begin{lemma}\label{lem:non-local_to_local}
    Fix $n^{-\frac1d} < \delta \leq \frac{\rho_{\min}}{8\Lip(\rho)}$, $\delta/\epsilon \leq C(\eta)$ and $0 \leq \lambda \leq\frac{\rho_{\min}}{8\rho_{\max}}$, and let $\rho_\delta \in L^\infty(\Omega)$ be the probability density given by \cref{thm:transportation}.

    Consider a realization of the random graph $\X_{n}$ such that \cref{thm:transportation_is_pushforward,thm:transportation_point_to_point,thm:transportation_distance_T} hold for the transport map $T_\delta \colon \Omega \to \X_{n}$, while \cref{thm:transportation_distance_rho} holds for $\rho_\delta$ and $\lambda$.
    
    Then, for any $u\in L^2(\Omega)$ and $f\in L^2(\Omega)$, we have
	\begin{align}\label{eq:non-local_to_local}
		I(\Lambda_{\epsilon,\delta}u \circ S_\eps;f,\rho_\delta)&-\ied(u;f,\rho_\delta)\notag\\
		\lesssim& \left(\frac{\delta}{\eps}+\epsilon + \lambda\right)\Ied1(u;\rho_\delta) +  \|f\|_{L^2(\Omega)}^2\eps + \|u\|_{L^2(\partial_{3R}\Omega)}\|f\|_{L^2(\partial_R\Omega)},
	\end{align}
    provided $0 < \eps \leq \eps_1$ for some positive constant $\eps_1(\Omega)$.
    %, holds with probability at least  $1-Cn\exp(-cn\delta^{d}\lambda^{2})$.
	%Here, $\rho_\delta \in L^\infty(\Omega)$ is the probability density given by \cref{thm:transportation}.
\end{lemma}
\begin{proof}
    We chose $\eps_1(\Omega) \leq \min(R,\eps_0)$ small enough, such that $|S_\eps(x)-x|\leq R$ for all $x \in \Omega$.
	
    By \cref{prop:Se_bounds,cor:nonlocal_to_local} and using \labelcref{eq:inclusion_stretching} we have
	\begin{align}\label{eq:non-local_to_local0}
		\I1(\Lambda_{\epsilon,\delta}u \circ S_\eps;\rho) 
        &= \frac{1}{2}\int_\Omega |\nabla \Lambda_{\epsilon,\delta}u \circ S_\eps|^2 \rho^2 \d x\notag\\
		&\leq \frac{1+C\epsilon}{2}\int_{\Omega_\eps}|\nabla \Lambda_{\epsilon,\delta} u|^2\rho^2\d x \notag\\
		&\leq \left(1 + C\left(\tfrac{\delta}{\epsilon} + \epsilon \right)\right) \Ied1(u;\rho).
	\end{align}
	Note that the restrictions on $\lambda$ and $\delta$ along with \cref{thm:transportation_distance_rho} ensure that $\rho_\delta \geq \frac{3}{4}\rho_{\min}>0$, and in particular, that
	\begin{equation}\label{eq:rhodelta_to_rho}
		\rho_\delta(x) \leq \rho(x)\left( 1 + C(\delta + \lambda)\right) \ \ \text{and} \ \ \rho(x) \leq \rho_\delta(x)\left( 1 + C(\delta + \lambda)\right)
	\end{equation}
	hold for all $x\in \Omega$. It follows that
	\[\I1(\Lambda_{\epsilon,\delta}u \circ S_\eps;\rho_\delta) \leq (1 + C(\delta + \lambda))\I1(\Lambda_{\epsilon,\delta}u \circ S_\eps;\rho),\]
	and
	\[\Ied1(u;\rho) \leq (1 + C(\delta + \lambda))\Ied1(u;\rho_\delta).\]
	Combining these observations with \labelcref{eq:non-local_to_local0} we obtain
	\[\I1(\Lambda_{\epsilon,\delta}u \circ S_\eps;\rho_\delta) \leq \left(1 + C\left(\tfrac{\delta}{\epsilon} + \epsilon + \lambda \right)\right) \Ied1(u;\rho_\delta),\]
	which can be rearranged to 
	\begin{equation}\label{eq:dirichlet_estimate}
		\I1(\Lambda_{\epsilon,\delta}u \circ S_\eps;\rho_\delta)  - \Ied1(u;\rho_\delta) \lesssim \left(\tfrac{\delta}{\epsilon} + \epsilon + \lambda \right)\Ied1(u;\rho_\delta).
	\end{equation}
	Since $S_\eps(x)=x$ for $x\in \Omega_{R}$ according to \labelcref{eq:support_j} we have 
	\begin{align}\label{eq:source1}
		\int_{\Omega_{R}} uf\rho_\delta\d x - \int_{\Omega_{R}} (\Lambda_{\epsilon,\delta}u\circ S_\eps) f \rho_\delta \d x&=\int_{\Omega_{R}} uf\rho_\delta\d x - \int_{\Omega_{R}} \Lambda_{\epsilon,\delta}u  f \rho_\delta  \d x\notag\\
		&=\int_{\Omega_{R}} (u-\Lambda_{\epsilon,\delta}u)f \rho_\delta\d x \notag \\
		&\lesssim\|u-\Lambda_{\epsilon,\delta}u\|_{L^2\left(\Omega_{R}\right)}\|f\|_{L^2(\Omega)}\notag \notag \\
		&\lesssim \sqrt{\Ied1(u;\one)}\|f\|_{L^2(\Omega)}\eps\notag \\
		&\lesssim \Ied1(u;\one)\eps + \|f\|^2_{L^2(\Omega)}\eps,
	\end{align}
	where we used \cref{lem:conv_diff} in the penultimate line. We also have
	\begin{align*}
		\int_{\partial_R\Omega} uf\rho_\delta\d x - \int_{\partial_R\Omega} (\Lambda_{\epsilon,\delta}u\circ S_\eps) f \rho_\delta \d x&\lesssim \left(\|u\|_{L^2(\partial_R\Omega)} + \|\Lambda_{\epsilon,\delta}u \circ S_\eps\|_{L^2(\partial_R\Omega)}\right)\|f\|_{L^2(\partial_R\Omega)}.
	\end{align*}
	We choose By \cref{prop:lp_bounds} and the fact that $|S_\eps(x)-x|\leq R$, we have using the restriction $0 < \eps \leq \eps_1$
	\begin{align*}
		\|\Lambda_{\epsilon,\delta}u \circ S_\eps\|^2_{L^2(\partial_R\Omega)}
        &= \int_\Omega |\Lambda_{\epsilon,\delta}u \circ S_\eps|^2\chi_{\partial_{R}\Omega}\d x\\
		&\lesssim \int_{\Omega_\eps}|\Lambda_{\epsilon,\delta}u|^2(\chi_{\partial_{R}\Omega} \circ S_\eps^{-1})\d x\\
		&\lesssim \int_{\Omega}|\Lambda_{\epsilon,\delta}u|^2\chi_{\partial_{2R}\Omega}\d x\\
		&\lesssim \|\Lambda_{\epsilon,\delta}u\|^2_{L^2(\partial_{2R}\Omega)}.
	\end{align*}
	By \cref{prop:smoothing_l1} we have
	\[\|\Lambda_{\epsilon,\delta}u\|_{L^2(\partial_{2R}\Omega)} \lesssim \|u\|_{L^2(\partial_{2R+\eps}\Omega)},\]
	and therefore
	\[\int_{\partial_R\Omega} uf\rho_\delta\d x - \int_{\partial_R\Omega} (\Lambda_{\epsilon,\delta}u\circ S_\eps) f \rho_\delta \d x \lesssim \|u\|_{L^2(\partial_{3R}\Omega)}\|f\|_{L^2(\partial_R\Omega)}.\]
	where we used $\eps \leq R$.  Combining this with \labelcref{eq:source1} we have
	\[\I2(\Lambda_{\epsilon,\delta}u \circ S_\eps;f,\rho_\delta) - \I2(u;f,\rho_\delta) \lesssim \Ied1(u;\one)\eps + \|f\|^2_{L^2(\Omega)}\eps + \|u\|_{L^2(\partial_{3R}\Omega)}\|f\|_{L^2(\partial_R\Omega)}.\]
	Combining this with \labelcref{eq:dirichlet_estimate} and using that $\Ied1(u;\one)\lesssim \Ied1(u;\rho_\delta)$ completes the proof.
\end{proof}

\subsubsection{Continuum perturbations}

The following lemma gives bounds of perturbations for the continuum Poisson equation.
\begin{lemma}\label{lem:density_perturbation}
	Let $\rho_1,\rho_2\in L^\infty(\Omega)$ be positive densities and let $f_1,f_2\in L^1(\Omega)$. Let us write $I_i(u) = I(u;f_i,\rho_i)$ and $u_i = \argmin_{u\in H^1_{\rho_i}(\Omega)}I_i(u)$ for $i=1,2$. Then there exists $C>0$ depending (in an increasing manner) on $\sup_\Omega \rho_i$ and $\left(\inf_\Omega \rho_i\right)^{-1}$ for $i=1,2$, such that 
	\[I_2(u_2) - I_1(u_1)\leq C \left( \|f_2\|_{L^1(\Omega)}\|u_1\|_{L^1(\Omega)} + \|f_1u_1\|_{L^1(\Omega)}\right)\|\rho_1-\rho_2\|_{L^\infty(\Omega)}+ \langle f_1-f_2, u_1 \rho_2\rangle_{L^2(\Omega)}.\]
\end{lemma}
\begin{proof}
	Throughout the proof we write $\|\cdot\|_p = \|\cdot\|_{L^p(\Omega)}$ for $p\in[1,\infty]$. 
	For $u\in H^1(\Omega)$ we have
	\begin{align*}
		\int_\Omega f_1 u \rho_1 \d x - \int_\Omega f_2 u \rho_2\d x &=\int_\Omega f_1 u (\rho_1-\rho_2) \d x + \int_\Omega (f_1-f_2) u \rho_2 \d x  \\
		&\leq \|f_1 u\|_1\|\rho_1 - \rho_2\|_\infty + \langle f_1-f_2, u_1 \rho_2\rangle_{L^2(\Omega)}
	\end{align*}
	Note also that
	\[\|\rho_1^2 - \rho_2^2\|_\infty\leq \|\rho_1 + \rho_2\|_\infty\|\rho_1 - \rho_2\|_\infty =C\|\rho_1 - \rho_2\|_\infty.\]
	Thus for $u\in H^1(\Omega)$ we have
	\[\left|\int_\Omega |\nabla u|^2 \rho_1^2 \d x - \int_\Omega |\nabla u|^2\rho_2^2\d x\right| \leq C\|\rho_1 - \rho_2\|_\infty\int_\Omega |\nabla u|^2\d x.\]
	For $i=1,2$, the weak form of the Euler--Lagrange equation yields
	\[\int_\Omega |\nabla u_i|^2\rho_i^2 \d x = \int_\Omega f_i u_i \rho_i \d x,\]
	and so 
	\[\int_\Omega |\nabla u_i|^2 \d x \leq C\|f_iu_i\|_1.\]
	Therefore
	\[I_2(u_1) - I_1(u_1) \leq C\|f_1 u_1\|_1\|\rho_1 - \rho_2\|_\infty +\langle f_1-f_2, u_1 \rho_2\rangle_{L^2(\Omega)}.\]
	%A similar computation shows that 
	%\[|I_1(u_2) - I_2(u_2)| \leq C\left(\|(f_1-f_2) u_2\|_1 +\|f_2 u_2\|_1\|\rho_1 - \rho_2\|_\infty \right).\]
	We also have
	\[|(u_1)_{\rho_2}| = |(u_1)_{\rho_2} - (u_1)_{\rho_1}|   \leq \frac{\int_\Omega u_1 |\rho_2^2 - \rho_1^2| \d x}{\int_\Omega \rho_2^2\, dx}  \leq C\|u_1\|_1\|\rho_1-\rho_2\|_\infty.\]
	%and similarly
	%\[|(u_2)_{\rho_1}|  \leq C_2\|u_2\|_{L^1(\Omega)}\|\rho_1-\rho_2\|_\infty.\]
	Noting that $(u_1 - (u_1)_{\rho_2})_{\rho_2} = 0$, % \ \ \text{and} \ \ (u_2 - (u_2)_{\rho_1})_{\rho_1} = 0,\]
	we can compute
	\begin{align*}
		I_2(u_2) - I_1(u_1) &=I_2(u_2) - I_2(u_1) + I_2(u_1) - I_1(u_1)\\
		&\leq I_2(u_1 - (u_1)_{\rho_2}) - I_2(u_1) + I_2(u_1) - I_1(u_1)\\
		&= (u_1)_{\rho_2}\int_\Omega f_2\rho_2 \d x+ I_2(u_1) - I_1(u_1)\\
		&\leq C\left( \|f_2\|_1\|u_1\|_1 + \|f_1u_1\|_1\right)\|\rho_1-\rho_2\|_\infty+\langle f_1-f_2, u_1 \rho_2\rangle_{L^2(\Omega)}.
	\end{align*}
	which completes the proof.
\end{proof}
As a consequence we obtain the following.
\begin{lemma}\label{lem:perturbation_of_density_rhs_extension}
    Let $f \in L^\infty(\Omega)$, let $u \in H^1_\rho(\Omega)$ be the unique minimizer of $I(\cdot; f, \rho)$ in $H^1_\rho(\Omega)$, and let $v \in H^1(\Omega)$.
    Further, fix $n^{-\frac1d} < \delta \leq \frac{\rho_{\min}}{8 \Lip(\rho)}$ and $0\leq \lambda \leq \frac{\rho_{\min}}{8\rho_{\max}}$, and let $\rho_\delta \in L^\infty(\Omega)$ be the probability density given by \cref{thm:transportation}.

    Consider a realization of the random graph $\X_{n}$, such that \cref{thm:transportation_is_pushforward,thm:transportation_point_to_point,thm:transportation_distance_T} hold for the transport map $T_\delta \colon \Omega \to \X_{n}$, while \cref{thm:transportation_distance_rho} holds for $\rho_\delta$ and $\lambda$. 
    
    %There exist constants $C \equiv C(\Omega)>0$ and $c \equiv c(\Omega, \rho_{\min}) > 0$, such that for any $0 \leq \lambda \leq \frac{\rho_{\min}}{8\rho_{\max}}$ the event that
    Then, for any $f_n \in \ell^2(\X_{n})$ satisfying the compatibility condition $\sprod{f_n,\one}{\ell^2(\X_{n})} = 0$, we have
    \begin{align*}
        &\phantom{{}={}}I(u; f, \rho) - I(v; E_\delta f_n, \rho_\delta) \\
        &\lesssim \left(\norm{f_n}_{\ell^2(\X_{n})}^2 + \norm{f}_{L^1(\Omega)}^2\right) \left(\delta + \lambda\right) + K(q)\norm{E_\delta f_n - f}_{L^1(\Omega)}\norm{f_n}_{\ell^q(\X_{n})}.
    \end{align*}
    %for any $f_n \in \ell^2(\X_{n})$ satisfying the compatibility condition $\ipg{f_n,\one} = 0$, holds with probability at least  $1-Cn\exp(-cn\delta^d \lambda^2$).
\end{lemma}
\begin{proof}
    %We take of realization of the random graph $\X_{n}$ such that the statements of \cref{thm:transportation} hold. This has probability at least $1 - Cn \exp(-cn\delta^d \lambda^2)$.
    %
    Let $w \in H_{\rho_\delta}^1(\Omega)$ be the unique minimizer of \labelcref{eq:def_I} with density $\rho_\delta$ and function $E_\delta f$  in $H^1_{\rho_\delta}(\Omega)$. Using \cref{lem:density_perturbation,thm:transportation_is_pushforward,thm:transportation_distance_rho}, we obtain
    \begin{align*}
        &\phantom{{}={}}I(u; f, \rho) - I(w; E_\delta f_n, \rho_\delta) \\
        &\leq C \left( \|f\|_{L^1(\Omega)}\|w\|_{L^1(\Omega)} + \|E_\delta f_n w\|_{L^1(\Omega)}\right)\|\rho-\rho_\delta\|_{L^\infty(\Omega)}+ \langle E_\delta f_n - f, w \rho \rangle_{L^2(\Omega)} \\
        &\lesssim  \left( \|f\|_{L^1(\Omega)} + \norm{E_\delta f_n}_{L^2(\Omega)}\right)\norm{w}_{L^2(\Omega)}\left(\lambda + \delta\right)+ \norm{E_\delta f_n - f}_{L^1(\Omega)} \norm{w}_{L^\infty(\Omega)} \\
        &\lesssim  \left( \|f\|_{L^1(\Omega)}^2 + \norm{E_\delta f_n}_{L^2(\Omega)}^2 + \norm{w}_{L^2(\Omega)}^2\right)\left(\lambda + \delta\right)+ \norm{E_\delta f_n - f}_{L^1(\Omega)} \norm{w}_{L^\infty(\Omega)},
    \end{align*}
    where we used Young's inequality in the last line.
    Since $w$ is a minimizer, we can test the corresponding Euler--Lagrange equation with itself and obtain the usual energy estimate $\norm{w}_{L^2(\Omega)} \lesssim \norm{E_\delta f_n}_{L^2(\Omega)} \lesssim \norm{f_n}_{\ell^2(\X_{n})}$. 
    
    Because $w$ is a minimizer of $I(\cdot; E_\delta f_n, \rho_\delta)$ in $H^1_\rho(\Omega)$, we can apply \cref{prop:global_regularity,rem:specific_bounds} to estimate, for any $q>d/2$, $\norm{w}_{L^\infty(\Omega)} \lesssim K(q) \norm{E_\delta f_n}_{L^q(\Omega)} \lesssim K(q) \norm{f_n}_{\ell^q(\X_{n})}$. %\todo{Do we want to use this estimate here or leave the $L^\infty$ norm for the moment?}
    
    Finally, since $\int_\Omega E_\delta f_n \rho_\delta \dx = 0$,
    \begin{equation*}
        I(v; E_\delta f_n, \rho_\delta) 
        = I(v - (v)_{\rho_\delta}; E_\delta f_n, \rho_\delta) 
        \geq I(w; E_\delta f_n, \rho_\delta)
    \end{equation*}
    which concludes the proof.
\end{proof}

\subsubsection{Combination}

\begin{proof}[Proof of \cref{prop:continuum_energy_continuum_solution_discrete_energy_discrete_solution}]
First, we know by \cref{prop:global_regularity} that $u$ is indeed Lipschitz. Since $f$ is assumed to be Borel-measurable, the right hand side is well defined as a random variable.

Let $\rho_\delta \in L^\infty(\Omega)$ be the probability density given by \cref{thm:transportation}. Choose a realization of the random graph $\X_{n}$ such that \cref{thm:transportation_distance_rho} holds for $\rho_\delta$ and $0 <\lambda \leq \tfrac{\rho_{\min}}{8\rho_{\max}}$, and \cref{thm:transportation_is_pushforward,thm:transportation_point_to_point,thm:transportation_distance_T} holds for the transport map $T_\delta \colon \Omega \to \X_{n}$. This has probability at least $1-C\exp(-cn\delta^d \lambda^2)$.

We have
\begin{align*}
    I(u; f, \rho) - \ene(u_{n,\eps};f_n) &= 
    I(u; f, \rho) - I(\Lambda_{\epsilon,\delta} E_\delta u_{n,\epsilon} \circ S_\epsilon; E_\delta f_n, \rho_\delta) \\
    &\qquad+ I(\Lambda_{\epsilon,\delta} E_\delta u_{n,\epsilon} \circ S_\epsilon; E_\delta f_n, \rho_\delta) - I_{\epsilon, \delta}(E_\delta u_{n, \epsilon}; E_\delta f_n, \rho_\delta) \\
    &\qquad+ I_{\epsilon, \delta}(E_\delta u_{n, \epsilon}; E_\delta f_n, \rho_\delta) - \ene(u_{n,\eps};f_n)
\end{align*}
where the operators $E_\delta$, $\Lambda_{\epsilon,\delta}$ and $S_\epsilon$ are defined in \labelcref{eq:def_extension_operator,def:smoothing-operator,eq:Seps} respectively.

The first difference is estimated using \cref{lem:perturbation_of_density_rhs_extension}, the second using \cref{lem:non-local_to_local} (provided $\eps$ is chosen small enough) and the third using \cref{lem:Discrete-to-non-local-1}, which yields
\begin{align*}
    I(u; f, \rho) - \ene(u_{n,\eps};f_n) 
    &\lesssim \left(\norm{f_n}_{\ell^2(\X_{n})}^2 + \norm{f}_{L^1(\Omega)}^2\right) \left(\delta + \lambda\right) + K(q)\norm{E_\delta f_n - f}_{L^1(\Omega)}\norm{f_n}_{\ell^q(\X_{n})}  \\
    &+\left(\frac{\delta}{\varepsilon} + \varepsilon + \lambda\right) I_{\varepsilon, \delta}^{(1)}(E_\delta u_{n,\varepsilon}; \rho_\delta) + \norm{E_\delta f_n}_{L^2(\Omega)}^2\varepsilon + \norm{E_\delta f_n}_{L^2(\partial_{R} \Omega)} \norm{E_\delta u_{n, \varepsilon}}_{L^2(\partial_{3R} \Omega)}
\end{align*}
Note that by \cref{thm:transportation_distance_rho}, we have $\norm{E_\delta f_n}_{L^2(\Omega)} \lesssim \norm{f_n}_{\ell^2(\X_{n})}$, $\norm{E_\delta f_n}_{L^2(\partial_R \Omega)} \lesssim \norm{f_n}_{\ell^2(\X_{n} \cap \partial_{2R} \Omega)}$ and $\norm{E_\delta u_{n, \varepsilon}}_{L^2(\partial_{3R} \Omega)} \lesssim \norm{u_{n, \varepsilon}}_{\ell^2(\X_{n} \cap \partial_{4R} \Omega)}$. 

Moreover, by \cref{lem:Discrete-to-non-local-1} we have $I_{\varepsilon, \delta}^{(1)}(E_\delta u_{n,\varepsilon}; \rho_\delta) \leq \ene^{(1)}(u_{n, \eps})$. Thus, because $\delta \lesssim \eps$,
\begin{equation*}
    I(u; f, \rho) - \ene(u_{n,\eps};f_n) 
    \lesssim B_n + \left(\frac{\delta}{\varepsilon} + \varepsilon + \lambda\right) \ene^{(1)}(u_{n, \eps}),
\end{equation*}
where
\begin{align*}
    B_n  &\defeq   \left(\norm{f_n}_{\ell^2(\X_{n})}^2 + \norm{f}_{L^1(\Omega)}^2\right) \left(\lambda + \eps\right)  + \norm{f_n}_{\ell^2(\X_{n} \cap \partial_{2R} \Omega)} \norm{u_{n, \eps}}_{\ell^2(\X_{n} \cap \partial_{4R} \Omega)} \\
    &\qquad\qquad  + K(q)\left(\norm{f_n - f}_{\ell^1(\X_{n})} + \norm{\osc_{B(\delta;\cdot) \cap \Omega} f}_{L^1(\Omega)}\right)\norm{f_n}_{\ell^q(\X_{n})},
\end{align*}
because $\norm{E_\delta f_n - E_\delta f}_{L^1(\Omega)} \lesssim \norm{f_n - f}_{\ell^1(\X_{n})}$ and $\abs{E_\delta f (x) - f(x)} \leq \osc_{B(\delta;x) \cap \Omega}f$.

By the Euler--Lagrange equation \labelcref{eq:discrete_pde_EL}, we have that $\ene^{(1)}(u_{n,\eps}) = -2 \ene(u_{n,\eps};f_n)$. Thus, with 
\begin{equation*}
    A\defeq  1-C(\Omega,\rho)\left(\frac{\delta}{\varepsilon} + \varepsilon + \lambda\right),
\end{equation*}
we have
\begin{align*}
    I(u; f, \rho) - A\cdot \ene(u_{n,\eps};f_n) 
    &\lesssim B_n.
\end{align*}
Let $C_3$, $\eps_1$ and $\lambda_1$ be small enough, such that for all $\eps \leq \eps_1$, $\lambda \leq \lambda_1$ it holds $C(\Omega,\rho)\left(\frac{\delta}{\varepsilon} + \varepsilon + \lambda\right) \leq \frac12$. Then
\begin{align*}
    I(u; f, \rho) -  \ene(u_{n,\eps};f_n) 
    &\lesssim A^{-1} B_n - (A^{-1} -1) I(u; f, \rho) \\
    &\lesssim B_n + \left(\frac{\delta}{\varepsilon} + \varepsilon + \lambda\right) I^{(1)}(u;\rho),
\end{align*}
where we used the assumption that $u$ is a minimizer to obtain $I(u;f,\rho) = -2I^{(1)}(u;\rho)$ and the fact that $A^{-1} -1 \lesssim  \left(\frac{\delta}{\varepsilon} + \varepsilon + \lambda\right)$. Moreover, testing the weak formulation of the Euler--Lagrange equation \labelcref{eq:continuum_pde_EL} with $u$, we obtain $I^{(1)}(u;\rho) \lesssim \norm{f}_{L^2(\Omega)}^2$. Absorbing the term involving $\norm{f}_{L^1(\Omega)}$ from the definition of $B_n$ into the error term involving $\norm{f}_{L^2(\Omega)}$ finishes the proof.
\end{proof}

\subsection{Proof of the main \texorpdfstring{\cref{thm:main_smooth}}{theorem}}

We will need to prove convergence results of the degree function $\deg_{n,\epsilon}$ defined in \labelcref{eq:def_deg_ne}. Its non-local counterpart is given by
\begin{equation}\label{eq:rhoe}
\hat\rho_\eps(x) \defeq   \int_\Omega \eta_\eps(|x-y|)\rho(y)\d y.
\end{equation}
We note that since $\Omega$ has a Lipschitz boundary, there exists $C>0$ such that
\begin{equation}\label{eq:rhoe_bounds}
C\rho_{\min}\leq \hat\rho_\eps(x) \leq \rho_{\max}.
\end{equation}
The adjustment for the bound below results from the fact that when $x$ is close to the boundary $\partial \Omega$, then part of the support of $\eta_\eps$ may lie outside of the domain $\Omega$.
We have the following relation of $\deg_{n,\epsilon}$ and $\hat\rho_\eps$ which is a consequence of Bernstein's inequality which we recall here for convenience.
The version stated here is taken from \cite[Theorem 5.12]{CalculusofVariationsLN}, and we also refer to \cite{boucheronconcentration}.
% \todo[inline]{In the definition of $\deg_{n,\eps}$ there are only $n-1$ terms, so in this bound there is an error term of $O\left(\frac{n}{n-1}-1\right)$ missing. We should decide if we redefine $\deg$ or divide by $n-1$ here.
% \\
% L: For fixed $x\in\Omega$ $\deg_{,n\eps}(x)$ has $n$ terms with probability one.
% }
\begin{theorem}[Bernstein's Inequality]\label{thm:bernstein}
Let $Y_1,Y_2\dots,Y_n$ be a sequence of \emph{i.i.d.}~real-valued random variables with finite expectation $\mu=\mathbb E[Y_i]$ and variance $\sigma^2=\mathbb V(Y_i)$, and write $S_n=\frac{1}{n}\sum_{i=1}^n Y_i$. Assume there exists $b>0$ such that $|Y_i-\mu|\leq b$ almost surely. Then for any $\lambda>0$ we have
\begin{equation}\label{eq:bernstein}
\P(S_n-\mu \geq \lambda)\leq \exp\left( -\frac{n\lambda^2}{2(\sigma^2 + \tfrac13 b\lambda)} \right).
\end{equation}
\end{theorem}

\begin{proposition}\label{prop:degree}
There exists a positive constant $C_1(\eta(0), \rho_{\min}, \rho_{\max})$, such that for any $0 < \lambda \leq 1$ and $x\in \Omega$ we have
\begin{equation}\label{eq:degree}
\left|\frac{\deg_{n,\epsilon}(x)}{n} - \hat\rho_\eps(x)\right| \leq \lambda \hat\rho_\eps(x)
\end{equation}
holds with probability at least $1-2\exp\left( -C_1 n \epsilon^d \lambda^2\right)$.
\end{proposition}
\begin{proof}
Fix $x\in \Omega$. Then
\begin{equation*}
    \deg_{n,\epsilon}(x) = \sum_{i=1}^n \eta_\eps\left(\abs{x-x_i}\right)
\end{equation*}
with probability one.
We can use Bernstein's inequality from \cref{thm:bernstein} with $Y_i\defeq  \eta_\eps(|x-x_i|)$. 
Then the mean is given by $\mu=\hat\rho_\eps(x)$ and we have $b = C \epsilon^{-d}$ and $\sigma^2 \leq C\hat\rho_\eps(x)\epsilon^{-d}$. 
It follows that \labelcref{eq:degree} holds without absolute value with probability at least $1-\exp\left(-\tfrac{\hat\rho_\eps(x)}{C} n \epsilon^d \lambda^2\right)$ for $0 < \lambda \leq 1$, where the value of $C$ changed by an absolute constant.
We use \labelcref{eq:rhoe_bounds}, redefine $C$, and apply the same argument to $-Y_i$ to complete the proof.
\end{proof}

% \red
% \begin{lemma}
%     \begin{equation*}
%         \abs{\frac{\deg_{n,\eps}(X_1)}{n-1} - \int_\Omega \rho(x)\hat\rho_\eps(x)\dx} \leq \rho_{\min} \lambda
%     \end{equation*}
%     has probability at least $1-2 \exp(-C_1 n \exp^{-d}\lambda^2)$.
% \end{lemma}
% \begin{proof}
%     Use Bernstein as before with $Y_i \defeq   \eta_\eps(\abs{X_1 - X_i})$, noting that
%     \begin{equation*}
%         \mathbb E[Y_i] = \int_\Omega \rho(x) \int_\Omega \rho(y) \eta_\eps(\abs{x-y}) \dy \dx = \int_\Omega \rho(x)\hat\rho_\eps(x) \dx.
%     \end{equation*}
% \end{proof}
% \nc

An immediate consequence is
\begin{lemma}\label{lem:estimate_deg_rho_average}
   There exists a positive constant $C_1(\eta(0), \rho_{\min}, \rho_{\max})$, such that for any $0 < \lambda \leq 1$ we have
    \begin{equation*}
         \abs{\int_\Omega \rho^2 \dx - \frac1{n^2}\sum_{x \in \X_n} \deg_{n,\eps}(x)} \lesssim \eps + \lambda
    \end{equation*}
    holds with probability at least $1-4 n \exp(-C_1 n\eps^d \lambda^2)$.
\end{lemma}
\begin{proof}
First, note that we can always assume $n\eps^d \lambda^2 \geq 1$ so that $\frac{1}{n\epsilon^d}\leq \lambda^2\leq \lambda$. Indeed, we can adjust the constant $C_1$ so that the probability lower bound in the lemma is negative when $n\eps^d \lambda^2 <1$, and so the statement is then trivially true.  Now, the first step in the proof is to apply  \cref{prop:degree} to the setting where $x=x_i$ for some $i$. Conditioning on $x_i=x\in \Omega$ we have
\[\deg_{n,\eps}(x_i) = \eta_\eps(0) + \sum_{j\neq i}\eta_{\eps}(|x - x_j|).\]
Applying  \cref{prop:degree} to the second sum, which is over $n-1$ \emph{i.i.d.} random variables we find that
\begin{equation}
\left|\frac{\deg_{n,\epsilon}(x_i) - \eta_\eps(0)}{n-1} - \hat\rho_\eps(x)\right| \leq \lambda \hat\rho_\eps(x)
\end{equation}
holds with conditional probability at least $1-2 \exp(-C_1 n\eps^d \lambda^2)$, where we used \cref{ass:neps} to bound $n-1\geq \frac{1}{2}n$. Multiplying by $(n-1)/n$ on both sides and using the law of conditional probability yields that
\begin{equation}\label{eq:deg_i}
\left|\frac{\deg_{n,\epsilon}(x_i)}{n} - \hat\rho_\eps(x_i)\right| \leq \left(\lambda + \tfrac{1}{n}\right)\hat\rho_\eps(x_i) + \tfrac{1}{n}\eta_\eps(0) \lesssim \lambda + \frac{1}{n\eps^d} \lesssim \lambda,
\end{equation}
holds with probability at least $1-2 \exp(-C_1 n\eps^d \lambda^2)$. We now union bound over $i=1,\dots,n$ to find that \labelcref{eq:deg_i} holds for all $i=1,\dots,n$ with probability at least $1-2 n\exp(-C_1 n\eps^d \lambda^2)$

Now, we have, using the triangle inequality, that
    \begin{align*}
         &\phantom{{}={}}\abs{\int_\Omega \rho^2 \dx - \frac1{n^2}\sum_{x \in \X_n} \deg_{n,\eps}(x)} \\
         &\leq \abs{\int_\Omega \rho^2 \dx - \frac1n \sum_{x \in \X_n} \hat\rho_\eps(x)} + \frac1{n}\sum_{x \in \X_n} \abs{\frac{\deg_{n,\eps}(x)}{n} - \hat\rho_\eps(x)} \\
         &\lesssim \int_\Omega \rho(x) \abs{\rho(x)-\hat\rho_\eps(x)}\dx + \abs{\int_\Omega \rho(x)\hat\rho_\eps(x) \dx - \frac1n \sum_{x \in \X_n} \hat\rho_\eps(x)} + \lambda.
    \end{align*}
    For all $x \in \Omega_{\eps}$ the rescaled kernel integrates to one and we have
    \begin{equation*}
        \abs{\rho(x) - \hat\rho_\eps(x)}
        \leq \int_\Omega \eta_\eps\left(\abs{x - y}\right) \abs{\rho(x) - \rho(y)} \dy
        \leq \Lip(\rho)\eps. 
    \end{equation*}
    On the other hand,
    \begin{equation*}
        \int_{\partial_{2\eps}\Omega} \abs{\rho(x) - \hat\rho_\eps(x)}\dx 
        \leq 2 \rho_{\max} \abs{\partial_{2\eps}\Omega}
        \lesssim \eps.
    \end{equation*}
    Therefore, the first term is bounded by
    \begin{equation*}
        \int_\Omega \rho(x) \abs{\rho(x)-\hat\rho_\eps(x)}\dx \lesssim \eps.
    \end{equation*}
    The second term is bounded using Hoeffding's inequality, similar to the proof of \cref{lem:non-local-to-discrete-energy}. Namely, by setting $Y_i \defeq   \hat\rho_\eps(X_i)$, we obtain that
    \begin{equation*}
        {\abs{\int_\Omega \rho(x)\hat\rho_\eps(x) \dx - \frac1n \sum_{x \in \X_n} \hat\rho_\eps(x)}} \leq \lambda \rho_{\max}
    \end{equation*}
    with probability at least $1-2\exp\left(-\frac14 n \lambda^2\right)$.
    %The last term gives
    %\begin{equation*}
    %    \frac{1}{n}(\lambda + \eps)\sum_{x \in \X_n} \hat\rho_\eps(x) \leq (\lambda +\eps)\rho_{\max}.
    %\end{equation*}
Therefore, the claim holds upon redefining $C_1$ appropriately.
\end{proof}

We can now show that the discrete mean is a good approximation of the continuum mean.

\begin{lemma}\label{lem:estimate_mean_values_Edeltau}
    There exist positive constants $C_1(\Omega)$, $C_2(\Omega, \eta(0), \rho_{\min}, \rho_{\max}$, $C_3(\rho)$, such that for every $n^{-\frac1d} < \delta \leq \frac{\rho_{\min}}{8 \Lip(\rho)}$, $\delta \leq \eps$, and $0 < \lambda \leq \min\left(1, \tfrac{\rho_{\min}}{8\Lip(\rho)}\right)$, $\lambda + \eps \leq C_3$, the event that
    \begin{equation*}
         \forall u \in \ell^2(\X_{n}) \colon \abs{(E_\delta u)_\rho - (u)_{\deg_{n,\eps}}} \lesssim (\sqrt\eps + \lambda)  \norm{u}_{\ell^2(\X_{n})}
    \end{equation*}
    has probability at least $1-C_1 n \exp(-C_2 n \delta^d \lambda^2)$.
\end{lemma}
\begin{proof}
    Consider a realization of the random graph $\X_n$ such that \labelcref{eq:degree} holds for all $x_i$, $i=1,\dots,n$, and such that the assertions of \cref{lem:estimate_deg_rho_average,thm:transportation} hold. Because $\delta \leq \eps$ this has probability at least $1 - C_1n \exp(-C_2 n \delta^d \lambda^2)$.

    We have
    \begin{align*}
        \abs{(E_\delta u)_\rho - (u)_{\deg}}
        &\leq \abs{{\int_\Omega \rho^2 \dx}}^{-1}\abs{\int_\Omega \rho^2 E_\delta u \dx - \frac{1}{n^2} \sum_{x \in \X_{n}} u(x) \deg_{n,\eps}(x)} \\
        &\qquad +
        \frac{1}{n^2} \sum_{x \in \X_{n}} \abs{u(x)} \deg_{n,\eps}(x) \abs{\left(\int_\Omega \rho^2 \dx\right)^{-1} - \left(\frac1{n^2} \sum_{x \in \X_{n}}\deg_{n,\eps}(x)\right)^{-1}}
    \end{align*}
    The first term is estimated by
    \begin{align*}
        &\phantom{{}={}}\abs{{\int_\Omega \rho^2 \dx}}^{-1}\abs{\int_\Omega \rho^2 E_\delta u \dx - \frac{1}{n^2} \sum_{x \in \X_{n}} u(x) \deg_{n,\eps}(x)} \\
        &\lesssim \abs{\int_\Omega \rho^2 E_\delta u \dx - \frac{1}{n} \sum_{x \in \X_{n}} u(x) \hat\rho_\eps(x)} + 
        \abs{\frac{1}{n} \sum_{x \in \X_{n}} u(x) \left(\hat\rho_\eps(x) - \frac{\deg_{n,\eps}(x)}{n}\right)} \\
        &\lesssim \abs{\int_\Omega \left(\rho^2 - \rho_\delta E_\delta \hat\rho_\eps\right) E_\delta u \dx}
        + \frac{\lambda}{n} \sum_{x \in \X_{n}} \abs{u(x)} \hat\rho_\eps(x)\\
        &\lesssim \int_\Omega \abs{\rho^2 - \rho_\delta E_\delta \hat\rho_\eps} \abs{E_\delta u} \dx
        + \lambda \rho_{\max} \norm{u}_{\ell^1(\X_{n})},
    \end{align*}
    where we used \labelcref{eq:rhoe_bounds} in the last line. 
    For $x \in \Omega$ we have that
    \begin{align*}
        &\phantom{{}={}} \abs{\rho(x)^2 - \rho_\delta(x) E_\delta \hat\rho_\eps(x)} \\
        &\leq \rho_{\max}\abs{\rho(x) - \hat\rho_\eps(T_n(x))} + \abs{\hat\rho_\eps(T_n(x))}\abs{\rho - \rho_\delta} \\
        &\leq \rho_{\max} \left(\Lip \rho \,\delta + \abs{\rho(T_n(x)) - \hat\rho_\eps(T_n(x))} + \lambda + \delta\right).
    \end{align*}
    Moreover, for $x \in \Omega_{\eps + \delta}$ we have that
    \begin{equation*}
        \abs{\rho(T_n(x)) - \hat\rho_\eps(T_n(x))}
        \leq \int_\Omega \eta_\eps\left(\abs{T_n(x) - y}\right)\abs{\rho(T_n(x)) - \rho(y)} \dy
        \leq \Lip\rho \left(\delta + \eps\right).
    \end{equation*}
    Finally, for $x \in \partial_{2\eps} \Omega$ we use the bound $\abs{\rho(T_n(x)) - \hat\rho_\eps(T_n(x))} \leq 2 \rho_{\max}$ to obtain
    \begin{align*}
        \int_\Omega \abs{\rho^2 - \rho_\delta E_\delta \hat\rho_\eps} \abs{E_\delta u} \dx
        &\lesssim (\eps + \lambda) \norm{E_\delta u}_{L^1(\Omega_{2\eps})} + \norm{E_\delta u}_{L^1(\partial_{2\eps} \Omega)} \\
        &\lesssim (\eps + \lambda) \norm{E_\delta u}_{L^1(\Omega_{2\eps})} + \sqrt\eps \norm{E_\delta u}_{L^2(\partial_{2\eps} \Omega)} \\
        &\lesssim (\sqrt\eps + \lambda)  \norm{u}_{\ell^2(\X_{n})},
    \end{align*}
    where we used $\abs{\partial_{2\eps} \Omega} \lesssim \eps$, $\delta \leq \eps$ and that by \cref{thm:transportation_distance_rho} we have $\rho_\delta \geq \frac34 \rho_{\min}$ and therefore $\norm{E_\delta u}_{L^p(\Omega)} \lesssim \norm{u}_{\ell^p(\X_{n})}$. 

    Note that, by \cref{lem:estimate_deg_rho_average}, we have for $\eps + \lambda \leq \frac12 \int_\Omega \rho^2 \dx =: C_3$, 
    \begin{equation*}
        \abs{\left(\int_\Omega \rho^2 \dx\right)^{-1} - \left(\frac1{n^2} \sum_{x \in \X_{n}}\deg_{n,\eps}(x)\right)^{-1}}
        \leq \frac{\abs{\int_\Omega \rho^2 \dx - \frac1{n^2} \sum_{x \in \X_{n}}\deg_{n,\eps}(x)}}{2 \int_\Omega \rho^2 \dx}
        \lesssim \eps + \lambda.
    \end{equation*}
    Moreover, using \labelcref{eq:degree} for every $x \in \X_n$, we obtain
    \begin{equation*}
        \frac{1}{n^2} \sum_{x \in \X_n} \abs{u(x)} \deg_{n,\eps}(x) 
        \leq \left(\lambda + 1\right) \frac1n \sum_{x \in \X_n} \abs{u(x)} \hat\rho_\eps(x) \lesssim \norm{u}_{\ell^1(\X_n)}.
    \end{equation*}
    Since $\norm{u}_{\ell^1(\X_n)} \leq \norm{u}_{\ell^2(\X_n)}$, this finishes the proof.
\end{proof}

\begin{lemma}\label{lem:estimate_mean_values_u}
    % Let $u: \Omega \to \R$ be bounded and Borel-measurable and  $0 < \eps \leq 1$.
    % There exist a positive constant $C_1(\eta(0), \rho_{\min}, \rho_{\max})$, such that for any $n^{-\frac1d} < \delta \leq \frac{\rho_{\min}}{8 \Lip(\rho)}$ and $0 < \lambda \leq 1$, the event that
    % \begin{equation*}
    %      \abs{(u)_\rho - (u)_{\deg_{n,\eps}}} \lesssim \left(\lambda + \sqrt\eps\right) \norm{u}_{L^2(\Omega)} + \eps^{\frac d2}\lambda \norm{u}_{L^\infty(\Omega)}
    % \end{equation*}
    % holds with probability at least $1 - 2n \exp(-C_1 n \eps^d \lambda^2)$.    

    Let $u: \Omega \to \R$ be bounded and Borel-measurable.
    There exist positive constants $C_1$, $C_2(\eta(0), \rho_{\min}, \rho_{\max})$, $C_3(\rho)$, such that for any $0 < \eps \leq 1$, $0 < \lambda \leq 1$, $\lambda + \eps \leq C_3$, the event that
    \begin{equation*}
          \abs{(u)_\rho - (u)_{\deg_{n,\eps}}} \lesssim \left(\lambda + \sqrt\eps\right) \norm{u}_{L^2(\Omega)} + \eps^{\frac d2}\lambda \norm{u}_{L^\infty(\Omega)}
    \end{equation*}
    has probability at least $1-C_1 n \exp(-C_2 n \eps^d \lambda^2)$.
\end{lemma}
\begin{proof}
    Consider a realization of the random graph $\X_n$ such that \labelcref{eq:degree} holds for all $x_i$,  $i=1,\dots,n$, and such that the assertion of \cref{lem:estimate_deg_rho_average} holds. This has probability at least $1 - 4 n \exp(-C_1 n \eps^d \lambda^2)$.

    We have
    \begin{align*}
        \abs{(u)_\rho - (u)_{\deg_{n,\eps}}}
        &\leq \abs{{\int_\Omega \rho^2 \dx}}^{-1}\abs{\int_\Omega \rho^2  u \dx - \frac{1}{n^2} \sum_{x \in \X_{n}} u(x) \deg_{n,\eps}(x)} \\
        &\qquad +
        \frac{1}{n^2} \sum_{x \in \X_{n}} \abs{u(x)} \deg_{n,\eps}(x) \abs{\left(\int_\Omega \rho^2 \dx\right)^{-1} - \left(\frac1{n^2} \sum_{x \in \X_{n}}\deg_{n,\eps}(x)\right)^{-1}}
    \end{align*}
    As in the proof of \cref{lem:estimate_mean_values_Edeltau} we obtain for $\lambda + \eps \leq \frac12 \int_\Omega \rho^2 \dx =: C_3$, that
    \begin{equation*}
        \frac{1}{n^2} \sum_{x \in \X_{n}} \abs{u(x)} \deg_{n,\eps}(x) \abs{\left(\int_\Omega \rho^2 \dx\right)^{-1} - \left(\frac1{n^2} \sum_{x \in \X_{n}}\deg_{n,\eps}(x)\right)^{-1}} 
        \lesssim \left(\lambda + \eps\right) \norm{u}_{\ell^1(\X_n)}.
    \end{equation*}
    Further,
    \begin{align*}
        &\phantom{{}={}}\abs{\int_\Omega \rho^2 \dx}^{-1} \abs{\int_\Omega u \rho^2 \dx - \frac{1}{n^2} \sum_{x \in \X_{n}} u(x) \deg(x)} \\
        &\lesssim \abs{\int_\Omega u \rho^2 \dx - \frac{1}{n} \sum_{x \in \X_{n}} u(x) \hat\rho_\eps(x)} + 
        \abs{\frac{1}{n} \sum_{x \in \X_{n}} u(x) \left(\hat\rho_\eps(x) - \frac{\deg(x)}{n}\right)} \\
        &\lesssim \int_\Omega u \rho \abs{\rho - \hat\rho_\eps} \dx + \abs{\int_\Omega u \hat\rho_\eps \rho \dx - \frac{1}{n} \sum_{x \in \X_{n}} u(x) \hat\rho_\eps(x)}
        + \lambda \rho_{\max} \norm{u}_{\ell^1(\X_{n})},
    \end{align*}
    where we used \labelcref{eq:degree,eq:rhoe_bounds} in the last line.

    As in the proof of \cref{lem:non-local-to-discrete-energy} we apply Hoeffding's inequality (to $f \equiv 1$ and $f \equiv \hat\rho_\eps$) to obtain with probability at least $1-4\exp(-\frac18 n\eps^d\lambda^2)$ that
    \begin{equation*}
        \abs{\norm{u}_{\ell^1(\X_{n})} - \int_\Omega \abs{u} \rho \dx} 
        \leq \norm{u}_{L^\infty(\Omega)} \eps^{\frac d2} \lambda
    \end{equation*}
    and
    \begin{equation*}
        \abs{\frac{1}{n} \sum_{x \in \X_{n}} u(x) \hat\rho_\eps(x) - \int_\Omega u \hat\rho_\eps \rho \dx} 
        \leq \rho_{\max} \norm{u}_{L^\infty(\Omega)} \eps^{\frac d2} \lambda.
    \end{equation*}
    Combining all estimates we obtain with probability at least $1-Cn \exp(-cn\eps^d \lambda^2)$ that
    \begin{align*}
        \abs{(u)_\rho - (u)_{\deg}}
        &\lesssim \abs{\int_\Omega u \rho \left(\rho - \hat\rho_\eps\right) \dx} + \eps^{\frac d2}\lambda \norm{u}_{L^\infty(\Omega)} + (\lambda + \eps) \norm{u}_{L^1(\Omega)} \\
        &\lesssim \eps^{\frac d2}\lambda \norm{u}_{L^\infty(\Omega)} + (\lambda + \eps) \norm{u}_{L^1(\Omega)} + \norm{\rho - \hat\rho_\eps}_{L^\infty(\Omega_\eps)} \norm{u}_{L^1(\Omega_\eps)} + \norm{u}_{L^1(\partial_\eps\Omega)} \\
        &\lesssim \eps^{\frac d2}\lambda \norm{u}_{L^\infty(\Omega)} + \left(\lambda + \sqrt \eps\right) \norm{u}_{L^2(\Omega)},
    \end{align*}
    where we used that $\abs{\rho - \hat\rho_\eps} \lesssim \eps$ in $\Omega_\eps$, and $\abs{\partial_\eps \Omega} \lesssim \eps$ in the last line, and that $0 < \eps \leq 1$, implies $\eps \leq \sqrt \eps$.
\end{proof}
We can now prove a discrete Poincaré inequality, which is uniform in $n$ and $\eps$. 
%\todo[inline]{The constant in the inequality depends on $\eps_1$, but is indeed independent of $\eps < \eps_1$. Should we mark this somehow? Usually, $\lesssim$ depends only on $\rho,\Omega,\eta$.
%\\
%L:
%$\eps_1$ depends on these quantities so this should be ok}
\begin{proposition}\label{prop:discrete_poincare}
    There exist positive constants $C_1(\Omega)$, $C_2(\Omega, \eta(0), \rho_{\min}, \rho_{\max})$ and $\eps_1(\Omega, \rho, \eta)$, such that for any $n \in \N$, $0 < \eps \leq \eps_1$, $n^{-\frac1d} < \delta \leq \frac{\rho_{\min}}{8 \Lip(\rho)}$ and $\delta \leq \eps$ the event that
    \begin{equation*}
         \norm{u - (u)_{\deg}}_{\ell^2(\X_{n})} \lesssim  \norm{\nabla_{n,\eps} u}_{\ell^2(\X_{n}^2)}
    \end{equation*}
    holds for all $u \in \ell^2(\X_{n})$ holds with probability at least $1 - C_1n \exp(-C_2 n \delta^d)$.
\end{proposition}
\begin{proof}
    Let $0 < \lambda \leq \min\left(1, \tfrac{\rho_{\min}}{8 \rho_{\max}}\right)$.
    We fix a realization $\X_{n}$ of the graph, such that \cref{lem:estimate_mean_values_Edeltau} and such that the assertions of \cref{thm:transportation} hold. For this we have to take $\varepsilon$ small enough by \cref{lem:estimate_mean_values_Edeltau}.
    This has probability at least $1-C_1 n \exp(-C_2 n \delta^d \lambda^2)$, where we used that $\delta \leq \eps$ and possibly chose other constants $C_1$ and $C_2$.
    
    Let $w \defeq   u - (u)_{\deg}$. By \cref{thm:transportation_distance_rho}, we have that $|\rho_\delta(x) - \rho(x)| \leq \frac{\rho(x)}{4}$ holds for all $x\in \Omega$. Therefore $\rho \lesssim \rho_\delta$ and $\rho_\delta \lesssim \rho$.  Therefore, by \cref{lem:conv_diff} 
	\begin{align*}
		\norm{w}_{\ell^2(\X_{n})} \lesssim\|E_\delta w\|_{L^2(\Omega)}\lesssim \|\Lambda_{\epsilon,\delta}E_\delta w\|_{L^2(\Omega)} + \eps \left(\Ied1(E_\delta w;\one)\right)^{\frac12}.
	\end{align*}
	Now, we use the continuum Poincaré inequality, \cref{lem:Discrete-to-non-local-1,cor:poincare_bound} and $\Ene1(w) = \frac12 \norm{\nabla_{n,\eps} w}_{\ell^2(\X_{n}^2)}^2$ to obtain
	\begin{align*}\label{eq:half_poincare}
		\norm{w}_{\ell^2(\X_{n})} 
        &\lesssim \norm{\nabla \Lambda_{\epsilon,\delta}E_\delta w}_{L^2(\Omega)} + (\Lambda_{\epsilon,\delta}E_\delta w)_\rho + \eps \left(\Ied1(E_\delta w;\one)\right)^{\frac12} \\
        &\lesssim \left(\I1(\Lambda_{\epsilon,\delta}E_\delta w;\rho)\right)^{\frac12} + (\Lambda_{\epsilon,\delta}E_\delta w)_\rho + \eps \left(\Ied1(E_\delta w;\rho_\delta)\right)^{\frac12} \\
        &\lesssim (1+\eps) \norm{\nabla_{n,\eps} w}_{\ell^2(\X_{n}^2)} + (\Lambda_{\epsilon,\delta}E_\delta w)_\rho
	\end{align*}
	Further, we have
	\begin{align*}
		\int_\Omega \rho^2\d x\left|(\Lambda_{\epsilon,\delta}E_\delta w)_\rho - (E_\delta w)_\rho\right| &\leq
		\int_\Omega \left|\Lambda_{\epsilon,\delta}E_\delta w - E_\delta w\right| \rho^2\d x\\
		&\leq \|\rho^2\|_{L^2(\Omega)} \|\Lambda_{\epsilon,\delta}E_\delta w - E_\delta w\|_{L^2(\Omega)}.
	\end{align*}
	Applying again \cref{lem:conv_diff} we obtain
	\begin{equation}\label{eq:smoothing_mean}
		\abs{(\Lambda_{\epsilon,\delta}E_\delta w)_\rho - (E_\delta w)_\rho} 
        \lesssim \eps \left(\Ied1(E_\delta w;\one)\right)^{\frac12} 
        \lesssim \eps \norm{\nabla_{n,\eps} w}_{\ell^2(\X_{n}^2)}.
	\end{equation}
	By \cref{lem:estimate_mean_values_Edeltau}, using that $(w)_{\deg} = 0$, we have
    \begin{equation*}
        \abs{(E_\delta w)_\rho} \lesssim \left(\sqrt \eps + \lambda\right) \norm{w}_{\ell^2(\X_{n})}.
    \end{equation*}
	Therefore,
    \begin{equation*}
        \norm{w}_{\ell^2(\X_{n})} \lesssim (1+2\eps) \norm{\nabla_{n,\eps} w}_{\ell^2(\X_{n}^2)} + \left(\sqrt \eps + \lambda\right) \norm{w}_{\ell^2(\X_{n})}.
    \end{equation*}
    Choosing $\eps_1$ and $\lambda$ small enough, such that $\left(\sqrt \eps + \lambda\right) \norm{w}_{\ell^2(\X_{n})}$ can be absorbed into the left hand side finishes the proof.
\end{proof}
We finally give the proof of the main theorem of this section. 
\begin{proof}[Proof of \cref{thm:main_smooth}]
    Let $\eps_1$ and $\hat\lambda_1$ be chosen small enough, such that the assertions of \cref{prop:discrete_poincare,lem:estimate_mean_values_u} hold for $\eps$ and $\lambda_1$. Fix a realization $\X_{n}$ of the random graph, such that all estimates of \cref{prop:discrete_continuum_energy_continuum_function,,prop:discrete_poincare,lem:estimate_mean_values_u} hold with $\lambda_1$ and \cref{prop:continuum_energy_continuum_solution_discrete_energy_discrete_solution} hold with $\lambda_2$. This has probability at least $1-4n\exp(-C_1 n \eps^d \lambda_1^2) - C_2n\exp(-C_3 n \delta^d \lambda_2^2)$.
	
	By the discrete Poincaré inequality \cref{prop:discrete_poincare} we have
	\begin{equation*}
		\|u - u_{n,\epsilon}\|_{H^1(\X_{n})}^2 
		\lesssim \|\nabla_{n,\eps} \left(u - u_{n,\epsilon}\right)\|_{\ell^2(\X_{n}^2)}^2 + (u-u_{n,\epsilon})_{\deg}^2.
	\end{equation*}
	Since $(u_{n,\epsilon})_{\deg} = 0$ and $(u)_\rho=0$, we have by \cref{lem:estimate_mean_values_u}
	\begin{equation*}
		\abs{(u-u_{n,\epsilon})_{\deg}}^2
        = \abs{(u)_{\deg} - (u)_\rho}^2
		\lesssim \left(\lambda_1^2 + \eps\right) \norm{u}_{L^2(\Omega)}^2 + \eps^d \lambda_1^2 \norm{u}_{L^\infty(\Omega)}^2.
	\end{equation*}	
	Because $\sprod{f_n,\one}{\l2} = 0$, $u_{n,\epsilon}$ satisfies \labelcref{eq:discrete_pde_EL} for all $v \in \ell^2(\X_{n})$. Hence, testing with $u$, we obtain
	\begin{align*}
		&\phantom{{}={}}\frac12\|\nabla_n \left(u - u_{n,\epsilon}\right)\|_{\ell^2(\X_{n}^2)}^2 \\
        &= \frac12\norm{\nabla_{n,\eps} u_{n,\eps}}_{\ell^2(\X_{n}^2)}^2  + \frac12\norm{\nabla_{n,\eps} u}_{\ell^2(\X_{n}^2)}^2 - \sprod{\nabla_{n,\eps} u_{n,\eps}, \nabla_{n,\eps} u}{\ell^2(\X_{n}^2)} \\
        &= \frac12\norm{\nabla_{n,\eps} u_{n,\eps}}_{\ell^2(\X_{n}^2)}^2 - \sprod{f_n,u}{\ell(\X_{n})} - \frac12\norm{\nabla_{n,\eps} u}_{\ell^2(\X_{n}^2)}^2 + \sprod{\nabla_{n,\eps} u_{n,\eps}, \nabla_{n,\eps} u_{n,\eps}}{\ell^2(\X_{n}^2)}\\
		&= \ene(u;f_n) - \ene(u_{n,\varepsilon}, f_n) \\
		&= \left(\ene(u;f_n) - I(u;f,\rho)\right) + \left(I(u;f,\rho) - \ene(u_{n,\varepsilon}, f_n)\right).
	\end{align*}
	The first difference is estimated using \cref{prop:discrete_continuum_energy_continuum_function}. % giving
%	\begin{equation}
%		\abs{\ene(u;f_n) - \ene(u;f)} 
%		\leq \norm{f-f_n}_{\ell^1(\X_{n})} \norm{u}_{\ell^\infty(\X_{n})} 
%		\leq \norm{f-f_n}_{\ell^1(\X_{n})} \norm{u}_{L^\infty(\Omega)}.
%	\end{equation}
	%
	%For the second difference, we employ \cref{prop:discrete_continuum_energy_continuum_function}. %which yields
%	\begin{equation*}
%		{\ene(u;f) - I(u;f,\rho)} 
%		\lesssim \left(\Lip(u)^2 + \norm{fu}_{L^\infty(\Omega)} \varepsilon^{\frac d2}\right) \lambda + \left(\norm{f}_{L^2(\Omega)}^2 + \norm{\nabla u}_{L^\infty(\partial_{2\varepsilon}\Omega)}^2 \right)\varepsilon.
%	\end{equation*}
%
	The second difference is estimated using \cref{prop:continuum_energy_continuum_solution_discrete_energy_discrete_solution}. Combining all estimates yields
	\begin{align*}
		\|u - u_{n,\epsilon}\|_{H^1(\X_{n})}^2 
        &\lesssim \left(\lambda_1^2 + \eps\right) \norm{u}_{L^2(\Omega)}^2 + \eps^d \lambda_1^2 \norm{u}_{L^\infty(\Omega)}^2 \\
        & + \left(\norm{u}_{L^\infty(\Omega)}  + K(q) \norm{f_n}_{\ell^q(\X_{n})}\right) \norm{f - f_n}_{\ell^1(\X_{n})} \\
        & + \left(\Lip(u)^2 + \norm{fu}_{L^\infty(\Omega)} \varepsilon^{\frac d2}\right) \lambda_1 +             
            %\norm{\nabla u}_{L^\infty(\partial_{2\varepsilon}\Omega)}^2 
            \Lip(u;\partial_{2\epsilon}\Omega)^2\eps \\
        & + \left(\frac{\delta}{\varepsilon} + \varepsilon + \lambda_2\right) \norm{f}_{L^2(\Omega)} ^2+ \norm{u_{n,\eps}}_{\ell^2(\X_{n} \cap \partial_{4R} \Omega)} \norm{f_n}_{\ell^2(\X_{n} \cap \partial_{2R} \Omega)} \\
		& + (\eps + \lambda_2) \norm{f_n}_{\ell^2(\X_{n})}^2 + K(q) \norm{\osc_{B(\delta;\cdot)}f}_{L^1(\Omega)} \norm{f_n}_{\ell^q(\X_{n})}.
	\end{align*}
    Now, by Poincaré and the fact that $u$ is a minimizer,
    \begin{equation*}
        \norm{u}_{L^2(\Omega)}^2 
        \lesssim \norm{\nabla u}_{L^2(\Omega)}^2 
        \lesssim \norm{f}_{L^2(\Omega)}^2.
    \end{equation*}
    Moreover, by the discrete Poincaré inequality and the fact that $u_{n,\eps}$ solves the weak Euler--Lagrange equation \labelcref{eq:discrete_pde_EL}
    \begin{equation*}
        \norm{u_{n,\eps}}_{\ell^2(\X_{n} \cap \partial_{4R} \Omega)}^2
        \lesssim \norm{\nabla_{n, \eps} u_{n,\eps}}_{\ell^2(\X_{n}^2)}^2
        \lesssim \norm{f_n}_{\ell^2(\X_{n})}^2.
    \end{equation*}
    Substituting those estimates back gives the inequality of the statement.
\end{proof}

\section{Fine asymptotics of the graph heat kernel}
\label{sec:heat_kernel}

Here we study the heat kernel on a graph and prove the estimates needed to use the heat kernel to mollify the solutions of graph Poisson equations. 

\subsection{General heat kernel properties}

We start by working with a general graph with vertices $\X_n$ and symmetric edge weights $w_{xy}\geq 0$ for $x,y\in \X_n$, as introduced in  \cref{sec:graph_calculus}. To introduce the heat kernel in this setting, we define the function $\delta_x\in \ell^2(\X_n)$ by $\delta_x(y)=n$ if $x=y$, and $\delta_x(y)=0$ otherwise. The graph function $\delta_x$ is a discrete approximation of the Dirac delta distribution centered at $x$, since it satisfies $\ip{u,\delta_x} = u(x)$ for any $u\in \ell^2(\X_n)$. 
% 
% We now define the heat kernel.
\begin{definition}\label{def:heatkernel}
For $x\in \X_n$ and $k\geq 0$, the \emph{heat kernel}  $\H^x_k\in \ell^2(\X_n)$ is the solution of the graph heat equation
\begin{equation}\label{eq:heatkernel}
\left\{
\begin{aligned}
\H^x_{k+1} &= \H^x_k - \Lr^T \H^x_k \ ,&& \text{for } k\geq 0 \\
\H^x_0 &= \delta_x,&&
\end{aligned}
\right.
\end{equation}
\end{definition}
We remind the reader of the definition of the adjoint random walk graph Laplacian $\Lr^T$ in \labelcref{eq:def_rw_graph_laplacian_adjoint}. In particular, using this definition, the propagation equation \labelcref{eq:heatkernel} for the heat kernel $\H^x_k$ centered at $x\in \X_n$ is simply
 \begin{equation}\label{eq:heat_kernel_prop}
 \H^x_{k+1}(y) =  \sum_{z \in \X_n} \frac{w_{yz}}{\deg(z)} \H^x_k(z) \ \ \text{for all} \ \ y \in \X_n,
 \end{equation}
which is simply a diffusion equation on the graph. In fact, the heat kernel $\H^x_k$ is exactly $n$ times the probability distribution for a random walk on the graph after $k$ steps, starting from node $x$ at time $k=0$ (though we do not explicitly use this property). The random walk has probability $w_{xy}/\deg(x)$ of stepping from $x$ to $y$ for $x\neq y$.  We note that we can also write the heat kernel as 
\begin{equation}\label{eq:heatkernelpowers}
\H_k^x = (I-\Lr^T)^k\delta_x.
\end{equation}

By taking inner products with the constant $\one$ function on both sides of \labelcref{eq:heatkernel} and using that $\Lr \one = 0$, we have
\[\ip{\H^x_{k+1},\one} = \ip{\H^x_{k} - \Lr^T\H^x_k,\one}   =\ip{\H^x_{k},\one}  - \ip{\H^x_k,\Lr\one} =  \ip{\H^x_k,\one}.\]
Since $\H^x_0 = \delta_x$ has unit mass, i.e., $\ip{\delta_x,\one}=1$, all heat kernels $\H^x_k$ also have unit mass, that is
\begin{equation}\label{eq:unitmass_heatkernel}
\ip{\H^x_k,\one}=1 \ \ \text{for all } k \geq 1.
\end{equation}
Furthermore, $\H_k^x$ is non-negative for all $k\geq 0$ which can be seen from \labelcref{eq:heat_kernel_prop}.

We denote by $\H_k:\X_n\times \X_n\to \R$ the function $(x,y) \mapsto \H^x_k(y)$. 
\begin{definition}\label{def:convolution_heatkernel}
For a function $u\in\ell^2(\X_n)$, we define the convolution $\H_k * u$ as the $\ell^2(\X_n)$ function
\begin{equation}\label{eq:convolution}
(\H_k * u)(x) = \ip{\H^x_k,u}.
\end{equation}
\end{definition}
Note that we have $\H_0 * u = u$ and
\[(\H_k * u)(x) = \frac{1}{n}\sum_{y\in \X_n} \H^x_k(y) u(y).\]
It turns out that, as one may expect, convolution with the heat kernel is equivalent to solving the heat equation.
\begin{proposition}\label{prop:heat}
The sequence of functions $u_k = \H_k * u$ satisfies
\begin{equation}\label{eq:heateq}
\left\{
\begin{aligned}
u_{k+1} &= u_k - \Lr u_k \ ,&& \text{for } k\geq 0 \\
u_0 &= u.&&
\end{aligned}
\right.
\end{equation}
\end{proposition}
\begin{proof}
We note that $u_k = (I-\Lr)^ku$ and so we have
\[u_k(x) =\ip{\delta_x,(I-\Lr)^ku} = \ip{(I-\Lr^T)^k \delta_x,u} = \ip{\H^x_k,u} = (\H_k * u)(x),\]
which completes the proof.
\end{proof}

\begin{remark}\label{rem:semigroup}
It follows immediately from \cref{prop:heat} that convolution with $\H_k$ satsifies a semigroup property, that is, for any $k,\ell \geq 0$ we have $\H_k * (\H_\ell * u) = \H_{k+\ell}*u$. In addition, \cref{prop:heat} gives the alternative form for the convolution with the heat kernel
\begin{equation}\label{eq:alt_convolution}
\H_k * u = (I - \Lr)^k u.
\end{equation}
\end{remark}

We caution the reader that in general $\H_k * \delta_x \neq \H_k^x$, but the two quantities are closely related.
\begin{proposition}\label{prop:Hkdelta}
For all $k\geq 1$ and $x\in \X_n$ it holds that
\begin{equation}\label{eq:Hkdelta}
\H_k*\delta_x = \deg(x)\deg^{-1}\H_k^x.
\end{equation}
\end{proposition}
\begin{proof}
It follows from \labelcref{eq:Lr} that $(I - \Lr^T)u = \deg\, (I - \Lr)(\deg^{-1}u)$ for any $u\in \ell^2(\X_n)$. By iterating this and taking $u = \delta_x$ we obtain
\[\H^x_k = (I - \Lr^T)^k\delta_x = \deg\, (I - \Lr)^k(\deg^{-1}\delta_x) = \deg(x)^{-1}\deg\, (I - \Lr)^k\delta_x.\]
The proof is completed by noting that \labelcref{eq:alt_convolution} implies that $\H_k*\delta_x = (I - \Lr)^k\delta_x$. 
\end{proof}
\begin{remark}\label{rem:symmetry}
We note that \cref{prop:Hkdelta} shows that
\[\H_k^y(x) = (\H_k*\delta_x)(y) = \deg(x)\deg(y)^{-1}\H_k^x(y).\]
In particular, the heat kernel is symmetric, i.e., $\H_k^x(y) = \H_k^y(x)$, whenever $\deg(x)=\deg(y)$. On a so-called regular graph with constant degree the heat kernel is symmetric for all $x,y$. 
\end{remark}

We can now commute the Laplacian $\Lr$ with convolution by the heat kernel $\H_k$. 
\begin{lemma}\label{lem:commute}
For all $k\geq 0$ we have $\H_k * \Lr u = \Lr (\H_k * u)$.
\end{lemma}
\begin{proof}
Write $f = \Lr u$, and let $u_k = \H_k * u$ and $f_k = \H_k * f$. Then we need to show that $\Lr u_k = f_k$.  We have $\Lr u_0 = \Lr u = f = f_0$. Now suppose that $\Lr u_j = f_j$ for all $j\leq k$. Then 
\[\Lr u_{k+1} = \Lr (u_k - \Lr u_k) = \Lr u_k - \Lr \Lr u_k = f_k - \Lr f_k = f_{k+1},\]
which completes the proof.
\end{proof}

Combining \cref{prop:Hkdelta,lem:commute} allows us to convolve solutions of Poisson equations with the heat kernel.
\begin{theorem}\label{thm:smooth_graph_poisson}
Suppose that $u\in \ell^2(\X_n)$ satisfies the Poisson equation
\begin{equation}\label{eq:graph_poisson}
\L u = \sum_{y\in \Gamma}a_y\delta_y,
\end{equation}
where $\Gamma \subset \X_n$ and $a_y\in \R$ for $y\in \Gamma$. Then $u_k \defeq   \H_k * u$ satisfies the Poisson equation
\begin{equation}\label{eq:graph_poisson_smoothed}
\L u_k = \sum_{y\in \Gamma}a_y\H^y_k.
\end{equation}
\end{theorem}
\begin{proof}
By the identity $\L u = \deg \Lr u$ we find that $u$ satisfies
\[\Lr u = \deg^{-1}\sum_{y\in \Gamma}a_y\delta_y = \sum_{y\in \Gamma}a_y\deg(y)^{-1}\delta_y.\]
Convolving with the heat kernel $\H_k$ on both sides and using \cref{lem:commute} yields
\[\Lr u_k =\Lr (\H_k * u)= \H_k * \Lr u = \sum_{y\in \Gamma}a_y\deg(y)^{-1}\H_k*\delta_y.\]
We now use \cref{prop:Hkdelta} to find that
\[\Lr u_k = \sum_{y\in \Gamma}a_y\deg(y)^{-1}\deg(y)\deg^{-1}\H_k^y = \deg^{-1}\sum_{y\in \Gamma}a_y\H_k^y.\]
The result follows by multiplying by $\deg$ on both sides and using that $\deg \Lr u_k = \L u_k$. 
\end{proof}
\cref{thm:smooth_graph_poisson} suggests that we can replace the singular Poisson equation \labelcref{eq:graph_poisson} with the Poisson equation \labelcref{eq:graph_poisson_smoothed} with smoothed source terms, i.e., we replace the delta functions with the heat kernel for some, possibly large, number of steps $k$. This gives us an effective way to smooth the solution of a graph Poisson equation. The final ingredient we need is some control on the difference $u-\H_k*u$ between $u$ and its convolution with the heat kernel. For this, we use the mean value property.
%We also have the mean value property, whose proof follows directly from \cref{prop:heat}.
%\begin{lemma}\label{lem:mvp}
%Suppose that $\Lr u=0$. Then $u = \H_k * u$ for all $k\geq 0$.  
%\end{lemma}
%The mean value property in \cref{lem:mvp} is not so useful in practice, since the only harmonic functions on a connected graph are the constant functions, for which $u=\H_k *u$ is trivial. We give a more useful quantitative version of the mean value property below.
\begin{lemma}[Mean Value Property]\label{lem:mvp_quant}
Suppose that $\Lr u =f$. Then for any $k\geq 1$ we have
\begin{equation}\label{eq:mvp_quant}
u = \H_k * u + \sum_{j=0}^{k-1}\H_j * f.
\end{equation}
\end{lemma}
\begin{proof}
Let us write $u_j= \H_j *u$ and $f_j = \H_j * f$.  By \cref{prop:heat,lem:commute}  we have
\[u_{j+1} = u_j - \Lr u_j = u_j - \Lr (\H_j * u) = u_j - \H_j * \Lr u = u_j - f_j.\]
Therefore
\[u_{j} - u_{j+1} =f_j.\]
We can sum this from $j=0$ to $j=k-1$ and use $u_0=u$ to obtain
\[u - u_k =\sum_{j=0}^{k-1}f_j.\]
which completes the proof.
\end{proof}

An important consequence of these results is that we have precise control over the solutions of the graph Poisson equation for regular data, and data which is convolved with the graph heat kernel. We recall for the reader that the degree-weighted average $(u)_{\deg}$ was defined in  \cref{sec:graph_calculus}. 

\begin{theorem}\label{thm:difference_graph_smoothed}
Suppose that $u\in \ell^2(\X_n)$ satisfies the Poisson equation \labelcref{eq:graph_poisson}, and assume that the compatibility condition $\sum_{y\in \Gamma}a_y = 0$  holds. Then $u_k \defeq   \H_k * u$ satisfies $(u_k)_{\deg} = (u)_{\deg}$ and
\begin{equation}\label{eq:difference_graph_smoothed}
u - u_k = \deg^{-1}\sum_{y\in \Gamma}a_y\sum_{j=0}^{k-1}\H^y_j.
\end{equation}
\end{theorem}
\begin{proof}
As in the proof of \cref{thm:smooth_graph_poisson} we have $\Lr u = f$, where
\[f=\sum_{y\in \Gamma}a_y\deg(y)^{-1}\delta_y, \ \ \text{and}  \ \ \H_j * f = \deg^{-1}\sum_{y\in \Gamma}a_y \H_j^y.\] 
Then \labelcref{eq:difference_graph_smoothed} follows directly from the mean value property in \cref{lem:mvp_quant}. Taking the inner product with $\deg$ on both sides of \labelcref{eq:difference_graph_smoothed} and using \labelcref{eq:unitmass_heatkernel} yields
\[\ip{u-u_k,\deg} = \sum_{y\in \Gamma}a_y\sum_{j=0}^{k-1}\ip{\H^y_j,\one} = k\sum_{y\in \Gamma}a_y = 0.\]
It follows that $(u)_{\deg} = (u_k)_{\deg}$, which completes the proof.
\end{proof}

\subsection{Heat kernel asymptotics}\label{sec:hka}

% We now consider the setting of a random geometric graph. We let $x_1,\dots,x_n$ be a sequence of \emph{i.i.d.}~random variables on $\Omega \subset \R^d$ with a Lipschitz continuous density $\rho:\Omega \to \R$ that is bounded above and below by positive constants. That is, there exists $0 < \rho_{\min} \leq \rho_{\max}<\infty$ such that
% \[\rho_{\min} \leq \rho \leq \rho_{\max}.\]
% We let $\X_n = \{x_1,\dots,x_n\}$.  Let $\eta:[0,\infty)\to [0,\infty)$ be nonincreasing and satisfy $\eta(t)=0$ for $t\geq 1$. We also define $\sigma_\eta = \int_{\R^d} \eta(|x|) x_d^2 \d x$, and we assume that
% \begin{equation}\label{eq:unitmass}
% \int_{\R^d}\eta(|x|) \d x = 1.
% \end{equation}
% We also assume that $\eta$ is continuous at $t=0$. For $\varepsilon >0$, define $\eta_\varepsilon \defeq   \tfrac{1}{\varepsilon^d} \eta(\tfrac{t}{\varepsilon})$.  
% %Throughout this section we will identify $\eta$ with the function $x\mapsto \eta(|x|)$. For $k\geq 1$ and $\epsilon>0$ we define
% %\[\psi_{k, \varepsilon} \defeq   \underbrace{\eta_\varepsilon * \cdots * \eta_\varepsilon}_{k\textrm{ times}},\] 
% A random geometric graph has edge weights 
% \[w_{n,\epsilon}(x,y) = \eta_\eps(|x-y|),\]
% and degree
% \[\deg_{n,\epsilon}(x) = \sum_{i=1}^n\eta_\eps(|x-x_i|),\]
% for $x,y\in \X_n$. However, this definition also makes sense for any $x,y\in\Omega$.
% % While we usually are only concerned with the degree $\deg_{n,\epsilon}(x_j)$ for nodes $x_j$ in the graph, the definition for any $x\in \Omega$ will be convenient below.

Now we will return to the random geometric graph setting introduced in  \cref{sec:graph_calculus_rg}, and studied previously in \cref{sec:continuum_limits}. Throughout the rest of this section, we assume that  $\eta$ satisfies \cref{ass:eta}, $\Omega$ satisfies \cref{ass:omega_lipschitz} and $\rho$ satisfies \cref{ass:rho_c1a} for some $0 < \alpha \leq 1$. All constants will be implicitly taken to depend on $\Omega$, $d$, $\rho_{min}$, $\rho_{max}$, and $\eta$. When constants depend on $[\nabla \rho]_\alpha$, we will explicitly denote this dependence. 

Our object of study is the heat kernel $\H^x_k$ and its asymptotics as $k\to \infty$ on a random geometric graph.  Towards this end, we first note that the heat kernel propagation equation \labelcref{eq:heat_kernel_prop} can be rewritten as for a random geometric graph as
\begin{equation}\label{eq:heat_prop}
\H^{x}_{k+1}(x_i) = \sum_{j=1}^n\frac{\eta_\eps(|x_i-x_j|)}{\deg_{n,\epsilon}(x_j)}\H^{x}_k(x_j),
\end{equation}
where $x\in \Omega$. We remind the reader of the definition of the degree $\deg_{n,\eps}$ in \labelcref{eq:def_deg_ne}. In particular, since $H^{x}_0(x_j) = \delta_{x}(x_j)$ we have
\begin{equation}\label{eq:H1}
\frac{1}{n}\H^{x}_{1}(x_i) = \frac{\eta_\eps(|x_i-x|)}{\deg_{n,\epsilon}(x)}.
\end{equation}
We remark, in particular, that we allow any $x\in \Omega$ in the heat kernel and not just nodes $x\in \X_n$ in the graph.

In order to study the propagation of the heat kernel from $k$ to $k+1$, we define, for a bounded Borel measurable function $\phi:\Omega\to\R$, the function $\M_{n,\epsilon}\phi:\Omega \to \R$ by
\begin{equation}\label{eq:Mne}
\M_{n,\epsilon}\phi(x) = \sum_{j=1}^n \frac{\eta_\eps(|x-x_j|)}{\deg_{n,\epsilon}(x_j)} \phi(x_j).
\end{equation}
We also define the associated nonlocal averaging operator $\M_{\epsilon}:L^2(\Omega) \to L^2(\Omega)$ by 
\begin{equation}\label{eq:Me}
\M_{\epsilon}\phi(x) = \int_\Omega\eta_\eps(|x-y|)\hat\rho_\eps(y)^{-1}\rho(y)\phi(y)\d y,
\end{equation}
where we remind the reader of the definition
\begin{equation*}
\hat\rho_\eps(x) \defeq   \int_\Omega \eta_\eps(|x-y|)\rho(y)\d y
\end{equation*}
from \labelcref{eq:rhoe}. Recall also the following bounds, given in \labelcref{eq:rhoe_bounds}.%We note that since $\Omega$ has a Lipschitz boundary, there exists $C>0$ such that
\begin{equation*}%\label{eq:rhoe_bounds}
C\rho_{\min}\leq \hat\rho_\eps(x) \leq \rho_{\max},
\end{equation*}
for some $C>0$ depending of the domain $\Omega$.
%The adjustment for the bound below results from the fact that when $x$ is close to the boundary $\partial \Omega$, then part of the support of $\eta_\eps$ may lie outside of the domain $\Omega$.
\begin{remark}\label{rem:rhoe}
Since $\rho$ is Lipschitz, whenever $B(x,\epsilon)\subset \Omega$ we can make a change of variables $z = (y-x)/\epsilon$ and Taylor expand to obtain
\[\hat\rho_\eps(x) = \int_{B(x,\epsilon)} \eta_\eps(|x-y|)\rho(y)\d y = \int_{B(0,1)}\eta(|z|)\rho(x + \epsilon z)\d z = \rho(x) + \O(\epsilon).\]
Thus, in the interior of the domain $\Omega$, $\hat\rho_\eps$ offers an $\O(\epsilon)$ accurate approximation of $\rho$ and we have
\[\M_{\epsilon}\phi(x) = \int_{B(x,\epsilon)}\eta_\eps(|x-y|)\phi(y)\d y + \O(\|\phi\|_{L^\infty(B(x,\epsilon)}\epsilon),\]
provided $B(x,\epsilon) \subset \Omega$. Thus, at least to the first order in $\epsilon$, the averaging operator $\M_\eps$ is independent of $\rho$. Furthermore, if $\rho\in C^{1,\alpha}(\Omega)$ for $\alpha \in (0,1]$, then we can use the improved estimate $\hat\rho_\eps(x) = \rho(x) + \O([\nabla \rho]_\alpha\epsilon^{1+\alpha})$ to obtain
\[\M_{\epsilon}\phi(x) = \int_{B(x,\epsilon)}\eta_\eps(|x-y|)\phi(y)\d y + \O([\nabla \rho]_\alpha\|\phi\|_{L^\infty(B(x,\epsilon)}\epsilon^{1+\alpha}),\]
provided, again, that $B(x,\epsilon) \subset \Omega$. 
% These approximations  will be useful later in Section ??.
\end{remark}
We will require some basic estimates on $\M_\epsilon$. For this, we introduce the notation
\begin{equation}\label{eq:etax}
\eta_\eps^x(y) = \eta_\eps(|x-y|).
\end{equation}
\begin{proposition}\label{prop:Me}
The following hold for all $\phi\in L^\infty(\Omega)$, $k\geq 1$, and $\eps>0$.
\begin{enumerate}[label=(\roman*)]
\item $\displaystyle\int_\Omega \rho \,\M_\epsilon^k \phi \d x = \int_\Omega \rho \,\phi \d x$.
\item $\displaystyle\|\hat\rho_\eps^{-1}\M^k_\epsilon\phi\|_{L^\infty(\Omega)} \leq \|\hat\rho_\eps^{-1}\phi\|_{L^\infty(\Omega)}$.
\item $\displaystyle\|\M^k_\epsilon\phi\|_{L^\infty(\Omega)} \leq \|\hat\rho_\eps\|_{L^\infty(\Omega)}\|\hat\rho_\eps^{-1}\phi\|_{L^\infty(\Omega)}$.
\item For all $x,y\in \Omega$ we have $\M_\eps^k\eta_\eps^x(y) = \M^k_\eps\eta^y_\eps(x)$. 
\end{enumerate}
\end{proposition}
\begin{proof}
For (i), we have
\begin{align*}
\int_{\Omega}\rho(x)\M_{\epsilon}\phi(x)\d x &= \int_\Omega\rho(x) \int_\Omega\eta_\eps(|x-y|)\hat\rho_\eps(y)^{-1}\rho(y)\phi(y)\d y\d x\\
&=\int_\Omega \hat\rho_\eps(y)^{-1}\rho(y)\phi(y)\int_\Omega\rho(x)\eta_\eps(|x-y|)\d x\d y=\int_\Omega \rho(y)\phi(y)\d y.
\end{align*}
The result then follows by induction.

For (ii) and (iii), we define the inner product
\[(u,v)_\rho = \int_\Omega u\!\:v\rho\d x,\]
and the adjoint operator
\[\M^*_\eps \phi(x) =\hat\rho_\eps(x)^{-1}\int_\Omega\eta_\eps(|x-y|)\rho(y)\phi(y)\d y,\]
which satisfies $(\M_\eps u,v)_\rho= (u,M_\eps^* v)_\rho$. We note that $\M_\eps \phi = \hat\rho_\eps \M^*_\eps(\hat\rho_\eps^{\,-1}\phi)$ and hence $\M^k_\eps\phi = \hat\rho_\eps{\M^*_\eps}^k(\hat\rho_\eps^{\,-1}\phi)$. Furthermore, we clearly have $\|\M^*_\eps \phi\|_{L^\infty(\Omega)} \leq \|\phi\|_{L^\infty(\Omega)}$, and so 
\[\|\hat\rho_\eps^{\,-1}\M^k_\epsilon\phi\|_{L^\infty(\Omega)} = \|{\M^*_\eps}^k(\rho^{-1}_\eps \phi)\|_{L^\infty(\Omega)} \leq \|\hat\rho_\eps^{\,-1}\phi\|_{L^\infty(\Omega)},\]
which establishes (ii), and (iii) follows by writing $\|\M^k_\epsilon\phi\|_{L^\infty(\Omega)} = \|\hat\rho_\eps\hat\rho_\eps^{\,-1}\M^k_\eps \phi\|_{L^\infty(\Omega)}$ and applying (ii).

To prove (iv), we first note that $\M_\eps \phi(x) = (\hat\rho_\eps^{\,-1}\eta^x_\eps,\phi)_\rho$ for any $\phi$. Thus, for $k=1$ we have
\[\M_\eps \eta^y_\eps(x) = (\hat\rho_\eps^{\,-1}\eta^x_\eps,\eta^y_\eps)_\rho = (\hat\rho_\eps^{\,-1}\eta^y_\eps,\eta^x_\eps)_\rho = \M_\eps \eta^x(y). \]
Now, for any $k\geq 1$ we use the identity $\M^k_\eps\phi = \hat\rho_\eps{\M^*_\eps}^k(\hat\rho_\eps^{\,-1}\phi)$ with $\phi=\eta_\eps^x$ to obtain
\begin{align*}
\M^{k+1}_\eps \eta^y_\eps(x) &= (\hat\rho_\eps^{\,-1}\eta^x_\eps,\M^k_\eps \eta_\eps^y)_\rho = ({\M^*_\eps}^k(\hat\rho_\eps^{\,-1}\eta^x_\eps),\eta_\eps^y)_\rho \\
&= (\hat\rho_\eps^{\,-1}\M^k_\eps \eta_\eps^x,\eta_\eps^y)_\rho=(\hat\rho_\eps^{\,-1}\eta_\eps^y,\M^k_\eps \eta_\eps^x)_\rho = \M^{k+1}_\eps \eta^x_\eps(y).\qedhere
\end{align*}
\end{proof}

Our main tool in this section is Bernstein's inequality from \cref{thm:bernstein}. Using it, we can prove a concentration result for $\M_{n,\epsilon}$. 
\begin{lemma}\label{lem:iteration}
There exists $C>0$ such that for all Borel measurable and bounded $\phi:\Omega\to\R$ and $0 < \lambda \leq 1$ we have that
\begin{equation}\label{eq:iteration}
\M_{n,\epsilon}(\phi(x_i)+\theta) = \M_\eps \phi(x_i)+ \O\left(\|\phi\|_{L^\infty(\Omega\cap B(x_i,\epsilon))} \lambda + |\theta|\right)
\end{equation}
holds for all $i=1,\dots,n$ and $\theta\in \R$ with probability at least $1 - 4n\exp\left( -C n\epsilon^d \lambda^2\right)$.
\end{lemma}
\begin{proof}
As in the proof of  \cref{lem:estimate_deg_rho_average}, for $0 < \lambda \leq 1$, we can use  \cref{prop:degree} and a union bound to show that 
\begin{equation}\label{eq:degreej}
\left|\frac{1}{n-1}\deg_{n,\epsilon}(x_i) - \hat\rho_\eps(x_i)\right| \leq \frac{\lambda}{4}\hat\rho_\eps(x_i) + \frac{1}{n-1}\eta_\eps(0) \leq \frac{\lambda}{2}\hat \rho_\eps(x_i)
\end{equation}
holds for all $i$ has probability at least $1-2n\exp\left( -C n \epsilon^d \lambda^2\right)$, provided that $\lambda n\eps^d \geq K$ where $K$ depends on $\rho_{min}$, $\Omega$, and $\eta(0)$ and we used \cref{ass:neps}. By adjusting the constant $C$ in the probability lower bound in the lemma, as we did in  \cref{lem:estimate_deg_rho_average}, we can restrict our attention to the case that $\lambda n \eps^d \geq K$. 

Now, assuming this event holds, we have
\begin{equation}\label{eq:dne1}
\left(1-\frac{\lambda}{2}\right)\hat\rho_\eps(x_i) \leq \frac{1}{n-1}\deg_{n,\epsilon}(x_i)  \leq \left(1+\frac{\lambda}{2}\right)\hat\rho_\eps(x_i)
\end{equation}
and so
\[\frac{n-1}{\deg_{n,\epsilon}(x_i)} = \hat\rho_\eps(x_i)^{-1} + \O(\lambda). \]
Therefore
\begin{align*}
\M_{n,\epsilon}\phi(x) &= \frac{1}{n-1}\sum_{j=1}^n \left(\hat\rho_\eps(x_i)^{-1} + \O(\lambda)\right)\eta_\eps(|x-x_j|)\phi(x_j)\\
&=\frac{1}{n-1}\sum_{j=1}^n \eta_\eps(|x_i-x_j|)\hat\rho_\eps(x_j)^{-1} \phi(x_j) + \O\left(\|\phi\|_{L^\infty(\Omega\cap B(x_i,\epsilon))}\lambda \right)\\
&=\frac{1}{n-1}\sum_{i\neq j=1}^n \eta_\eps(|x_i-x_j|)\hat\rho_\eps(x_j)^{-1} \phi(x_j) +  \O\left(\|\phi\|_{L^\infty(\Omega\cap B(x_i,\epsilon))}\lambda \right),
\end{align*}
where in the last line we again used that $\lambda n \eps^d \geq K$ and  \cref{ass:neps} to remove the term $j=i$ from the sum and absorb it into the error term.

We now condition on $x_i$ and apply Bernstein's inequality from \cref{thm:bernstein} to the sum over $j\neq i$ above.  We have $\mu= \M_\eps\varphi(x_i)$,  $b = C \|\phi\|_{L^\infty(\Omega\cap B(x_i,\epsilon))}\epsilon^{-d}$ and 
\[\sigma^2 \leq C \|\phi\|_{L^\infty(\Omega\cap B(x_i,\epsilon))}^2 \|\eta_\eps(|\cdot - x_i|)^2\|_{L^1(B(x_i,\eps))} \leq C \|\phi\|_{L^\infty(\Omega\cap B(x_i,\epsilon))}^2\epsilon^{-d}.\]
Combining this with a union bound over $i=1,\dots,n$ completes the proof when $\theta=0$. 

When $\theta\neq 0$ we have 
\[\M_{n,\epsilon}(\phi(x_i)+\theta) =\M_{n,\epsilon}\phi(x_i) + \O\left(|\theta|\sum_{j=1}^n \frac{\eta_\eps (|x-x_j|)}{\deg_{n,\epsilon}(x_j)} \right) = \M_{n,\epsilon}\phi(x_i) + \O(|\theta|)\]
due to \labelcref{eq:dne1}. 
\end{proof}

We now have our main result in this section which states an asymptotic expansion of the heat kernel in terms of a repeated averaging operator.

\begin{theorem}\label{thm:heatkernel_nonlocal}
There exists $C>0$ such that for all $x_0\in \Omega$ and $0<\lambda \leq 1$
\begin{equation}\label{eq:heatkerne_nonlocal}
\H^{x}_{k}(x_i) = \hat\rho_\eps(x)^{-1}\M_{\eps}^{k-1}\eta_\eps^x(x_i) + \O\left(\lambda\sum_{j=0}^{k-2}\|\M_{\eps}^j\eta_\eps^x\|_{L^\infty(\Omega\cap B(x_i,\epsilon))}\right)
\end{equation}
holds for all $x\in \X_n\cup \{x_0\}$, $i=1,\dots,n$, and $k=2,\dots,m$ with probability at least $1-12mn^2\exp\left( -C n \epsilon^d \lambda^2\right)$.
\end{theorem}
% \begin{remark}
%     \todo{do we keep this remark? make it corollary for combination}
%     Using (iii) in \cref{prop:Me}, we can bound the error term as follows
%     \begin{align*}
%         &\lambda\sum_{j=0}^{k-2}\|\M_{\eps}^j\eta_\eps^x\|_{L^\infty(\Omega \cap B(x_i,\epsilon))}
%         \leq 
%         \lambda
%         (k-1)
%         \|{\hat\rho_\eps}\|_{L^\infty(\Omega)}
%         \|{\hat\rho_\eps^{-1}\eta_\eps^x}\|_{L^\infty(\Omega)}
%         \\&\qquad
%         \leq 
%         \lambda
%         (k-1)
%         \frac{\rho_{\max} + \O(\eps)}{\rho_{\min}}
%         \|\eta_\eps^x\|_{L^\infty(\Omega)}
%         % \\        &
%         \leq 
%         \lambda (k-1) C(\rho,\eta)\eps^{-d}
%     \end{align*}
% \end{remark}
\begin{proof}
The proof proceeds by induction. We first establish the base case. Let $x\in \Omega$. By \labelcref{eq:heat_prop}, \labelcref{eq:H1} and \cref{lem:iteration} we have
\begin{align*}
\frac{1}{n}\deg_{n,\epsilon}(x)\H^{x}_{2}(x_i) &= \sum_{j=1}^n\frac{\eta_\eps(|x_i-x_j|)}{\deg_{n,\epsilon}(x_j)}\eta_\eps(|x_j-x|)\\
&= \M_{n,\epsilon}\eta_\eps^x(x_i) = \M_{\epsilon}\eta_\eps^x(x_i) + \O(\|\eta_\eps^x\|_{L^\infty(\Omega \cap B(x_i,\epsilon))}\lambda),
\end{align*}
holds for all $i$ with probability at least $1-4n\exp\left( -C n \epsilon^d \lambda^2\right)$. Invoking \cref{prop:degree}  as in \labelcref{eq:degreej} we have
\[\hat\rho_\eps(x)\H^{x}_{2}(x_i) = \M_{\epsilon}\eta_\eps^x(x_i) + \O(\|\eta_\eps^x\|_{L^\infty(\Omega \cap B(x_i,\epsilon))}\lambda),\]
for all $i$ with probability at least $1-6n\exp\left( -C n \epsilon^d \lambda^2\right)$. As before, we union bound over $x\in \X_n\cup \{x_0\}$, and  the probability decreases to $1-12n^2\exp\left( -C n \epsilon^d \lambda^2\right)$.

For the inductive step, let us assume that for some $m\geq 2$ we have that
\[\hat\rho_\eps(x)\H^{x}_{k}(x_i)=\M_{\eps}^{k-1}\eta_\eps^x(x_i) + \O\left(\lambda\beta_k\right), \ \ \beta_k = \sum_{j=0}^{k-2}\|\M_{\eps}^j\eta_\eps^x\|_{L^\infty(\Omega\cap B(x_i,\epsilon))},\]
holds for all $i$ and $k=2,\dots,m$ with probability at least $1-12mn^2\exp\left( -C n \epsilon^d \lambda^2\right)$.  Let $k=m$. By \labelcref{eq:heat_prop}, \cref{prop:Me} (ii) and \cref{lem:iteration} we have
\begin{align*}
\hat\rho_\eps(x)\H^{x}_{k+1}(x_i) &= \sum_{j=1}^n\frac{\eta_\eps(|x_i-x_j|)}{\deg_{n,\epsilon}(x_j)}\left( \M_{\eps}^{k-1}\eta_\eps^x(x_j) + \O(\lambda\beta_k)\right)\\
&=\M_{n,\epsilon}\left(\M_{\eps}^{k-1}\eta_\eps^x + \O(\lambda\beta_k)\right)(x_i)\\
&=\M_{\eps}^{k}\eta_\eps^x(x_i) + \O\left(\lambda\beta_k + \lambda \|\M_{\eps}^{k-1}\eta_\eps^x\|_{L^\infty(\Omega\cap B(x_i,\epsilon))}\right)\\
&=\M_{\eps}^{k}\eta_\eps^x(x_i) + \O\left(\lambda\beta_{k+1}\right)
\end{align*}
for all $i$ and $x\in \X_n\cup\{x_0\}$ with probability at least $1-12(k+1)n^2\exp\left( -C n \epsilon^d \lambda^2\right)$, which completes the proof.
\end{proof}

\subsection{Estimates for repeated averaging}\label{sec:rep_avg}

We now turn to study the repeated averaging operator $\M^k_\epsilon$. In particular, we show that $\M^k_\eps$ is asymptotic in the repeated convolution
\[\psi_{k, \varepsilon} \defeq   \underbrace{\eta_\varepsilon * \cdots * \eta_\varepsilon}_{k\textrm{ times}},\] 
where throughout this section we will identify $\eta$ with the function $x\mapsto \eta(|x|)$. We note that we can write
\begin{equation}\label{eq:psi_rewrite}
\psi_{k, \varepsilon}(x) = (\eta_\varepsilon * \cdots * \eta_\varepsilon)(x) = \eps^{-d}(\eta * \cdots *\eta)\left( \frac{x}{\epsilon}\right) = \epsilon^{-d}\psi_{k}\left( \frac{x}{\epsilon}\right),
\end{equation}
where $\psi_k\defeq  \psi_{k,1}$. Thus, we will often focus our attention on $\psi_k$. Throughout this section we use the notation $\eps_k =  \sqrt{k}\epsilon$, and make the following standing assumption.
\begin{assumption}\label{ass:epsk}
We assume that $0 < \epsilon \leq \frac{1}{2}$ and $k\geq 1$ satisfy $\epsilon_k = \epsilon \sqrt k \leq 1$. 
\end{assumption}
In particular, \cref{ass:epsk} implies that $\eps_k \leq 1$, so $k \leq \eps^{-2}$.  We also define
\begin{equation}\label{eq:Rkdef}
R_k = 5\epsilon + \epsilon_k\sqrt{8d \log(k\epsilon^{-(d+2)})},
\end{equation}
for $k\geq 1$. The radius $R_k$ is sufficiently large to contain the effective support of $\psi_{k,\epsilon}$, up to error terms that are exponentially small. In particular, it is straightforward to check that for $k\geq 1$ we have 
\begin{equation}\label{eq:Rk_prop}
R_k \geq 5\eps, \ \ R_k \leq 6(d+2)\eps_k \log(\eps^{-1})^{\frac{1}{2}}, \ \ \text{and} \ \ \exp\left( -\frac{(R_k - \eps)^2}{8d\eps_k^2}\right)\leq k^{-1}\eps^{d+2}.
\end{equation}
Some of the results in this section hold for smaller values of $R_k$; in particular, the logarithmic terms are not always sharp. For simplicity, we have fixed one value of $R_k$ that works in the majority of the paper. We also mention that some of the proofs in this section are elementary results that we believe are relatively well-known. For completeness we include the proofs of such results in \cref{app:rep_avg}.

Our first result is a standard estimate based on the Hoeffding bounding method, which shows that the bulk of the mass of $\psi_{k,\eps}$ is concentrated in a ball of radius $R_k$. 
Note that $\psi_{k,\eps}$ has unit mass on $\R^d$ as a consequence of \cref{ass:eta} which states that $\eta$ has unit mass.
\begin{lemma}\label{lem:Hoeffding}
For any $k\geq 1$, $\epsilon>0$ and $t>0$ we have
\begin{equation}\label{eq:Hoeffding}
\int_{\{|x| > t\}} \psi_{k,\epsilon}(x) \d x \leq 2d\exp\left(-\frac{t^2}{2d\eps_k^2}\right).
\end{equation}
\end{lemma}
The proof of \cref{lem:Hoeffding} is given in  \cref{app:rep_avg}.  We now upgrade this tail bound to pointwise Gaussian upper bounds. 
\begin{proposition}\label{prop:psi_gaussian_upper}
For all $x\in \R^d$ and $\eps>0$ the following hold. 
\begin{enumerate}[label=(\roman*)]
\item There exists $C_1>0$, depending only on $\eta$ and $d$, such that for all $k\geq 1$ we have
\begin{equation}\label{eq:psi_pointwise}
|\psi_{k,\epsilon}(x)| \leq C_1\min\left\{ \eps_k^{-d},\eps^{-d}\exp\left( -\frac{|x|^2}{8d\eps_k^2}\right)\right\}.
\end{equation}
\item 
% If there exists $M>0$ and $s>0$ such that 
% \begin{equation}\label{eq:eta_fourier_decay}
% |\hat{\eta}(y)| \leq \frac{M}{1 + |y|^s} \ \ \text{for all } y\in \R^d,
% \end{equation}
% \todo[inline]{Look for a reference or try to modify proof of Riemann-Lebesgue Lemma. L: Done with $s=1$}
% then for $k> \frac{3}{s}$ 
For $k\geq 3$ it holds that $\psi_{k,\epsilon}$ is continuously differentiable and there exists $C_2>0$ depending only on $\eta$ and $d$ such that 
\[|\nabla \psi_{k,\epsilon}(x)|\leq C_2 \epsilon_k^{-(d+1)}.\]
\item Let $T:\Omega\to \Omega$ be Borel measureable and satisfy $|T(x)-x|\leq \epsilon_k$ for all $x\in \Omega$.  There exists $C_3>0$, depending only on $\eta$, $d$, and $|\Omega|$, such that for all $k\geq 1$ and $x_0\in \Omega$ we have
\begin{equation}\label{eq:psi_lp}
\|\psi_{k,\epsilon}(T(\cdot)-x_0)\|_{L^p(\Omega)} \leq C_3\eps_k^{-d + \frac{d}{p}}\log(\eps^{-1})^{\frac{d}{2p}}.
\end{equation}
\item There exists $C_4>0$ such that for every $x_0\in \Omega$, $\eps,\delta > 0$, and $k\geq 3$ we have
\begin{equation}\label{eq:oscillation_bound_M_ke}
\|\osc_{\Omega \cap B(\cdot,\delta)} \psi_{k,\epsilon}(\cdot - x_0)\|_{L^1(\Omega)} \leq  C_4 ( \eps_k^{-1}\delta\log(\eps^{-1})^{\frac{d}{2}} + \eps^2).
\end{equation}
\end{enumerate}
\end{proposition}
The proof of \cref{prop:psi_gaussian_upper} uses Fourier techniques and is given in  \cref{app:rep_avg}. 
We now establish related bounds in the graph setting with high probability. 
\begin{corollary}\label{lem:psik_graph_bounds}
There exists $C_1,C_2>0$ such that the following hold.
\begin{enumerate}[label=(\roman*)]
\item  For any $p\geq 1$ there exists $C_3>0$ such that the event that
\[\ng{p}{\psi_{k,\epsilon}(\cdot-x_0)} \leq C_3 \epsilon_k^{-d + \frac{d}{p}}\log(\epsilon^{-1})^{\frac{d}{2p}},\]
holds for all $x_0\in \Omega$ has probability at least $1-C_1 n \exp\left( -C_2 n\epsilon^d\right)$. 
\item There exists $C_4>0$ such that for all $k\geq 3$ and $\eps,\delta>0$ the event that
\[\ng{1}{\psi_{k,\epsilon}(\cdot - x_0) - \psi_{k,\epsilon}(\cdot - y_0)}  \leq C_4 \left(\eps_k^{-1}\delta\log(\eps^{-1})^{\frac{d}{2}} + \eps^2\right),\]
holds for all $x_0,y_0\in \Omega$ with $|x_0-y_0| \leq \delta$ has probability at least $1-C_1 n \exp\left( -C_2 n\delta^d\right)$. 
\end{enumerate}
\end{corollary}
\begin{proof}
To prove (i), we assume \cref{thm:transportation} holds with $\delta=\eps\leq \epsilon_k$ and $\lambda=\frac{\rho_{\min}}{8\rho_{\max}}$, which has probability at least $1-C_1 n \exp\left( -C_2 n\epsilon^d\right)$. Using \cref{prop:psi_gaussian_upper} (iii) we have
\begin{align*}
\ng{p}{\psi_{k,\epsilon}(\cdot-x_0)} &= \left(\frac{1}{n}\sum_{i=1}^n \psi_{k,\epsilon}(x_i - x_0)^p\right)^{\frac{1}{p}} \\
&= \left(\int_\Omega \psi_{k,\epsilon}(T_\delta(y)-x_0)^p\rho_\delta(y)\, dy\right)^{\frac{1}{p}} \\
&\leq C\|\psi_{k,\epsilon}(T_\delta(\cdot)-x_0)\|_{L^p(\Omega)} \leq C \epsilon_k^{-d + \frac{d}{p}}\log(\epsilon^{-1})^{\frac{d}{2p}}.
\end{align*}

To prove (ii) we assume \cref{thm:transportation} holds with $\delta\geq |x_0-y_0|$ and $\lambda=\frac{\rho_{\min}}{8\rho_{\max}}$, which has probability at least $1-C_1 n \exp\left( -C_2 n\delta^d\right)$. Using \cref{prop:psi_gaussian_upper} (iv) we have
\begin{align*}
\ng{1}{\psi_{k,\epsilon}(\cdot - x_0) - \psi_{k,\epsilon}(\cdot - y_0)} &= \frac{1}{n}\sum_{i=1}^n|\psi_{k,\epsilon}(x_i - x_0) - \psi_{k,\epsilon}(x_i - y_0)|\\
&=\int_\Omega |\psi_{k,\epsilon}(T_\delta(x) - x_0) - \psi_{k,\epsilon}(T_\delta(x) - y_0)| \rho_\delta(x)\d x\\
&\leq C\int_\Omega |\psi_{k,\epsilon}(T_\delta(x) - x_0) - \psi_{k,\epsilon}(T_\delta(x) - y_0)|\d x\\
&\leq C \| \osc_{\Omega\cap B(x,2\delta)}\psi_{k,\epsilon}(\cdot - x_0)\|_{L^1(\Omega)}\\
&\leq C (\eps_k^{-1}\delta\log(\eps^{-1})^{\frac{d}{2}} + \eps^2).\qedhere
\end{align*}
\end{proof}

We now prove our main result in this section, which relates the repeated averaging operator $\M^k_\eps$ to the repeated convolution $\psi_{k,\epsilon}$, centered at a point sufficiently far from the boundary of the domain. 
\begin{theorem}\label{thm:Me}
Let $R\geq R_k$, and let $x_0\in \Omega$ with $B(x_0,R)\subset \Omega$. Then for $[\nabla \rho]_\alpha k\epsilon^{1+\alpha} \ll 1$, $\eps \ll 1$ and all $x\in \Omega$ we have
\begin{equation}\label{eq:Me_psi}
\M^{k}_\eps\eta^{x_0}_\eps(x)  = (1 + \O([\nabla \rho]_\alpha k\epsilon^{1+\alpha}))\psi_{k+1,\epsilon}(x-x_0) + \O\left( \epsilon^{2}\exp\left( -\frac{R^2}{4d\epsilon_k^2}\right) \right).
\end{equation}
\end{theorem}
\begin{proof}
By \cref{lem:Hoeffding}, the restriction $R\geq R_k$, and \labelcref{eq:Rk_prop}, for any $j\leq k$ we have
\begin{equation}\label{eq:tail}
\int_{\{|x|> R\}}\psi_{j,\epsilon}(x)\d x \leq 2dk^{-1}\epsilon^{d+2} \exp\left( -\frac{R^2}{4d\epsilon_k^2}\right).
\end{equation}
Also by \labelcref{eq:Rk_prop} we have $B(x_0,2\epsilon)\subset \Omega$ as well. Since $\rho\in C^{1,\alpha}(\Omega)$ and $B(x_0,2\epsilon)\subset \Omega$, we can show, as in \cref{rem:rhoe}, that for $\epsilon \ll 1$ we have $\hat\rho_\eps(y) = \rho(y) + \O([\nabla \rho]_\alpha\epsilon^{1+\alpha})$ for $y\in B(x,\epsilon)$ and so 
\[\M_\eps\eta^{x_0}_\eps(x) = (1+\O([\nabla \rho]_\alpha \eps^{1+\alpha}))\int_\Omega\eta_\eps(x-y)\eta_\eps(x_0-y)\d y.\]
Since $B(x_0,2\epsilon)\subset \Omega$ we have
\begin{align*}
\M_\eps\eta^{x_0}_\eps(x) &= (1+\O([\nabla \rho]_\alpha \eps^{1+\alpha})) \int_{\R^n}\eta_\eps(x-y)\eta_\eps(x_0-y)\d y \\
&= (1+\O([\nabla \rho]_\alpha \eps^{1+\alpha})) \int_{\R^n}\eta_\eps(z)\eta_\eps(x-x_0-z)\d z\\
&= (1+\O([\nabla \rho]_\alpha \eps^{1+\alpha})) (\eta_\eps *\eta_\eps)(x-x_0)\\
&= (1+\O([\nabla \rho]_\alpha \eps^{1+\alpha}))\psi_{2,\epsilon}(x-x_0)=(1+\O([\nabla \rho]_\alpha \eps^{1+\alpha}))\psi^{x_0}_{2,\epsilon}(x)
\end{align*}
for all $x\in \Omega$, where we set $\psi^{x_0}_{k,\epsilon}(x) \defeq  \psi_{k,\epsilon}(x-x_0)$ in the last line for notational simplicity. Now assume by way of induction that
\[\M^{j}_\eps\eta^{x_0}_\eps  = (1+\O(2j[\nabla \rho]_\alpha \eps^{1+\alpha}))\psi^{x_0}_{j+1,\epsilon} + \O\left(2j\epsilon^2C_{k,\epsilon}\right)\ \ \text{on } \Omega\]
for some $1 \leq j \leq k-1$, where $C_{k,\epsilon}=k^{-1}\exp\left( -\frac{R^2}{4dk\epsilon^2}\right)$. Then we have 
\[\M^{j+1}_\eps\eta^{x_0}_\eps  = \M_\eps \M^j_\eps \eta^{x_0}_\eps = (1+\O(2j[\nabla \rho]_\alpha \eps^{1+\alpha}))\M_\eps \psi^{x_0}_{j+1,\epsilon} + \O(2j\epsilon^2C_{k,\epsilon}).\]
By \labelcref{eq:tail} and $\hat\rho_\eps(y) = \rho(y) + \O([\nabla \rho]_\alpha \eps^{1+\alpha})$ for $y\in B(x_0,R)$ we have
\begin{align*}
&\phantom{{}={}}
\M^{j+1}_\eps\eta^{x_0}_\eps(x) 
\\
&= (1+\O(2j[\nabla \rho]_\alpha \eps^{1+\alpha}))\int_\Omega\eta_\eps(x-y) \rho_{\epsilon}(y)^{-1}\rho(y) \psi_{j+1,\epsilon}(y-x_0)\d y+ \O(2j\epsilon^2C_{k,\epsilon})\\
&= (1+\O(2j[\nabla \rho]_\alpha \eps^{1+\alpha}))\left(\int_{B(x_0,R)}\eta_\eps(x-y) \rho_{\epsilon}(y)^{-1}\rho(y) \psi_{j+1,\epsilon}(y-x_0)\d y + \O(\epsilon^2C_{k,\epsilon})\right)\\
&\hspace{4in}+ \O(2j\epsilon^2C_{k,\epsilon})\\
&= (1+\O(2j[\nabla \rho]_\alpha \eps^{1+\alpha}))(1 + \O([\nabla \rho]_\alpha \eps^{1+\alpha}))\int_{B(x_0,R)}\eta_\eps(x-y) \psi_{j+1,\epsilon}(y-x_0)\d y \\
&\hspace{4in}+ \O((2j+1)\epsilon^2C_{k,\epsilon})\\
&= (1+\O((2j+2)[\nabla \rho]_\alpha \eps^{1+\alpha})))\int_{\R^d}\eta_\eps(x-y) \psi_{j+1,\epsilon}(y-x_0)\d y + \O(2(j+1)\epsilon^2C_{k,\epsilon})\\
&= (1+\O(2(j+1)[\nabla \rho]_\alpha \eps^{1+\alpha})))\int_{\R^d}\eta_\eps(z) \psi_{j+1,\epsilon}(x-z-x_0)\d z + \O(2(j+1)\epsilon^2C_{k,\epsilon})\\
&= (1+\O(2(j+1)[\nabla \rho]_\alpha \eps^{1+\alpha})))\psi_{j+2,\epsilon}(x-x_0) + \O(2(j+1)\epsilon^2C_{k,\epsilon})
\end{align*}
for all $x\in \Omega$, which completes the proof. 
\end{proof}

\cref{thm:Me} allows us to establish the analogous results to \cref{lem:psik_graph_bounds} for the repeated averaging operator $\M_{\eps}^{k-1}\eta_\eps^x$, which is the content of \cref{cor:Me_graph_bounds}. 
The main ingredient for this is an estimate for the gradient of the repeated averaging operator. 
\begin{proposition}\label{prop:Me_grad}
Let $k\geq 2$ and $x_0\in \Omega$ such that $B(x_0,R_k)\subset \Omega$. Then $\M^{k}_\eps\eta^{x_0}_\eps$ is differentiable, and there exists $C$ depending on $\rho$ and $\eta, \eta'$ such that for $[\nabla \rho]_\alpha k\epsilon^{1+\alpha} \ll 1$ and all $x\in \Omega$ we have
\begin{equation}\label{eq:Me_grad}
|\nabla \M^{k}_\eps\eta^{x_0}_\eps(x)|  \leq C\epsilon_k^{-d}\eps^{-1}
\end{equation}
\end{proposition}
The idea of the proof of  \cref{prop:Me_grad} is to write
\[\M^k_\eps \eta^{x_0}_\eps(x) = \epsilon^{-d}\int_{\Omega}\eta\left(\frac{|x-y|}{\epsilon}\right)\hat\rho_\eps(y)^{-1}\rho(y)\M^{k-1}_\eps\eta^{x_0}_\eps(y)\d y,\]
and differentiate in $x$ under the integral, which produces the extra $\eps^{-1}$ term. The proof is somewhat technical since we do not require Lipschitzness of $\eta$, and instead produce the desired estimates using that $\eta$ is radially decreasing, and thus has bounded variation. We defer the proof to \cref{app:rep_avg}.

\begin{corollary}\label{cor:Me_graph_bounds}
Let $\eps \geq \delta > 0$. There exists $C_1,C_2,C_3,C_4>0$ such that the following hold.
\begin{enumerate}[label=(\roman*)]
\item   For $x_0\in \Omega$ with $B(x_0,R_k)\subset \Omega$ it holds that
\begin{equation}\label{eq:oscillation_bound_psi_ke}
\|\osc_{\Omega \cap B(\cdot,\delta)} \M^{k}_\eps\eta^{x_0}_\eps\|_{L^1(\Omega)} \leq  C_1 (\eps^{-1}\delta\log(\eps^{-1})^{\frac{d}{2}} + \eps^2).
\end{equation}
\item  For $\delta \ll 1$, the event that
\[\ng{1}{\M^{k}_\eps\eta^{x_0}_\eps - \M^{k}_\eps\eta^{y_0}_\eps}  \leq C_2 (\eps^{-1}\delta\log(\eps^{-1})^{\frac{d}{2}} + \eps^2)\]
holds for all $x_0,y_0\in \Omega$ with $|x_0-y_0| \leq \delta$ and $B(x_0,R_k)\cup B(y_0,R_k)\subset \Omega$  has probability at least $1-C_3 n \exp\left( -C_4 n\delta^d\right)$. 
\end{enumerate}
\end{corollary}
\begin{proof}
We first prove (i). As in the proof of  \cref{prop:psi_gaussian_upper} (iv), we use  \cref{prop:Me_grad} to show that since $R_k\geq \eps \geq \delta$ we have
\begin{align*}
\|\osc_{\Omega \cap B(\cdot,\delta)} \M^k_\eps\eta^{x_0}_\eps\|_{L^1(\Omega)} &\leq CR_k^d \delta\|\nabla \M^k_\eps \eta^{x_0}_\eps\|_{L^\infty(\Omega)} + 2|\Omega| \|\M^k_\eps \eta^{x_0}_\eps\|_{L^\infty(\Omega\setminus B(x_0,R_k))}\\
&\leq CR_k^d \delta\eps^{-1}\eps_k^{-d} + 2|\Omega|\|\M^k_\eps \eta^{x_0}_\eps\|_{L^\infty(\Omega\setminus B(x_0,R_k))}.
\end{align*}
Using \cref{thm:Me} we can follow the proof of  \cref{prop:psi_gaussian_upper} (iv), but here we note that since $R_k$ satisfies \labelcref{eq:Rkdef,eq:Rk_prop} we have
\begin{equation}\label{eq:linf_Me}
\|\M^k_\eps \eta^{x_0}_\eps\|_{L^\infty(\Omega\setminus B(x_0,R_k))} \leq C\eps^{-d}\exp\left( -\frac{R_k^2}{8d\eps_k^2}\right) \leq \eps^{2}
\end{equation}
By \labelcref{eq:Rk_prop} we have  $R_k \leq C\epsilon_k \log(\eps^{-1})^{\frac{1}{2}}$ and so
\[\|\osc_{\Omega \cap B(\cdot,\delta)} \M^k_\eps\eta^{x_0}_\eps\|_{L^1(\Omega)} \leq C(R_k^d \delta \eps^{-1}\eps_k^{-d} + \eps^2) \leq C (\eps^{-1}\delta \log(\eps^{-1})^{\frac{d}{2}} + \eps^2).\]

To prove (ii) we assume \cref{thm:transportation} holds with $\delta\geq |x_0-y_0|$ and $\lambda=\frac{\rho_{\min}}{8\rho_{\max}}$, which has probability at least $1-C_3 n \exp\left( -C_4 n\delta^d\right)$. Using \cref{prop:Me} (iv), \cref{prop:Me_grad}, as well as \labelcref{eq:linf_Me,eq:Rk_prop} we have
\begin{align*}
\ng{1}{\M^{k}_\eps\eta^{x_0}_\eps - \M^{k}_\eps\eta^{y_0}_\eps} &= \frac{1}{n}\sum_{i=1}^n|\M^{k}_\eps\eta^{x_0}_\eps(x_i) - \M^{k}_\eps\eta^{y_0}_\eps(x_i)|\\
&=\int_\Omega |\M^{k}_\eps\eta^{x_0}_\eps(T_\delta(x)) - \M^{k}_\eps\eta^{y_0}_\eps(T_\delta(x))| \rho_\delta(x)\d x\\
&\leq C\int_{B(x_0,2R_k)}|\M^{k}_\eps\eta^{T_\delta(x)}_\eps(x_0) - \M^{k}_\eps\eta^{T_\delta(x)}_\eps(y_0)| \d x\\
&\qquad
+ C\|\M^k_\eps \eta^{x_0}_\eps\|_{L^\infty(\Omega\setminus B(x_0,2R_k-\delta))} + C\|\M^k_\eps \eta^{y_0}_\eps\|_{L^\infty(\Omega\setminus B(x_0,2R_k-\delta))}\\
&\leq C\int_{B(x_0,2R_k)} \|\nabla \M^k_\eps \eta_\eps^{T_\delta(x)}\|_{L^\infty(\Omega)} |x_0-y_0| \d x\\
&\qquad + C\|\M^k_\eps \eta^{x_0}_\eps\|_{L^\infty(\Omega\setminus B(x_0,R_k))} + C\|\M^k_\eps \eta^{y_0}_\eps\|_{L^\infty(\Omega\setminus B(y_0,R_k))}\\
&\leq C(R_k^d \eps_k^{-d}\eps^{-1}\delta + \eps^2) \leq C (\eps^{-1}\delta \log(\eps^{-1})^{\frac{d}{2}} + \eps^2).\qedhere
\end{align*}
\end{proof}

\subsection{Interior asymptotics of the heat kernel}

We now prove our main result in this section, which is an refinement of \cref{thm:heatkernel_nonlocal} and controls the arising error terms in the interior of the domain, where the boundary effects of the heat kernel can be ignored. 
\begin{theorem}\label{thm:heatkernel_final}
Let $x_0\in \Omega$, $m\geq 2$ and suppose $R\geq R_m$. Let us define $\phi$ and $\cdk$ by
\begin{equation}\label{eq:Kconstant}
\phi(z) = \min\left\{\cdk,k\exp\left( -\frac{(|z|-\eps)_+^2}{8d\epsilon_k^2}\right) \right\}, \text{ and } \cdk\defeq  
\begin{cases}
\displaystyle \sqrt{k},& \text{if } d=1\\
\displaystyle \log(k+1),& \text{if } d=2\\
\displaystyle \frac{d}{d-2},& \text{if } d\geq 3.
\end{cases}
\end{equation}
There exists $C>0$ such that for $m\epsilon^{1+\alpha} \ll 1$ and $0<\lambda \leq 1$ the estimate
\begin{equation}\label{eq:heat_kernel_asymptotic_Me}
\H^{x}_{k}(x_i) = \hat\rho_\eps(x)^{-1}\M_{\eps}^{k-1}\eta_\eps^x(x_i) + \O\left(\lambda\epsilon^{-d}\phi(x-x_i) + \lambda\eps_k^2 \exp\left( -\frac{R^2}{4d\epsilon_k^2}\right)\right)
\end{equation}
holds for all $i=1,\dots,n$, $x\in \X_n\cup\{x_0\}$ for which $B(x,R)\subset \Omega$, and $k=2,\dots,m$ with probability at least $1-12mn^2\exp\left( -C n \epsilon^d \lambda^2\right)$.
\end{theorem}
\begin{proof}
By \cref{thm:Me,prop:psi_gaussian_upper} we have for any $j= 0,...,k-2$ the estimate
\begin{align*}
\|\M_{\eps}^j\eta_\eps^x\|_{L^\infty(\Omega\cap B(x_i,\epsilon))} &\leq C\Bigg[\eps^{-d}\min\left\{(j+1)^{-\frac{d}{2}},\exp\left( -\frac{(|x-x_i|-\eps)_+^2}{8d(j+1)\epsilon^2}\right) \right\}\\
&\hspace{2.5in}+ \epsilon^2 \exp\left( -\frac{R^2}{4dj\epsilon_j^2}\right)\Bigg]
\end{align*}
Note that
\[\sum_{j=0}^{k-2}(j+1)^{-\frac{d}{2}} \leq 1 + \int_0^{k-2} (x+1)^{-\frac{d}{2}}\d x \leq 2\cdk.\]
Therefore, for any $0 < \lambda \leq 1$ we have
\begin{align*}
\lambda\sum_{j=0}^{k-2}\|\M_{\eps}^j\eta_\eps^x\|_{L^\infty(\Omega\cap B(x_i,\epsilon))} &\leq C\lambda\eps^{-d}\min\left\{\cdk,k\exp\left( -\frac{(|x-x_i|-\eps)_+^2}{8d\epsilon_k^2}\right) \right\}\\
&\hspace{2in} + C\lambda(k-1)\epsilon^2 \exp\left( -\frac{R^2}{4d\epsilon_k^2}\right)\\
&=C\lambda\left(\epsilon^{-d}\phi(x-x_i) + \eps_k^2 \exp\left( -\frac{R^2}{4d\epsilon_k^2}\right)\right).
\end{align*}
The proof is completed by inserting the estimates above into \cref{thm:heatkernel_nonlocal}.
\end{proof}
\begin{remark}\label{rem:heat_kernel_asymptotic}
Note that by \cref{thm:Me}, we can replace \labelcref{eq:heat_kernel_asymptotic_Me} with the estimate
\begin{align}\label{eq:heat_kernel_asymptotic}
\H^{x}_{k}(x_i) &= (1+\O([\nabla \rho]_\alpha k\epsilon^{1+\alpha}))\rho(x)^{-1}\psi_{k,\epsilon}(x_i-x)\notag \\
&\hspace{2in} +\O\left(\lambda \eps^{-d}\phi(x-x_i) +\eps_k^2 \exp\left( -\frac{R^2}{4d\epsilon_k^2}\right) \right).
\end{align}
However, the additional error term $\O([\nabla \rho]_\alpha k\epsilon^{1+\alpha})$ that arises is too large for the discrete to continuum applications later in  \cref{sec:combination} when $\rho$ is nonconstant. The estimate \labelcref{eq:heat_kernel_asymptotic} is is mainly useful when $\rho$ is constant, as well as for proving rough estimates on the decay of the heat kernel and its $\ell^p$ norms, as we shall see below.
\end{remark}
We now turn to a decay estimate on the heat kernel. 
\begin{corollary}\label{cor:heat_kernel_decay}
Let $x_0\in \Omega$ with $B(x_0,R_k)\subset \Omega$ and assume $[\nabla \rho]_\alpha k\epsilon^{1+\alpha} \ll 1$ and $\eps \ll 1$. There exists $C_1,C_2>0$, with $C_1$ additionally depending on $[\nabla \rho]_\alpha$, such that the event that $\H^x_k(x_i) \leq C_1\epsilon$ holds for all $i=1,\dots,n$ and $x\in \X_n\cup\{x_0\}$ for which $B(x,R_k)\subset \Omega$ and $|x-x_i|\geq R_k$ has probability at least $1-12kn^2\exp\left( -C_2 n \epsilon^d \lambda^2\right)$.
\end{corollary}
\begin{proof}
The case $k=1$ is trivial. For $k\geq 2$, we apply  \cref{thm:heatkernel_final} with $\lambda=1$ and $R=R_k$ and  \cref{thm:Me} as outlined in  \cref{rem:heat_kernel_asymptotic}. The restriction $|x-x_i| \geq R_k$ implies that 
\[\psi_{k,\epsilon}(x-x_i) \leq \epsilon, \ \ k\epsilon^{-d}\exp\left( -\frac{(|x-x_i|-\eps)_+^2}{8d\eps_k^2}\right) \leq \eps \ \ \text{and} \ \ \eps_k^2 \exp\left( -\frac{R^2}{4d\eps_k^2}\right) \leq \epsilon,\]
provided $|x-x_i|\geq R_k$. Inserting these estimates into \labelcref{eq:heat_kernel_asymptotic} completes the proof.
\end{proof}

In order use  \cref{thm:heatkernel_final} to obtain $\l p$ bounds, we will need to control the $\l p$ norm of the function $\phi$ defined in \labelcref{eq:Kconstant}. 
\begin{lemma}\label{lem:phi_graph_lp}
Let $\phi$ be defined by \labelcref{eq:Kconstant}. There exists $C_1,C_2,C_3>0$ such that the event that
\begin{equation}\label{eq:phi_lp_bound}
\ng{p}{\phi(\cdot-x_i)} \leq C_3\cdk\epsilon_k^{\frac{d}{p}}\log(\eps^{-1})^{\frac{d}{2p}}
\end{equation}
holds for all $x_0\in \Omega$ has probability at least $1-C_1 n \exp\left( -C_2 n\epsilon^d\right)$. 
\end{lemma}
\begin{proof}
The proof is similar to \cref{lem:psik_graph_bounds}, so we give a brief sketch.  We note that $\phi(z) \leq \cdk\epsilon_k^d$ whenever 
\[|z| \geq R \defeq   \epsilon + \epsilon_k\sqrt{8d\log(\epsilon^{-1})}.\]
We use \cref{thm:transportation} with $\delta=\epsilon$ and $\lambda=\frac{\rho_{\min}}{8\rho_{\max}}$, which has probability at least $1-C_1 n \exp\left( -C_2 n\eps^d\right)$. As in  \cref{lem:psik_graph_bounds} we have
\begin{align*}
\ng{p}{\phi(\cdot-x_0)} &\leq C\|\phi(T_\delta(\cdot)-x_0)\|_{L^p(\Omega)}. 
\end{align*}
Since $|T_\delta(x)-x| \leq \epsilon \leq R$, if $|x-x_0|\geq 2R$ then $|T(x)-x_0|\geq R$. Thus $|\phi(T_\delta(x)-x_0)|\leq \cdk\eps_k^d$ for $x\in \Omega\setminus B(x_0,2R)$, while for $x\in B(x_0,2R)$ we have $|\phi(x-x_0)|\leq \cdk$ by \labelcref{eq:Kconstant}. As in  \cref{lem:psik_graph_bounds} we can use these two estimates to obtain
\[\int_{\Omega} |\phi(T_\delta(x)-x_0)|^p \d x \leq C\cdk^{p}\epsilon_k^{d}\log(\eps^{-1})^{\frac{d}{2}}. \]
Taking the $p^{\rm th}$ root and inserting above completes the proof. 
%Full proof below
%The proof is similar to \cref{lem:psik_graph_bounds}, so we give a brief sketch.  We note that $\phi(z) \leq \cdk\epsilon_k^d$ whenever 
%\[|z| \geq R \defeq   \epsilon + \epsilon_k\sqrt{8d\log(\epsilon^{-1})}.\]
%We use \cref{thm:transportation} with $\delta=\epsilon$ and $\lambda=\frac{\rho_{\min}}{8\rho_{\max}}$, which has probability at least $1-C_1 n \exp\left( -C_2 n\eps^d\right)$. Then we have
%\begin{align*}
%\ng{p}{\phi(\cdot-x_0)} &= \left(\frac{1}{n}\sum_{i=1}^n \phi(x_i - x_0)^p\right)^{\frac{1}{p}} \\
%&= \left(\int_\Omega \phi(T_\delta(y)-x_0)^p\rho_\delta(y)\, dy\right)^{\frac{1}{p}} \\
%&\leq C\|\phi(T_\delta(\cdot)-x_0)\|_{L^p(\Omega)}. 
%\end{align*}
%Since $|T_\delta(x)-x| \leq \epsilon \leq R$, if $|x-x_0|\geq 2R$ then $|T(x)-x_0|\geq R$. Thus $|\phi(T_\delta(x)-x_0)|\leq \cdk\eps_k^d$ for $x\in \Omega\setminus B(x_0,2R)$. It follows that
%\[\int_{\Omega \setminus B(x_0,2R)} |\phi(T_\delta(x)-x_0)|^p \d x \leq |\Omega| \cdk^p \epsilon_k^{dp} \leq |\Omega| \cdk^p \epsilon_k^{d},\]
%since $p\geq 1$ and $\eps_k \leq 1$.
%For $x\in B(x_0,2R)$ we use the estimate $|\phi(x-x_0)|\leq \cdk$ to obtain
%\[\int_{\Omega \cap B(x_0,2R)} |\phi(T_\delta(x)-x_0)|^p \d x \leq \cdk^p|B(x_0,2R)|\leq C_1\cdk^{p}R^d. \]
%Combining the two estimates above yields
%\[\int_{\Omega} |\phi(T_\delta(x)-x_0)|^p \d x \leq C_2\cdk^{p}\epsilon_k^{d}\log(\eps^{-1})^{\frac{d}{2}}. \]
%Taking the $p^{\rm th}$ root and inerting above completes the proof. 
\end{proof}

We now use the interior heat kernel asymptotic in \cref{thm:heatkernel_final} along with  \cref{lem:phi_graph_lp} to obtain $\l1$ versions of  \cref{thm:heatkernel_final}, as well an $\l p$ estimate on the heat kernel $\H^x_k$. 
These results are contained in the following three corollaries, the first two depending on whether $\rho$ is constant or not. 
\begin{corollary}\label{cor:heatkernel_l1_asymptotic_rho}
Let $m\geq 2$ and $x_0\in \Omega$ with $B(x_0,R_m)\subset \Omega$. There exists $C_1,C_2>0$ such that for $0<\lambda \leq 1$ the estimate
\begin{equation}\label{eq:heat_kernel_l1_asymptotic_rho}
\ng{1}{\H^{x}_{k} - \hat\rho_\eps(x)^{-1}\M_{\eps}^{k-1}\eta_\eps^x} \leq C_1\left(\lambda \eps^{-d}\cdk\epsilon_k^{d}\log(\eps^{-1})^{\frac{d}{2}} +\lambda\eps\right)
\end{equation}
holds for all $i=1,\dots,n$, $x\in \X_n\cup\{x_0\}$ for which $B(x,R_m)\subset \Omega$, and $k=2,\dots,m$ with probability at least $1-14mn^2\exp\left( -C_2 n \epsilon^d \lambda^2\right)$.
\end{corollary}
\begin{proof}
The proof follows from \labelcref{eq:heat_kernel_asymptotic_Me},  \cref{thm:heatkernel_final} and \cref{lem:phi_graph_lp} with~\mbox{$p=1$}. 
\end{proof}
\begin{corollary}\label{cor:heatkernel_l1_asymptotic}
Assume $\rho\equiv |\Omega|^{-1}$ is constant, and let $m\geq 2$ and $x_0\in \Omega$ with $B(x_0,R_m)\subset \Omega$. There exists $C_1,C_2>0$ such that for $0<\lambda \leq 1$ the estimate
\begin{equation}\label{eq:heat_kernel_l1_asymptotic}
\ng{1}{\H^{x}_{k} - \rho^{-1}\psi_{k,\epsilon}(\cdot-x)} \leq C_1\left(\lambda \eps^{-d}\cdk\epsilon_k^{d}\log(\eps^{-1})^{\frac{d}{2}} +\eps\right)
\end{equation}
holds for all $i=1,\dots,n$, $x\in \X_n\cup\{x_0\}$ for which $B(x,R_m)\subset \Omega$, and $k=2,\dots,m$ with probability at least $1-14mn^2\exp\left( -C_2 n \epsilon^d \lambda^2\right)$.
\end{corollary}
\begin{proof}
The proof follows from \labelcref{eq:heat_kernel_asymptotic}, \cref{rem:heat_kernel_asymptotic} and  \cref{lem:phi_graph_lp} with $p=1$. 
\end{proof}

Of course, it is certainly possible to prove versions of \cref{cor:heatkernel_l1_asymptotic,cor:heatkernel_l1_asymptotic_rho} where the convergence is to a Gaussian density of variance $\eps_k^2$. However, in addition to the $[\nabla \rho]_\alpha k \eps^{1+\alpha}$ error term from  \cref{rem:heat_kernel_asymptotic}, we would incur an additional $\O(k^{-1})$ error term. This would only be sufficiently sharp in the constant density case when $[\nabla \rho]_\alpha  = 0$, but in this case the repeated convolution $\psi_{k,\epsilon}$ is sufficient for our purposes.

Finally, we turn to $\ell^p$ bounds on the heat kernel.
\begin{corollary}\label{cor:heatkernel_lp}
Let $m\geq 2$ and $x_0\in \Omega$ with $B(x_0,R_m)\subset \Omega$.  There exists $C_1,C_2>0$, with $C_1$ additionally depending on $[\nabla \rho]_\alpha$, such that for $[\nabla \rho]_\alpha m\epsilon^{1+\alpha} \ll 1$ and $0<\lambda \leq 1$ the estimate
\begin{equation}\label{eq:heat_kernel_lp}
\ng{p}{\H^x_k} \leq C_1 \epsilon_k^{\frac{d}{p}}\log(\epsilon^{-1})^{\frac{d}{2p}}\left(\epsilon_k^{-d} + \lambda \epsilon^{-d}\cdk\right)
\end{equation}
holds for all $i=1,\dots,n$, $x\in \X_n\cup\{x_0\}$ for which $B(x,R_m)\subset \Omega$, and $k=2,\dots,m$ with probability at least $1-14mn^2\exp\left( -C_2 n \epsilon^d \lambda^2\right)$.
\end{corollary}
\begin{proof} 
Fix $x\in \X_n\cup\{x_0\}$ for which $B(x,R_m)\subset \Omega$. By \labelcref{eq:heat_kernel_asymptotic},  \cref{lem:phi_graph_lp}, and the fact that $\ng{p}{1}=1$ we have
\[\ng{p}{\H^{x}_{k}} \leq C\left( \ng{p}{\psi_{k,\epsilon}(\cdot-x)} + \lambda \eps^{-d}\cdk\epsilon_k^{\frac{d}{p}}\log(\eps^{-1})^{\frac{d}{2p}} +\eps_k^2\right)\]
for all $k=2,\dots,m$ with probability at least $1-13mn^2\exp\left( -C_2 n \epsilon^d \lambda^2\right)$. The proof is completed by invoking  \cref{lem:psik_graph_bounds}. 
\end{proof}

\section{Combination: Convergence rates for Poisson learning with measure data}
\label{sec:combination}

In this section, we combine the results from \cref{sec:smoothed_poisson,sec:continuum_limits,sec:heat_kernel} 
to prove quantitative convergence rates for the continuum limit of Poisson learning. Throughout this section we continue to use the notation $\epsilon_k = \epsilon \sqrt{k}$ from  \cref{sec:heat_kernel}. We also utilize the constant $\cdk$ defined in  \labelcref{eq:Kconstant} and $R_k$ defined in \labelcref{eq:Rkdef}, which we recall satisfies \labelcref{eq:Rk_prop}.  We also make the standing \cref{ass:rho_c1a,ass:omega_c11,ass:neps,ass:epsk} with $\alpha=1$ throughout this section. The constants in this section depend on all of the quantities in these assumptions.

In order to set up the results, let $\Gamma\subset \Omega$ be a finite set of \emph{continuum} labels with corresponding label data $a:\Gamma\to \R$, which we denote by $a_x\defeq  a(x)$ for $x\in\Gamma$. For $x\in \Gamma$, denote by $\tau(x)\in \X_n$ any closest point to $x$ in the point cloud $\X_n$, with ties broken arbitrarily. Let us also define
\[\Gamma_n = \{\tau(x) \, : \, x\in \Gamma\},\]
keeping in mind that $\Gamma_n$ may have fewer points than $\Gamma$ if the closest points collide. We will center labels at the projections $\tau(x)$ of the points $x\in \Gamma$ to the point cloud.
%Let $\Gamma_n\subset \X_n$ be the Euclidean closest point projections of $\Gamma$ to $\X_n$, and write again $a:\Gamma_n\to \R$ for the copied labels. 
In particular, let $u_{n,\epsilon}\in \l2$ be the solution of the Poisson learning problem with data on $\Gamma$, that is
\begin{equation}\label{eq:une}
u_{n,\epsilon} = \argmin_{u\in \lo2}\ene\left(u;\sum_{x\in \Gamma} a_x \delta_{\tau(x)} \right).
\end{equation}
We assume the compatibility condition 
\begin{equation}\label{eq:comp_ax}
%\sum_{x\in \Gamma} a_x = \sum_{x\in \Gamma_n} a_x = 0
\sum_{x\in \Gamma} a_x = 0
\end{equation}
holds, which ensures that $u_{n,\epsilon}$ also satisfies the Poisson equation
\begin{equation}\label{eq:final_graph_poisson}
\L_{n,\epsilon}u_{n,\epsilon} = \sum_{x\in \Gamma} a_x \delta_{\tau(x)} \ \ \text{on} \ \ \X_n.
\end{equation}
Notice that the compatibility condition \labelcref{eq:comp_ax} ensures that the right hand side of \labelcref{eq:final_graph_poisson} has mean zero, since
\[\sum_{y\in \X_n}\sum_{x\in \Gamma} a_x \delta_{\tau(x)}(y) =\sum_{x\in \Gamma}a_x\sum_{y\in \X_n}  \delta_{\tau(x)}(y) = \sum_{x\in \Gamma} na_x = 0.\]
Above, we used that $\delta_{\tau(x)}(y)=n$ if $\tau(x)=y$ and $\delta_{\tau(x)}(y)=0$ if $\tau(x)\neq y$.

Our main result in this section is a convergence rate to the function $u \in W^{1,p}(\Omega)$ for $p < \frac{d}{d-2}$ defined by
\begin{equation}\label{eq:ucontinuum}
u = \sum_{x\in \Gamma} a_x G^{x},
\end{equation}
where $G^x$ is the Green's function for $\varrho=\rho^2$ defined in \cref{sec:smoothed_poisson}. By the results in  \cref{sec:smoothed_poisson}, the function $u\in W^{1,p}(\Omega)$ defined in \labelcref{eq:ucontinuum} is the unique distributional solution of the continuum Poisson equation
\begin{equation}\label{eq:final_poisson_continuum}
-\div(\rho^2 \nabla u) = \sum_{x\in \Gamma} a_x \delta_{x},
\end{equation}
where $\delta_x$ is the Dirac delta measure centered at $x$, with homogeneous Neumann boundary conditions, and mean zero condition $\int_\Omega \rho^2 u\d x=0$. We recall that commit the minor abuse of notation of using the same symbol $\delta_x$ to denote the Dirac delta measure and its graph approximation.

Our main result in this paper is the following quantitative convergence rate. 
\begin{theorem}[Main theorem]\label{thm:main_general}
We make the \cref{ass:rho_c1a,ass:omega_c11,ass:neps,ass:epsk} throughout this section, and additionally assume that $k\geq 2$, $\eps \ll 1$, $n\eps^{2d} \geq 1$, $\eps\leq \eps_k^d\leq 1$ and $\eps_k \log(\eps^{-1})^{\frac{1}{2}}\leq \frac{\dist(\Gamma,\partial\Omega)}{24(d+2)}$. Let $u_{n,\epsilon}\in \lo2$ be defined by \labelcref{eq:une} and let $u\in W^{1,p}(\Omega)$ for $1\leq p<\frac{d}{d-1}$ be defined by \labelcref{eq:ucontinuum}.  
Then there exist constants $C_1,C_2>0$ such that for 
\begin{equation}\label{eq:q_prob}
q = kn^2\exp\left( -C_2\cdk^{-1}  n \epsilon^{3d+2}\epsilon_k^{-2(d + 1)}\right) + \exp(-C_2n\epsilon^{d+2}\epsilon_k^{2d})+\exp(-C_2n\eps^{2d+2})
\end{equation}
the following hold.
\begin{enumerate}[label=(\roman*)]
\item If $\rho\equiv |\Omega|^{-1}$ is constant, then for any $0 < \sigma < 1$ we have that
\begin{equation}\label{eq:constant_rate}
\ng{1}{u - u_{n,\eps}} \lesssim \sum_{x\in \Gamma}|a_x|\log(\eps^{-1})^{\frac{d}{2}+1}\left(\eps_k^{2-\sigma} + \eps_k^{-\frac{d}{2}}\eps^{\frac{1}{2}}\right)
\end{equation}
holds with probability at least $1 - C_1q - 2^{2+\gamma}n^{-\gamma}$, where $\gamma = \frac{\sigma}{d-\sigma}$. 
\item In general we have that
\begin{equation}\label{eq:nonconstant_rate}
\ng{1}{u - u_{n,\eps}} \lesssim \sum_{x\in \Gamma}|a_x|\log(\eps^{-1})^{\frac{d}{2}}\left(\eps_k + \eps_k^{-\frac{d}{2}}\eps^{\frac{1}{2}}\right)
\end{equation}
holds with probability at least $1 - C_1q -  2^{2+\gamma}n^{-\gamma}$, where 
\begin{equation}\label{eq:gamma_d}
\gamma = 
\begin{cases}
\frac{1}{2(d-1)},& \text{if } d\geq 2\\
1,& \text{if } d=1.
\end{cases}
\end{equation}
\end{enumerate}
\end{theorem}
\begin{proof}
The proof of  \cref{thm:main_general} is a combination of  \cref{lem:main_smooth_graph,lem:discrete_to_continuum,lem:continuum_smoothing}, proved below, taking $\delta=\eps^2$ in  \cref{lem:discrete_to_continuum} and $r = \frac{d}{d-\sigma}$ in part (i). 
\end{proof}
\begin{remark}\label{rem:Rkepsk}
Note that the condition $\eps_k \log(\eps^{-1})^{\frac{1}{2}}\leq \frac{\dist(\Gamma,\partial\Omega)}{24(d+2)}$ in  \cref{thm:main_general} ensures that $R_k\leq \frac{1}{4}\dist(\Gamma,\partial\Omega)$, due to \labelcref{eq:Rk_prop}, which is required in  \cref{lem:discrete_to_continuum}. This is a requirement since our heat kernel estimates are only valid in the interior of the domain. An interesting future problem would be to address the setting where one of the labels $x\in \Gamma$ falls on the boundary $x\in \partial \Omega$. In this case, we would need to analyze the heat kernel asymptotics at the boundary, in which case we expect the analysis to be substantially different.  
\end{remark}
\begin{remark}\label{rem:lprates}
The main results used to prove  \cref{thm:main_general}, namely \cref{lem:main_smooth_graph,lem:discrete_to_continuum,lem:continuum_smoothing}, all provide $\l p$ rates for $p>1$ as well. Thus, it is possible to prove a version of  \cref{thm:main_general} that holds in the graph $\l p$ norm for $p>1$. Since the result is more complicated to state, and there are very strong restrictions on $p$ which essentially require $p\approx 1$, we do not state a formal result in this setting. We do mention, however, that when $p>1$ there is an additional error term arising from the graph mollification result in  \cref{lem:main_smooth_graph} of the form $\O(\eps_k^2 n^{1-\frac{1}{p}}\eps^2)$. In order to, for example,  absorb this into the $\O(\eps_k^2)$ error term in  \cref{thm:main_general} (ii) when $\rho$ is constant, we would require that 
\[n^{1-\frac{1}{p}}\eps^2 \leq 1 \iff \eps \leq n^{-\frac{p-1}{2p}}.\]
This puts an \emph{upper bound} restriction on $\eps$, which is in opposition to the lower bound restrictions required to ensure events hold with high probability. There is precedent for such upper bound restrictions on length scales in discrete to continuum results in graph-based learning with the $p$-Laplacian \cite{slepcev2019analysis} and the properly weighted Laplacian \cite{calder2020properly}. In a sense, upper bound bound restrictions on $\eps$ ensure that the graph problem ``looks'' more like the continuum PDE than the intermediate nonlocal integral equation. 
\end{remark}

We can obtain the sharpest convergence rates in  \cref{thm:main_general} by balancing the two error terms. 
\begin{corollary}\label{cor:main_nonconstant}
In the context of  \cref{thm:main_general} we in general have that
\begin{equation}\label{eq:nonconstant_rate_optimal}
\ng{1}{u - u_{n,\eps}} \lesssim \sum_{x\in \Gamma}|a_x|\log(\eps^{-1})^{\frac{d}{2}}\eps^{\frac{1}{d+2}},
\end{equation}
holds with probability at least 
\[1  -  2^{2+\gamma}n^{-\gamma} - C_1n^3\exp\left( -C_2 \Theta_{d,\eps^{-2}}^{-1}n\eps^{3d + \frac{2}{d+2}}\right).\]
where $\gamma$ is given in \labelcref{eq:gamma_d}. 
\end{corollary}
\begin{proof}
We choose $k$ to balance $\eps_k = \eps_k^{-\frac{d}{2}}\eps^{\frac{1}{2}}$, which amounts to $k = \eps^{-\frac{2(d+1)}{d+2}}$ and $\eps = \eps_k^{d+2}$, so the rate is $\eps_k = \eps^{\frac{1}{d+2}}$.  The probability $q$ defined in \labelcref{eq:q_prob} simplifies using that $k\leq \eps^{-2} \leq n^{\frac{1}{d}}$, and the first term in $q$ again dominates. 
\end{proof}
\begin{remark}\label{rem:eps_nonconstant}
In order to ensure that the probability in  \cref{cor:main_constant} is close to one, we require that
\begin{equation}\label{eq:eps_constant}
\Theta_{d,\eps^{-2}}^{-1}n\eps^{3d + \frac{2}{d+2}} \geq C\log n,
\end{equation}
for a possibly large constant $C>0$. For $d\geq 3$,  $\cdk$ is independent of $k$, and so \labelcref{eq:eps_constant} simplifies to 
\begin{equation}\label{eq:eps_nonconstant_d3}
\eps \geq C\left( \frac{\log n}{n}\right)^{\frac{d+2}{3d^2 + 6d + 2}},
\end{equation}
for a different constant $C$. For $d=2$ we have $\Theta_{d,\eps^{-2}} = \log(1 + \eps^{-2})$ and so we require
\[\eps \geq C\left(\frac{\log n}{n}\right)^{\frac{d+2-\tau}{3d^2 + 6d + 2}}\]
for some fixed small $\tau>0$. For $d=1$ we have $\Theta_{d,\eps^{-2}} = \eps^{-1}$ which yields the lower bound
\[\eps \geq C\left( \frac{\log n}{n}\right)^{\frac{d+2}{3d^2 + 7d + 4}}.\]
\end{remark}

When the density $\rho$ is constant we can obtain an improved convergence rate. 
\begin{corollary}\label{cor:main_constant}
In the context of  \cref{thm:main_general}, assuming that $\rho\equiv |\Omega|^{-1}$ is constant, there exists $C_1,C_2>0$ such that for any $0 < \sigma < 1$ we have that
\begin{equation}\label{eq:constant_rate_optimal}
\ng{1}{u - u_{n,\eps}} \lesssim \sum_{x\in \Gamma}|a_x|\log(\eps^{-1})^{\frac{d}{2}+1}\eps^{\frac{2-\sigma}{d+4}},
\end{equation}
holds with probability at least 
\[1  -  2^{2+\gamma}n^{-\gamma} - C_1n^3\exp\left( -C_2 \Theta_{d,\eps^{-2}}^{-1}n\eps^{3d + \frac{6}{d+4}}\right), \ \ \text{where} \ \ \gamma = \frac{\sigma}{d-\sigma}.\]
\end{corollary}
\begin{proof}
We choose $k$ to balance $\eps_k^2 = \eps_k^{-\frac{d}{2}}\eps^{\frac{1}{2}}$, which amounts to $k = \eps^{-\frac{2(d+3)}{d+4}}$ and $\eps = \eps_k^{d+4}$, so that $\eps_k = \eps^{\frac{1}{d+4}}$. The probability $q$ defined in \labelcref{eq:q_prob} simplifies using that $k\leq \eps^{-2} \leq n^{\frac{1}{d}}$, and the first term in $q$ dominates. 
\end{proof}
\begin{remark}\label{rem:eps_constant}
As in  \cref{rem:eps_nonconstant}, for  \cref{cor:main_constant} to hold with probability close to one we require that 
\[\eps \geq C\left( \frac{\log n}{n}\right)^{q}, \ \ \text{where} \ \ 
q = 
\begin{cases}
\frac{d+4}{3d^2 + 13d + 10},& \text{if } d=1\\
\frac{d+4-\tau}{3d^2 + 12d + 6},& \text{if } d=2\\
\frac{d+4}{3d^2 + 12d + 6},& \text{if } d\geq 3.
\end{cases}
\]
\end{remark}

Our last result in concerned with Poisson equations with general compactly supported Radon measure as source terms.  For a signed Radon measure $f\in\M(\overline\Omega)$ with $f(\overline\Omega)=0$ we let $f=f^+-f^-$ be the decomposition into positive and negative parts, which are measures with the same mass. Furthermore, we let $W_1(\mu,\nu)\defeq  \sup\left\lbrace\int_{\overline\Omega}\phi\d(\mu-\nu)\st \Lip(\phi)\leq 1\right\rbrace$ denote the Wasserstein-1 distance of two measures $\mu$ and $\nu$ with the same mass.
Combing our main result \cref{thm:main_general} for right hand sides that are linear combinations of Dirac deltas with the $L^p$-stability result for measure data from \cref{thm:Lp_stability}, we easily infer the following result which has an additional error term that compares the source terms.
\begin{theorem}[Main theorem for general measures]\label{thm:rate_general_measure}
Let $f\in\M(\overline\Omega)$ be a Radon measure with $f(\overline{\Omega})=0$, let $f_\Gamma=\sum_{x\in\Gamma}a_x\delta_x$ and assume that $f^\pm(\overline\Omega)=f_\Gamma^\pm(\overline\Omega)$.
We make \cref{ass:rho_c1a,ass:omega_c11,ass:neps,ass:epsk} and additionally assume that $k\geq 2$, $\eps \ll 1$, $n\eps^{2d} \geq 1$, $\eps\leq \eps_k^d\leq 1$ and $\eps_k \log(\eps^{-1})^{\frac{1}{2}}\leq \frac{\dist(\Gamma,\partial\Omega)}{24(d+2)}$. Let $u_{n,\epsilon}\in \lo2$ be defined by \labelcref{eq:une} and let $u\in W^{1,p}(\Omega)$ for $1\leq p<\frac{d}{d-1}$ be the distributional solution of $-\div(\rho^2\nabla u)=f$ in the sense of \cref{def:distr_sol_Poisson_eq} with $\varrho=\rho^2$.
Then there exist constants $C_1,C_2>0$ such that for $q$ defined by \labelcref{eq:q_prob} the following hold.
\begin{enumerate}[label=(\roman*)]
\item If $\rho\equiv |\Omega|^{-1}$ is constant, then for any $0 < \sigma < 1$ we have that
\begin{equation*}
\ng{1}{u - u_{n,\eps}} \lesssim \sum_{x\in \Gamma}|a_x|\log(\eps^{-1})^{\frac{d}{2}+1}\left(\eps_k^{2-\sigma} + \eps_k^{-\frac{d}{2}}\eps^{\frac{1}{2}}\right)
+
W_1(f^+,f_\Gamma^+)
+
W_1(f^-,f_\Gamma^-).
\end{equation*}
holds with probability at least $1 - C_1q - 2^{3+\gamma}n^{-\gamma}$, where $\gamma = \frac{\sigma}{d-\sigma}$. 
\item In general we have that
\begin{equation*}
\ng{1}{u - u_{n,\eps}} \lesssim \sum_{x\in \Gamma}|a_x|\log(\eps^{-1})^{\frac{d}{2}}\left(\eps_k + \eps_k^{-\frac{d}{2}}\eps^{\frac{1}{2}}\right)
+
W_1(f^+,f_\Gamma^+)
+
W_1(f^-,f_\Gamma^-)
\end{equation*}
holds with probability at least $1 - C_1q -  2^{3+\gamma}n^{-\gamma}$, where $\gamma$ is given in \labelcref{eq:gamma_d}. 
\end{enumerate}
\end{theorem}
\begin{proof}
    Let $u_\Gamma$ be defined as \labelcref{eq:ucontinuum} and hence it is a distributional solution of $-\div(\rho^2\nabla u_\Gamma)=\sum_{x\in\Gamma}a_x\delta_x=f_\Gamma$.
    Applying \cref{thm:Lp_stability} we have for $1\leq p < \frac{d}{d-1}$ that
    \begin{align*}
        \norm{u-u_\Gamma}_{L^p(\Omega)}\leq 
        C\sup
        \left\lbrace 
        \int_{\overline{\Omega}}
        \psi\d(f-f_\Gamma)
        \st 
        \psi \in C^{1,\beta}(\Omega),\,
        \norm{\psi}_{C^{1,\beta}(\Omega)}\leq 1
        \right\rbrace
    \end{align*}
    where $\beta \defeq   1-\frac{d(p-1)}{p}>0$.
    Using that $(f-f_\Gamma)(\overline\Omega)=0$ and that $C^{1,\beta}(\Omega)$ is continuously embedded in $C^{0,1}(\Omega)$, we can estimate the right hand side as follows
    \begin{align*}
        &\phantom{{}={}}
        \sup
        \left\lbrace 
        \int_{\overline{\Omega}}
        \psi\d(f-f_\Gamma)
        \st 
        \psi \in C^{1,\beta}(\Omega),\,
        \norm{\psi}_{C^{1,\beta}(\Omega)}\leq 1
        \right\rbrace
        \\
        &\lesssim
        \sup        
        \left\lbrace 
        \int_{\overline{\Omega}}
        \psi\d(f-f_\Gamma)
        \st 
        \Lip(\psi)\leq 1
        \right\rbrace
        \\
        &=
        \sup        
        \left\lbrace 
        \int_{\overline{\Omega}}
        \psi\d(f^+-f_\Gamma^+)
        +
        \int_{\overline\Omega}
        \phi\d(f_\Gamma^-
         -f^-)
        \st 
        \Lip(\psi)\leq 1
        \right\rbrace
        \\
        &\leq 
        \sup        
        \left\lbrace 
        \int_{\overline{\Omega}}
        \psi\d(f^+-f_\Gamma^+)
        \st 
        \Lip(\psi)\leq 1
        \right\rbrace
        +
        \sup        
        \left\lbrace 
        \int_{\overline\Omega}
        \phi\d(f_\Gamma^-
         -f^-)
        \st 
        \Lip(\psi)\leq 1
        \right\rbrace
        \\
        &=
        W_1(f^+,f_\Gamma^+)
        +
        W_1(f^-,f_\Gamma^-).
    \end{align*}
    An application of a standard Monte Carlo estimate (see \cref{prop:lp_graph_control} below) shows that
    \begin{align*}
        \norm{u-u_\Gamma}_{\ell^1(\X_n)}
        \lesssim
        W_1(f^+,f_\Gamma^+)
        +
        W_1(f^-,f_\Gamma^-)
    \end{align*}
    with probability at least $1-2^{1+p}n^{1-p}$.
    Applying \cref{thm:main_general}, using the triangle inequality, and applying a union bound yields the desired result.
\end{proof}
As before we get effective rates of convergence with high probability by optimizing over the parameters.
\begin{corollary}\label{cor:general_measure}
    In the context of \cref{thm:rate_general_measure} we in general have that
    \begin{equation}\label{eq:nonconstant_rate_optimal_measure}
        \ng{1}{u - u_{n,\eps}} \lesssim \sum_{x\in \Gamma}|a_x|\log(\eps^{-1})^{\frac{d}{2}}\eps^{\frac{1}{d+2}}+
        W_1(f^+,f_\Gamma^+)
        +
        W_1(f^-,f_\Gamma^-),
    \end{equation}
    holds with probability at least 
    \[1  -  2^{3+\gamma}n^{-\gamma} - C_1n^3\exp\left( -C_2 \Theta_{d,\eps^{-2}}^{-1}n\eps^{3d + \frac{2}{d+2}}\right).\]
    where $\gamma$ is given in \labelcref{eq:gamma_d}.
    Assuming that $\rho\equiv |\Omega|^{-1}$ is constant, there exists $C_1,C_2>0$ such that for any $0 < \sigma < 1$ we have that
    \begin{equation}\label{eq:constant_rate_optimal_measure}
        \ng{1}{u - u_{n,\eps}} \lesssim \sum_{x\in \Gamma}|a_x|\log(\eps^{-1})^{\frac{d}{2}+1}\eps^{\frac{2-\sigma}{d+4}}
        +
        W_1(f^+,f_\Gamma^+)
        +
        W_1(f^-,f_\Gamma^-),
    \end{equation}
    holds with probability at least 
    \[1  -  2^{3+\gamma}n^{-\gamma} - C_1n^3\exp\left( -C_2 \Theta_{d,\eps^{-2}}^{-1}n\eps^{3d + \frac{6}{d+4}}\right), \ \ \text{where} \ \ \gamma = \frac{\sigma}{d-\sigma}.\]
\end{corollary}
\begin{remark}\label{rem:general_measure}
    To make the these results more interpretable we consider the case where $\dist(\supp f,\partial\Omega)>0$ and the measures $f^\pm$ are approximates by empirical measures consisting of $m$ atomic measures. 
    This can be done, for instance, by letting $\Gamma$ be the union of $m/2$ \emph{i.i.d.} samples from $f^\pm$ (independent of the random geometric graph construction!) and letting $a_x=\frac{\pm 1}{m}$ for $x\sim f^\pm$.
    Here we assume w.l.o.g. that $f^\pm$ are probability measures.
    In this case we have that $\sum_{x\in\Gamma}\abs{a_x}=1$ and furthermore $W_1(f^\pm,f_\Gamma^\pm)\sim \left(\frac{\log m}{m}\right)^{1/d}$ with high probability \cite{trillos2015rate}.
\end{remark}

\subsection{Regularizing the graph equation}

We now proceed with the three key lemmas for proving  \cref{thm:main_general}. We start with a discrete smoothing result on the graph that allows us to replace the graph delta functions with the heat kernel.  Let $\H_k$ be the heat kernel on the random geometric graph, as defined in \cref{sec:hka}. We let $k\geq 2$ and define $u_{n,\epsilon,k} = \H_k * u_{n,\epsilon}$. Our first result relates $u_{n,\epsilon,k}$ to $u_{n,\epsilon}$ and shows that $u_{n,\epsilon,k}$ solves a regularized graph problem.

\begin{lemma}\label{lem:main_smooth_graph}
Let $k\geq 2$ and assume $\Gamma_n + B(0,R_k) \subset \Omega$.  Define the smoothed graph function $u_{n,\epsilon,k}=\H_k*u_{n,\epsilon}$. Then the following hold.
\begin{enumerate}[label=(\roman*)]
\item We have
\begin{equation}\label{eq:smooth_graph_problem}
u_{n,\epsilon,k}= \argmin_{u\in \lo2}\ene\left(u;f_{n,\epsilon,k}\right),
\end{equation}
where
\begin{equation}\label{eq:fnek}
f_{n,\epsilon,k} = \sum_{x\in \Gamma} a_x \H^{\tau(x)}_k.
\end{equation}
\item If $p=1$, then 
\begin{equation}\label{eq:main_graph_l1}
\ng1{u_{n,\epsilon} - u_{n,\epsilon,k}} \lesssim \sum_{x\in \Gamma}|a_x|\eps_k^2,
\end{equation}
holds with probability at least $1 - 2n\exp\left( -Cn\epsilon^d\right)$.
\item If $k\epsilon^2 \ll 1$ then for any $p>1$ we have
\begin{equation}\label{eq:main_graph_lp}
\ng{p}{u_{n,\epsilon} - u_{n,\epsilon,k}} \lesssim\sum_{x\in \Gamma}|a_x| \epsilon_k^{2}\left[\log(\epsilon^{-1})^{\frac{d}{2p}}\eps_k^{\frac{d}{p}}\left( \epsilon_k^{-d} + \lambda \cdk \eps^{-d}\right) + n^{1-\frac{1}{p}}\epsilon^2\right]
\end{equation}
holds for $0 < \lambda \leq 1$ with probability at least $1-15kn^2\exp\left( -C n \epsilon^d \lambda^2\right)$.
\end{enumerate}
\end{lemma}
\begin{remark}
This result is largely consistent with the continuum case (cf. \cref{thm:convergence_rate_smoothing,rem:actual_rate}), where both the rate $\Ri^{2-d\frac{p-1}{p}}$ and the validity of the rate for $p<\frac{d}{d-2}$ are the same by identifying $\Ri\sim\eps_k$. There is, however, an extra error term $n^{1-\frac{1}{p}}\eps^2 \eps_k^2$ that arises in the graph setting. 
\end{remark}
\begin{proof}
We recall from \cref{sec:continuum_limits} and the definition of the graph Laplacian $\L_{n,\eps}$ in \labelcref{eq:graph-laplacian-geometric} that $u_{n,\epsilon}$ is the solution of
\[\L_{n,\epsilon} u_{n,\epsilon} = \sum_{x\in \Gamma} a_x \delta_{\tau(x)}\]
satisfying $(u_{n,\epsilon})_{\deg} = 0$, where $\delta_x(y)=n$ for $x=y$ and $\delta_x(y)=0$ otherwise. We recall the definition of the unnormalized Laplacian
\[\L u(x) = \sum_{i=1}^n\eta_\eps(\abs{x_i-x_j})\left(u(x_i)-u(x_j)\right),\]
and note from \labelcref{eq:graph-laplacian-geometric} that 
\begin{equation}\label{eq:laplacian_exchange}
\L_{n,\epsilon} u(x) = \frac{1}{\sigma_\eta(n-1)\epsilon^2}\L u(x).
\end{equation}
Therefore, $u_{n,\epsilon}$ satisfies
\begin{equation}\label{eq:ceq}
\L u_{n,\epsilon} = \sigma_\eta(n-1)\epsilon^2\sum_{x\in \Gamma} a_x \delta_{\tau(x)}.
\end{equation}
By \cref{thm:smooth_graph_poisson} and \labelcref{eq:ceq} we have
\[\L u_{n,\epsilon,k} = \sigma_\eta(n-1)\epsilon^2\sum_{x\in \Gamma} a_x \H^{\tau(x)}_k,\]
which is equivalent, by \labelcref{eq:laplacian_exchange}, to 
\begin{equation}\label{eq:smoothed_graph_eq}
\L_{n,\epsilon} u_{n,\epsilon,k} = \sum_{x\in \Gamma} a_x \H^{\tau(x)}_k.
\end{equation}
By \labelcref{eq:ceq} and \cref{thm:difference_graph_smoothed} we have $(u_{n,\epsilon,k})_{\deg}=0$ and 
\begin{equation}\label{eq:graph_smooth_diff}
u_{n,\epsilon} - u_{n,\epsilon,k} = \deg_{n,\epsilon}^{-1}\sigma_\eta(n-1)\epsilon^2\sum_{x\in \Gamma}a_x\sum_{j=0}^{k-1}\H_j^{\tau(x)}.
\end{equation}
Therefore, $u_{n,\epsilon,k}$ is the solution of \labelcref{eq:smooth_graph_problem}, which establishes (i).

For (ii), we take $\ell^p$ norms on both sides of \labelcref{eq:graph_smooth_diff} and apply a union bound with \cref{prop:degree} for $\lambda=1$ to obtain
\begin{equation}\label{eq:lpbound_1}
\ng{p}{u_{n,\epsilon} - u_{n,\epsilon,k}} \lesssim \epsilon^2\sum_{x\in \Gamma}|a_x|\sum_{j=0}^{k-1}\ng{p}{\H_j^{\tau(x)}}
\end{equation}
with the claimed probability.
When $p=1$ we use $\|\H_j^{\tau(x)}\|_{\ell^1(\X_n)}=1$ to complete the proof of (ii). 
When $p>1$ we first note that $\H_0^{\tau(x)} = \delta_x$ and so $\ng{p}{\H^{\tau(x)}_0} = n^{1-\frac{1}{p}}$. For $j\geq 1$ we  use \labelcref{eq:heat_kernel_lp} from \cref{cor:heatkernel_lp}, which yields
\begin{equation}\label{eq:sum_heat_kernel}
\sum_{j=1}^{k-1}\ng{p}{\H_j^{\tau(x)}} \lesssim \log(\epsilon^{-1})^{\frac{d}{2p}}\left(\sum_{j=1}^{k-1}\epsilon_j^{-d( 1-\frac{1}{p})} + \lambda \epsilon^{-d}\sum_{j=1}^{k-1}\cdj\eps_j^{\frac{d}{p}}\right).
\end{equation}
Now, note that  $0 < \frac{d}{2}( 1-\tfrac{1}{p}) < 1$ holds for all $p>1$ when $d=1,2$, and for $1 < p < \frac{d}{d-2}$ when $d\geq 3$. Hence, the remaining terms in the first sum can be bounded by 
\[\sum_{j=1}^{k-1}\epsilon_j^{-d(1-\frac{1}{p})} = \epsilon^{-d(1-\frac{1}{p})}\sum_{j=1}^{k-1}j^{-\frac{d}{2}(1-\frac{1}{p})} \lesssim \epsilon^{-d(1-\frac{1}{p})}k^{1-\frac{d}{2}(1 - \frac{1}{p})} = k\epsilon_k^{-d(1-\frac{1}{p})} .\]
Since $\cdj \eps_j^{\frac{d}{p}} \leq \cdk \eps_k^{\frac{d}{p}}$, the second sum is bounded by $k\:\cdk \eps_k^{\frac{d}{p}}$. 
Combining these bounds with $\ng{p}{\H^{\tau(x)}_0} = n^{1-\frac{1}{p}}$ yields
\[\sum_{j=0}^{k-1}\ng{p}{\H_j^{\tau(x)}} \lesssim k \log(\epsilon^{-1})^{\frac{d}{2p}}\eps_k^{\frac{d}{p}}\left(k^{-1}n^{1-\frac{1}{p}}\eps_k^{-\frac{d}{p}} + \epsilon_k^{-d} + \lambda \cdk\eps^{-d}\right) + n^{1-\frac{1}{p}}.\]
Substituting this into \labelcref{eq:lpbound_1} completes the proof.
\end{proof}

\subsection{Discrete to continuum convergence}

We now proceed to our main discrete to continuum convergence result (\cref{lem:discrete_to_continuum} below), which is a direct application of \cref{thm:main_smooth} (\cref{thm:main_smooth_simplified} to be more precise).
We define the continuum source term
\begin{equation}\label{eq:fk}
f_{k} = \sum_{x\in \Gamma}a_x 
\begin{cases}
\displaystyle\rho^{-1}\psi_{k,\epsilon}(\cdot - x),& \text{if } \rho \text{ is constant,} \\
\displaystyle\widehat\rho_\eps(x)^{-1}\M^{k-1}_\eps\eta_\eps^{x},& \text{otherwise.} 
\end{cases}
\end{equation}
Note that when $\rho$ is constant, we have $\rho \equiv |\Omega|^{-1}$ since $\rho$ is a probability density. We define $u_k \in H^1_\rho(\Omega)$ by 
\begin{equation}\label{eq:uk_var}
u_k = \argmin_{u\in H^1_\rho(\Omega)}\int_\Omega \frac{\rho^2}{2} |\nabla u|^2 - f_ku\rho \d x.
\end{equation}
The source term $f_k$ and corresponding solution $u_k$ are the continuum objects that we shall show, in  \cref{lem:discrete_to_continuum} below, the regularized graph solution $u_{n,\epsilon,k}$ converges to as $n\to \infty$ and $\epsilon\to 0$, as is verified in the following result.

\begin{lemma}\label{lem:discrete_to_continuum}
Assume that $n^{-\frac{1}{d}} \leq \delta \ll \epsilon \ll 1$, $k\geq 2$, and $k\eps \geq 1$. If $R_k \leq \frac{1}{4}\dist(\Gamma,\partial\Omega)$ and $\eps \lesssim \eps_k^d$ then the probability that
\begin{equation}\label{eq:discrete_to_continuum}
\|u_k - u_{n,\epsilon,k}\|_{H^1(\X_n)}  \lesssim \sum_{x\in \Gamma}|a_x|\log(\eps^{-1})^\frac{d}{2}\epsilon_k^{-\frac{d}{2}}\sqrt{\tfrac{\delta}{\epsilon} + \epsilon}
\end{equation}
holds is at least
\[1-C_1\left(kn^2\exp\left( -C_2\cdk^{-1}  n \epsilon^{3d+2}\epsilon_k^{-2(d + 1)}\right) + \exp(-C_2n\epsilon^{d+2}\epsilon_k^{2d})+\exp(-C_2n\delta^d\eps^2)\right).\]
\end{lemma}
\begin{proof}
We first give the proof in the case that $\rho\equiv |\Omega|^{-1}$ is constant. Note that since $\sum_{x\in \Gamma_n}a_x=0$,  $f_{n,\epsilon,k}\in \l2$ satisfies the compatibility condition $\ipg{f_{n,\epsilon,k},\one} = 0$. Thus, we may apply \cref{thm:main_smooth_simplified} to compare the discrete pair $(u_{n,\epsilon,k},f_{n,\epsilon,k})$ to the continuum counterpart $(u_k,f_k)$. We use the parameters
\begin{equation}\label{eq:Gamma_boundary}
R = \frac{1}{4}\dist(\Gamma,\partial \Omega),
\end{equation}
and we choose $q>\frac{d}{2}$ as $q \defeq   d$. In particular, this means that $R_k \geq R$. This yields 
\begin{align}\label{eq:rate_before}
\|u_k - u_{n,\epsilon,k}\|_{H^1(\X_n)}^2 \lesssim& \left(\ng1{f_{n,\epsilon,k}-f_k} +\|\osc_{\Omega\cap B(\cdot,\delta)}f_k\|_{L^1(\Omega)} \right)\left(\|f_k\|_{L^{q}(\Omega)} + \ng{q}{f_{n,\epsilon,k}}\right)\notag \\
&+\|f_k\|_{L^\infty(\Omega)}^2\lambda_1+\|f_k\|_{L^\infty(\partial_{4\epsilon}\Omega)}^2\eps+ \ng2{f_{n,\epsilon,k}} \|f_{n,\epsilon,k}\|_{\ell^2(\X_n\cap \partial_{2R}\Omega)} \notag \\
&+\left(\ng2{f_{n,\epsilon,k}}^2 + \|f_k\|^2_{L^2(\Omega)}\right)\left(\frac{\delta}{\eps}+\epsilon +\lambda_1^2 + \lambda_2\right),
\end{align}
with probability at least $1-C_1\left(\exp(-C_2n\epsilon^d\lambda_1^2)+\exp(-C_2n\delta^d\lambda_2^2)\right)$.  Since the proof of  \cref{thm:main_smooth_simplified} invokes  \cref{thm:transportation}, we may assume the latter holds for the rest of the proof.

The rest of the proof amounts to estimating the terms above, and choosing the parameters $\lambda,\lambda_1,\lambda_2$ appropriately. For notational simplicity we write $\sa = \sum_{x\in \Gamma}|a_x|$ throughout the rest of the proof.  For each $x\in \Gamma$, we let $\cp x\in \X_n$ denote the closest point projection, so that $|\tau(x) - x|\leq \delta$ and
\[f_{n,\epsilon,k} = \sum_{x\in \Gamma} a_x \H^{\cp{x}}_k.\]
Then on $\X_n$ we have
\[f_{k} - f_{n,\epsilon,k} = \sum_{x\in \Gamma}a_x \left(\rho^{-1} \psi_{k,\epsilon}(\cdot - x) - \H^{\cp{x}}_k\right),\]
and so using \cref{lem:psik_graph_bounds,cor:heatkernel_l1_asymptotic} we have
\begin{align*}
\ng{1}{f_{k} - f_{n,\epsilon,k}} &\leq \sum_{x\in \Gamma}|a_x|\ng{1}{\rho^{-1} \psi_{k,\epsilon}(\cdot - x) - \H^{\cp{x}}_k}\\
&\lesssim \sum_{x\in \Gamma}|a_x| \ng{1}{\psi_{k,\epsilon}(\cdot - x) - \psi_{k,\epsilon}(\cdot - \cp x)}\\
&\hspace{1in}+  \sum_{x\in \Gamma}|a_x|  \ng{1}{\rho\psi_{k,\epsilon}(\cdot - \cp x) - \H^{\cp{x}}_k}\\
&\lesssim \sa\left(\eps_k^{-1}\delta\log(\eps^{-1})^{\frac{d}{2}} +  \lambda \eps^{-d}\cdk \epsilon_k^{d}\log(\eps^{-1})^{\frac{d}{2}}+\eps\right),
\end{align*}
with probability at least $1-C_1kn^2\exp\left( -C_2 n \epsilon^d \lambda^2\right)$.

By \cref{prop:psi_gaussian_upper} (i) we have
\begin{equation}\label{eq:fk_infty}
\|f_k\|_{L^\infty(\Omega)} \lesssim \sa\eps_k^{-d} \ \ \text{and} \ \ \|f_k\|_{L^\infty(\partial_{4\epsilon}\Omega)} \lesssim 1,
\end{equation}
the latter holding due to \labelcref{eq:Rk_prop} and the fact that $\dist(\Gamma,\partial\Omega) \geq 4R_k \geq R_k + 5\eps$ so that $\dist(\Gamma,\partial_{4\eps}\Omega) \geq R_k + \eps$.
%\[\dist(\Gamma,\partial\Omega) \geq \epsilon_k\sqrt{8d^2 \log(\epsilon^{-1})} + 4\epsilon + \delta,\]
%which is true due to the assumption that $\dist(\Gamma,\partial\Omega) \geq 3R_k$. 
By  \cref{prop:psi_gaussian_upper} (iii) and using $q=d$ we have
\begin{equation}\label{eq:fk_Lq}
\|f_k\|_{L^q(\Omega)} \lesssim \sa\eps_k^{-d + 1}\log(\eps^{-1})^{\frac{1}{2}}, \ \ \text{and} \ \   \|f_k\|_{L^2(\Omega)} \lesssim \sa\eps_k^{-\frac{d}{2}}\log(\eps^{-1})^{\frac{d}{4}},
\end{equation}
and by \cref{prop:psi_gaussian_upper} (iv) we have
\[\|\osc_{\Omega\cap B(\cdot,\delta)}f_k\|_{L^1(\Omega)} \leq \sum_{x\in \Gamma}|a_x|\|\osc_{\Omega\cap B(\cdot,\delta)}\psi_{k,\epsilon}(\cdot-x)\|_{L^1(\Omega)} \lesssim \sa\left(\eps_k^{-1}\delta\log(\eps^{-1})^{\frac{d}{2}} + \eps^2\right).\]
Now, using \cref{cor:heatkernel_lp} with $p=q=d$ we have
\begin{equation}\label{eq:Hkpq}
\ng{q}{f_{n,\epsilon,k}} \lesssim \sa\epsilon_k\log(\epsilon^{-1})^{\frac{1}{2}}\left(\epsilon_k^{-d} + \lambda \epsilon^{-d}\cdk\right),
\end{equation}
and with $p=2$ we have
\begin{equation}\label{eq:Hkp2}
\ng{2}{f_{n,\epsilon,k}} \lesssim \sa\epsilon_k^{\frac{d}{2}}\log(\epsilon^{-1})^{\frac{d}{4}}\left(\epsilon_k^{-d} + \lambda \epsilon^{-d}\cdk\right).
\end{equation}
Now, by \labelcref{eq:Gamma_boundary} we have 
\[\dist(\Gamma_n,\partial_{2R}\Omega) \geq \dist(\Gamma,\partial\Omega) - 2R - \delta =2R-\delta \geq R_k.\]
It follows from \cref{cor:heat_kernel_decay} that
\begin{equation}\label{eq:fneboundary}
\|f_{n,\epsilon,k}\|_{\ell^2(\X_n\cap \partial_{2R}\Omega)} \lesssim \epsilon \ng{2}{1} \lesssim \eps.
\end{equation}

Inserting all of these estimates above and simplifying the logarithmic terms we have
\begin{align}\label{eq:simplified_combination}
\|u_k - u_{n,\epsilon,k}\|_{H^1(\X_n)}^2  &\lesssim A^2\log(\eps^{-1})^d\Bigg[ \lambda \cdk \epsilon^{-d}\epsilon_k + \epsilon_k^{-2d}\lambda_1 \notag\\
&\hspace{1in}  + \left( \epsilon_k^{-d} + \lambda^2\cdk^2 \epsilon^{-2d}\epsilon_k^d\right)\left( \frac{\delta}{\epsilon} + \epsilon + \lambda_1^2 + \lambda_2\right) \Bigg],
\end{align}
with probability at least
\[1-C_1\left(kn^2\exp\left( -C_2 n \epsilon^d \lambda^2\right) + \exp(-C_2n\epsilon^d\lambda_1^2)+\exp(-C_2n\delta^d\lambda_2^2)\right),\]
provided that $\lambda \cdk \epsilon^{-d}\eps_k \lesssim 1$.  The largest error term without tunable parameters is $\eps_k^{-d}\epsilon$, so we will choose parameters to match this error term. Hence, we now choose $\lambda$ so that
\[\lambda \cdk \epsilon^{-d}\epsilon_k = \epsilon_k^{-d}\epsilon, \ \ \text{that is} \ \ \lambda =\cdk^{-1} \epsilon^{d+1} \epsilon_k^{-(d + 1)},\]
and we require $\epsilon \lesssim \eps_k^d$.  Since $\eps\leq 1$ and $k\epsilon\geq 1$ we have
\[\lambda^2 \cdk^2 \epsilon^{-2d}\eps_k^d = \epsilon_k^{-d}\epsilon_k^{-2} \eps^2 = \epsilon_k^{-d}k^{-1} \leq \epsilon_k^{-d}.\]
Inserting this above yields
\begin{align*}
\|u_k - u_{n,\epsilon,k}\|_{H^1(\X_n)}^2  &\lesssim A^2\log(\eps^{-1})^d\left[ \epsilon_k^{-2d}\lambda_1  + \epsilon_k^{-d}\left( \frac{\delta}{\epsilon} + \epsilon + \lambda_1^2 + \lambda_2\right) \right].
\end{align*}
The proof when $\rho$ is constant is completed by choosing $\lambda_1 = \epsilon_k^d\epsilon$ and $\lambda_2 = \eps$.

The proof when $\rho$ is not constant is very similar, with differences in only two estimates. First, using \cref{cor:heatkernel_l1_asymptotic_rho,thm:heatkernel_final,cor:Me_graph_bounds,thm:Me} we have
\begin{align*}
\ng{1}{f_{k} - f_{n,\epsilon,k}} &\leq \sum_{x\in \Gamma}|a_x|\ng{1}{\widehat\rho_\eps(x)^{-1}\M^{k-1}_\eps\eta_\eps^{x}- \H^{\cp{x}}_k}\\
&\lesssim \sum_{x\in \Gamma}|a_x| \ng{1}{\widehat\rho_\eps(x)^{-1}\left(\M^{k-1}_\eps\eta_\eps^{x}-\M^{k-1}_\eps\eta_\eps^{\tau(x)}\right)}\\
&\hspace{0.5in}+\sum_{x\in \Gamma}|a_x| \ng{1}{\left(\widehat\rho_\eps(x)^{-1}-\widehat\rho_\eps(\tau(x))^{-1}\right)\M^{k-1}_\eps\eta_\eps^{\tau(x)}}\\
&\hspace{1in}+  \sum_{x\in \Gamma}|a_x|  \ng{1}{\widehat\rho_\eps(\tau(x))^{-1}\M^{k-1}_\eps\eta_\eps^{\tau(x)}- \H^{\cp{x}}_k}\\
&\lesssim \sa\left(\eps^{-1}\delta\log(\eps^{-1})^{\frac{d}{2}} +  \lambda \eps^{-d}\cdk \epsilon_k^{d}\log(\eps^{-1})^{\frac{d}{2}}+\eps\right),
\end{align*}
with probability at least $1-C_1kn^2\exp\left( -C_2 n \epsilon^d \lambda^2\right)$. By \cref{prop:psi_gaussian_upper} (i), (iii) and  \cref{thm:Me} we have that \labelcref{eq:fk_infty} and \labelcref{eq:fk_Lq} hold. By  \cref{cor:Me_graph_bounds} (ii) we have
\[\|\osc_{\Omega\cap B(\cdot,\delta)}f_k\|_{L^1(\Omega)} \leq \sum_{x\in \Gamma}|a_x|\|\widehat\rho_\eps(x)^{-1}\M^{k-1}_\eps\eta_\eps^{x}\|_{L^1(\Omega)} \lesssim \sa\left(\eps^{-1}\delta\log(\eps^{-1})^{\frac{d}{2}} + \eps^2\right).\]
As before we have \labelcref{eq:Hkp2,eq:Hkpq,eq:fneboundary}, since these do not depend on $\rho$ being constant. Inserting all of these estimates into \labelcref{eq:rate_before} we find that \labelcref{eq:simplified_combination} holds as before, and the rest of the proof is the same. 
\end{proof}

\subsection{Regularizing the continuum PDE}

The final step is to perform the smoothing of the continuum PDE. To pull this back to the graph, we require a standard Monte Carlo estimate, which is stated in  \cref{prop:lp_graph_control} below, and whose proof is postponed to  \cref{app:monte}. 
\begin{proposition}\label{prop:lp_graph_control}
Let $r>1$. If $u\in L^r(\Omega)$ is Borel measurable, then for any $p \in [1,r)$ we have
\[\P\left(\ng{p}{u} \leq 2^{\frac{1}{p}}\rho_{\max}^{\frac{1}{r}}\norm{u}_{L^r(\Omega)}\right) \geq 1 - 2^{1+\frac{r}{p}}n^{1-\frac{r}{p}}.\]
\end{proposition}

Our main continuum smoothing result is as follows.
\begin{lemma}\label{lem:continuum_smoothing}
If $\eps \ll 1$ then the following hold. 
\begin{enumerate}[label=(\roman*)]
\item If $\rho\equiv |\Omega|^{-1}$ is constant, then for $1 \leq p < r < \frac{d}{d-2}$ we have that
\begin{equation}\label{eq:continuum_smoothing}
\ng{p}{u - u_k} \lesssim \sum_{x\in \Gamma}|a_x| \epsilon_k^{2-d + \frac{d}{r}} 
\log(\eps^{-1})^{\frac{d}{2r}+1}
\end{equation}
holds with probability at least $1 - 2^{1 + \frac{r}{p}}n^{1-\frac{r}{p}}$.
\item In general, for $1 \leq p < \frac{d}{d-1}$ we have that
\begin{equation}\label{eq:continuum_smoothing_lp}
\ng{p}{u - u_k} \lesssim \sum_{x\in \Gamma}|a_x|\eps_k \log(\eps^{-1})^{\frac{1}{2}}
\end{equation}
%holds with probability at least $1 - 2^{1 + \frac{r}{p}}n^{1-\frac{r}{p}}$, where
holds with probability at least $1 - 2^{2+\gamma}n^{-\gamma}$, where $\gamma > 0$ is given by $\gamma=1$ when $d=1$ and for 
\[\gamma = \frac{d}{2p(d-1)} - \frac{1}{2} \ \ \text{for} \ \ d \geq 2.\]
\end{enumerate}
\end{lemma}
\begin{proof}
We first give the proof of (i) when $\rho$ is constant. The first part of the proof is an application of \cref{thm:convergence_rate_smoothing}. We define $f,g\in L^\infty(\Omega)$ by 
\begin{equation}\label{eq:f_g}
f = \sum_{x\in \Gamma}a_x \psi_{k,\epsilon}(\cdot - x), \ \ \text{and} \ \ g = \sum_{x\in \Gamma} a_x \psi_{k,\epsilon}(\cdot-x)\one_{B(x,R_k)}.
\end{equation}
Then $u_k\in H^1_\rho(\Omega)$ defined as the solution of \labelcref{eq:uk_var} is also the solution of \labelcref{eq:variational_problem}  with $\varrho=\rho^2$ and $f$ given in \labelcref{eq:f_g}.  Thus, \cref{thm:convergence_rate_smoothing} allows us to estimate $\|u_k - u\|_{L^p(\Omega)}$.  By  \cref{lem:Hoeffding} and \labelcref{eq:Rk_prop} we have
\begin{equation}\label{eq:psi_k_eps_bound}
1 - \frac{2d}{k}\eps^{d+2} \leq \int_{B(0,R_k)} \psi_{k,\epsilon}\d x \leq 1.
\end{equation}
Also by \cref{prop:psi_gaussian_upper} (i) and \labelcref{eq:Rk_prop} we have
\[\|f - g\|_{L^\infty(\Omega)} \leq \sum_{x\in \Gamma}|a_x|\|\psi_{k,\epsilon}\|_{L^\infty(\R^d \setminus B(0,R_k))} \leq \sum_{x\in \Gamma}|a_x|\eps^{-d}k^{-1}\eps^{d+2}=\sum_{x\in \Gamma}|a_x| k^{-1} \eps^{2}.\]
Finally, by  \cref{prop:psi_gaussian_upper} (iii)  we have
\[\|\psi_{k,\epsilon}(\cdot-x)\|_{L^r(B(x,R_k))} \lesssim \epsilon_k^{-d + \frac{d}{r}}\log(\epsilon^{-1})^{\frac{d}{2r}}.\]
We now use  \cref{thm:convergence_rate_smoothing} with $r_i=R_k$ and $b_i = \int_{B(0,R_k)} \psi_{k,\epsilon}\, dx$, along with \labelcref{eq:Rk_prop},  to obtain 
\[\|u - u_k\|_{L^r(\Omega)} \lesssim \sum_{x\in \Gamma}|a_x| \epsilon_k^{2-d + \frac{d}{r}} \log(\eps^{-1})^{\frac{d}{2r}+1}\]
for $1 < r < \frac{d}{d-2}$. The last part of the proof, for $\rho$ constant, is an application of \cref{prop:lp_graph_control}.

To prove (ii), where $\rho$ is not constant, we use \cref{thm:convergence_rate_varrho} instead of \cref{thm:convergence_rate_smoothing}. We define 
\[f = \sum_{x\in \Gamma} a_x \widehat \rho_\eps(x)^{-1}\rho\,\M^{k-1}_\eps \eta^x_\eps,\]
and define the $g$ as in \labelcref{eq:f_g}. Then we have
\begin{align*}
\|f-g\|_{L^1(\Omega)} &\leq \sum_{x\in \Gamma}|a_x| \|\widehat \rho_\eps(x)^{-1}\rho\,\M^{k-1}_\eps \eta^x_\eps - \psi_{k,\epsilon}(\cdot-x)\one_{B(x,R_k)}\|_{L^1(\Omega)}\\
&\lesssim \sum_{x\in \Gamma}|a_x|\Big(\|\widehat \rho_\eps(x)^{-1}\rho\,\M^{k-1}_\eps \eta^x_\eps - \psi_{k,\epsilon}(\cdot-x)\|_{L^1(B(x,R_k))} \\
&\hspace{2.5in}+ \|\M^{k-1}_\eps \eta^x_\eps\|_{L^1(\Omega\setminus B(x,R_k))}\Big).
\end{align*}
By \cref{thm:Me} and the first part of the proof we have $\|\M^{k-1}_\eps \eta^x_\eps\|_{L^1(\Omega\setminus B(x,R_k))} \lesssim \eps^2$. Since $\rho$ is $C^{1,1}(\Omega)$ and $B(x,\eps)\subset \Omega$ we have (as in  \cref{rem:rhoe}) that $\widehat \rho_\eps(x) = \rho(x) + \O(\eps^2)$. Thus, for $y\in B(x,R_k)$ and $\eps\ll 1$ we have 
\[\widehat \rho_\eps(x)^{-1}\rho(y) = (\rho(x)^{-1} + \O(\eps^2))(\rho(x) + \O(R_k)) = 1 + \O(R_k).\] 
Therefore, using  \cref{thm:Me} again we have
\begin{align*}
\|\widehat \rho_\eps(x)^{-1}&\rho\,\M^{k-1}_\eps \eta^x_\eps - \psi_{k,\epsilon}(\cdot-x)\|_{L^1(B(x,R_k))} \\
&\lesssim \|\M^{k-1}_\eps \eta^x_\eps - \psi_{k,\epsilon}(\cdot-x)\|_{L^1(B(x,R_k))}  + \|\left(\widehat \rho_\eps(x)^{-1}\rho - 1\right)\M^{k-1}_\eps \eta^x_\eps\|_{L^1(B(x,R_k))}\\
&\lesssim \eps_k^2 + R_k \lesssim R_k.
\end{align*}
It follows that 
\[\|f-g\|_{L^1(\Omega)} \lesssim \sum_{x\in \Gamma}|a_x|R_k.\]
We now apply  \cref{thm:convergence_rate_varrho} with $1 \leq p < r < \frac{d}{d-1}$  and use \labelcref{eq:Rk_prop} and  \cref{prop:lp_graph_control}  to obtain
\[\ng{p}{u - u_k} \lesssim \|u - u_k\|_{L^r(\Omega)} \lesssim \sum_{x\in \Gamma}|a_x| R_k \lesssim \sum_{x\in \Gamma}|a_x|\eps_k \log(\eps^{-1})^{\frac{1}{2}}.\]
holds with probability at least $1 - 2^{1 + \frac{r}{p}}n^{1-\frac{r}{p}}$. The proof is completed by selecting $r = \frac{1}{2}\left( p + \frac{d}{d-1}\right)$ when $d\geq 2$ and $r=2p$ when $d=1$. 
\end{proof}

\phantomsection
\addcontentsline{toc}{section}{References}
\bibliography{ref}
\bibliographystyle{abbrv}

\appendix

%%%%%%%%%%%%%%%%%%%%%%%%%%%%%%%%%%%%%%%%%%%%%%%%%%%%
%%%%%%%%%%%%%%%%%%%%%%%%%%%%%%%%%%%%%%%%%%%%%%%%%%%%
\section{Proofs from \texorpdfstring{\cref{sec:smoothed_poisson}}{Section 2}}\label{app:smoothed_poisson}

\subsection{Existence of distributional solutions}

\begin{proof}[Proof of \cref{prop:existence_solution_continuum_pde}]
Without loss of generality, we assume $d \geq 2$ since for $d=1$ by the Sobolev embedding theorem, all $W^{1,2}$ functions are continuous, and we can apply classical variational techniques. 

Fix $p \in \left[1, d/(d-1)\right)$ for the rest of the proof and let $q$ be the dual exponent $1/p + 1/q = 1$. Note that $q \geq 2$. 
The proof now consists of two steps. First, we show that there exist functions $u_n$ which solve \labelcref{eq:continuum_pde} with right-hand side $f_n$ which are uniformly bounded in $W^{1,p}(\Omega)$. Second, if $u$ is a cluster point of the $u_n$, we show that $u$ is indeed a solution of the original PDE with right-hand side $f$.

\textit{Step 1:} For the first argument, fix $n \in \N$. By standard variational techniques, there exists a unique weak solution $u_n \in W^{1, 2}(\Omega)$ to \labelcref{eq:continuum_pde} with right-hand side $f_n$.
Therefore, for any $v \in W^{1, 2}(\Omega)$ it holds
\begin{equation}\label{eq:weak_solution_un}
    \int_\Omega \varrho \, \nabla u_n \cdot \nabla v \d x = \int_\Omega v f_n \d x.% + \int_{\partial\Omega} v g_n \d\H^{d-1}.
\end{equation}
For $m \in \N$ we define
\begin{equation*}
    \psi_m(s) := 
    \begin{dcases}
        1 & \text{if } s > m+1,\\
        s - m & \text{if } m\leq s \leq m+1,\\
        0 & \text{if } -m < s < m,\\
        s + m & \text{if } -(m+1) \leq s \leq -m,\\
        -1 & \text{if } s < -(m+1).
    \end{dcases}
\end{equation*}
Since $\psi_m$ is Lipschitz, $v := \psi_m \circ u_n \in W^{1, 2}(\Omega)$ defines a valid test function. Inserting into~\labelcref{eq:weak_solution_un} yields
\begin{equation}\label{eq:bound_un_Dm}
    \int_{D_{m,n}} \abs{\nabla u_n}^2 \varrho \d x = \int_\Omega \left(\psi_m \circ u_n\right) f_n \d x,
    %+\int_{\partial\Omega} \left(\psi_m \circ u_n\right) g_n \d\H^{d-1},
\end{equation}
where $D_{m,n} := \left\{x \in \Omega: m \leq \abs{u_n(x)} \leq m+1\right\}$. Since $\abs{\psi_m} \leq 1$ and $\varrho$ is bounded from below by a positive constant, we can use \labelcref{eq:bound_un_Dm} to bound
\begin{equation*}
    \int_{D_{m,n}} \abs{\nabla u_n}^2 \d x \leq \varrho_{\min}^{-1} 
    \norm{f_n}_{L_1(\Omega)}
    \leq C,
\end{equation*}
where $C$ is independent of $n$ or $m$ by assumption. Now, we apply Hölder's inequality to see
\begin{equation*}
    \int_{D_{m,n}} \abs{\nabla u_n}^p \d x 
    \leq \left(\int_{D_{m,n}} \abs{\nabla u_n}^2 \d x\right)^{\frac p2} \abs{D_{m,n}}^{\frac{2-p}{2}} \leq C \abs{D_{m,n}}^{\frac{2-p}{2}}.
\end{equation*}
To bound the size of $D_{m,n}$, we choose some $p^*>1$ to be determined later and estimate
\begin{equation*}
    \abs{D_{m,n}}
    = \int_{D_{m,n}} 1 \d x
    = \frac1{m^{p^*}} \int_{D_{m,n}} m^{p^*} \d x
    \leq \frac1{m^{p^*}} \int_{D_{m,n}} \abs{u_n}^{p^*} \d x.
\end{equation*}
We fix $m_0$ to be determined later and write
\begin{equation}\label{eq:un_Lpnorm_gradient_bounded_by_sum}
\begin{aligned}
    \int_\Omega \abs{\nabla u_n}^p \d x
    &= \sum_{m \in \N} \int_{D_{m,n}} \abs{\nabla u_n}^p \d x \\
    &= \sum_{m=0}^{m_0} \int_{D_{m,n}} \abs{\nabla u_n}^p \d x + \sum_{m > m_0} \int_{D_{m,n}} \abs{\nabla u_n}^p \d x \\
    &\leq C \abs{\Omega}^{\frac{2-p}{2}} (m_0+1) + C \sum_{m > m_0} m^{p^* \left(\frac{p-2}2\right)} \left(\int_{D_{m,n}} \abs{u_n}^{p^*} \right)^{\frac{2-p}{2}} \\
    &\leq C \abs{\Omega}^{\frac{2-p}{2}} (m_0+1) + C \left(\sum_{m > m_0} m^{p^* \left(\frac{p-2}p\right)}\right)^{\frac p2} \left(\sum_{m > m_0} \int_{D_{m,n}} \abs{u_n}^{p^*} \right)^{\frac{2-p}{2}} \\
    &\leq C \abs{\Omega}^{\frac{2-p}{2}} (m_0+1) + C \norm{u_n}_{L^{p^*}(\Omega)}^{p^*\left(\frac{2-p}{2}\right)} \left(\sum_{m > m_0} m^{p^* \left(\frac{p-2}p\right)}\right)^{\frac p2},
\end{aligned}
\end{equation}
In the second to last inequality, we used Hölder's inequality for infinite sums with exponents $2/(2-p)$ and $p/2$. 

Since the $u_n$ satisfy a weighted zero mean condition $\int_\Omega u_n\varrho\d x=0$, using the Poincaré inequality implies that $\norm{u_n}_{L^{p}(\Omega)} \leq C \norm{\nabla u_n}_{L^{p}(\Omega)}$, so it suffices to bound the latter quantity for the goal of Step 1.
% Usually, the Poincaré inequality is stated for functions satisfying $\int_\Omega u \d x = 0$. However, under our assumptions on $\varrho$, the classical proof (cf. \cite[Theorem 1 in chapter 5.8]{Evans:PartialDifferentialEquations2010}) shows that it applies in our case as well since $u \equiv const$ and $\int_\Omega u \varrho \d x=0$ together imply $u \equiv 0$

We choose $p^*$ as the Sobolev exponent, i.e., such that $\frac1{p^*} = \frac 1p - \frac 1d$, and use the embedding $W^{1,p} \hookrightarrow L^{p^*}$:
\begin{equation}\label{eq:un_Lpnorm_bound_by_sum}
\begin{aligned}
    \norm{u_n}_{L^{p^*}(\Omega)}^p 
    &\leq C \int_\Omega \abs{\nabla u_n}^p \d x \\
    &\leq C (m_0+1) + C \norm{u_n}_{L^{p^*}(\Omega)}^{p^*\left(\frac{2-p}{2}\right)} \left(\sum_{m > m_0} m^{p^* \left(\frac{p-2}p\right)}\right)^{\frac p2}.
\end{aligned}
\end{equation}
Since $1\leq p < d/(d-1) \leq 2 \leq d$, we conclude $p^*\tfrac{2-p}{2} \leq p$ and $p^*\tfrac{2-p}{p}>1$. In particular, the sum in \labelcref{eq:un_Lpnorm_bound_by_sum} is convergent.
So if $p^* \tfrac{2-p}{2}=p$, we can choose $m_0$ large enough and absorb the rightmost term into the left hand side.
Otherwise, an application of Young's inequality allows us to do the same, using that the sum on the right is convergent.
In both cases we obtain
\begin{equation*}
    \norm{u_n}_{L^{p^*}(\Omega)}^p \leq C.
\end{equation*}
In particular, by \labelcref{eq:un_Lpnorm_gradient_bounded_by_sum}, this implies that $\norm{\nabla u_n}_{L^p(\Omega)} \leq C$ where $C$ does not depend on $n$.

% If $d > 2$, then $p^* \left(\frac{2-p}{2}\right) < p$ and $p^* \left(\frac{2-p}{p}\right) > 1$. As before, the sum in \labelcref{eq:un_Lpnorm_bound_by_sum} converges. Choosing $m_0 = 0$, we see as before that 
% \begin{equation*}
%     \norm{u_n}_{L^{p^*}(\Omega)}^p \leq C.
% \end{equation*}
% Hence, by \labelcref{eq:un_Lpnorm_gradient_bounded_by_sum}, this implies that $\norm{\nabla u_n}_{L^p(\Omega)} \leq C$ where $C$ depends only on $\varrho$, $p$, $d$, $\Omega$, $\norm{f_n}_{L^1(\Omega)}$.

To summarize, we showed that there exists a constant 
$$C=C(\varrho, p, d, \Omega, \sup_{n \in \N} \norm{f_n}_{L^1(\Omega)}),$$ 
such that
\begin{equation*}
    \sup_{n \in \N} \norm{u_n}_{W^{1, p}(\Omega)} \leq C,
\end{equation*}
which concludes the first part of the proof.

\textit{Step 2:} The second part of the proof now follows from Step 1 by weak convergence of the $u_n$ up to a subsequence. 
From Step 1 we deduce that $u_{n_k} \rightharpoonup u$ weakly in $W^{1,p}(\Omega)$ for some subsequence which we relabel to be $u_n$. Fix any $\psi \in C^\infty(\overline\Omega)$. Then, since $C^\infty(\overline\Omega) \subseteq W^{1, q}(\Omega)$,
\begin{equation*}
    a_\psi(v) := \int_\Omega \varrho \, \nabla v \cdot \nabla \psi \d x
\end{equation*}
defines a linear functional on $W^{1, p}(\Omega)$, and hence 
\begin{equation}\label{eq:convergence_aun_to_au}
    a_\psi(u_n) \to a_\psi(u).
\end{equation}

On the other hand, since $u_n$ are weak solutions, it holds
\begin{equation*}
    a_\psi(u_n) = \int_\Omega \psi f_n \d x.%+\int_{\partial\Omega} \psi g_n \d\H^{d-1}.
\end{equation*}
By assumption, we know
\begin{equation}\label{eq:convergence_fn_f}
    \int_\Omega \psi f_n \d x 
    %+ \int_{\partial\Omega} \psi g_n \d\H^{d-1} 
    \to \int_{\overline\Omega} \psi \d f.
    %+\int_{\partial\Omega} \psi \d g.
\end{equation}
Combing \labelcref{eq:convergence_aun_to_au} and \labelcref{eq:convergence_fn_f} yields that $u$ solves the PDE \labelcref{eq:continuum_pde} with right hand side $f$ in the sense of distributions (cf. \cref{def:distr_sol_Poisson_eq}), and the weak convergence $u_n\rightharpoonup u$ together with boundedness of $\varrho$ also implies that $u$ satisfies the weighted zero mean condition $\int_\Omega u\varrho\d x =0$.
\end{proof}

\subsection{Properties of Green's functions}

The goal of the remainder of this section is to provide the auxiliary results regarding properties of Green's functions for the continuum PDE \labelcref{eq:continuum_pde}. To this end, we start by recalling the construction introduced in \cref{sec:Greensfunctions}: 
 for a fixed $y\in\Omega$, define 
\begin{equation*}
        f_n := \phi_n^y - \varrho,
\end{equation*}
    where $\supp(\phi)=B(0;1)$,   $\phi \in L^\infty(B(0;1))$ with $\phi \geq 0$ such that $\norm \phi _{L^1(B(0;1))}=1$ and $\phi_n^y(x) := n^{d}\phi\left(n\left(x-y\right)\right)$, defined for $n>\frac{1}{\dist(y,\partial\Omega)}$.
    By \cref{prop:existence_solution_continuum_pde}, there exist weak solutions  $G_n^y \in W^{1,2}(\Omega)$ that converge (up to a subsequence) to some function $G^y \in W^{1,p}\left(\Omega\right)$, which is a distributional solution to \labelcref{eq:continuum_pde}. In other words, we define (up to a subsequence),
    \begin{equation}\label{eq:construction_Gy}
        G^y := \lim_{n\to \infty} G_n^y \quad \text{ in }  W^{1,p}(\Omega)\,.
    \end{equation}
% \begin{remark}\label{eq:construction_Gy}
%     In this setting, \cref{prop:existence_solution_continuum_pde} gives existence of functions $G^y \in W^{1,p}(\Omega)$ as limits of functions $G_n^y \in W^{1,2}(\Omega)$, which are weak solutions to \labelcref{eq:continuum_pde} with right hand side 
%     \begin{equation*}
%         f_n := \phi_n^y - \varrho,
%     \end{equation*}
%     where  $\phi \in L^2(B(0;1))$ with $\phi \geq 0$ such that $\norm \phi _{L^1(B(0;1))}=1$ and $\phi_n^y(x) := n^{d}\phi\left(n\left(x-y\right)\right)$, defined for $n>\frac{1}{\dist(y,\partial\Omega)}$. 
%     We will show that $G^y \in W_{\loc}^{1,2}(\Omega \setminus \set{y})$, so $G^y$ is a Green's function as in \cref{def:green_function}.
% \end{remark}

We now prove some classical regularity of Green's functions outside their poles.

\begin{lemma}[Interior regularity of Green's functions]\label{thm:C1alpha_approximation_Gy}
    Let $\alpha \in (0,1)$. Let $\Omega$ be a bounded Lipschitz domain and $\varrho \in C^{0,\alpha}(\Omega)$ bounded above and below by positive constants, i.e., $0 < \varrho_{\min} \leq \varrho \leq \varrho_{\max}$. 
    For $y \in \Omega$ let $G^y \in W^{1,p}$, $p \in [1, d/d-1)$ be constructed as in \labelcref{eq:construction_Gy}. 
        
    Further, let $x_0 \in \Omega$ and $R > 0$ such that $\dist \left(x_0, \partial \Omega\right) > 5R$ and $\abs{x_0 - y} > 5R$. Then, there exists a constant $C = C(d, \varrho, R, \Omega, \alpha, \beta)>0$ which does not depend on $x_0$ and $y$, such that
    \begin{equation*}
        \norm{G^y}_{C^{1,\beta}\left(B(x_0;R)\right)} \leq C.
    \end{equation*}
    for all $\beta \in (0, \alpha)$. In particular, $G^y$ is continuous for any $x \in \Omega$, $x\neq y$.
\end{lemma}
\begin{proof}
We choose $n \in \N$ large enough such that $n>1/R$. If not indicated differently, all balls will be centered in $x_0$, and we will henceforth omit it from the notation.
Note that $G_n^y$ satisfies
\begin{equation}\label{eq:Gnyint}
    \div \varrho \nabla G_n^y = \varrho \qquad \text{in } B_{4R}.
\end{equation}
%
% First, assume only that $\varrho \in L^\infty(\Omega)$. Since $G_n^y \in W^{1,2}(\Omega)$, we can apply regularity theory for elliptic PDEs. By a classical result of De Giorgi and Nash (cf. \cite[Theorem 8.22]{Gilbarg:EllipticPartialDifferential2001}), we can estimate for all $R_0>0$ and $z \in B_{4R}$ such that $B_{R_0} \subset B_{4R}$ and $r \leq R_0$,
% \begin{align*}
%     \osc_{B(z;r)} G_n^y &:= \sup_{x\in B(z;r)} G_n^y(x)\, - \inf_{x\in B(z;r)} G_n^y(x) \\
%     &\leq C(d, \varrho, R_0, \Omega) \,r^\alpha \left(\left(R_0\right)^{-\alpha} \norm{G_n^y}_{L^\infty(B(z;R_0))} + \frac{\varrho_{\max}}{\varrho_{\min}}\right),
% \end{align*}
% where $\alpha=\alpha(d, \varrho, R_0)>0$. This result implies that \todo{Should I write the steps down? YES, at least some comment so those less familiar can follow, or a REF}
% \begin{equation*}
%     \left[G_n^y\right]_{C^{0, \alpha}(B_R)} \leq C(d, \varrho, R, \Omega) \left(\norm{G_n^y}_{L^\infty\left(B_{2R}\right)} + 1\right).
% \end{equation*}
% \todo{In the estimate above, I think Theorem 8.32 in Gilbarg/Trudinger can be used to get $C^{1,\alpha}$.}
By a classical result on interior regularity (cf. \cite[Theorem 8.32]{Gilbarg:EllipticPartialDifferential2001}), we obtain the following bound on the $C^{1,\alpha}(B_R)$ norm of $G_n^y$:
\begin{equation*}
    \norm{G_n^y}_{C^{1,\alpha}(B_R)} \leq C(R, d, \rho, \alpha) \left(\norm{G_n^y}_{L^\infty(B_{2R})} + \norm{\rho}_{L^\infty(\Omega)}\right).
\end{equation*}
By a theorem due to Moser (cf. \cite[Theorem 8.17]{Gilbarg:EllipticPartialDifferential2001}), we further estimate
\begin{equation}\label{eq:Gny_bound_Linfty}
    \norm{G_n^y}_{L^\infty\left(B_{2R}\right)} \leq C \left(R^{-\frac dp} \norm{G_n^y}_{L^p\left(B_{4R}\right)} + R \abs{\Omega}^{\frac1d} \norm\varrho_{L^\infty(\Omega)}  \right),
\end{equation}
where $C = C(d, \varrho, R)$ is a constant independent of $x_0$ and $y$. By \cref{lemma:Greens_functions_W1p_bound} it follows that for  any choice of $p\in [1,d/(d-1))$ the $L^p(\Omega)$ norm on the right hand side of \labelcref{eq:Gny_bound_Linfty} can be bounded uniformly in $y$ and $n$ and therefore $G_n^y$ is $C^1$-Hölder continuous with exponent $\alpha$ in $B(x_0;R)$ with a uniform bound on the $C^1$-Hölder norms.

Compact embeddings of Hölder spaces with smaller exponents imply that for any $\beta \in (0, \alpha)$, we can choose a subsequence, which we do not relabel, such that 
\begin{equation*}
    G_n^y \to \tilde G^y \qquad \text{in } C^{1, \beta}(B_R)
\end{equation*}
and
\begin{equation*}
    \norm{\tilde G^y}_{C^{1, \beta}(B_R)} \leq (2R)^{\alpha-\beta} \limsup_{n \to \infty} \norm{G_n^y}_{C^{1, \alpha}(B_R)}.
\end{equation*}
However, by $L^p$ convergence of the $G_n^y$, it follows that $G^y \equiv \tilde G^y$ in $B_R$. 
Continuity of $G^y$ at $x\neq y$ follows by taking $x_0 = x$ and $R$ small enough.
%
% \fh{[REMOVE]
% We now further assume that $\varrho \in C^{0,1}(\Omega)$. By classical interior estimates in elliptic regularity (c.f. \cite[Theorem 5.20]{Giaquinta:IntroductionRegularityTheory2012}), it follows that $G_n^y \in C^{1, \beta}(B_{4R})$ for some $\beta \in \left(\alpha,1\right)$. 
%
% Interior elliptic estimates (c.f. \cite[Theorem 8.32]{Gilbarg:EllipticPartialDifferential2001}) yield
% \begin{equation*}
%     \norm{G_n^y}_{C^{1,\beta}\left(B_R\right)} \leq C \left(\norm{G_n^y}_{L^\infty\left(B_{2R}\right)} + \norm\varrho_{L^\infty(\Omega)}\right),
% \end{equation*}
% where $C = C(d, \varrho, R)$ is a constant independent of $x_0$ and $y$. Therefore, using again \labelcref{eq:Gny_bound_Linfty},
% \begin{equation*}
%     \norm{G_n^y}_{C^{1,\beta}\left(B_R\right)} \leq C
% \end{equation*}
% for a constant $C$ independent of $x_0$, $y$ and $n$. In particular, by the compact embeddings of Hölder spaces $C^{1, \alpha}(B_R) \subset \subset C^{1, \beta}(B_R)$, it follows that for a subsequence which we do not relabel
% \begin{equation*}
%     G_n^y \to \tilde G^y \qquad \text{in } C^{1, \alpha}(B_R)
% \end{equation*}
% and
% \begin{equation*}
%     \norm{G^y}_{C^{1, \alpha}(B_R)} \leq (2R)^{\beta-\alpha} \limsup_{n \to \infty} \norm{G_n^y}_{C^{1, \beta}(B_R)}.
% \end{equation*}
%
% However, by $L^p$ convergence of the $G_n^y$, it follows that $G^y \equiv \tilde G^y$ in $B_R$, which concludes the proof.
% }
\end{proof}
In order to prove regularity of the Greens function close to the boundary, we employ the following technical lemma.
\begin{lemma}\label{lem:boundary_estimate}
    Let $\alpha \in (0,1)$. Let $\Omega$ be a bounded $C^{1,\alpha}$ domain and $\varrho \in C^{0,\alpha}(\Omega)$ bounded above and below by positive constants, i.e., $0 < \varrho_{\min} \leq \varrho \leq \varrho_{\max}$. 

    Let $\Omega''\subset\subset \Omega' \subset\subset \Omega$. Let $f\in L^\infty(\Omega)$ with $\int_\Omega f \d x = 0$, and let $v\in W^{1,2}(\Omega)$ be the weak solution of \labelcref{eq:continuum_pde}. Then there exist $C_1=C_1(\varrho, d, \Omega)$ and $C_2=C(\Omega,\Omega',\Omega'',\varrho)$ such that
    \begin{equation}\label{eq:boundary_estimate}
        \|v\|_{C^{1, \alpha}(\Omega\setminus \Omega')} \leq C_1\left( \|f\|_{L^\infty(\Omega\setminus\Omega'')} + \norm{v}_{L^1(\Omega \setminus \Omega'')} + C_2\|v\|_{C^{1}(\Omega'\setminus\Omega'')}\right).
    \end{equation}
\end{lemma}
\begin{proof}
    Let $\chi\in C^\infty(\Omega)$ be a smooth cutoff function satisfying $\chi \equiv 1$ on $\Omega\setminus \Omega'$, $\chi \equiv 0$ on $\Omega''$ and $0 \leq \chi \leq 1$ on $\Omega'\setminus \Omega''$. Let $w = \chi v\in H^1(\Omega)$ and $\phi \in H^1(\Omega)$. We use that $u$ is a weak solution of \labelcref{eq:continuum_pde} and $\frac{\partial \chi}{\partial \nu}=0$ on $\partial\Omega$ to compute
    \begin{align*}
        \int_\Omega \varrho \nabla w \cdot \nabla \phi \d x 
        &= \int_\Omega \varrho\chi \nabla v\cdot \nabla \phi + \varrho v \nabla \chi\cdot \nabla \phi \d x\\
        &=\int_\Omega \varrho\nabla v\cdot \nabla (\chi \phi) - \varrho \phi \nabla v\cdot \nabla \chi + \varrho v \nabla \chi\cdot \nabla \phi \d x\\
        %&=\int_\Omega \chi \phi f \d x - \int_\Omega \varrho \phi \nabla v\cdot \nabla \chi \d x- \int_\Omega \phi \div(\varrho v \nabla \chi)  \d x\\
        &=\int_\Omega \phi \left(\chi f - \varrho \nabla v\cdot \nabla \chi\right) + \varrho v \nabla \chi \cdot \nabla \phi \d x.% = \int_\Omega \phi g \d x,
    \end{align*}
    This shows that $w \in W^{1,2}(\Omega)$ is a weak solution to 
    \begin{equation}\label{eq:w_rhs}
        \begin{cases}
            -\div \varrho \nabla w = \chi f  - \varrho \nabla \chi \cdot \nabla v - \div(\varrho v \nabla \chi) & \text{in } \Omega, \\
            \frac{\partial w}{\partial \nu} = 0 &\text{on } \partial \Omega.
        \end{cases}
    \end{equation}
    where the divergence on the right hand side of \labelcref{eq:w_rhs} is to be understood in a weak sense.

    We first invoke \cite[Theorem 5.54]{lieberman2003oblique} to see that $v \in C^{1,\alpha}(\Omega)$.
    Second, we also invoke \cite[Theorem 5.54]{lieberman2003oblique} for $w$, which yields
    \begin{equation}\label{eq:wC1alpha_wg}
    \begin{aligned}
        \norm{w}_{C^{1,\alpha}(\Omega)} 
        &\leq C \left(\norm{w}_{L^\infty(\Omega)} + \norm{f \chi - \varrho \nabla v \cdot \nabla \chi}_{L^\infty(\Omega)} + \norm{\varrho v \nabla \chi}_{C^{0,\alpha}(\Omega)}\right) \\
        &\leq C \Big(\norm{w}_{L^\infty(\Omega)} + \norm{f}_{L^\infty(\Omega \setminus \Omega'')} + \norm{\varrho \nabla \chi}_{L^\infty(\Omega)} \norm{\nabla v}_{L^\infty(\Omega' \setminus \Omega'')} \\
        &\qquad
        + \norm{\varrho \nabla \chi}_{C^{0,\alpha}(\Omega)} \norm{v}_{C^{0,1}(\Omega' \setminus \Omega'')}\Big)
    \end{aligned}
    \end{equation}
    Here, $C$ depends only on $d, \Omega$ and $\varrho$.
    One can easily adapt \cref{lem:v_as_Greens_convolution} to show 
    % \todo{check
    % \blue 
    % L: should be ok since $\rho v\nabla\chi$ are smooth and $G_n^x\rightharpoonup G^x$ weakly in $W^{1,p}$}
    \begin{equation*}
        w(x) - \int_\Omega \varrho w \d y = \int_{\Omega} G^x \left(f \chi - \varrho \nabla v \cdot \nabla \chi\right) + \varrho v \nabla \chi \cdot \nabla_y G^x \d y.
    \end{equation*}
    Hence, using \cref{lemma:Greens_functions_W1p_bound}, we see
    \begin{equation*}
    \begin{aligned}
        \norm{w}_{L^\infty(\Omega)} 
        &\leq \sup_{x\in \Omega}\norm{G^x}_{L^1} \norm{f \chi - \varrho \nabla v \cdot \nabla \chi}_{L^\infty}
        + \sup_{x\in \Omega}\norm{\nabla G^x}_{L^1} \norm{\varrho v \nabla \chi}_{L^\infty} + \norm{\varrho}_{L^\infty}\norm{w}_{L^1} \\
        &\leq C_1 \left(\norm{f}_{L^\infty(\Omega \setminus \Omega'')} + \norm{w}_{L^1(\Omega)} + C_2 \norm{v}_{C^{1}(\Omega' \setminus \Omega'')}\right)
    \end{aligned}
    \end{equation*}
    with $C_1 = C_1(\Omega, \varrho, d) > 0$ and $C_2 = C_2(\Omega, \Omega', \Omega'', \varrho) > 0$.
    Combining these estimates with \labelcref{eq:wC1alpha_wg} and using that $\norm{v}_{C^{0,1}(\Omega'\setminus\Omega'')} \leq C_2 \norm{\nabla v}_{L^\infty(\Omega'\setminus\Omega'')}$, we obtain
    \begin{equation*}
        \norm{w}_{C^{1, \alpha}(\Omega)} 
        \leq C_1 \left(\norm{f}_{L^\infty(\Omega \setminus \Omega'')} + \norm{v}_{L^1(\Omega \setminus \Omega'')} + C_2 \norm{v}_{C^1(\Omega' \setminus \Omega'')}\right).
    \end{equation*}
    Since $\|v\|_{C^{1, \alpha}(\Omega\setminus \Omega')}\leq \|w\|_{C^{1, \alpha}(\Omega)}$, this concludes the proof. 
\end{proof}
\begin{lemma}[Boundary regularity of Green's functions]\label{thm:Gy_boundary_regularity}
    Let $y \in \Omega$ and define $G^y \in W^{1,p}(\Omega)$ as in \labelcref{eq:construction_Gy}. Moreover, assume that $\Omega$ is a $C^{1,\alpha}$ domain, and that $\rho \in C^{0,\alpha}(\Omega)$ for some $\alpha \in (0,1)$. 
    Further, let $0 < R < \frac13 \dist(y, \partial \Omega)$. Then,
    \begin{equation*}
    \norm{G^y}_{C^{1, \beta}\left(\Omega \setminus B(y;R)\right)} \leq C.
    \end{equation*}
    for all $\beta \in (0,\alpha)$, where $C = C(d, \varrho, R, \Omega, \alpha, \beta)>0$ is a constant which does not depend on $y$.
\end{lemma}
\begin{proof}
    % Set $r := R/5$ and consider the sets $D_R := \set{x \in \Omega \colon \abs{x-y} > R \text{ and } \dist(x, \partial \Omega) > R}$ and $\partial_R \Omega := \set{x \in \Omega \colon \dist(x, \partial \Omega) < R}$.
    %
    % By \cref{thm:C1alpha_approximation_Gy}, for every $\beta \in (0, \alpha)$, there exists a constant $C = C(d, \varrho, r, \Omega, \alpha, \beta) > 0$ such that for all $x \in D_R$ we have that $\norm{G^y}_{C^{1, \beta}(B(x; r))} \leq C$. In particular, this implies
    % \begin{equation*}
    %     \norm{G^y}_{C^{1, \beta}(D_R)} \leq C
    % \end{equation*}
    % by a covering argument. \todo{check this}
    %
    Let $G_n^y \in W^{1,2}(\Omega)$ be the sequence used to construct $G^y$. By \cref{lem:boundary_estimate} we have %We show boundary regularity by applying \cref{lem:boundary_estimate} to $G_n^y$. In particular, we obtain that
    % \begin{equation*}
    %     \norm{G_n^y}_{C^{1, \alpha}(\partial_R \Omega)} \leq C \left(\norm{\phi_n^y - \rho}_{L^\infty(\partial_{2R} \Omega)} + \norm{G_n^y}_{L^1(\partial_{2R} \Omega)} + C \norm{G_n^y}_{C^{1}(\partial_{2R} \Omega \setminus \partial_{R} \Omega)}\right).
    % \end{equation*}
    \begin{multline*}
        \norm{G_n^y}_{C^{1,\alpha}(\Omega \setminus B(y;R))}
        \leq C(\Omega, d, \varrho) \biggl( \norm{\phi_n^y - \varrho}_{L^\infty(\Omega \setminus B(y;R/2))} + \\ + \norm{G_n^y}_{L^1(\Omega \setminus B(y;R/2))} + C(\varrho, R) \norm{G_n^y}_{C^{1}(B(y;R) \setminus B(y;R/2))}\biggr)
    \end{multline*}
    For $n > 3/R$ the right hand side can be estimated independently of $n$. For the first term we use that $\phi_n$ has compact support in $B(y;1/n)$. For the second term the uniform $W^{1,p}(\Omega)$ bound on $G_n^y$ from \cref{lemma:Greens_functions_W1p_bound}. Finally, the third term can be controlled using interior regularity results on $G_n^y$, e.g. by using \cref{thm:C1alpha_approximation_Gy}:
    \begin{equation*}
        \norm{G_n^y}_{C^{1}(B(y;R) \setminus B(y;R/2))} \leq C(R, d, \varrho, \alpha).
    \end{equation*}
    Therefore,
    \begin{equation*}
        \norm{G_n^y}_{C^{1, \alpha}(\Omega \setminus B(y;R))} \leq C(d, \varrho, R, \Omega, \alpha).
    \end{equation*}
    As in the proof of \cref{thm:C1alpha_approximation_Gy} we use the compact embedding of Hölder spaces and the convergence $G_n^y \to G^y$ in $L^p(\Omega)$, to obtain
    \begin{equation*}
        \norm{G^y}_{C^{1, \beta}(\Omega \setminus B(y;R))} \leq C(d, \varrho, R, \Omega, \alpha, \beta),
    \end{equation*}
    concluding the proof.
\end{proof}

Next, we show that Green's function $G^y(x)$ and its gradient are symmetric when exchanging the roles of $x$ and $y$.
\begin{lemma}\label{thm:Gy_symmetric}
Let $x,y \in \Omega$. 
Then, $G^x(y) = G^y(x)$ and $\nabla_y G^x(y) = \nabla_y G^y(x)$ for almost every $x,y\in\Omega$.    
\end{lemma}
\begin{proof}
We compute for $x \neq y$
    \begin{align}
        G^y(x) 
            &= \lim_{n \to \infty} \int_\Omega G^y(z) \left(\phi_n^x(z) - \varrho(z)\right) \d z \label{eq:Gy_symmetry_1} \\
            &= \lim_{n \to \infty} \lim_{m \to \infty}\int_\Omega G_m^y(z) \left(\phi_n^x(z) - \varrho(z)\right) \d z \label{eq:Gy_symmetry_2} \\
            &= \lim_{n \to \infty} \lim_{m \to \infty}\int_\Omega \varrho(z) \nabla G_m^y(z)\cdot\nabla G_n^x(z) \d z \label{eq:Gy_symmetry_3} \\
            &= \lim_{n \to \infty} \lim_{m \to \infty}\int_\Omega \left(\phi_m^y(z) - \varrho(z)\right) G_n^x(z) \d z \label{eq:Gy_symmetry_4} \\
            &= \lim_{n \to \infty} G_n^x(y) \label{eq:Gy_symmetry_5} \\
            &= G^x(y). \label{eq:Gy_symmetry_6}
    \end{align}
Equalities \labelcref{eq:Gy_symmetry_1} and \labelcref{eq:Gy_symmetry_5} are due to the continuity of $G^y$ and $G_n^x$ respectively away from their poles (due to \cref{thm:C1alpha_approximation_Gy}) and their zero mean condition. Equality \labelcref{eq:Gy_symmetry_2} follows from the $L^p(\Omega)$ convergence of $G_m^y$, while \labelcref{eq:Gy_symmetry_6} follows from the $C_{\loc}^{0,\alpha}(\Omega)$ convergence of $G_n^x$ established in the proof of \cref{thm:C1alpha_approximation_Gy}. Equalities \labelcref{eq:Gy_symmetry_3} and \labelcref{eq:Gy_symmetry_4} is the weak formulation of the PDEs solved by the two Green's functions.

To prove the symmetry of the gradient we let $\phi\in C^\infty_c(\Omega;\R^d)$ and compute, using the symmetry $G^y(x)=G^x(y)$ that we just proved: for any $x\in \Omega$,
\begin{align*}
    \int_\Omega \nabla_y G^x(y) \cdot \phi(y) \d y
    &=
    -\int_\Omega G^x(y) \div_y\phi(y) \d y
    \\
    &=
    -\int_\Omega G^y(x) \div_y\phi(y) \d y
    \\
    &=
    \int_\Omega \nabla_y G^y(x) \cdot \phi(y) \d y.
\end{align*}
Since $\phi$ was arbitrary, this shows $\nabla_y G^y(x)=\nabla_y G^x(y)$ almost everywhere.
\end{proof}

\section{Proofs from \texorpdfstring{\cref{sec:continuum_limits}}{Section 3}}\label{sec:stability}

\subsection{Transportation maps}\label{app:transportation}

\begin{proposition}\label{prop:partition}
Suppose $\Omega\subset \R^d$ is open and bounded with a Lipschitz boundary. Then there exists a constant $C>0$, depending only on $\Omega$, such that for every $h>0$ there exists a partition $B_1,\dots,B_M$ of $\Omega$ satisfying the following.
\begin{enumerate}[label=(\roman*)]
\item $M \leq C|\Omega|h^{-d}$,
\item $|B_i| \geq C^{-1}h^d$, and
\item Each $B_i$ is contained in a ball of radius $h$. 
\end{enumerate}
\end{proposition}
\begin{proof}
Let $X=\{x_1,\dots,x_M\}\subset \Omega$ be an $h$-net on $\Omega$, which means that 
\begin{equation}\label{eq:covering}
\Omega \subset \cup_{i=1}^M B(x_i,h),
\end{equation}
and the family of sets $\{B(x_i,\tfrac{h}{2}\}_{i=1}^M$ are disjoint.  Therefore
\[|\Omega| \geq \sum_{i=1}^M\left|\Omega\cap B\left(x_i,\frac{h}{2}\right)\right|.\]
Since the boundary of $\Omega$ is Lipschitz, then there exists a constant $C>0$ depending on $\Omega$ such that 
\[\left|\Omega\cap B\left(x,\frac{h}{2}\right)\right| \geq C_1^{-1}h^d\]
for all $x\in \Omega$. Therefore
\[|\Omega| \geq \sum_{i=1}^MC_1^{-1}h^d = C_1^{-1}h^d M,\]
and so $M \leq C|\Omega|h^{-d}$. We now define the Voronoi cells
\begin{equation}\label{eq:voronoi}
B_i = \{x\in \Omega\, : \, |x-x_i| \leq |x-x_j| \text{ for } j >i \text{ and } |x-x_i| < |x-x_j| \text{ if } j < i \}.
\end{equation}
Then the sets $B_1,\dots,B_M$ form a partition of $\Omega$. Furthermore, since the sets $\{B(x_i,\frac{h}{2})\}_{i=1}^{M}$ are disjoint, we have that $\Omega\cap B(x_i,\frac{h}{2}) \subset B_i$ and so 
\begin{equation}\label{eq:lower_measure}
|B_i| \geq \left|\Omega\cap B\left(x_i,\frac{h}{2}\right)\right| \geq C_1^{-1}h^d.
\end{equation}
This establishes (i) and (ii).

To prove (iii), note that by \labelcref{eq:covering} we have that whenever $x\not\in B(x_i,h)$ we must have $x\in B(x_j,h)$ for some $j\neq i$. Therefore, if $|x-x_i|> h$ then $|x-x_j|<h < |x-x_i|$ for some $j\neq i$. Therefore $x\not\in B_i$. This implies that $B_i\subset B(x_i,h)$, and completes the proof.
\end{proof}

We now give the proof of \cref{thm:transportation}.
\begin{proof}[Proof of \cref{thm:transportation}]
By  \cref{prop:partition}, there exists a constant $C > 0$ depending only on $\Omega$ such that for each $h > 0$, there is a partition $B_{1}, B_{2},\dots, B_{M}$ of $\Omega$ for which $M \leq C|\Omega|h^{-d}$, $|B_i|\geq C^{-1}h^d$, and each $B_{i}$ is contained in a ball of radius $h$. Let $\delta > 0$ and set $h = \frac{\delta}{2}$. Let $\rho_{\delta}$ be the histogram density estimator
\begin{equation}
    \rho_{\delta} = \frac{1}{n}\sum^{M}_{i=1}\frac{n_{i}}{|B_{i}|} \mathds{1}_{B_{i(x)}},
\end{equation}
where $n_{i}$ is the number of points from $x_1,\dots,x_n$ that fall in $B_{i}$. We easily check that
\begin{equation}
    \int_{\Omega}\rho_{\delta}(x) \d x = \frac{1}{n}\sum^{M}_{i=1}n_{i} = 1.
\end{equation}

We now prove property (iv). Note that each $n_{i}$ is a Bernoulli random variable with parameter $p_{i} = \int_{B_{i}}\rho(x)\d x$. By the Chernoff bounds, we have
\begin{equation}
    \mathds{P}(|n_{i}-np_{i}|\geq \lambda n p_{i}) \leq 2 \exp{\left(-\frac{3}{8}n p_{i} \lambda^{2}\right)}
\end{equation}
for $0 < \lambda \leq 1$. We note that 
\begin{equation}
    p_{i} = \int_{B_{i}}\rho(x) \, \d x \geq \rho_{\min}|B_i| \geq C^{-1}\rho_{\min}h^{d} = C^{-1}2^{-d}\rho_{\min}\delta^d,
\end{equation}
and similarly $p_i \leq \rho_{\max}|B_i|$. Union bounding over $i= 1,\dots,M$ and using that $M \leq 2^dC|\Omega|\delta^{-d}\leq 2^dC|\Omega|n$ since $n\delta^d\geq 1$, we have that for any $0 < \lambda \leq 1$ with probability at least 
\begin{equation}\label{eq:prob}
1 - 2^{d+1}C|\Omega|n\exp{\left(-\frac{3}{8}C^{-1}2^{-d}\rho_{\min}n\delta^{d}\lambda\right)}
\end{equation}
it holds that
\begin{equation} \label{eq:inequality-rho}
\bigg|\frac{n_i}{n}-\int_{B_i}\rho(x)\,\d x\bigg| \leq \rho_{\max}|B_{i}|\lambda
\end{equation}
for all $i= 1,\dots,M$. For the rest of the proof we assume this event holds.

Let $x \in \Omega$. Then $x\in B_{i}$ for some $i$ and using \labelcref{eq:inequality-rho} we have
\begin{flalign*}
|\rho(x) - \rho_{\delta}(x)| &= \left| \rho(x) - \frac{n_{i}}{n|B_i|}\right|\\
&= \left| \rho(x) - \frac{1}{|B_i|}\int_{B_i}\rho(y) \d y+ \frac{1}{|B_i|}\int_{\Omega_i}\rho(y) \d y - \frac{n_{i}}{n|B_i|}\right|\\
&\leq \left| \rho(x) - \tfrac{1}{|B_i|}\int_{B_i}\rho(y) \d y \right| + \left|\frac{1}{|B_i|}\int_{B_i}\rho(y) \d y - \frac{n_{i}}{n|B_i|}\right|\\
&\leq \frac{1}{|B_i|}\int_{B_i}|\rho(x) - \rho(y)|\d y+ \rho_{\max}\lambda\\
&\leq \Lip(\rho)\delta + \rho_{\max}\lambda,
\end{flalign*}
which establishes property (iv). We now assume that $\delta \leq \frac{\rho_{\min}}{8\Lip(\rho)}$ and $\lambda \leq\frac{\rho_{\min}}{8\rho_{\max}}$, so that $\rho_{\delta}\geq 3\frac{\rho_{\min}}{4} > 0$. This means, in particular, that each $B_i$ is nonempty, so $n_i \geq 1$ for all $i$. 

We now construct a partition $\Omega_1,\dots,\Omega_n$ of $\Omega$ satisfying
\begin{equation}\label{eq:definition-rho-delta}
\Omega_{i} \subset B(x_i, \delta), \ \ \text{and} \ \  \int_{\Omega_{i}} \rho_{\delta} \, dx = \frac{1}{n} \quad\text{ for all } 1 \leq i \leq n.
\end{equation}
To do this, we construct a partition $B_{i,1},B_{i,2},\dots,B_{i, n_{i}}$ of each $B_{i}$ consisting of sets of equal measure. Let $k_{i,1},k_{i,2},\dots,k_{i,n_{i}},$ denote the indices of
the random variables that fall in $B_{i}$ and set $\Omega_{k_{i},j} = B_{i,j}$. That is, the partition of $B_{i,j}$ is assigned one-to-one with the random samples that fall in $B_{i}$. We can also arrange that $x_{k_{i,j}}\in B_{i,j}$. Then for some $1\leq j \leq M$ and $1 \leq k \leq n_{j}$ we have 
\begin{equation}
    \int_{\Omega_{i}}\rho_{\delta}\d x =\frac{n_{j}}{n|B_{j}|}\int_{B_{j,k}}\d x = \frac{1}{n}.
\end{equation}
Fix $1\leq j\leq n$ and fix $i$ such that $x_j \in B_i$. Since $B_{i}$ is contained in a ball of radius $\frac{\delta}{2}$ we have $B(x_j, \delta) \supset B_i \supset \Omega_j$. This proves \labelcref{eq:definition-rho-delta}.

Finally, we define the map $T_\delta:\Omega\to \X_n$ by setting $T_\delta(x)=x_i$ for all $x\in \Omega_i$. 
Since the sets $\Omega_i$ are measurable, so is the map $T$.
Since $x_i\in \Omega_i$ we have $T_\delta(x_i)=x_i$.  Since $T_\delta^{-1}(\{x_i\}) = \Omega_i$,  \labelcref{eq:definition-rho-delta} implies that ${T_\delta}_\#(\rho_\delta \d x) = \mu_n$. Since $\Omega_{i} \subset B(x_i, \delta)$ we also have $|T_\delta(x) - x|\leq \delta$ for all $x\in \Omega$, which completes the proof.
\end{proof}

\subsection{Mollification}\label{app:mollification}

\begin{proof}[Proof of \cref{lem:sigma_eta_bound}]
The right inequality in \labelcref{eq:lip} is trivial, so we focus on the left inequality. We change to polar coordinates to write
\[\sigma_{\eta,t} = \frac{1}{d}\int_{\R^{d}}\eta(|z| + 2t) |z|^{2}\d z =\omega_d\int_{0}^\infty r^{d+1}\eta(r + 2t) \d r.\]
We now make the change of variables $s=r+2t$ to obtain
\[\sigma_{\eta,t} =\omega_d\int_{2t}^\infty s^{d+1}\left(1-\frac{2t}{s}\right)^{d+1}\eta(s) \d s\geq \omega_d\int_{2t}^\infty s^{d+1}\left(1-\frac{2(d+t)t}{s}\right)\eta(s) \d s,\]
where we used that $(1-x)^{d+1}\geq 1-(d+1)x$ for $x\in [0,1]$ in the last inequality. Therefore we have
\begin{align*}
\sigma_{\eta,t} &\geq \omega_d\int_{2t}^\infty s^{d+1}\eta(s) \d s -2(d+1)\omega_dt\int_{2t}^\infty s^{d}\eta(s) \d s\\
&= \sigma_{\eta,0} - \omega_d\int_{0}^{2t} s^{d+1}\eta(s) \d s -2(d+1)\omega_dt\int_{2t}^\infty s^{d}\eta(s) \d s\\
&\geq \sigma_{\eta,0} - 2\omega_dt\int_{0}^{2t} s^{d}\eta(s) \d s -2(d+1)\omega_dt\int_{2t}^\infty s^{d}\eta(s) \d s\\
&\geq \sigma_{\eta,0} - 2(d+1)t\int_{0}^\infty \omega_ds^{d}\eta(s) \d s\\
&\geq \sigma_{\eta,0} - 4t\int_{0}^\infty d \omega_ds^{d}\eta(s) \d s\\
&= \sigma_{\eta,0} - 4t\int_{\R^d} \eta(|z|)|z| \d z.
\end{align*}
This completes the proof.
\end{proof}

\begin{proof}[Proof of  \cref{lem:psi_moll}]
We first prove (i). Since  $\eta$ is continuous and positive at $0$, there  exists $r_0>0$ such that $\eta(t) \geq \eta(0)/2$ for $0 \leq t \leq r_0$. Assume that $\delta/\epsilon \leq r_0/4$. Then  by \cref{lem:sigma_eta_bound}, for any $|x|\leq \frac{r_0}{4}$ we have
\[\psi_{1,\delta/\eps}(x) = \frac{1}{\sigma_{\eta,\delta/\eps}}\int_{|x|}^\infty \eta\left(s + 2\frac{\delta}{\epsilon}\right) s \d s \geq \frac{\eta(0)}{2\sigma_{\eta}}\int_{|x|}^{r_0-2\frac{\delta}{\epsilon}}  s \d s \geq \frac{\eta(0)}{2\sigma_{\eta}}\int_{\frac{r_0}{4}}^{\frac{r_0}{2}} s \d s = \frac{3\eta(0)r_0^2}{16\sigma_{\eta}}. \]
It follows that
\[\ted(x) = \int_{\Omega} \psi_{\epsilon, \delta}(x-y) \d y=\frac{1}{\epsilon^d}\int_{B(x,\epsilon-2\delta)\cap \Omega}\psi_{1,\delta/\eps}\left( \frac{x-y}{\epsilon}\right)\d y\geq \frac{3\eta(0)r_0^2}{16\sigma_{\eta}\epsilon^d}\left|B\left(x,\frac{\epsilon r_0}{4}\right)\cap \Omega\right|.\]
Since the boundary $\partial\Omega$ is Lipschitz continuous, there exists $C_\Omega>0$ such that  $\left|B\left(x,r\right)\cap \Omega\right| \geq C_\Omega r^d$ for $r>0$ suffciently small (in fact, $r<1$ is sufficient). It follows that
\[\ted(x) \geq \frac{3C_\Omega\eta(0)r_0^{d+2}}{16^2\sigma_{\eta}},\]
which establishes the lower bound in (i).

To prove (ii) (and the upper bound in (i)), note that since $\eta_\eps(s + 2\delta) = 0$ for $s\geq \epsilon-2\delta$ we have $\psi_{\eps,\delta}(x-y) = 0$ if $|x-y| \geq \epsilon-2\delta$. Since $\dist(x,\partial\Omega)\geq \epsilon-2\delta$, the support of $y\mapsto \psi_{\epsilon,\delta}(x-y)$ is contained in $\bar{\Omega}$, and therefore
\begin{align*}
\int_{\Omega} \psi_{\epsilon, \delta}(x-y) \d y &= \int_{B(x,\epsilon-2\delta)}\psi_{\epsilon,\delta}(x-y) \d y\\
&=\int_{B(0,1)} \psi_{1, \delta/\epsilon}(z) \d z \\
&= \int_{0}^{1} d\omega_d \tau^{d-1}\psi_{1, \delta/\epsilon}(\tau e_1) \d \tau \\
&= \frac{1}{\sigma_{\eta, \delta/\eps}}\int_{0}^{1} d\omega_d \tau^{d-1} \int_{\tau}^1 \eta\left(s + \tfrac{2\delta}{\epsilon}\right) s \d s\d \tau\\
&= \frac{1}{\sigma_{\eta, t}}\int_{0}^{1} \omega_d  \eta\left(s + \tfrac{2\delta}{\epsilon}\right) s \left(\int_{0}^{s} d \tau^{d-1}\d \tau\right)\d s\\
&= \frac{1}{\sigma_{\eta, t}}\int_{0}^{1} \omega_d s^{d+1}  \eta\left(s + \tfrac{2\delta}{\epsilon}\right) \d s\\
&= \frac{1}{d\sigma_{\eta, t}}\int_{\R^d} \eta\left(|z| + \tfrac{2\delta}{\epsilon}\right) |z|^{2}\d z
=  1,
\end{align*}
where we used \labelcref{eq:sigma_alt} in the last equality. The upper bound of $\theta_{\eps,\delta}(x)\leq 1$ follows from the argument above, and the observation that
\[\int_{\Omega} \psi_{\epsilon, \delta}(x-y) \d y\leq \int_{B(x,\epsilon-2\delta)}\psi_{\epsilon,\delta}(x-y) \d y\]
holds for all $x\in \Omega$. 
\end{proof}

\begin{proof}[Proof of \cref{prop:smoothing_l1}]
	The case of $p=\infty$ is immediate. Let $1 \leq p< \infty$. By \cref{lem:psi_moll} (i)  Jensen's inequality we have
	\begin{align*}
		\|\Lambda_{\epsilon, \delta}u(x)\|^p_{L^p(\Omega')}&=\int_{\Omega'}|\Lambda_{\epsilon, \delta}u(x)|^p \d x\\
		&=\int_{\Omega'}\left|\frac{1}{\ted(x)}\int_{\Omega}\psi_{\eps,\delta}(x-y)|u(y)|\d y\right|^p \d x\\
		&=\int_{\Omega'}\frac{1}{\ted(x)}\int_{\Omega}\psi_{\eps,\delta}(x-y)|u(y)|^p\d y \d x\\
		&\leq C\int_{\Omega'}\int_{(\Omega'+B_\eps)\cap \Omega}\psi_{\eps,\delta}(x-y)|u(y)|^p\d y \d x\\
		&= C\int_{(\Omega'+B_\eps)\cap \Omega}\int_{\Omega'}\psi_{\eps,\delta}(x-y)|u(y)|^p\d x \d y\\
		&\leq C\int_{(\Omega'+B_\eps)\cap \Omega}|u(y)|^p\d y\\
		&=C\|u\|^p_{L^p((\Omega'+B_\eps)\cap \Omega)},
	\end{align*}
	which completes the proof.
\end{proof}

\section{Proofs from \texorpdfstring{\cref{sec:heat_kernel}}{Section 4}}\label{app:rep_avg}

This section contains many of the proofs from \cref{sec:heat_kernel}. Some of the proofs follow \cite{lawler2010random} rather closely. 

\begin{proof}[Proof of \cref{lem:Hoeffding}]
By \labelcref{eq:psi_rewrite}, we may prove the result for $\epsilon=1$. Fix any $b\in \R^d$ with $|b|=1$. For any two functions $f,g:\R^d\to \R$ with compact support on $\R^d$ we have
\begin{align*}
\int_{\R^d}e^{b\cdot x}(f*g)(x) \d x &=\int_{\R^d} \int_{\R^d}e^{b\cdot x} f(x-y)g(y) \d x \d y\\
&=\int_{\R^d} e^{b\cdot y}g(y)\left(\int_{\R^d}e^{b\cdot (x-y)} f(x-y) \d x \right) \d y\\
&=\int_{\R^d} e^{b\cdot y}g(y)\left(\int_{\R^d}e^{b\cdot z} f(z) \d z \right) \d y\\
&=\left(\int_{\R^d} e^{b\cdot x}g(x)\d x\right)\left(\int_{\R^d}e^{b\cdot x} f(x) \d x \right).
\end{align*}
Therefore
\[\int_{\R^d} e^{b\cdot x}\psi_k(x) \d x= \int_{\R^d} e^{b\cdot x}(\eta * \cdots * \eta)(x) \d x= \left(\int_{\R^d} e^{b\cdot x}\eta(x) \d x\right)^k.\]
Thus for any $s>0$ we have
\begin{align*}
\int_{\{b\cdot x > t\}} \psi_k(x) \d x &= \int_{\{e^{sb\cdot x} > e^{st}\}} \psi_k(x) \d x\\
&\leq e^{-st}\int_{\R^d}e^{sb\cdot x}\psi_k(x)\d x=e^{-st}\left(\int_{B_1} e^{s b\cdot x}\eta(x) \d x\right)^k.
\end{align*}
By convexity, we have
\[e^{s \tau} \leq e^{-s} + \frac{1}{2}\left( e^s - e^{-s}\right)(\tau + 1)\]
for any $\tau \in [-1,1]$. Since $|b \cdot x| \leq |b| |x| \leq 1$ for any $x\in B_1:=B(0,1)$, we have
\begin{align*}
\int_{B_1} e^{s b\cdot x}\eta(x) \d x &\leq \int_{B_1}\left(e^{-s} + \frac{1}{2}\left( e^s - e^{-s}\right)(b\cdot x + 1) \right)\eta(x)\d x\\
&=\frac{1}{2}(e^s + e^{-s})\int_{B_1}\eta(x)\d x = \frac{1}{2}(e^s + e^{-s}) \leq e^{\frac{s^2}{2}}.
\end{align*}
Therefore
\[\int_{\{b\cdot x > t\}} \psi_k(x) \d x \leq \exp\left(-st + \frac{ks^2}{2}\right).\]
Optimizing over $s>0$ we choose $s=\frac{t}{k}$ and find that
\[\int_{\{b\cdot x > t\}} \psi_k(x) \d x \leq \exp\left(-\frac{t^2}{2k}\right).\]
Choosing $b=\pm e_i$ for $i=1,\dots, d$ yields
\[\int_{\{|x|_1 > t\}} \psi_k(x) \d x \leq 2d\exp\left(-\frac{t^2}{2k}\right),\]
where $|x|_1=\sum_{i=1}^d|x_i|$ is the 1-norm. The proof is completed by using that $|x| \leq  \sqrt{d}|x|_1$.
\end{proof}

The rest of the proofs in this section require the Fourier transform of the kernel $\eta$, defined by
\begin{align*}
\hat{\eta}(y) = \int_{\R^d} \eta(|z|) e^{-2\pi i z \cdot y} \d z.
\end{align*}
Since $\eta$ is radially symmetric, the Fourier transform $\hat{\eta}$ is real-valued. Also note that $|\hat{\eta}(y)|\le 1$ for all $y\in\R^d$. Using the Fourier convolution property and inversion formula, we may write
\begin{equation}\label{eq:psik_fourier}
\psi_k(x) = \int_{\R^d}e^{2\pi i y\cdot x} \hat{\eta}(y)^k \d y.
\end{equation}
We also note that if 
\begin{equation}\label{eq:eta_moment}
\int_{\R^d}|y| |\hat{\eta}(y)|^k \d y < \infty,
\end{equation}
then $\psi_k$ is continuously differentiable and 
\begin{equation}\label{eq:psik_grad}
\nabla \psi_k(x) = 2\pi i \int_{\R^d}e^{2\pi i y\cdot x} y\, \hat{\eta}(y)^k \d y.
\end{equation}

Our first lemma bounds $\hat{\eta}$ near the origin.
\begin{lemma}\label{lem:near_origin}
For $|y|\ll 1$ we have
\begin{align*}
\hat{\eta}(y) = \exp\left(-2\pi^2 \sigma_\eta |y|^2 + \O(|y|^4) \right).
\end{align*}
\end{lemma}
\begin{proof}
By Taylor expansion, we have
\begin{align*}
\hat{\eta}(y) &= \int_{B_1} \eta(|z|) \left( 1-2\pi i z \cdot y- 2\pi^2 (z \cdot y)^2 + \frac{4}{3} \pi^3 i (z \cdot y)^3 + \O(| z \cdot y|^4)\right) \d z  \\
% &= 1 - 2\pi^2 \sum_{j=1}^d y_j^2 \int_{\R^d}\eta(|z|) z_j^2 \d z  + \O(|y|^4) \notag\\ 
&= 1 - 2\pi^2 \sigma_\eta  |y|^2  + \O(|y|^4), \notag
\end{align*}
where we note that the first and third order integrals vanish due to symmetry, and where we recall the identity for $\sigma_\eta$ in \labelcref{eq:sigma_eta_identity}.  For $|y|\ll 1$ we have
\[\log \hat{\eta}(y) = \log\left(1-  2\pi^2 \sigma_\eta  |y|^2  + \O(|y|^4)\right) = -2\pi^2 \sigma_\eta |y|^2 + \O(|y|^4).\]
Exponentiating both sides completes the proof.
\end{proof}

We now use \cref{lem:near_origin} to bound $\hat{\eta}$ on larger balls $B_R$. 
\begin{lemma}\label{lem:away_origin} 
For each $R > 0$, there exists a constant $b_R > 0$ such that 
\begin{align*}
|\hat{\eta}(y)| \leq 1 - b_R |y|^2 \ \ \text{for } y\in B_R.
\end{align*}
\end{lemma}
\begin{proof}
We split the proof into two steps.

\textbf{Step 1:} We first show that $|\hat{\eta}(y)| < 1$ for $y\neq 0$. To prove this, we assume, by way of contradiction, that $|\hat{\eta}(y)| = 1$ for some $y\neq 0$. Without loss of generality, we may assume $\hat{\eta}(y)=1$. Then we have
\[1 = \int_{\R^d} \eta(|z|) e^{-2\pi i y\cdot z}\d z = \int_{\R^d}\eta(|z|) \cos(2\pi y \cdot z)\d z,\]
and therefore
\[\int_{\R^d} \eta(|z|) (1-\cos(2\pi y \cdot z)) \d z = 0.\]
Since $\eta$ is non-increasing, non-negative and continuous at the origin, there exists $r,c >0$ such that $\eta(|z|) \geq c$ for all $z\in B_r$. We must therefore have
\begin{align*}
\int_{B_r} 1- \cos(2\pi y\cdot z ) \d z= 0,
\end{align*}
which implies that $\cos(2 \pi y\cdot z) =1$ for all $z \in B_r$, arriving at a contradiction.

\textbf{Step 2:} By \cref{lem:near_origin} there exists $r>0$ and $b>0$ such that
\[\hat{\eta}(y) \leq  1 - b|y|^2 \ \ \text{for all } y\in B_r.\]
We now define 
\[b_R' = \min_{r \leq |y| \leq R} \frac{1- |\hat{\eta}(y)|}{|y|^2}.\]
By Part 1 we have $b_R'>0$, and by definition we have 
\[|\hat{\eta}(y)|\leq 1 - b_R'|y|^2 \ \ \text{for all } y\in B_R\setminus B_r.\]
The proof is completed by setting $b_R = \min\{b,b_R'\}$.
\end{proof}
Before we can establish our Gaussian upper bounds as well as gradient estimates for $\psi_{k,\epsilon}$, we need a quantitative version of the Riemann--Lebesgue lemma (which in its original version says that the Fourier transform of an $L^1$-function decays to zero) in order to control the decay of the Fourier transform $\hat\eta$, cf. \labelcref{eq:psik_grad}. 
\begin{lemma}[Quantitative Riemann--Lebesgue lemma]\label{lem:riemann-lebesgue}
    Let $\eta:[0,\infty)\to [0,\infty)$ be non-increasing and satisfy $\eta(t)=0$ for $t\geq 1$. 
    Then its Fourier transform, defined as $\hat\eta(y):=\int_{\R^d} \eta(|x|)e^{-2\pi i x\cdot y}\d x$ for $y\in\R^d$, satisfies
    \begin{align*}
        |\hat\eta(y)| \leq \min\left\{\frac{\omega_{d-1}}{|y|},\omega_d\right\}\eta(0),\qquad\forall y\in\R^d,
    \end{align*}
    where $\omega_k$ denotes the volume of the unit ball in $\R^k$ for $k\in\N$.
\end{lemma}
\begin{proof}   
    By definition of the Fourier transform and using the properties of $\eta$ we have 
    \begin{align*}
        |\hat\eta(y)|
        \leq
        \int_{\R^d}\eta(|x|)\d x
        =
        \omega_d\int_0^1\eta(t)dt
        \leq \omega_d\eta(0)    
    \end{align*}
    which establishes the trivial upper bound.
    To show the non-trivial one, we use a change of variables to obtain:
    \begin{align*}
        \hat\eta(y) 
        &= \int_{\R^d} \eta(|x|)e^{-2\pi i x\cdot y}\d x
        =
        \int_{\R^d} \eta\left(\left|x+\frac{y}{2|y^2|}\right|\right)e^{-2\pi i\left(x+\frac{y}{2|y|^2}\right)\cdot y}\d x
        \\
        &=
        \int_{\R^d} \eta\left(\left|x+\frac{y}{2|y|^2}\right|\right)e^{-2\pi i x\cdot y}
        e^{-\pi i}\d x
        =
        -\int_{\R^d}\eta\left(\left|x+\frac{y}{2|y|^2}\right|\right)e^{-2\pi i x\cdot y}\d x.
    \end{align*}
    This is another formula for the Fourier transform $\hat\eta$. 
    Averaging it and the original definition yields    
    \begin{align*}
        \hat\eta(y) 
        &= 
        \frac{1}{2}\int_{\R^d} \left(\eta(|x|)-\eta\left(\left|x+\frac{y}{2|y|^2}\right|\right)\right)e^{-2\pi i x\cdot y}\d x.
    \end{align*}
    Since $\hat\eta$ is by definition a radial function, it suffices to consider the case $y=t e_d$ with $t\in\R$ and without loss of generality we can assume $t\geq 0$.
    By splitting the integral over $\R^d$ into two halfspaces and using that $\eta$ is non-increasing, we obtain
    \begin{align*}
        2|\hat\eta(y)|
        &\leq
        \int_{\R^d} \left|\eta(|x|)-\eta\left(\left|x+\frac{e_d}{2t}\right|\right)\right|\d x
        \\
        &=
        \int_{\{x_d\geq 0\}}
        \left|\eta(|x|)-\eta\left(\left|x+\frac{e_d}{2t}\right|\right)\right|\d x
        +
        \int_{\{x_d< 0\}}
        \left|\eta(|x|)-\eta\left(\left|x+\frac{e_d}{2t}\right|\right)\right|\d x
        \\
        &=
        \int_{\R^{d-1}}
        \int_0^\infty
        \eta(|(\bar x,s)|)-\eta\left(\left|\left(\bar x,s+\frac{1}{2t}\right)\right|\right)
        \d s\d\bar x
        \\
        &\qquad
        +
        \int_{\R^{d-1}}
        \int_{\frac{1}{2t}}^\infty
        \eta\left(\left|\left(\bar x,s-\frac{1}{2t}\right)\right|\right)
        -
        \eta(|(\bar x,s)|)
        \d s\d\bar x
        \\
        &\qquad
        +
        \int_{\R^{d-1}}
        \int_0^{\frac{1}{2t}}
        \left|
        \eta(|(\bar x,s)|)
        -
        \eta\left(\left|\left(\bar x,s-\frac{1}{2t}\right)\right|\right)
        \right|
        \d s\d\bar x
        \\
        &=
        2\int_{\R^{d-1}}
        \int_0^\infty
        \eta(|(\bar x,s)|)-\eta\left(\left|\left(\bar x,s+\frac{1}{2t}\right)\right|\right)
        \d s\d\bar x
        \\
        &\qquad
        +
        \int_{\R^{d-1}}
        \int_0^{\frac{1}{2t}}
        \left|
        \eta(|(\bar x,s)|)
        -
        \eta\left(\left|\left(\bar x,s-\frac{1}{2t}\right)\right|\right)
        \right|
        \d s\d\bar x
\end{align*}
where we used that 
\begin{align*}
    \left|\left(\bar x, s+\frac{1}{2t}\right)\right| &\geq |(\bar x,s)|\qquad\forall\bar x\in\R^{d-1},\; s\geq 0,\\
    \left|\left(\bar x, s-\frac{1}{2t}\right)\right| &\leq |(\bar x,s)|\qquad\forall\bar x\in\R^{d-1},\; s\geq \frac{1}{2t}.
\end{align*}
We begin by estimating the first integral. 
Performing a change of variables and using that $\supp(\eta)\subset[0,1]$ we get
\begin{align*}
    &\phantom{{}={}}
    2\int_{\R^{d-1}}
    \int_0^\infty
    \eta(|(\bar x,s)|)-\eta\left(\left|\left(\bar x,s+\frac{1}{2t}\right)\right|\right)
    \d s\d\bar x
    \\
    &=
    2\int_{B^{d-1}(0,1)}
    \left[
    \int_0^\infty
    \eta(|(\bar x,s)|)
    \d s
    -
    \int_{\frac{1}{2t}}^\infty
    \eta\left(\left|\left(\bar x,s\right)\right|\right)
    \d s
    \right]
    \d\bar x 
    \\
    &=
    2\int_{B^{d-1}(0,1)}
    \int_0^{\frac{1}{2t}}
    \eta(|(\bar x,s)|)\d s\d\bar x
    \leq 
    \frac{\omega_{d-1}\eta(0)}{t}, 
\end{align*}
where $B^{d-1}(0,1)$ denotes the unit ball in $\R^{d-1}$.
The second integral is bounded as follows:
\begin{align*}
    &\phantom{{}={}}
    \int_{\R^{d-1}}
    \int_0^{\frac{1}{2t}}
    \left|
    \eta(|(\bar x,s)|)
    -
    \eta\left(\left|\left(\bar x,s-\frac{1}{2t}\right)\right|\right)
    \right|
    \d s\d\bar x
    \\
    &=
    \int_{B^{d-1}(0,1)}
    \int_0^{\frac{1}{2t}}
    \left|
    \eta(|(\bar x,s)|)
    -
    \eta\left(\left|\left(\bar x,s-\frac{1}{2t}\right)\right|\right)
    \right|
    \d s\d\bar x
    \leq 
    \frac{\omega_{d-1}\eta(0)}{t}.
\end{align*}
Combining these two estimates we obtain
\begin{align*}
    2|\hat\eta(y)|
    \leq 
    \frac{2\omega_{d-1}\eta(0)}{t}
\end{align*}
which concludes the proof.
\end{proof}

We conclude this appendix with the proofs of  \cref{prop:psi_gaussian_upper,prop:Me_grad}. 
\begin{proof}[Proof of  \cref{prop:psi_gaussian_upper}]
To prove (i), we have to prove the two estimates
\begin{equation}\label{eq:psi1}
|\psi_{k,\epsilon}(x)| \leq C_1\eps_k^{-d},
\end{equation}
and
\begin{equation}\label{eq:psi2}
|\psi_{k,\epsilon}(x)| \leq C_1\eps^{-d}\exp\left( -\frac{|x|^2}{8d\eps_k^2}\right),
\end{equation}
for all $\eps>0$, $x\in \R^d$ and $k\geq 1$.  
We first prove \labelcref{eq:psi1}. By \labelcref{eq:psi_rewrite} we can restrict our attention to $\eps=1$. Since $|\psi_1| = |\eta_\eps| \leq \eta(0)\eps^{-d}$, we may also restrict our attention to $k\geq 2$. By \cref{lem:riemann-lebesgue} there exists $R>0$ just depending on $d$ and $\eta$ such that $|\hat{\eta}(y)|\leq \frac{1}{2}$ for all $|y|\geq R$ and since by assumption $k\geq 2$ we have $|\hat{\eta}(y)|^{k-2} \leq \frac{4}{2^k}$ for $|y|\geq R$. 
By \cref{lem:away_origin} there exists $b>0$ such that $|\hat{\eta}(y)| \leq 1 - b |y|^2$ for $y\in B_R$.
Using these two inequalities as well as \labelcref{eq:psik_fourier} we have
\begin{align*}
|\psi_k(x)| &\leq \int_{\R^d}|\hat{\eta}(y)|^k \d y
% \\&
= \int_{B_R}|\hat{\eta}(y)|^k\d y + \int_{\R^d\setminus B_R}|\hat{\eta}(y)|^{k-2}|\hat{\eta}(y)|^2 \d y\\
&\leq \int_{B_R}(1-b|y|^2)^k\d y + 4\int_{\R^d\setminus B_R}\frac{1}{2^k}|\hat{\eta}(y)|^2 \d y\\
&\leq \int_{\R^d}e^{-b k |y|^2}\d y  + \frac{4}{2^k}\|\hat{\eta}\|^2_{L^2(\R^d)}
% \\&
=\left( \frac{\pi}{bk}\right)^{\frac{d}{2}} + \frac{4}{2^k}\|\eta\|_{L^2(\R^d)}^2
% \\&
\leq C k^{-\frac{d}{2}},
\end{align*}
where we used that $\int_{\R^d} e^{-a|x|^2 } \d x = \left(\frac{\pi}{a}\right)^{\frac{d}{2}}$  and $2^{-k} \leq C(d)k^{-d/2}$ in the last line. This establishes \labelcref{eq:psi1}. 

To prove \labelcref{eq:psi2}, we first note that for $|x|\leq 2\epsilon$ we have
\[|\psi_{k,\epsilon}(x)| \leq C\epsilon^{-d} \leq 2d C\epsilon^{-d}\exp\left(-\frac{1}{2dk}\right) \leq 2dC \epsilon^{-d}\exp\left( -\frac{|x|^2}{8d\epsilon_k^2}\right).\]
Assume now that $|x|\geq 2\epsilon$. For $k=1$ we have $\psi_{1,\epsilon}(x)=0$ for $|x|\geq \epsilon$, so the result is immediate. For $k\geq 2$, we use \cref{lem:Hoeffding} and the inclusion $B(x,\epsilon) \subset \{y\, : \, |y|\geq |x|-\epsilon\}$ to obtain that
\begin{align*}
\psi_{k,\epsilon}(x) &= (\eta_\eps * \psi_{k-1,\epsilon})(x)
% \\&
=\int_{B(x,\epsilon)}\eta_{\eps}(x-y)\psi_{k-1,\epsilon}(y)\d y
\\&
\leq C\epsilon^{-d}\exp\left(-\frac{(|x|-\epsilon)^2}{2d(k-1)\epsilon^2}\right)
% \\&
\leq C\epsilon^{-d}\exp\left(-\frac{|x|^2}{8dk\epsilon^2}\right)=C\epsilon^{-d}\exp\left(-\frac{|x|^2}{8d\epsilon_k^2}\right),
\end{align*}
for a constant only depending on $d$ and $\eta(0)$, where we used that $k-1\leq k$ and $\epsilon \leq \frac{1}{2}|x|$ in the last line. 
This establishes \labelcref{eq:psi2}. 

We now prove (ii), which largely follows the proof of \labelcref{eq:psi1} in part 1. 
Again, by \labelcref{eq:psi_rewrite} we may restrict our attention to $\eps=1$. 
% By \cref{lem:riemann-lebesgue} the values of $R$ and $b$ in part (i) depend only on $\eta(0)$ and $d$. 
% Let $k_0$ be the smallest integer such that $sk_0 \geq 1$, so $\frac{1}{s}\leq k_0 \leq \frac{1}{s}+1$. Then by \labelcref{eq:eta_fourier_decay} we have that $|y||\hat{\eta}(y)|^{k_0}$ is bounded on $\R^d$. Also note that for $sk - 1 > 2$, or $k > \frac{3}{s} \geq k_0$, we have
% \[\int_{\R^d}|y| |\hat{\eta}(y)|^k \d y \leq C\int_{\R^d} \frac{|y|}{(1 + |y|^s)^k}\d y < \infty.\]
First, by \cref{lem:riemann-lebesgue} we have for $k\geq 3$ that
\begin{align*}
\int_{\R^d}|y| |\hat{\eta}(y)|^k \d y &=  C + \int_{\R^d\setminus B(0,1)}|y|  |\hat{\eta}(y)|^{k-2}|\hat{\eta}(y)|^2\d y    \\
&\leq C + \left(\eta(0)\omega_{d-1}\right)^{k-2}\int_{\R^d\setminus B(0,1)}|y|^{3-k} |\hat{\eta}(y)|^2 \d y \\
&\leq C + \left(\eta(0)\omega_{d-1}\right)^{k-2}\int_{\R^d\setminus B(0,1)} |\hat{\eta}(y)|^2 \d y  
\\
&\leq 
C + \left(\eta(0)\omega_{d-1}\right)^{k-2}\|\hat\eta\|_{L^2(\R^d)} < \infty,
\end{align*}
since $\eta,\hat\eta\in L^2(\R^d)$. 
Here the constant $C$ just depends on $d$ and $\eta(0)$.
Hence, in light of \labelcref{eq:psik_grad}, $\psi_k$ is continuously differentiable for $k \geq 3$ and---using also that by \cref{lem:riemann-lebesgue} the function $y\mapsto|y||\hat{\eta}(y)|$ is bounded on $\R^d$ and that $|\hat{\eta}(y)|\leq \frac{1}{2}$ for all $|y|\geq R$---we have
\begin{align*}
\frac{1}{2\pi}|\nabla \psi_k(x)| &\leq \int_{\R^d}|y||\hat{\eta}(y)|^k \d y\\
&= \int_{B_R}|y||\hat{\eta}(y)|^k\d y + \int_{\R^d\setminus B_R}|y||\hat{\eta}(y)||\hat{\eta}(y)|^{k-3}|\hat{\eta}(y)|^2 \d y\\
&\leq \int_{B_R}|y|(1-b|y|^2)^k\d y + \frac{C}{2^{k-3}}\int_{\R^d\setminus B_R}|\hat{\eta}(y)|^2 \d y\\
&\leq \int_{\R^d}|y|e^{-b k |y|^2}\d y  + \frac{C}{2^k},
\end{align*}
where $C$ changes from line to line and depends on $d$, $\eta(0)$, and $\|\eta\|_{L^2(\R^d)}$. We now compute
\[\int_{\R^d}|y|e^{-b k |y|^2}\d y = (bk)^{-\frac{d}{2}} \int_{\R^d}(bk)^{-\frac{1}{2}}x e^{-|x|^2}\d x \leq Ck^{-\frac{d+1}{2}},\]
where $C$ depends as well on $b$ now. Inserting this above and using that $2^{-k} \leq C(d)k^{-\frac{d+1}{2}}$ completes the proof.

We now prove (iii).  Let us set 
\[R^2  = 8d^2 \eps_k^2\log(\eps^{-1}) \]
so that
\[\eps^{-d}\exp\left( -\frac{R^2}{8d\eps_k^2}\right) = 1. \]
Since $|T(x)-x| \leq \epsilon_k\leq \frac{R}{8}\leq R$, if $|x-x_0|\geq 2R$ then $|T(x)-x_0|\geq R$. Thus, by part (i) we have $|\psi_{k,\epsilon}(T(x)-x_0)|\leq C_1$ for $x\in \Omega\setminus B(x_0,2R)$. It follows that
\[\int_{\Omega \setminus B(x_0,2R)} |\psi_{k,\epsilon}(T(x)-x_0)|^p \d x \leq C|\Omega| \leq C\eps_k^{-d(p-1)}, \]
since $\eps_k \leq 1$.  For $x\in B(x_0,2R)$ we use the estimate $|\psi_{k,\epsilon}(x-x_0)|\leq C\eps_k^{-d}$ from part (i) to obtain
\[\int_{\Omega \cap B(x_0,2R)} |\psi_{k,\epsilon}(T(x)-x_0)|^p \d x \leq C|B(x_0,2R)| \eps_k^{-pd} \leq CR^d\eps_k^{-pd} = C\eps_k^{-d(p-1)}\log(\eps^{-1})^{\frac{d}{2}}. \]
Combining the two estimates above yields
\[\int_\Omega |\psi_{k,\epsilon}(T(x)-x_0)|^p \d x \leq C\eps_k^{-d(p-1)}\log(\eps^{-1})^{\frac{d}{2}}.\]

To prove (iv), let $R\geq \delta$ and use part (ii) to obtain
\begin{align*}
\|\osc_{\Omega \cap B(\cdot,\delta)} \psi_{k,\epsilon}(\cdot - x_0)\|_{L^1(\Omega)} &=\int_{\Omega}\sup_{y,z\in B(x,\delta)}|\psi_{k,\epsilon}(y-x_0) - \psi_{k,\epsilon}(z-x_0)|\,dx\\
&=\int_{\Omega\cap B(x_0,2R)}\sup_{y,z\in B(x,\delta)}|\psi_{k,\epsilon}(y-x_0) - \psi_{k,\epsilon}(z-x_0)|\,dx \\
&\hspace{0.25in}+\int_{\Omega\setminus B(x_0,2R)}\sup_{y,z\in B(x,\delta)}|\psi_{k,\epsilon}(y-x_0) - \psi_{k,\epsilon}(z-x_0)|\,dx \\
&\leq CR^d \delta\|\nabla \psi_{k,\epsilon}\|_{L^\infty(\R^d)} + 2|\Omega| \|\psi_{k,\epsilon}\|_{L^\infty(\R^d\setminus B(0,R))}\\
&\leq CR^d \delta\eps_k^{-(d+1)} + 2|\Omega| \|\psi_{k,\epsilon}\|_{L^\infty(\R^d\setminus B(0,R))}.
\end{align*}
By part (i) we have
\[\|\psi_{k,\epsilon}\|_{L^\infty(\R^d\setminus B(0,R))} \leq \eps^{-d}\exp\left( -\frac{R^2}{8d\eps_k^2}\right) = \eps^2, \]
provided $R$ is chosen so that
%\[\exp\left( -\frac{R^2}{8d\eps_k^2}\right) = \eps^d\delta \epsilon_k^{-1}\]
%\[\frac{R^2}{8d\eps_k^2} = \log(\eps^{-d}\delta^{-1} \epsilon_k)\]
\[R^2 = 8d(d+2)\eps_k^2\log(\eps^{-1}).\]
Since $R \leq C \epsilon_k \log(\eps^{-1})^{\frac{1}{2}}$, we have $R^d \leq C \eps_k^d \log(\eps^{-1})^{\frac{d}{2}}$, which completes the proof.
\end{proof}

\begin{proof}[Proof of \cref{prop:Me_grad}]
% We first write
% \[\M^k_\eps \eta^{x_0}_\eps(x) = \epsilon^{-d}\int_{\Omega}\eta\left(\frac{|x-y|}{\epsilon}\right)\hat\rho_\eps(y)^{-1}\rho(y)\M^{k-1}_\eps\eta^{x_0}_\eps(y)\d y.\]
% Differentiating in $x$ yields
% \[\nabla \M^k_\eps \eta^{x_0}_\eps(x) = \epsilon^{-(d+1)}\int_{\Omega}\eta'\left( \frac{|x-y|}{\epsilon}\right)\frac{(x-y)}{|x-y|}\hat\rho_\eps(y)^{-1}\rho(y)\M^{k-1}_\eps\eta^{x_0}_\eps(y)\d y.\]
% Therefore
% \begin{align*}
% |\nabla \M^k_\eps \eta^{x_0}_\eps(x)| &\leq C\epsilon^{-(d+1)} \|\eta'\|_{L^\infty(0,1)}|\Omega\cap B(x,\epsilon)| \|\M_\eps^{k-1}\eta^{x_0}_\epsilon\|_{L^\infty(\Omega\cap B(x,\epsilon))}\\
% &\leq C\epsilon^{-1}\|\M_\eps^{k-1}\eta^{x_0}_\epsilon\|_{L^\infty(\Omega\cap B(x,\epsilon))},
% \end{align*}
% where $C$ depends on $\rho$ and $\eta'$.
We first write
\begin{align}\label{eq:Mk_proof}
    \M^k_\eps \eta^{x_0}_\eps(x) = \epsilon^{-d}\int_{\Omega}\eta\left(\frac{|x-y|}{\epsilon}\right)\hat\rho_\eps(y)^{-1}\rho(y)\M^{k-1}_\eps\eta^{x_0}_\eps(y)\d y
\end{align}
We first show by induction that for $k\geq 2$ the function
\begin{align}\label{eq:continuous_integrand}
    y \mapsto \hat\rho_\eps(y)^{-1}\rho(y)\M^{k-1}_\eps\eta^{x_0}(y)
\end{align}
is continuous.
Note that for $k=1$ this is not necessarily the case since then $\M^{k-1}_\eps\eta^{x_0}(y)=\eta^{x_0}(y)$ can be discontinuous.

For proving the result we note that $\rho$ is continuous by assumption and hence $\hat\rho_\eps$, defined in \labelcref{eq:rhoe}, is a convolution of an integrable and a bounded function and hence also continuous. 
Furthermore, $\hat\rho_\eps$ is bounded from below by the assumption that $\rho$ is bounded from below which makes $\hat\rho_\eps^{-1}$ continuous.
Also $\M_\eps\eta^{x_0}$, defined in \labelcref{eq:Me}, is also a convolution of an integrable and a bounded function and therefore continuous. 
This proves the continuity of \labelcref{eq:continuous_integrand} for $k=2$. 
For the induction step not that if \labelcref{eq:continuous_integrand} is continuous, then \labelcref{eq:Mk_proof} is a convolution of a integrable and a bounded function and therefore continuous. 
This shows the continuity of \labelcref{eq:continuous_integrand} for $k-1\to k$.

For fixed $y\in\Omega$ let us define the Lipschitz map $x\mapsto\varphi_\eps(x;y) := y+\eps x$ with Lipschitz continuous inverse $z\mapsto\varphi_\eps^{-1}(x;y):=\frac{x-y}{\eps}$.
Since $\eta$ is non-increasing, by the coarea formula for $BV$-functions (see, e.g., \cite[Theorem 3.40]{ambrosio2000functions}) the function $\bar{\eta} := \eta(\abs{\cdot})$ lies in $BV(\R^d)$.
As a consequence \cite[Theorem 3.16]{ambrosio2000functions} implies that for every $y\in\Omega$ the function $x\mapsto f_\eps(x;y) := \eta\left(\frac{\abs{x-y}}{\eps}\right) =\bar{\eta}\circ\phi_\eps^{-1}(x;y)$ lies in $BV(\Omega)$ and by symmetry also $y\mapsto f_\eps(x;y)$ lies in $BV(\Omega)$ for every $x\in\Omega$.
By symmetry it also holds that $\d D_x f_\eps(x;y)=-\d D_y f_\eps(x,y)$ in the sense of measures.
Furthermore, \cite[Theorem 3.16]{ambrosio2000functions} provides the following symmetric bounds for the derivative
\begin{subequations}\label{eq:BV_bound_derivative}
    \begin{align}
        \abs{D_x f_\eps(\cdot;y)} 
        &\leq 
        \Lip(\varphi_\eps(\cdot;y))^{d-1} 
        \varphi_\eps(\cdot;y)_\sharp 
        \abs{D\bar\eta}
        \\
        \abs{D_y f_\eps(x;\cdot)} 
        &\leq 
        \Lip(\varphi_\eps(x;\cdot))^{d-1} 
        \varphi_\eps(x;\cdot)_\sharp 
        \abs{D\bar\eta}
\end{align}
\end{subequations}
which are to be understood as inequalities of non-negative Radon measures.
Noting also that both $f_\eps(x;\cdot)$ and therefore $D_x f_\eps(x;\cdot)$ are supported in $\overline{B(x,\eps)}$ and using the symmetry, we get from \labelcref{eq:Mk_proof,eq:BV_bound_derivative} and the continuity of \labelcref{eq:continuous_integrand} that
\begin{align*}
    \nabla \M^k_\eps \eta^{x_0}_\eps(x) 
    &= 
    \epsilon^{-d}
    \int_{\overline{B(x,\eps)}}
    \hat\rho_\eps(y)^{-1}\rho(y)\M^{k-1}_\eps\eta^{x_0}_\eps(y)\d D_x f_\eps(x;y).
    \\
    &= 
    -\epsilon^{-d}
    \int_{\overline{B(x,\eps)}}
    \hat\rho_\eps(y)^{-1}\rho(y)\M^{k-1}_\eps\eta^{x_0}_\eps(y)\d D_y f_\eps(x;y)
\end{align*}
and therefore
\begin{align*}
|\nabla \M^k_\eps \eta^{x_0}_\eps(x)| 
&\leq 
C\epsilon^{-d} 
\abs{D_y f_\eps(x;\cdot)}(\overline{B(x,\eps)})
\,
\|\M_\eps^{k-1}\eta^{x_0}_\epsilon\|_{L^\infty(\Omega\cap B(x,\epsilon))}
\\
&\leq 
C\epsilon^{-d} 
\Lip(\varphi_\eps(x;\cdot))^{d-1}
\varphi_\eps(x;\cdot)_\sharp\abs{D\bar\eta}(\overline{B(x,\eps)})
\,
\|\M_\eps^{k-1}\eta^{x_0}_\epsilon\|_{L^\infty(\Omega\cap B(x,\epsilon))}
\\
&\leq 
C\eps^{-1}
\,
\|\M_\eps^{k-1}\eta^{x_0}_\epsilon\|_{L^\infty(\Omega\cap B(x,\epsilon))},
\end{align*}
where $C$ changed its value between the lines and depends on $\rho$ and $\bar\eta$.
In the last inequality we used that $\Lip(\varphi_\eps(x;\cdot))=\eps$ and 
\begin{align*}
    \varphi_\eps(x;\cdot)_\sharp\abs{D\bar\eta}(\overline{B(x,\eps)})
    =
    \abs{D\bar\eta}(\overline{B(0,1)}).
\end{align*}
The proof is completed by invoking \cref{prop:psi_gaussian_upper} (i) and \cref{thm:Me}.
\end{proof}

\section{Proofs from \texorpdfstring{\cref{sec:combination}}{Section 5}}\label{app:monte}

\begin{proof}[Proof of  \cref{prop:lp_graph_control}]
Since $u$ is Borel measurable, the composition $u(X_i)$ is a random variable.  Note that
\[\ng{p}{u}^p = \frac{1}{n}\sum_{i=1}^n |u(x_i)|^p,\]
and
\[\bE \ng{p}{u}^p = \int_\Omega |u(x)|^p \rho \d x.\]
Let $Z_i = |u(X_i)|^p -\int_\Omega |u(x)|^p \rho \d x$. Then the random variables $Z_1,\dots,Z_n$ are \emph{i.i.d.}~with mean zero. Let $S_n = \sum_{i=1}^n Z_i$. By Markov's inequality we have 
\[\P\left(\ng{p}{u}^p \geq \int_\Omega |u(x)|^p \rho \dx + t\right)  = \P(S_n > n t) \leq \P(|S_n| \geq nt) \leq (nt)^{-q}\bE[|S_n|^q]\]
for any $t>0$ and $1 < q \leq r/p$. By the Bahr--Esseen bounds \cite{von1965inequalities} we have
\begin{align*}
\bE |S_n|^q &\leq 2\sum_{i=1}^n \bE|Z_i|^q = 2n\bE |Z_1|^q \\
&= 2n\int_\Omega\left| |u(z)|^p -\int_\Omega |u(x)|^p \rho(x) \d x\right|^q\rho(z) \d z\\
&\leq 2^{q}n\left(\int_\Omega |u(z)|^{pq}\rho(z)\d z  + \left|\int_\Omega |u(x)|^p \rho(x) \d x\right|^q\right)\\
&\leq 2^{q}n\left(\int_\Omega |u|^{pq} \rho\d z  + \int_\Omega |u|^{pq} \rho \d x\right) \leq 2^{q+1}\rho_{\max} n \|u\|_{L^{pq}(\Omega)}^{pq},
\end{align*}
where we used $(a+b)^q \leq 2^{q-1}(a^q + b^q)$ for $a,b\geq 0$ in the second inequality, and that $\rho$ is a probability density (which makes Jensen's inequality applicable) in the second to last inequality. Therefore
\[\P\left(\ng{p}{u}^p \geq \int_\Omega |u(x)|^p \rho \dx + t\right)  \leq 2^{q+1}\rho_{\max}n^{1-q}t^{-q}\|u\|_{L^{pq}(\Omega)}^{pq}.\]
Setting $t = \rho_{\max}^{\frac{1}{q}} \|u\|_{L^{pq}(\Omega)}^p$ and using Jensen's inequality to bound
\[\int_\Omega |u(x)|^p \rho \dx \leq \left(\int_\Omega |u(x)|^{pq} \rho \dx\right)^{\frac{1}{q}} \leq \rho_{\max}^{\frac{1}{q}}\|u\|_{L^{pq}(\Omega)}^p,\]
we have
\[\P\left(\ng{p}{u}^p \geq 2\rho_{\max}^{\frac{1}{q}}\|u\|_{L^{pq}(\Omega)}^p \right) \leq 2^{q+1}n^{1-q}.\]
The proof is completed by choosing $q=r/p$ and taking the $p^{\rm th}$ root of the inequality.
\end{proof}

\end{document}